\documentclass[nohyperref]{article}

\usepackage{microtype}
\usepackage{graphicx}
\usepackage{subfigure}
\usepackage{booktabs} % for professional tables

\usepackage[nohyperref,accepted]{icml2023}

\usepackage{amsmath}
\usepackage{amssymb}
\usepackage{mathtools}
\usepackage{amsthm}

\usepackage[pagebackref=true]{hyperref}  % hyperlinks
\renewcommand*{\backrefalt}[4]{%
    \ifcase #1 \footnotesize{(Not cited.)}%
    \or        \footnotesize{(Cited on page~#2)}%
    \else      \footnotesize{(Cited on pages~#2)}%
    \fi}

\usepackage[capitalize,noabbrev]{cleveref}

\theoremstyle{plain}
\newtheorem{theorem}{Theorem}[section]

\newtheorem{lemma}[theorem]{Lemma}

\theoremstyle{definition}

\newtheorem{assumption}[theorem]{Assumption}
\theoremstyle{remark}

\icmltitlerunning{High-Probability Bounds for Stochastic Optimization and Variational Inequalities: the Case of Unbounded Variance}

\usepackage{url}            % simple URL typesetting
\usepackage{booktabs}       % professional-quality tables
\usepackage{amsfonts}       % blackboard math symbols
\usepackage{nicefrac}       % compact symbols for 1/2, etc.
\usepackage{microtype}      % microtypography

\usepackage{amssymb, amsmath, amsthm, latexsym}
\usepackage{url}
\usepackage{algorithm}
\usepackage{algorithmic}
\usepackage{tabularx}
\usepackage{paralist}
\usepackage{mathtools}

\usepackage{bbm} %for indicators 1
\usepackage{wrapfig}
\usepackage{makecell}
\usepackage{multirow}
\usepackage{booktabs}

\usepackage{nicefrac}       % nice fractions

\usepackage[flushleft]{threeparttable} % http://ctan.org/pkg/threeparttable

\usepackage{caption}
\usepackage{multirow}
\usepackage{colortbl}
\definecolor{bgcolor}{rgb}{0.8,1,1}
\definecolor{bgcolor2}{rgb}{0.8,1,0.8}
\definecolor{niceblue}{rgb}{0.0,0.19,0.56}

\usepackage{hyperref}
\hypersetup{colorlinks,linkcolor={blue},citecolor={niceblue},urlcolor={blue}}

\renewcommand{\algorithmicrequire}{\textbf{Input:}}
\renewcommand{\algorithmicensure}{\textbf{Output:}}

\usepackage{pifont}
\definecolor{PineGreen}{RGB}{0,110,51}
\definecolor{BrickRed}{RGB}{143,20,2}
\newcommand{\cmark}{{\color{PineGreen}\ding{51}}}%
\newcommand{\xmark}{{\color{BrickRed}\ding{55}}}%

\usepackage{tikz-cd} %for a diagram

\usepackage{subfigure}

\newcommand{\R}{\mathbb{R}}
\newcommand{\eqdef}{\stackrel{\text{def}}{=}}

\def\<#1,#2>{\left\langle #1,#2\right\rangle}

\newcolumntype{Y}{>{\centering\arraybackslash}X}

\usepackage{xspace}

\newcommand{\algname}[1]{{\sf  #1}\xspace}

\newcommand{\circledOne}{\text{\ding{172}}}
\newcommand{\circledTwo}{\text{\ding{173}}}
\newcommand{\circledThree}{\text{\ding{174}}}
\newcommand{\circledFour}{\text{\ding{175}}}
\newcommand{\circledFive}{\text{\ding{176}}}
\newcommand{\circledSix}{\text{\ding{177}}}
\newcommand{\circledSeven}{\text{\ding{178}}}

\usepackage[colorinlistoftodos,bordercolor=orange,backgroundcolor=orange!20,linecolor=orange,textsize=scriptsize]{todonotes}

\newcommand{\Id}{\mathrm{Id}}

\newcommand{\cD}{{\cal D}}

\newcommand{\cO}{{\cal O}}
\newcommand{\cP}{{\cal P}}

\newcommand{\cX}{{\cal X}}

\newcommand{\Var}{\mathrm{Var}}

\newcommand{\EE}{\mathbb{E}}
\newcommand{\PP}{\mathbb{P}}

\newcommand{\tx}{\widetilde{x}}
\newcommand{\tX}{\widetilde{X}}

\newcommand{\tF}{\widetilde{F}}

\newcommand{\tnabla}{\widetilde{\nabla}}

\def\clip{\texttt{clip}}
\def\avg{\texttt{avg}}
\def\gap{\texttt{Gap}}

\usepackage{hyperref}
\graphicspath{{plots/}}

\usepackage{makecell}

\usepackage{accents}
\newlength{\dhatheight}

\usepackage{pgfplotstable} % loads pgfplots which loads tikz
\usetikzlibrary{automata, positioning, arrows, shapes, fit, calc, intersections}
\usepgfplotslibrary{statistics}
\include{figure}

\def\la{\langle}
\def\ra{\rangle}

\begin{document}

\twocolumn[
\icmltitle{High-Probability Bounds for Stochastic Optimization and\\ Variational Inequalities: the Case of Unbounded Variance}

% It is OKAY to include author information, even for blind
% submissions: the style file will automatically remove it for you
% unless you've provided the [accepted] option to the icml2022
% package.

% List of affiliations: The first argument should be a (short)
% identifier you will use later to specify author affiliations
% Academic affiliations should list Department, University, City, Region, Country
% Industry affiliations should list Company, City, Region, Country

% You can specify symbols, otherwise they are numbered in order.
% Ideally, you should not use this facility. Affiliations will be numbered
% in order of appearance and this is the preferred way.
% \icmlsetsymbol{equal}{*}

\begin{icmlauthorlist}
\icmlauthor{Abdurakhmon Sadiev}{kaust}
\icmlauthor{Marina Danilova}{mipt}
\icmlauthor{Eduard Gorbunov}{mbzuai}
\icmlauthor{Samuel Horv\'ath}{mbzuai}
\icmlauthor{Gauthier Gidel}{udem,cifar}
\icmlauthor{Pavel Dvurechensky}{wias}
\icmlauthor{Alexander Gasnikov}{mipt,skoltech,iitp}
\icmlauthor{Peter Richt\'arik}{kaust}
\end{icmlauthorlist}

\icmlaffiliation{kaust}{King Abdullah University of Science and Technology, KSA}
\icmlaffiliation{mipt}{Moscow Institute of Physics and Technology, Russia}
\icmlaffiliation{mbzuai}{Mohamed bin Zayed University of Artificial Intelligence, UAE}
\icmlaffiliation{udem}{Universit\'e de Montr\'eal and Mila, Canada}
\icmlaffiliation{cifar}{Canada CIFAR AI Chair}
\icmlaffiliation{wias}{Weierstrass Institute for Applied Analysis and Stochastics, Germany}
\icmlaffiliation{iitp}{Institute for Information Transmission Problems RAS, Russia}
\icmlaffiliation{skoltech}{Skolkovo Institute of Science and Technology, Russia}

\icmlcorrespondingauthor{Eduard Gorbunov}{eduard.gorbunov@mbzuai.ac.ae}

% You may provide any keywords that you
% find helpful for describing your paper; these are used to populate
% the "keywords" metadata in the PDF but will not be shown in the document
\icmlkeywords{High-probability convergence, stochastic optimization, gradient clipping, unbounded variance}

\vskip 0.3in
]

% this must go after the closing bracket ] following \twocolumn[ ...

% This command actually creates the footnote in the first column
% listing the affiliations and the copyright notice.
% The command takes one argument, which is text to display at the start of the footnote.
% The \icmlEqualContribution command is standard text for equal contribution.
% Remove it (just {}) if you do not need this facility.

\printAffiliationsAndNotice{}  % leave blank if no need to mention equal contribution
% \printAffiliationsAndNotice{\icmlEqualContribution} % otherwise use the standard text.

\begin{abstract}
During recent years the interest of optimization and machine learning communities in high-probability convergence of stochastic optimization methods has been growing. One of the main reasons for this is that high-probability complexity bounds are more accurate and less studied than in-expectation ones. However, SOTA high-probability non-asymptotic convergence results are derived under strong assumptions such as the boundedness of the gradient noise variance or of the objective's gradient itself. In this paper, we propose several algorithms with high-probability convergence results under less restrictive assumptions. In particular, we derive new high-probability convergence results under the assumption that the gradient/operator noise has bounded central $\alpha$-th moment for $\alpha \in (1,2]$ in the following setups: (i) smooth non-convex / Polyak-{\L}ojasiewicz / convex / strongly convex / quasi-strongly convex  minimization problems, (ii) Lipschitz / star-cocoercive and monotone / quasi-strongly monotone variational inequalities. These results justify the usage of the considered methods for solving problems that do not fit standard functional classes studied in stochastic optimization.
\end{abstract}

\section{Introduction}\label{sec:intro}

Training of machine learning models is usually performed via stochastic first-order optimization methods, e.g., Stochastic Gradient Descent (\algname{SGD}) \citep{robbins1951stochastic}
\begin{equation}
    x^{k+1} = x^k - \gamma \nabla f_{\xi^k}(x^k), \label{eq:SGD_update}
\end{equation}
where $\nabla f_{\xi^k}(x^k)$ represents the stochastic gradient of the objective/loss function $f$ at point $x^k$. Despite numerous empirical studies and observations validating the good performance of such methods, it is also important for the field to understand their theoretical convergence properties, e.g., under what assumptions a method converges and what the rate is. However, since the methods of interest are stochastic, one needs to specify what type of convergence is considered before moving on to further questions.

Typically, the convergence of the stochastic methods is studied only in expectation, i.e., for some performance metric\footnote{Examples of performance metrics for minimization of function $f$: $\cP(x) = f(x) - f(x^*)$, $\cP(x) = \|\nabla f(x)\|^2$, $\cP(x) = \|x - x^*\|^2$, where $x^* \in \arg\min_{x\in\R^d} f(x)$.} $\cP(x)$, upper bounds are derived for the number of iterations $K$ needed to achieve $\EE[\cP(x^K)] \leq \varepsilon$, where $x^K$ is the output of the method after $K$ steps, $\varepsilon$ is an optimization error, and $\EE[\cdot]$ is the full expectation. These bounds can be ``blind'' to some important properties like light-/heavy-tailedness of the noise distribution and, as a result, such guarantees do not accurately describe the methods' convergence in practice \citep{gorbunov2020stochastic}. In contrast, high-probability convergence guarantees are more sensitive to the noise distribution and thus are more accurate. Such results provide upper bounds for the number of iterations $K$ needed to achieve $\PP\{\cP(x^K) \leq \varepsilon\} \geq 1 - \beta$ for some confidence level $\beta \in (0,1]$, where $\PP\{\cdot\}$ denotes some probability measure determined by a setup.

% Recent works on high-probability convergence \citep{nazin2019algorithms, davis2021low, gorbunov2020stochastic, gorbunov2021near, gorbunov2022clipped, cutkosky2021high} focus mostly on the relaxing the assumptions under which these guarantees are derived. This is a very important direction with the ultimate goal of bridging the theory and practice. \emph{Our paper contributes to this line of works in two main aspects: we derive new high-probability results allowing the variance of the noise and the gradient of the objective to be unbounded.} 
With the ultimate goal of bridging the theory and practice of stochastic methods, recent works on high-probability convergence guarantees \citep{nazin2019algorithms, davis2021low, gorbunov2020stochastic, gorbunov2021near, gorbunov2022clipped, cutkosky2021high} focus on an important direction of the relaxing the assumptions under which these guarantees are derived. \emph{Our paper further extensively complements this line of works in two main aspects: for a plethora of settings, we derive new high-probability results allowing the variance of the noise and the gradient of the objective to be unbounded.}

\subsection{Technical Preliminaries}
Before we move on to the main part of the paper, we introduce the problems considered in the work and all necessary preliminaries. In particular, we consider stochastic unconstrained optimization problems 
\begin{equation}
    \min\limits_{x\in\R^d} \left\{f(x) = \EE_{\xi\sim\cD}\left[f_{\xi}(x)\right]\right\}, \label{eq:min_problem}
\end{equation}
where $\xi$ is a random variable with distribution $\cD$. Such problems often arise in machine learning, where $f_{\xi}(x)$ represents the loss function on the data sample $\xi$ \citep{shalev2014understanding}.

Another class of problems that we consider this work is unconstrained variational inequality problems (VIP), i.e., non-linear equations \citep{harker1990finite, ryu2021large}:
\begin{equation}
    \text{find } x^*\in \R^d \text{ such that } F(x^*) = 0, \label{eq:VIP}
\end{equation}
where $F(x) = \EE_{\xi\sim \cD}[F_{\xi}(x)]$. These problems arise in adversarial/game formulations of machine learning tasks \citep{goodfellow2014generative, gidel2019variational}.

\textbf{Notation.} We use standard notation: $\|x\| = \sqrt{\langle x, x\rangle}$ denotes the standard Euclidean norm in $\R^d$, $\EE_{\xi}[\cdot]$ denotes an expectation w.r.t.\ the randomness coming from random variable $\xi$, $B_{R}(x) = \{y \in \R^d\mid \|y-x\| \leq R\}$ is a ball with center at $x$ and radius $R$. We define restricted gap-function as $\gap_{R}(x) = \max_{y \in B_{R}(x^*)}\langle F(y), x - y \rangle$ -- a standard convergence criterion for monotone VIP \citep{nesterov2007dual}. Finally, $\cO(\cdot)$ hides numerical factors and $\widetilde\cO(\cdot)$ hides poly-logarithmic and numerical factors.

\textbf{Assumptions on a subset.} Although we consider unconstrained problems, our analysis does not require any assumptions to hold on the whole space. For our purposes, it is sufficient to introduce all assumptions only on some subset of $\R^d$, since we prove that the considered methods do not leave some ball around the solution or some level-set of the objective function with high probability. This allows us to consider quite large classes of problems.

\textbf{Stochastic oracle.} We assume that at given point $x$ we have an access to the unbiased stochastic oracle returning $\nabla f_{\xi}(x)$ or $F_{\xi}(x)$ that satisfy the following conditions.
\begin{assumption}\label{as:bounded_alpha_moment}
    We assume that there exist some set $Q \subseteq \R^d$ and values $\sigma \geq 0$, $\alpha \in (1,2]$ such that for all $x \in Q$
    \begin{itemize}
        \item[(i)] for problem \eqref{eq:min_problem} $\EE_{\xi\sim \cD}[\nabla f_{\xi}(x)] = \nabla f(x)$ and
    \begin{equation}
        \EE_{\xi\sim \cD}[\left\|\nabla f_{\xi}(x) - \nabla f(x)\right\|^\alpha] \leq \sigma^\alpha, \label{eq:bounded_alpha_moment_gradient}
    \end{equation}
        \item[(ii)] for problem \eqref{eq:VIP} $\EE_{\xi\sim \cD}[F_{\xi}(x)] = F(x)$ and
    \begin{equation}
        \EE_{\xi\sim \cD}[\left\|F_{\xi}(x) - F(x)\right\|^\alpha] \leq \sigma^\alpha. \label{eq:bounded_alpha_moment_operator}
    \end{equation}
    \end{itemize}
\end{assumption}
When $\alpha = 2$, the above assumption recovers the standard uniformly bounded variance assumption \citep{nemirovski2009robust, ghadimi2012optimal, ghadimi2013stochastic}. However, Assumption~\ref{as:bounded_alpha_moment} allows the variance of the estimator to be \emph{unbounded} when $\alpha \in (1,2)$, i.e., the noise can follow some heavy-tailed distribution. For example, the distribution of the gradient noise in the training of large attention models resembles L\'evy $\alpha$-stable distribution with $\alpha < 2$ \citep{zhang2020adaptive}. There exist also other versions of Assumption~\ref{as:bounded_alpha_moment}, see \citep{patel2022global}.

\textbf{Assumptions on $f$.} We start with a very mild assumption since without it, problem~\eqref{eq:min_problem} does not make sense.
\begin{assumption}\label{as:lower_boundedness}
    We assume that there exist some set $Q \subseteq \R^d$ such that $f$ is uniformly lower-bounded on $Q$: $f_* = \inf_{x \in Q}f(x) > -\infty$.
\end{assumption}

Moreover, when working with minimization problems \eqref{eq:min_problem}, we always assume smoothness of $f$.
\begin{assumption}\label{as:L_smoothness}
    We assume that there exist some set $Q \subseteq \R^d$ and constant $L > 0$ such that for all $x, y \in Q$
    \begin{eqnarray}
        \|\nabla f(x) - \nabla f(y)\| &\leq& L\|x - y\|, \label{eq:L_smoothness}\\
        \|\nabla f(x)\|^2 &\leq& 2L\left(f(x) - f_*\right), \label{eq:L_smoothness_cor_2}
    \end{eqnarray}
    where $f_* = \inf_{x \in Q}f(x) > -\infty$.
\end{assumption}
We notice here that \eqref{eq:L_smoothness_cor_2} follows from \eqref{eq:L_smoothness} for $Q = \R^d$, but in the general case, the implication is slightly more involved (see the details in Appendix~\ref{appendix:useful_facts}). When $Q$ is a compact set, the function $f$ is allowed to be non-$L$-smooth on the whole $\R^d$, which is related to local-Lipschitzness of the gradients \citep{patel2022global, patel2022gradient}.

In each particular special case, we also make \emph{one of the following assumptions} about the structured non-convexity of the objective function. The previous two assumptions hold for a very broad class of functions. The next assumption -- Polyak-{\L}ojasiewicz condition \citep{polyak1963gradient, lojasiewicz1963topological} -- narrows the class of non-convex functions.
\begin{assumption}\label{as:PL}
    We assume that there exist some set $Q \subseteq \R^d$ and constant $\mu > 0$ such that $f$ satisfies Polyak-{\L}ojasiewicz (P{\L}) condition/inequality on $Q$, i.e., for all $x \in Q$ and $x^* = \arg\min_{x\in \R^d} f(x)$
    \begin{equation}
        \|\nabla f(x)\|^2 \geq 2\mu\left(f(x) - f(x^*)\right). \label{eq:PL}
    \end{equation}
\end{assumption}
When function $f$ is $\mu$-strongly convex, it satisfies P{\L} condition. However, {P{\L}} inequality can hold even for non-convex functions. Some analogs of this assumption have been observed for over-parameterized models \citep{liu2022loss}.

We also consider another relaxation of convexity.
\begin{assumption}\label{as:QSC}
    We assume that there exist some set $Q \subseteq \R^d$ and constant $\mu \geq 0$ such that $f$ is $\mu$-quasi-strongly convex, i.e., for all $x \in Q$ and $x^* = \arg\min_{x\in \R^d} f(x)$
    \begin{equation}
        f(x^*) \geq f(x) + \langle \nabla f(x), x^* - x \rangle + \frac{\mu}{2}\|x - x^*\|^2. \label{eq:QSC}
    \end{equation}
    % When $\mu = 0$ function $f$ is called star-convex.
\end{assumption}
As P{\L} condition, this assumption holds for any $\mu$-strongly convex function but does not imply convexity. Nevertheless, for the above two assumptions, some standard deterministic methods such as Gradient Descent (\algname{GD}) converge linearly; see more details and examples in \citep{necoara2019linear}.

In the analysis of the accelerated method, we also need standard (strong) convexity.
\begin{assumption}\label{as:str_cvx}
    We assume that there exist some set $Q \subseteq \R^d$ and constant $\mu \geq 0$ such that $f$ is $\mu$-strongly convex, i.e., for all $x, y \in Q$
    \begin{equation}
        f(y) \geq f(x) + \langle \nabla f(x), y - x \rangle + \frac{\mu}{2}\|y - x\|^2. \label{eq:str_cvx}
    \end{equation}
    When $\mu = 0$ function $f$ is called convex.
\end{assumption}

\textbf{Assumptions on $F$.} In the context of solving \eqref{eq:VIP}, we assume Lipschitzness of $F$ -- a standard assumption for VIP.

\begin{assumption}\label{as:L_Lip}
    We assume that there exist some set $Q \subseteq \R^d$ and constant $L > 0$ such that for all $x, y \in Q$
    \begin{eqnarray}
        \|F(x) - F(y)\| \leq L\|x - y\|, \label{eq:L_Lip}
    \end{eqnarray}
\end{assumption}

Similarly to the case of minimization problems, we make \emph{one or two of the following assumptions} about the structured non-monotonicity of the operator $F$. The first assumption we consider is the standard monotonicity.

\begin{assumption}\label{as:monotonicity}
    We assume that there exist some set $Q \subseteq \R^d$ such that $F$ is monotone on $Q$, i.e., for all $x, y \in Q$
    \begin{equation}
        \langle F(x) - F(y), x - y \rangle \geq  0.\label{eq:monotonicity}
    \end{equation}
\end{assumption}
Monotonicity can be seen as an analog of convexity for VIP. When \eqref{eq:monotonicity} holds with $\mu\|x-y\|^2$ in the r.h.s. instead of just 0, operator $F$ is called $\mu$-strongly monotone.

Next, we consider quasi-strong monotonicity \citep{mertikopoulos2019learning, song2020optimistic, loizou2021stochastic} -- a relaxation of strong monotonicity. There exist examples of non-monotone problems such that the assumption below holds \citep[Appendix A.6]{loizou2021stochastic}.
\begin{assumption}\label{as:QSM}
    We assume that there exist some set $Q \subseteq \R^d$ and constant $\mu > 0$ such that $F$ is $\mu$-quasi strongly monotone on $Q$, i.e., for all $x \in Q$ and $x^*$ such that $F(x^*) = 0$ we have
    \begin{equation}
        \langle F(x), x - x^* \rangle \geq  \mu \|x - x^*\|^2.\label{eq:QSM}
    \end{equation}
\end{assumption}

Another structured non-monotonicity assumption that we consider in this paper is star-cocoercivity.
\begin{assumption}\label{as:star-cocoercivity}
    We assume that there exist some set $Q \subseteq \R^d$ and constant $\ell > 0$ such that $F$ is star-cocoercive on $Q$, i.e., for all $x \in Q$ and $x^*$ such that $F(x^*) = 0$
    \begin{equation}
        \|F(x)\|^2 \leq \ell\langle F(x), x - x^* \rangle.\label{eq:star-cocoercivity}
    \end{equation}
\end{assumption}
This assumption can be seen as a relaxation of the standard cocoercivity: $\|F(x) - F(y)\|^2 \leq \ell \langle F(x) - F(y), x - y \rangle$. However, unlike cocoercivity, star-cocoercivity implies neither monotonicity nor Lipschitzness of operator $F$ \citep[Appendix A.6]{loizou2021stochastic}.

\subsection{Closely Related Works and Our Contributions}

\begin{table*}[t]
    \centering
    \scriptsize
    \caption{\scriptsize Summary of known and new high-probability complexity results for solving smooth problem \eqref{eq:min_problem}. Column ``Setup'' indicates the assumptions made in addition to Assumptions~\ref{as:bounded_alpha_moment} and \ref{as:L_smoothness}. All assumptions are made only on some ball around the solution with radius $\sim R \geq \|x^0 - x^*\|$ (unless the opposite is indicated). By the complexity we mean the number of stochastic oracle calls needed for a method to guarantee that $\PP\{\text{Metric} \leq \varepsilon\} \geq 1 - \beta$ for some $\varepsilon> 0$, $\beta \in (0,1]$ and ``Metric'' is taken from the corresponding column. For simplicity, we omit numerical and logarithmic factors in the complexity bounds. Column ``$\alpha$'' shows the allowed values of $\alpha$, ``UD?'' shows whether the analysis works on unbounded domains, and ``UG?'' indicates whether the analysis works without assuming boundedness of the gradient. Notation: $L$ = Lipschitz constant; $D$ = diameter of the domain (for the result from \citep{nazin2019algorithms}); $\sigma$ = parameter from Assumption~\ref{as:bounded_alpha_moment}; $R$ = any upper bound on $\|x^0 - x^*\|$; $\mu$ = (quasi-)strong convexity/Polyak-{\L}ojasiewicz parameter; $\Delta$ = any upper bound on $f(x^0) - f_*$; $G$ = parameter such that $\EE_{\xi\sim\cD}\|\nabla f_\xi(x)\|^\alpha \leq G^\alpha$ (for the result from \citep{cutkosky2021high}). The results of this paper are highlighted in blue.
    }
    \label{tab:comparison_of_rates_minimization}
    % \vspace{-2.5mm}
    \begin{threeparttable}
        \begin{tabular}{|c|c c c c c c c|}
        \hline
        Setup & Method & Citation & Metric & Complexity & $\alpha$ & UD? & UG?\\
        \hline\hline
        \multirow{8}{*}{\makecell{As.~\ref{as:str_cvx}\\ ($\mu = 0$)}} & \algname{RSMD} & \citep{nazin2019algorithms}\tnote{\color{blue}(1)} & $f(\overline{x}^K) - f(x^*)$ & $\max\left\{\frac{LD^2}{\varepsilon}, \frac{\sigma^2 D^2}{\varepsilon^2}\right\}$ & {\color{BrickRed}$2$} & \xmark & \cmark\\
        & \algname{clipped-SGD} & \makecell{\citep{gorbunov2020stochastic}\\\citep{gorbunov2021near}} & $f(\overline{x}^K) - f(x^*)$ & $\max\left\{\frac{LR^2}{\varepsilon}, \frac{\sigma^2 R^2}{\varepsilon^2}\right\}$ & {\color{BrickRed}$2$} & \cmark & \cmark \\
        & \algname{clipped-SSTM} & \makecell{\citep{gorbunov2020stochastic}\\\citep{gorbunov2021near}} & $f(y^K) - f(x^*)$ & $\max\left\{\sqrt{\frac{LR^2}{\varepsilon}}, \frac{\sigma^2 R^2}{\varepsilon^2}\right\}$ & {\color{BrickRed}$2$} & \cmark & \cmark \\
        & \cellcolor{bgcolor}\algname{clipped-SGD} & \cellcolor{bgcolor}Theorems~\ref{thm:clipped_SGD_main_theorem} \& \ref{thm:clipped_SGD_convex_case_appendix} & \cellcolor{bgcolor}$f(\overline{x}^K) - f(x^*)$ &\cellcolor{bgcolor} $\max\left\{\frac{LR^2}{\varepsilon}, \left(\frac{\sigma R}{\varepsilon}\right)^{\frac{\alpha}{\alpha-1}}\right\}$ &\cellcolor{bgcolor} {\color{PineGreen}$(1,2]$} &\cellcolor{bgcolor} \cmark &\cellcolor{bgcolor} \cmark \\
        & \cellcolor{bgcolor} \algname{clipped-SSTM} &\cellcolor{bgcolor}Theorems~\ref{thm:clipped_SSTM_main_theorem} \& \ref{thm:clipped_SSTM_convex_case_appendix} &\cellcolor{bgcolor} $f(y^K) - f(x^*)$ &\cellcolor{bgcolor} $\max\left\{\sqrt{\frac{LR^2}{\varepsilon}}, \left(\frac{\sigma R}{\varepsilon}\right)^{\frac{\alpha}{\alpha-1}}\right\}$ & \cellcolor{bgcolor}{\color{PineGreen}$(1,2]$} &\cellcolor{bgcolor} \cmark &\cellcolor{bgcolor} \cmark \\
        \hline
        \multirow{8}{*}{\makecell{As.~\ref{as:str_cvx}\\ ($\mu > 0$)}} & \algname{restarted-RSMD} & \citep{nazin2019algorithms}\tnote{\color{blue}(1)} & $f(\overline{x}^K) - f(x^*)$ & $\max\left\{\frac{L}{\mu}, \frac{\sigma^2}{\mu\varepsilon}\right\}$ & {\color{BrickRed}$2$} & \xmark & \cmark\\
        & \algname{proxBoost} & \citep{davis2021low}\tnote{\color{blue}(1)} & $f(\overline{x}^K) - f(x^*)$ & $\max\left\{\sqrt{\frac{L}{\mu}}, \frac{\sigma^2}{\mu\varepsilon}\right\}$\tnote{\color{blue}(2)} & {\color{BrickRed}$2$} & \cmark & \cmark\\
        & \algname{R-clipped-SGD} & \makecell{\citep{gorbunov2020stochastic}\\\citep{gorbunov2021near}} & $f(\overline{x}^K) - f(x^*)$ & $\max\left\{\frac{L}{\mu}, \frac{\sigma^2}{\mu\varepsilon}\right\}$ & {\color{BrickRed}$2$} & \cmark & \cmark \\
        & \algname{R-clipped-SSTM} & \makecell{\citep{gorbunov2020stochastic}\\\citep{gorbunov2021near}} & $f(y^K) - f(x^*)$ & $\max\left\{\sqrt{\frac{L}{\mu}}, \frac{\sigma^2}{\mu\varepsilon}\right\}$ & {\color{BrickRed}$2$} & \cmark & \cmark \\
        & \cellcolor{bgcolor}\algname{R-clipped-SSTM} & \cellcolor{bgcolor}Theorems~\ref{thm:clipped_SSTM_main_theorem} \& \ref{thm:R_clipped_SSTM_main_theorem_appendix} & \cellcolor{bgcolor}$f(y^K) - f(x^*)$ & \cellcolor{bgcolor}$\max\left\{\sqrt{\frac{L}{\mu}}, \left(\frac{\sigma^2}{\mu\varepsilon}\right)^{\frac{\alpha}{2(\alpha-1)}}\right\}$ &\cellcolor{bgcolor} {\color{PineGreen}$(1,2]$} & \cellcolor{bgcolor}\cmark &\cellcolor{bgcolor} \cmark \\
        \hline
        \makecell{As.~\ref{as:QSC}\\ ($\mu > 0$)} & \cellcolor{bgcolor}\algname{clipped-SGD} & \cellcolor{bgcolor}Theorems~\ref{thm:clipped_SGD_main_theorem} \& \ref{thm:main_result_QSC_SGD} &\cellcolor{bgcolor} $\|x^K - x^*\|^2$ &\cellcolor{bgcolor} $\max\left\{\frac{L}{\mu}, \left(\frac{\sigma^2}{\mu^2\varepsilon}\right)^{\frac{\alpha}{2(\alpha-1)}}\right\}$ &\cellcolor{bgcolor} {\color{PineGreen}$(1,2]$} &\cellcolor{bgcolor} \cmark &\cellcolor{bgcolor} \cmark \\
        \hline
        \multirow{5.5}{*}{As.~\ref{as:lower_boundedness}} & \algname{MSGD} & \citep{li2020high}\tnote{\color{blue}(1)} & $\frac{1}{K+1}\sum\limits_{k=0}^K\|\nabla f(x^k)\|^2$ & $\max\left\{\frac{L^2 \Delta^2}{\varepsilon}, \frac{\sigma^4}{\varepsilon^2}\right\}$ & {\color{BrickRed}\xmark}\tnote{\color{blue}(3)} & \cmark & \cmark\\
        & \algname{clipped-NMSGD} & \citep{cutkosky2021high}\tnote{\color{blue}(1)} & $\left(\frac{1}{K+1}\sum\limits_{k=0}^K\|\nabla f(x^k)\|\right)^2$ \tnote{\color{blue}(4)} & $\left(\frac{G^2}{\varepsilon}\right)^{\frac{3\alpha-2}{2\alpha-2}}$ & {\color{PineGreen}$(1,2]$} & \cmark & \xmark\\
        &\cellcolor{bgcolor} \algname{clipped-SGD} &\cellcolor{bgcolor} Theorems~\ref{thm:clipped_SGD_main_theorem} \& \ref{thm:clipped_SGD_non_convex_main} \tnote{\color{blue}(5)}&\cellcolor{bgcolor} $\frac{1}{K+1}\sum\limits_{k=0}^K\|\nabla f(x^k)\|^2$ &\cellcolor{bgcolor} $\max\left\{\frac{L \Delta}{\varepsilon}, \left(\frac{\sqrt{L\Delta} \sigma}{\varepsilon}\right)^{\frac{\alpha}{\alpha-1}}\right\}$ &\cellcolor{bgcolor} {\color{PineGreen}$(1,2]$} &\cellcolor{bgcolor} \cmark & \cellcolor{bgcolor}\cmark \\
        \hline
        As.~\ref{as:PL} & \cellcolor{bgcolor}\algname{clipped-SGD} & \cellcolor{bgcolor}Theorems~\ref{thm:clipped_SGD_main_theorem} \& \ref{thm:clipped_SGD_PL_main} \tnote{\color{blue}(5)}& \cellcolor{bgcolor}$f(x^K) - f(x^*)$ & \cellcolor{bgcolor}$\max\left\{\frac{L}{\mu}, \left(\frac{L\sigma^2}{\mu^2\varepsilon}\right)^{\frac{\alpha}{2(\alpha-1)}}\right\}$ &\cellcolor{bgcolor} {\color{PineGreen}$(1,2]$} &\cellcolor{bgcolor} \cmark &\cellcolor{bgcolor} \cmark \\
        \hline
    \end{tabular}
    \begin{tablenotes}
        {\scriptsize \item [{\color{blue}(1)}] All assumptions are made on the whole domain.
        \item [{\color{blue}(2)}] Complexity has extra logarithmic factor of $\ln(\nicefrac{L}{\mu})$.
        \item [{\color{blue}(3)}] \citet{li2020high} assume that the noise is sub-Gaussian: $\EE\left[\exp\left(\nicefrac{\|\nabla f_{\xi}(x) - \nabla f(x)\|^2}{\sigma^2}\right)\right] \leq \exp(1)$ for all $x$ from the domain.
        \item [{\color{blue}(4)}] We notice that $\left(\frac{1}{K+1}\sum_{k=0}^K\|\nabla f(x^k)\|\right)^2 \leq \frac{1}{K+1}\sum_{k=0}^K\|\nabla f(x^k)\|^2$ and in the worst case the left-hand side is $K+1$ times smaller than the right-hand side.
        \item [{\color{blue}(5)}] All assumptions are made on the level set $Q = \{x \in \R^d\mid \exists y\in \R^d:\; f(y) \leq f_* + 2\Delta \text{ and } \|x-y\|\leq \nicefrac{\sqrt{\Delta}}{20\sqrt{L}}\}$.
        }
    \end{tablenotes}
    \end{threeparttable}
    % \vspace{-5mm}
\end{table*}

\begin{table*}[t]
    \centering
    \scriptsize
    \caption{\scriptsize Summary of known and new high-probability complexity results for solving \eqref{eq:VIP}. Column ``Setup'' indicates the assumptions made in addition to Assumption~\ref{as:bounded_alpha_moment}. All assumptions are made only on some ball around the solution with radius $\sim R \geq \|x^0 - x^*\|$ (unless the opposite is indicated). By the complexity we mean the number of stochastic oracle calls needed for a method to guarantee that $\PP\{\text{Metric} \leq \varepsilon\} \geq 1 - \beta$ for some $\varepsilon> 0$, $\beta \in (0,1]$ and ``Metric'' is taken from the corresponding column. For simplicity, we omit numerical and logarithmic factors in the complexity bounds. Column ``$\alpha$'' shows the allowed values of $\alpha$, ``UD?'' shows whether the analysis works on unbounded domains, and ``UG?'' indicates whether the analysis works without assuming boundedness of the gradient. Notation: $\tx^K_{\text{avg}} = \frac{1}{K+1}\sum_{k=0}^K \tx^k$ (for \algname{clipped-SEG}), $x^K_{\avg} = \frac{1}{K+1}\sum_{k=0}^K x^k$ (for \algname{clipped-SGDA}); $L$ = Lipschitz constant; $D$ = diameter of the domain (used in \citep{juditsky2011solving}); $\gap_{D}(x) = \max_{y \in \cX}\langle F(y), x - y\rangle$, where $\cX$ is a bounded domain with diameter $D$ where the problem is defined (used in \citep{juditsky2011solving}); $D$ = diameter of the domain (for the result from \citep{juditsky2011solving}); $\sigma$ = parameter from Assumption~\ref{as:bounded_alpha_moment}; $R$ = any upper bound on $\|x^0 - x^*\|$; $\mu$ = quasi-strong monotonicity parameter; $\ell$ = star-cocoercivity parameter. The results of this paper are highlighted in blue.}
    \label{tab:comparison_of_rates_VIP}
    % \vspace{-2.5mm}
    \begin{threeparttable}
        \begin{tabular}{|c|c c c c c c c|}
        \hline
        Setup & Method & Citation & Metric & Complexity & $\alpha$ & UD? & UG?\\
        \hline\hline
        \multirow{4}{*}{As.~\ref{as:L_Lip} \& \ref{as:monotonicity}} & \algname{Mirror-Prox} & \citep{juditsky2011solving}\tnote{\color{blue}(1)} & $\gap_{D}(\tx_{\avg}^K)$ & $\max\left\{\frac{LD^2}{\varepsilon}, \frac{\sigma^2 D^2}{\varepsilon^2}\right\}$ & \xmark\tnote{\color{blue}(2)} & \xmark & \cmark\\
        & \algname{clipped-SEG} & \citep{gorbunov2022clipped} & $\gap_{R}(\tx_{\avg}^K)$ & $\max\left\{\frac{LR^2}{\varepsilon}, \frac{\sigma^2 R^2}{\varepsilon^2}\right\}$ & {\color{BrickRed}$2$} & \cmark & \cmark \\
        & \cellcolor{bgcolor}\algname{clipped-SEG} & \cellcolor{bgcolor}Theorems~\ref{thm:clipped_SEG_main_theorem} \& \ref{thm:main_result_gap_SEG} & \cellcolor{bgcolor}$\gap_{R}(\tx_{\avg}^K)$ &\cellcolor{bgcolor} $\max\left\{\frac{LR^2}{\varepsilon}, \left(\frac{\sigma R}{\varepsilon}\right)^{\frac{\alpha}{\alpha-1}}\right\}$ &\cellcolor{bgcolor} {\color{PineGreen}$(1,2]$} &\cellcolor{bgcolor} \cmark &\cellcolor{bgcolor} \cmark \\
        \hline
        \multirow{2.5}{*}{As.~\ref{as:L_Lip} \& \ref{as:QSM}} & \algname{clipped-SEG} & \citep{gorbunov2022clipped} & $\|x^k - x^*\|^2$ & $\max\left\{\frac{L}{\mu}, \frac{\sigma^2}{\mu^2\varepsilon}\right\}$ & {\color{BrickRed}$2$} & \cmark & \cmark\\
        & \cellcolor{bgcolor}\algname{clipped-SEG} & \cellcolor{bgcolor}Theorems~\ref{thm:clipped_SEG_main_theorem} \& \ref{thm:main_result_str_mon_SEG}& \cellcolor{bgcolor}$\|x^k - x^*\|^2$ &\cellcolor{bgcolor} $\max\left\{\frac{L}{\mu}, \left(\frac{\sigma^2}{\mu^2\varepsilon}\right)^{\frac{\alpha}{2(\alpha-1})}\right\}$ &\cellcolor{bgcolor} {\color{PineGreen}$(1,2]$} &\cellcolor{bgcolor} \cmark &\cellcolor{bgcolor} \cmark \\
        \hline\hline
        \multirow{2.5}{*}{As.~\ref{as:monotonicity} \& \ref{as:star-cocoercivity}} & \algname{clipped-SGDA} & \citep{gorbunov2022clipped} & $\gap_{R}(x_{\avg}^K)$ & $\max\left\{\frac{\ell R^2}{\varepsilon}, \frac{\sigma^2R^2}{\varepsilon^2}\right\}$ & {\color{BrickRed}$2$} & \cmark & \cmark\\
        & \cellcolor{bgcolor}\algname{clipped-SGDA} & \cellcolor{bgcolor}Theorems~\ref{thm:clipped_SGDA_main_theorem} \& \ref{thm:main_result_gap_SGDA} & \cellcolor{bgcolor}$\gap_{R}(x_{\avg}^K)$ &\cellcolor{bgcolor} $\max\left\{\frac{\ell R^2}{\varepsilon}, \left(\frac{\sigma R}{\varepsilon}\right)^{\frac{\alpha}{\alpha-1}}\right\}$ &\cellcolor{bgcolor} {\color{PineGreen}$(1,2]$} &\cellcolor{bgcolor} \cmark &\cellcolor{bgcolor} \cmark \\
        \hline
        \multirow{3}{*}{As.~\ref{as:star-cocoercivity}} & \algname{clipped-SGDA} & \citep{gorbunov2022clipped} & $\frac{1}{K+1}\sum\limits_{k=0}^{K}\|F(x^k)\|^2$ & $\max\left\{\frac{\ell^2 R^2}{\varepsilon}, \frac{\ell^2\sigma^2R^2}{\varepsilon^2}\right\}$ & {\color{BrickRed}$2$} & \cmark & \cmark\\
        & \cellcolor{bgcolor}\algname{clipped-SGDA} & \cellcolor{bgcolor}Theorems~\ref{thm:clipped_SGDA_main_theorem} \& \ref{thm:main_result_non_mon_SGDA} & \cellcolor{bgcolor}$\frac{1}{K+1}\sum\limits_{k=0}^{K}\|F(x^k)\|^2$ &\cellcolor{bgcolor} $\max\left\{\frac{\ell^2 R^2}{\varepsilon}, \left(\frac{\ell \sigma R}{\varepsilon}\right)^{\frac{\alpha}{\alpha-1}}\right\}$ &\cellcolor{bgcolor} {\color{PineGreen}$(1,2]$} &\cellcolor{bgcolor} \cmark &\cellcolor{bgcolor} \cmark \\
        \hline
        \multirow{2.5}{*}{As.~\ref{as:QSM} \& \ref{as:star-cocoercivity}} & \algname{clipped-SGDA} & \citep{gorbunov2022clipped} & $\|x^K - x^*\|^2$ & $\max\left\{\frac{\ell}{\mu}, \frac{\sigma^2}{\mu^2\varepsilon}\right\}$ & {\color{BrickRed}$2$} & \cmark & \cmark\\
        & \cellcolor{bgcolor}\algname{clipped-SGDA} & \cellcolor{bgcolor}Theorems~\ref{thm:clipped_SGDA_main_theorem} \& \ref{thm:main_result_str_mon_SGDA} & \cellcolor{bgcolor}$\|x^K - x^*\|^2$ &\cellcolor{bgcolor} $\max\left\{\frac{\ell}{\mu}, \left(\frac{\sigma^2}{\mu^2\varepsilon}\right)^{\frac{\alpha}{2(\alpha-1})}\right\}$ &\cellcolor{bgcolor} {\color{PineGreen}$(1,2]$} &\cellcolor{bgcolor} \cmark &\cellcolor{bgcolor} \cmark \\
        \hline
    \end{tabular}
    \begin{tablenotes}
        {\scriptsize \item [{\color{blue}(1)}] All assumptions are made on the whole domain.
        \item [{\color{blue}(2)}] \citet{juditsky2011solving} assume that the noise is sub-Gaussian: $\EE\left[\exp\left(\nicefrac{\|F_{\xi}(x) - F(x)\|^2}{\sigma^2}\right)\right] \leq \exp(1)$ for all $x$ from the domain.
        }
    \end{tablenotes}
    \end{threeparttable}
    % \vspace{-6mm}
\end{table*}

In this subsection, we overview closely related works and describe the contributions of our work. Additional related works are discussed in Appendix~\ref{appendix:additional_related}.

\textbf{Convex optimization and monotone VIPs.} Classical high-probability results for (strongly) convex minimization \citep{nemirovski2009robust, ghadimi2012optimal} and monotone VIP \citep{juditsky2011solving} are derived under the so-called light-tails assumption, meaning that the noise in the stochastic gradients/operators is assumed to be sub-Gaussian: $\EE_{\xi \sim \cD}[\exp(\nicefrac{\|\nabla f_\xi(x) - \nabla f(x)\|^2}{\sigma^2})] \leq \exp(1)$ or $\EE_{\xi \sim \cD}[\exp(\nicefrac{\|F_\xi(x) - F(x)\|^2}{\sigma^2})] \leq \exp(1)$. In these settings, optimal (up to logarithmic factors) rates of convergence are derived in the mentioned papers.

The first high-probability results with logarithmic dependence\footnote{Note that from in-expectation convergence guarantee, one can always get a high-probability one using Markov's inequality. For example, under bounded variance, smoothness, and strong convexity assumptions \algname{SGD} achieves $\EE\|x^k - x^*\|^2 \leq \varepsilon$ after $k = \widetilde\cO(\max\{\nicefrac{L}{\mu}, \nicefrac{\sigma^2}{\mu\varepsilon}\})$ iterations. Therefore, taking $k$ such that $\EE\|x^k - x^*\|^2 \leq \varepsilon\beta$ we get from Markov's inequality that $\PP\{\|x^k - x^*\|^2 \leq \varepsilon\} \leq \beta$. However, in this case, we get bound $k = \widetilde\cO(\max\{\nicefrac{L}{\mu}, \nicefrac{\sigma^2}{\mu\varepsilon\beta}\})$, having undesirable inverse-power dependence on $\beta$.} on $\nicefrac{1}{\beta}$ under just bounded variance assumption are given by \citet{nazin2019algorithms}, where the authors show non-accelerated rates of convergence for a version of Mirror Descent with a special truncation operator for smooth convex and strongly convex problems defined on the bounded sets. Then, \citet{davis2021low} derive accelerated rates in the strongly convex case using robust distance estimation techniques. \citet{gorbunov2020stochastic, gorbunov2021near} propose an accelerated method with clipping for unconstrained (strongly) convex problems with Lipschitz / H\"older continuous gradients and derive the first high-probability results for \algname{clipped-SGD}. In the context of VIP, \citet{gorbunov2022clipped} derive the first high-probability results for the stochastic methods for solving VIP under bounded variance assumption and different assumptions on structured non-monotonicity.

However, there are no high-probability results (with logarithmic dependence on the confidence level) for smooth (strongly) convex minimization problems and Lipschitz VIP without imposing bounded variance assumption. Only recently, \citet{zhang2022parameter} derived optimal regret-bounds under Assumption~\ref{as:bounded_alpha_moment} in the convex case with bounded gradients on $\R^d$. However, the bounded gradients assumption is quite restrictive when assumed on the whole space. Thus, a noticeable gap in the stochastic optimization literature remains.

\textbf{\emph{Contribution.}} We obtain new high-probability convergence results under Assumption~\ref{as:bounded_alpha_moment} for smooth convex minimization problems and Lipschitz VIP; see the summary in Tables~\ref{tab:comparison_of_rates_minimization} and \ref{tab:comparison_of_rates_VIP}. In particular, for Clipped Stochastic Similar Triangles Method (\algname{clipped-SSTM}) \citep{gorbunov2020stochastic} and its restarted version, we derive high-probability convergence results  for smooth convex and strongly convex problems. The high-probability complexity in the strongly convex case matches (up to logarithmic factors) the known in-expectation lower bound \citep{zhang2020adaptive} and deterministic lower bound \citep{nemirovskij1983problem}. In other words, we derive the first optimal high-probability complexity results for smooth strongly convex optimization. Noticeably, the derived results have clear separation between accelerated part and stochastic part that emphasizes a potential of \algname{clipped-SSTM} for efficient parallelization. Next, we derive high-probability results for \algname{clipped-SGD} for smooth star-convex and quasi-strongly convex objectives under Assumption~\ref{as:bounded_alpha_moment}. Finally, under the same assumption, we prove the high-probability convergence of Clipped Stochastic Extragradient (\algname{clipped-SEG}) \citep{korpelevich1976extragradient, juditsky2011solving, gorbunov2022clipped} for Lipschitz monotone and quasi-strongly monotone VIP and also obtain high-probability results for Clipped Stochastic Gradient Descent-Ascent \algname{(clipped-SGDA)} for star-cocoercive and monotone / quasi-strongly monotone VIP. In the special case of $\alpha = 2$, our analysis recovers SOTA high-probability results under bounded variance assumption.

\textbf{Non-convex optimization.} Under the light-tails and smoothness assumption \citet{li2020high} derive high-probability convergence rates to the first-order stationary point for \algname{SGD}. These rates match the known in-expectation guarantees for \algname{SGD} and are optimal up to logarithmic factors \citep{arjevani2022lower}. Recently, \citet{cutkosky2021high} derived the first high-probability results for non-convex optimization under Assumption~\ref{as:bounded_alpha_moment} for a version of \algname{SGD} with gradient clipping and normalization of the momentum. The results are obtained for the non-standard metric -- $\frac{1}{K+1}\sum_{k=0}^K\|\nabla f(x^k)\|$ -- and match in-expectation lower bound for the expected (non-squared) norm of the gradient from \citep{zhang2020adaptive}. However, \citet{cutkosky2021high} make an additional assumption that the norm of the gradient is bounded\footnote{More precisely, instead of Assumption~\ref{as:bounded_alpha_moment}, \citet{cutkosky2021high} assume $\EE_{\xi\sim}\|\nabla f_\xi(x)\|^\alpha \leq G^\alpha$ for some $G > 0$. This assumption implies Assumption~\ref{as:bounded_alpha_moment} and boundedness of $\|\nabla f(x)\|$.} on $\R^d$, which is quite restrictive.

\textbf{\emph{Contribution.}} We derive the first high-probability result with logarithmic dependence on the confidence level for finding first-order stationary points of smooth (possibly, non-convex) functions without bounded gradients assumption. The result is derived for simple \algname{clipped-SGD}. Moreover, we extend the analysis to the functions satisfying Polyak-{\L}ojasiewicz condition; see Table~\ref{tab:comparison_of_rates_minimization} for the summary.

\textbf{Gradient clipping} received a lot of attention in the machine learning community due to its successful empirical applications in the training of deep neural networks \citep{pascanu2013difficulty, goodfellow2016deep}. The clipping operator is defined as $\clip(x,\lambda) = \min\left\{1 , \nicefrac{\lambda}{\|x\|}\right\}x$ ($\clip(x,\lambda) = 0$, when $x = 0$). 
% \begin{equation}
%     \clip(x,\lambda) = \begin{cases}\min\left\{1 , \frac{\lambda}{\|x\|}\right\}x, & \text{if } x \neq 0,\\ 0,& \text{otherwise.} \end{cases} \label{eq:clipping_operator}
% \end{equation}
From the theoretical perspective, gradient clipping is used for multiple different purposes: to handle structured non-smoothness in the objective function \citep{zhang2020gradient}, to robustify aggregation \citep{karimireddy2021learning} and to provide privacy guarantees \citep{abadi2016deep} in the distributed training. Moreover, as we already mentioned before, gradient clipping is used to handle heavy-tailed noise (satisfying Assumption~\ref{as:bounded_alpha_moment}) in the stochastic gradients \citep{zhang2020adaptive} and, in particular, to derive better high-probability guarantees under bounded variance assumption \citep{nazin2019algorithms, gorbunov2020stochastic}. However, there are no results showing the necessity of modifying standard methods like \algname{SGD} and its accelerated variants to achieve high-probability convergence with logarithmic dependence on the confidence level under bounded variance assumption.

\textbf{\emph{Contribution.}} We construct an example of a strongly convex smooth problem and stochastic oracle with bounded variance such that to achieve $\PP\{\|x^k - x^*\|^2 > \varepsilon\} \leq \beta$ \algname{SGD} requires $\Omega\left(\nicefrac{\sigma^2}{\mu\sqrt{\varepsilon\beta}}\right)$ iterations, i.e., the algorithm has inverse-power dependence on the confidence level. This justifies the importance of using some non-linearity such as gradient clipping to achieve logarithmic dependence on the confidence level even in the bounded variance case.

\section{Failure of Standard \algname{SGD}}
It is known that \algname{SGD} $x^{k+1} = x^k - \gamma \nabla f_{\xi^k}(x^k)$ can diverge in expectation, when Assumption~\ref{as:bounded_alpha_moment} is satisfied with $\alpha < 2$ \citep[Remark 1]{zhang2020adaptive}. However, it does converge in expectation when $\alpha = 2$, i.e., when the variance is bounded. In contrast, there are no high-probability convergence results for \algname{SGD} having logarithmic dependence on $\nicefrac{1}{\beta}$. The next theorem establishes the impossibility of deriving such high-probability results.

\begin{theorem}\label{thm:SGD_high-prob_conv}
    For any $\varepsilon > 0$ and sufficiently small $\beta \in (0,1)$ there exist problem \eqref{eq:min_problem} such that Assumptions~\ref{as:bounded_alpha_moment}, \ref{as:L_smoothness}, and \ref{as:str_cvx} hold with $Q = \R^d$, $\alpha = 2$, $0 < \mu \leq L$ and for the iterates produced by \algname{SGD} with any stepsize $\gamma > 0$
    \begin{equation}
        \PP\left\{\|x^k - x^*\|^2 \geq \varepsilon\right\} \leq \beta \;\; \Longrightarrow\;\; k = \Omega\left(\frac{\sigma}{\mu\sqrt{\varepsilon\beta}}\right). \notag
    \end{equation}
\end{theorem}

The proof is deferred to Appendix~\ref{appendix:failure_of_SGD}. We believe that similar examples can be constructed for any stochastic first-order methods having linear dependence on the stochastic gradients in their update rules. Thus, Theorem~\ref{thm:SGD_high-prob_conv} motivates the use of non-linear operators such as gradient clipping in stochastic methods to achieve logarithmic dependence on the confidence level in the high-probability bounds.

\section{Main Results for Minimization Problems}

% In this section, we present the main results for solving \eqref{eq:min_problem}.

\subsection{\algname{SGD} with Clipping}
We start with \algname{clipped-SGD}:
\begin{gather}
    x^{k+1} = x^k - \gamma \cdot \clip\left(\nabla f_{\xi^k}(x^k), \lambda_k\right), \label{eq:clipped_SGD_update}
\end{gather}
where $\xi^k$ is sampled from $\cD_k$ independently from previous steps. We emphasize here and below that distribution of the noise is allowed to be dependent on $k$: we require just independence of $\xi^k$ from the the previous steps. Our main convergence results for \algname{clipped-SGD} are summarized in the following theorem.

\begin{theorem}[Convergence of \algname{clipped-SGD}]\label{thm:clipped_SGD_main_theorem}
{\color{white}}
Let $k\geq 0$ and $\beta \in (0,1]$ are such that $A = \ln \frac{4(K+1)}{\beta} \geq 1$.\newline
\textbf{Case 1.} Let Assumptions~\ref{as:bounded_alpha_moment}, \ref{as:lower_boundedness}, \ref{as:L_smoothness} hold for $Q = \{x \in \R^d\mid \exists y\in \R^d:\; f(y) \leq f_* + 2\Delta \text{ and } \|x-y\|\leq \nicefrac{\sqrt{\Delta}}{20\sqrt{L}}\}$, $\Delta \geq f(x^0) - f_*$ and $0 < \gamma \leq \cO\left(\min\{\nicefrac{1}{LA}, \nicefrac{\sqrt{\Delta}}{\sigma \sqrt{L}K^{\nicefrac{1}{\alpha}}A^{\nicefrac{(\alpha-1)}{\alpha}}} \}\right)$, $\lambda_k = \lambda = \Theta(\nicefrac{\sqrt{\Delta}}{\sqrt{L}\gamma A})$.\newline
\textbf{Case 2.} Let Assumptions~\ref{as:bounded_alpha_moment}, \ref{as:L_smoothness}, \ref{as:PL} hold for $Q = \{x \in \R^d\mid \exists y\in \R^d:\; f(y) \leq f_* + 2\Delta \text{ and } \|x-y\|\leq \nicefrac{\sqrt{\Delta}}{20\sqrt{L}}\}$, $\Delta \geq f(x^0) - f_*$ and $0 < \gamma =  \cO\left(\min\{\nicefrac{1}{LA}, \nicefrac{\ln(B_K)}{\mu(K+1)}\}\right)$, $B_K = \Theta\left(\max\{2, \nicefrac{(K+1)^{\nicefrac{2(\alpha-1)}{\alpha}}\mu^2 \Delta}{L\sigma^2 A^{\nicefrac{2(\alpha-1)}{\alpha}}\ln^2(B_K)}\}\right)$, $\lambda_k = \Theta(\nicefrac{\exp(-\gamma\mu(1+\nicefrac{k}{2}))\sqrt{\Delta}}{\sqrt{L}\gamma A})$.\newline
\textbf{Case 3.} Let Assumptions~\ref{as:bounded_alpha_moment}, \ref{as:L_smoothness}, \ref{as:str_cvx} with $\mu = 0$ hold for $Q = B_{3R}(x^*)$, $R \geq \|x^0 - x^*\|$ and $0 < \gamma \leq \cO(\min\{\nicefrac{1}{LA}, \nicefrac{R}{\sigma K^{\nicefrac{1}{\alpha}}A^{\nicefrac{(\alpha-1)}{\alpha}}}\})$, $\lambda_k = \lambda = \Theta(\nicefrac{R}{\gamma A})$.\newline
\textbf{Case 4.} Let Assumptions~\ref{as:bounded_alpha_moment}, \ref{as:L_smoothness}, \ref{as:QSC} with $\mu > 0$ hold for $Q = B_{3R}(x^*)$, $R \geq \|x^0 - x^*\|$ and $0 < \gamma =  \cO\left(\min\{\nicefrac{1}{LA}, \nicefrac{\ln(B_K)}{\mu(K+1)}\}\right)$, $B_K = \Theta\left(\max\{2, \nicefrac{(K+1)^{\nicefrac{2(\alpha-1)}{\alpha}}\mu^2 R^2}{\sigma^2 A^{\nicefrac{2(\alpha-1)}{\alpha}}\ln^2(B_K)}\}\right)$, $\lambda_k = \Theta(\nicefrac{\exp(-\gamma\mu(1+\nicefrac{k}{2}))R}{\gamma A})$.\newline
Then to guarantee $\frac{1}{K+1}\sum_{k=0}^k \|\nabla f(x^k)\|^2 \leq \varepsilon$ in \textbf{Case 1}, $f(x^K) - f(x^*) \leq \varepsilon$ in \textbf{Case 2}, $f(\bar{x}^K) - f(x^*) \leq \varepsilon$ in \textbf{Case 3} with $\bar{x}^K = \frac{1}{K+1}\sum_{k=0}^K x^k$, $\|x^K - x^*\|^2 \leq \varepsilon$ in \textbf{Case 4} with probability $\geq 1 - \beta$ \algname{clipped-SGD} requires
\begin{align}
    \text{{\bf \emph{Case 1}}:}&\quad \widetilde\cO\left(\max\left\{\frac{L \Delta}{\varepsilon}, \left(\frac{\sqrt{L\Delta} \sigma}{\varepsilon}\right)^{\frac{\alpha}{\alpha-1}}\right\}\right) \label{eq:SGD_main_result_non_cvx}\\
   \text{{\bf \emph{Case 2}}:}&\quad \widetilde\cO\left(\max\left\{\frac{L}{\mu}, \left(\frac{L\sigma^2}{\mu^2\varepsilon}\right)^{\frac{\alpha}{2(\alpha-1)}}\right\}\right) \label{eq:SGD_main_result_PL}\\
   \text{{\bf \emph{Case 3}}:}&\quad \widetilde\cO\left(\max\left\{\frac{LR^2}{\varepsilon}, \left(\frac{\sigma R}{\varepsilon}\right)^{\frac{\alpha}{\alpha-1}}\right\}\right) \label{eq:SGD_main_result_cvx}\\
    \text{{\bf \emph{Case 4}}:}&\quad \widetilde\cO\left(\max\left\{\frac{L}{\mu}, \left(\frac{\sigma^2}{\mu^2\varepsilon}\right)^{\frac{\alpha}{2(\alpha-1)}}\right\}\right) \label{eq:SGD_main_result_QSC}
\end{align}
oracle calls.
\end{theorem}

The complete formulation of the result and full proofs are deferred to Appendix~\ref{appendix:clipped_SGD}. As one can see from Table~\ref{tab:comparison_of_rates_minimization}, for $\alpha = 2$ the derived complexity bounds match the best-known ones for \algname{clipped-SGD} in the setups where it was analyzed. Next, we emphasize that the second term under the maximum in \eqref{eq:SGD_main_result_QSC} (quasi-strongly convex functions) is optimal up to logarithmic factors \citep{zhang2020adaptive}. In the convex case, there are no lower bounds, but we conjecture that the second term in \eqref{eq:SGD_main_result_cvx} is optimal (up to logarithms) in this case as well.

Next, in the case of P{\L}-functions, we are not aware of any high-probability convergence results in the literature. In the special case of $\alpha = 2$, the derived complexity bound \eqref{eq:SGD_main_result_PL} matches the best-known in-expectation complexity bound for \algname{SGD} \citep{karimi2016linear, khaled2020better} and the first term coincides (up to logarithms) with the lower bound for deterministic first-order methods in this setup \citep{yue2022lower}.

Finally, in the non-convex case, bound \eqref{eq:SGD_main_result_non_cvx} is the first high-probability result under Assumption~\ref{as:bounded_alpha_moment} without the additional assumption of the boundedness of the gradients. For $\alpha = 2$ it matches (up to logarithms) in-expectation lower bound \citep{arjevani2022lower}. However, when $\alpha < 2$, bound \eqref{eq:SGD_main_result_non_cvx} is inferior to the existing one $\widetilde\cO\left(\left(\nicefrac{G^2}{\varepsilon}\right)^{\nicefrac{(3\alpha-2)}{2(\alpha-1)}}\right)$ by \citet{cutkosky2021high}, which relies on the stronger assumption that $\EE_{\xi\sim \cD}\|\nabla f_{\xi}(x)\|^\alpha \leq G^\alpha$ for some $G > 0$ and all $x\in\R^d$, and also do not match the lower bound by \citet{zhang2020adaptive} derived for $\EE\|\nabla f(x^k)\|$, where $x^k$ is the output of the stochastic first-order method. It is also worth mentioning that \citet{cutkosky2021high} use a different performance metric: $\hat\cP_K = \left(\frac{1}{K+1}\sum_{k=0}^K\|\nabla f(x^k)\|\right)^2$. This metric is always smaller than $\cP_K =\frac{1}{K+1}\sum_{k=0}^K\|\nabla f(x^k)\|^2$, which we use in our result. In the worst case, $\cP_K$ can be $K+1$ times larger than $\hat\cP_K$. Moreover, the lower bound from \citep{zhang2020adaptive} is derived for $\EE\|\nabla f(x^k)\|$ that is also always smaller than the standard quantity of interest $\EE\|\nabla f(x^k)\|^2$. Therefore, the question of optimality of the bound \eqref{eq:SGD_main_result_non_cvx} remains open for $\alpha < 2$. Moreover, it will also be interesting to modify our analysis in this case to derive a better bound for metric $\hat\cP_K$ than \eqref{eq:SGD_main_result_non_cvx}.

\subsection{Acceleration}

Next, we focus on the accelerated version of \algname{clipped-SGD} called Clipped Stochastic Similar Triangles Method \algname{clipped-SSTM} \cite{gorbunov2020stochastic}. The method constructs three sequences of points $\{x^k\}_{k \geq 0}$, $\{y^k\}_{k\geq 0}$, $\{z^k\}_{k \geq 0}$ satisfying the following update rules: $x^0 = y^0 = z^0$ and
\begin{align}
    x^{k+1} &= \frac{A_k y^k + \alpha_{k+1} z^k}{A_{k+1}}, \label{eq:clipped_SSTM_x_update}\\
    z^{k+1} &= z^k - \alpha_{k+1} \cdot \clip\left(\nabla f_{\xi^k}(x^{k+1}), \lambda_k\right), \label{eq:clipped_SSTM_z_update}\\
    y^{k+1} &= \frac{A_k y^k + \alpha_{k+1} z^{k+1}}{A_{k+1}}, \label{eq:clipped_SSTM_y_update}
\end{align}
where $A_0 = \alpha_0 = 0$, $\alpha_{k+1} = \frac{k+2}{2aL}$, $A_{k+1} = A_k + \alpha_{k+1}$, and  $\xi^k$ is sampled from $\cD_k$ independently from previous steps. Our main convergence result for \algname{clipped-SSTM} is given in the following theorem.

\begin{theorem}[Convergence of \algname{clipped-SSTM}]\label{thm:clipped_SSTM_main_theorem} Let Assumptions~\ref{as:bounded_alpha_moment}, \ref{as:L_smoothness}, \ref{as:str_cvx} with $\mu = 0$ hold for $Q = B_{3R}(x^*)$, $R \geq \|x^0 - x^*\|^2$ and $a = \Theta(\max\{A^2, \nicefrac{\sigma K^{\nicefrac{(\alpha+1)}{\alpha}}A^{\nicefrac{(\alpha-1)}{\alpha}}}{LR}\})$, $\lambda_k = \Theta(\nicefrac{R}{(\alpha_{k+1}A)})$, where $A = \ln \frac{4K}{\beta}$, $\beta \in (0,1]$ are such that $A \geq 1$. Then to guarantee $f(y^K) - f(x^*) \leq \varepsilon$ with probability $\geq 1 - \beta$ \algname{clipped-SSTM} requires
\begin{equation}
    \widetilde\cO\left(\max\left\{\sqrt{\frac{LR^2}{\varepsilon}}, \left(\frac{\sigma R}{\varepsilon}\right)^{\frac{\alpha}{\alpha-1}}\right\}\right) \quad \text{oracle calls.} \label{eq:clipped_SSTM_main_result}
\end{equation}
Moreover, with probability $\geq 1-\beta$ the iterates of \algname{clipped-SSTM} stay in the ball $B_{2R}(x^*)$: $\{x^k\}_{k=0}^{K+1}, \{y^k\}_{k=0}^{K}, \{z^k\}_{k=0}^{K} \subseteq B_{2R}(x^*)$.
\end{theorem}

The derived high-probability bound matches (see the proof in Appendix~\ref{appendix:clipped_SSTM}) the best-known one in the case of $\alpha = 2$. For $\alpha < 2$ there are no lower bounds in the convex case. However, the first term in \eqref{eq:clipped_SSTM_main_result} is optimal and matches the deterministic lower bound in the convex case \citep{nemirovskij1983problem}. The second term is the same as in the bound for \algname{clipped-SGD} \eqref{eq:SGD_main_result_cvx} and we conjecture that it cannot be improved.

In the strongly convex case, we consider \algname{clipped-SSTM} with restarts (\algname{R-clipped-SSTM}). This method consists of $\tau$ stages. On the $t$-th stage \algname{R-clipped-SSTM} runs \algname{clipped-SSTM} for $K_0$ iterations from the starting point $\hat x^t$, which is the output from the previous stage ($\hat x^t = x^0$), and defines the obtained point as $\hat x^{t+1}$; see Algorithm~\ref{alg:R-clipped-SSTM} in Appendix~\ref{appendix:R_clipped_SSTM}. For this procedure we have the following result.

\begin{theorem}[Convergence of \algname{R-clipped-SSTM}]\label{thm:R_clipped_SSTM_main_theorem}
    Let Assumptions~\ref{as:bounded_alpha_moment}, \ref{as:L_smoothness}, \ref{as:str_cvx} with $\mu > 0$ hold for $Q = B_{3R}(x^*)$, $R \geq \|x^0 - x^*\|^2$ and \algname{R-clipped-SSTM} runs \algname{clipped-SSTM} $\tau = \lceil \log_2(\nicefrac{\mu R^2}{2\varepsilon}) \rceil$ times. Let $K_t = \widetilde\Theta(\max\{\sqrt{\nicefrac{LR_{t-1}^2}{\varepsilon_t}}, (\nicefrac{\sigma R_{t-1}}{\varepsilon_t})^{\nicefrac{\alpha}{(\alpha-1)}}\})$, $a_t = \widetilde\Theta(\max\{1, \nicefrac{\sigma K_t^{\nicefrac{\alpha+1}{\alpha}}}{LR_t}\})$, $\lambda_k^t = \widetilde{\Theta}(\nicefrac{R}{\alpha_{k+1}^t})$ be the parameters for the stage $t$ of \algname{R-clipped-SSTM}, where $R_{t-1} = 2^{-\nicefrac{(t-1)}{2}} R$, $\varepsilon_t = \nicefrac{\mu R_{t-1}^2}{4}$, $\ln \frac{4\tau K_t}{\beta} \geq 1$ for all $t=1,\ldots,\tau$ and some $\beta \in (0,1]$. Then to guarantee $f(\hat x^\tau) - f(x^*) \leq \varepsilon$ with probability $\geq 1 - \beta$ \algname{R-clipped-SSTM} requires
\begin{equation}
    \widetilde\cO\left(\max\left\{\sqrt{\frac{L}{\mu}}, \left(\frac{\sigma^2}{\mu\varepsilon}\right)^{\frac{\alpha}{2(\alpha-1)}}\right\}\right) \quad \text{oracle calls.} \label{eq:R_clipped_SSTM_main_result}
\end{equation}
Moreover, with probability $\geq 1-\beta$ the iterates of \algname{R-clipped-SSTM} at stage $t$ stay in the ball $B_{2R_{t-1}}(x^*)$.
\end{theorem}

The obtained complexity bound (see the proof in Appendix~\ref{appendix:R_clipped_SSTM}) is the first optimal (up to logarithms) high-probability complexity bound under Assumption~\ref{as:bounded_alpha_moment} for the smooth strongly convex problems. Indeed, the first term cannot be improved in view of the deterministic lower bound by \citet{nemirovskij1983problem}, and the second term is optimal due to \citet{zhang2020adaptive}.

\section{Main Results for Variational Inequalities}

% In this section, we present the main results for solving \eqref{eq:VIP}.

\subsection{Clipped Stochastic Extragradient}
For (quasi strongly) monotone VIPs we consider Clipped Stochastic Extragradient method (\algname{clipped-SEG}):
\begin{align}
    \tx^k &= x^k - \gamma \cdot \clip(F_{\xi_1^k}(x^k), \lambda_k), \label{eq:clipped_SEG_extr_step}\\
    x^{k+1} &= x^k - \gamma \cdot \clip(F_{\xi_2^k}(\tx^k), \lambda_k), \label{eq:clipped_SEG_upd_step}
\end{align}
where $\xi_1^k$, $\xi_2^k$ are sampled from $\cD_k$ independently from previous steps. Our main convergence results for \algname{clipped-SEG} are summarized below.

\begin{theorem}[Convergence of \algname{clipped-SEG}]\label{thm:clipped_SEG_main_theorem}
{\color{white} curiosity}
\textbf{Case 1.} Let Assumptions~\ref{as:bounded_alpha_moment}, \ref{as:L_smoothness}, \ref{as:monotonicity} hold for $Q = B_{4R}(x^*)$ and $0 < \gamma = \cO\left(\min\{\nicefrac{1}{LA}, \nicefrac{R}{K^{\nicefrac{1}{\alpha}}\sigma A^{\nicefrac{(\alpha-1)}{\alpha}}}\}\right)$, $\lambda_k = \lambda = \Theta\left(\nicefrac{R}{\gamma A}\right)$, where $A = \ln \frac{6(K+1)}{\beta} \geq 1$, $\beta \in (0,1]$.\newline
\textbf{Case 2.} Let Assumptions~\ref{as:bounded_alpha_moment}, \ref{as:L_smoothness}, \ref{as:QSM} with $\mu > 0$ hold for $Q = B_{3R}(x^*)$ and $0 < \gamma =  \cO\left(\min\{\nicefrac{1}{LA}, \nicefrac{\ln(B_K)}{\mu(K+1)}\}\right)$, $B_K = \Theta\left(\max\{2, \nicefrac{(K+1)^{\nicefrac{2(\alpha-1)}{\alpha}}\mu^2 R^2}{\sigma^2 A^{\nicefrac{2(\alpha-1)}{\alpha}}\ln^2(B_K)}\}\right)$, $\lambda_k = \Theta(\nicefrac{\exp(-\gamma\mu(1+\nicefrac{k}{2}))R}{\gamma A})$, where $A = \ln \frac{6(K+1)}{\beta}$, $\beta \in (0,1]$ are such that $A \geq 1$.\newline
Then to guarantee $\gap_{R}(\tx_{\text{avg}}^K) \leq \varepsilon$ in \textbf{Case 1} with $\tx_{\text{avg}}^K = \frac{1}{K+1}\sum_{k=0}^K \tx^k$, $\|x^K - x^*\|^2 \leq \varepsilon$ in \textbf{Case 2} with probability $\geq 1 - \beta$ \algname{clipped-SEG} requires
\begin{align}
    \text{{\bf \emph{Case 1}}:}&\quad \widetilde\cO\left(\max\left\{\frac{LR^2}{\varepsilon}, \left(\frac{\sigma R}{\varepsilon}\right)^{\frac{\alpha}{\alpha-1}}\right\}\right) \label{eq:SEG_main_result_monotone}\\
    \text{{\bf \emph{Case 2}}:}&\quad \widetilde\cO\left(\max\left\{\frac{L}{\mu}, \left(\frac{\sigma^2}{\mu^2\varepsilon}\right)^{\frac{\alpha}{2(\alpha-1)}}\right\}\right) \label{eq:SEG_main_result_QSM}
\end{align}
oracle calls.
\end{theorem}

The proofs are deferred to Appendix~\ref{appendix:SEG}. When $\alpha =2$, the above bounds recover SOTA high-probability bounds for monotone and quasi-strongly monotone Lipschitz VIP \citep{gorbunov2022clipped}. For the case of $\alpha < 2$ \eqref{eq:SEG_main_result_monotone} and \eqref{eq:SEG_main_result_QSM} are the first high-probability results for the mentioned classes. Next, the first terms in these complexity bounds are optimal (up to logarithms) due to the lower bounds for the deterministic methods \citep{ouyang2021lower, zhang2022lower}. The second term in \eqref{eq:SEG_main_result_QSM} is also optimal (up to logarithms) due to the lower bounds for stochastic strongly convex minimization \citep{zhang2020adaptive}. Similarly to the convex case in minimization, we conjecture that the second term in \eqref{eq:SEG_main_result_monotone} cannot be improved in the monotone case as well.

\subsection{Clipped Stochastic Gradient Descent-Ascent}

In the star-cocoercive case, we focus on Clipped Stochastic Gradient Descent-Ascent (\algname{clipped-SGDA}):
\begin{equation}
    x^{k+1} = x^k - \gamma \cdot \clip(F_{\xi^k}(x^k), \lambda_k), \label{eq:clipped_SGDA_update}
\end{equation}
where $\xi^k$ is sampled from $\cD_k$ independently from previous steps. For this method we derive the following convergence guarantees.
\begin{theorem}[Convergence of \algname{clipped-SGDA}]\label{thm:clipped_SGDA_main_theorem}
{\color{white} curios}
\textbf{Case 1.} Let Assumptions~\ref{as:bounded_alpha_moment}, \ref{as:star-cocoercivity}, \ref{as:monotonicity} hold for $Q = B_{2R}(x^*)$ and $0 < \gamma = \cO\left(\min\{\nicefrac{1}{\ell A}, \nicefrac{R}{K^{\nicefrac{1}{\alpha}}\sigma A^{\nicefrac{(\alpha-1)}{\alpha}}}\}\right)$, $\lambda_k = \lambda = \Theta\left(\nicefrac{R}{\gamma A}\right)$, $\beta \in (0,1]$ are such that $A \geq 1$.\newline
\textbf{Case 2.} Let Assumptions~\ref{as:bounded_alpha_moment}, \ref{as:star-cocoercivity} hold for $Q = B_{2R}(x^*)$ and $0 < \gamma = \cO\left(\min\{\nicefrac{1}{\ell A}, \nicefrac{R}{K^{\nicefrac{1}{\alpha}}\sigma A^{\nicefrac{(\alpha-1)}{\alpha}}}\}\right)$, $\lambda_k = \lambda = \Theta\left(\nicefrac{R}{\gamma A}\right)$, where $A = \ln \frac{4(K+1)}{\beta}$, $\beta \in (0,1]$ are such that $A \geq 1$.\newline
\textbf{Case 3.} Let Assumptions~\ref{as:bounded_alpha_moment}, \ref{as:star-cocoercivity}, \ref{as:QSM} with $\mu > 0$ hold for $Q = B_{2R}(x^*)$ and $0 < \gamma =  \cO\left(\min\{\nicefrac{1}{\ell A}, \nicefrac{\ln(B_K)}{\mu(K+1)}\}\right)$, $B_K = \Theta\left(\max\{2, \nicefrac{(K+1)^{\nicefrac{2(\alpha-1)}{\alpha}}\mu^2 R^2}{\sigma^2 A^{\nicefrac{2(\alpha-1)}{\alpha}}\ln^2(B_K)}\}\right)$, $\lambda_k = \Theta(\nicefrac{\exp(-\gamma\mu(1+\nicefrac{k}{2}))R}{\gamma A})$, where $A = \ln \frac{4(K+1)}{\beta}$, $\beta \in (0,1]$ are such that $A \geq 1$.\newline
Then to guarantee $\gap_{R}(\tx_{\text{avg}}^K) \leq \varepsilon$ in \textbf{Case 1} with $\tx_{\text{avg}}^K = \frac{1}{K+1}\sum_{k=0}^K \tx^k$, $\frac{1}{K+1}\sum_{k=0}^k \|F(x^k)\|^2 \leq \ell \varepsilon$ in \textbf{Case 2},  $\|x^K - x^*\|^2 \leq \varepsilon$ in \textbf{Case 3} with probability $\geq 1 - \beta$ \algname{clipped-SGDA} requires
\begin{align}
    \text{{\bf \emph{Case 1}} and {\bf \emph{2}}:}&\quad \widetilde\cO\left(\max\left\{\frac{\ell R^2}{\varepsilon}, \left(\frac{\sigma R}{\varepsilon}\right)^{\frac{\alpha}{\alpha-1}}\right\}\right) \label{eq:SGDA_main_result_monotone_and_star_cocoercive} \\
    \text{{\bf \emph{Case 2}}:}&\quad \widetilde\cO\left(\max\left\{\frac{\ell}{\mu}, \left(\frac{\sigma^2}{\mu^2\varepsilon}\right)^{\frac{\alpha}{2(\alpha-1)}}\right\}\right) \label{eq:SGDA_main_result_QSM}
\end{align}
oracle calls.
\end{theorem}

One can find the proofs in Appendix~\ref{appendix:SGDA}. The derived high-probability results generalize the existing SOTA results from the case of $\alpha = 2$ \citep{gorbunov2022clipped} to the case of $\alpha < 2$.

\section{Key Lemma and Intuition Behind the Proofs}

The proofs of all results in this paper follow a very similar pattern. To illustrate the main idea, we consider the analysis of \algname{clipped-SGD} in the non-convex case. Mimicking the proof of deterministic \algname{GD} we derive the following inequality:
\begin{align}
    \gamma\sum\limits_{k=0}^K\|\nabla f(x^k)\|^2 &\lessapprox \Delta_0 - \Delta_{K+1} \label{eq:induction_inequality_example}\\
    & - \gamma \sum\limits_{k=0}^K \langle \nabla f(x^k), \theta_k \rangle + L\gamma^2 \sum\limits_{k=0}^{K-1} \|\theta_k\|^2,\notag
\end{align}
where $\Delta_k = f(x^k) - f_*$ and $\theta_k = \tnabla f_{\xi^k}(x^{k}) - \nabla f(x^{k})$. In other words, we separate the deterministic part of the method from its stochastic part. To obtain the result of Theorem~\ref{thm:clipped_SGD_main_theorem} (Case 1) it remains to upper bound with high-probability the sums from the second line of the formula above. We do it with the help of Bernstein's inequality (Lemma~\ref{lem:Bernstein_ineq}). However, it requires several preliminary steps. In particular, Bernstein's inequality needs the random variables to be bounded. The magnitudes of summands depend on $\nabla f(x^k)$ that can be arbitrarily large due to the stochasticity in $x^k$. However, \eqref{eq:induction_inequality_example} allows to bound $\Delta_{K+1}$ inductively and, using smoothness, to bound $\|\nabla f(x^{K+1})\|$. Secondly, Bernstein's inequality requires knowing the bounds on the bias and variance of the clipped stochastic estimator. For such purposes, we derive the following result, which is a generalization of Lemma~F.5 from \citep{gorbunov2020stochastic}; see also Lemma~10 from \citep{zhang2020adaptive}.
\begin{lemma}\label{lem:bias_and_variance_clip}
    Let $X$ be a random vector in $\R^d$ and $\tX = \clip(X,\lambda)$. Then, $\|\tX - \EE[\tX]\| \leq 2\lambda$. 
    % \begin{equation}
    %     \left\|\tX - \EE[\tX]\right\| \leq 2\lambda. \label{eq:bound_X_main}
    % \end{equation} 
    Moreover, if for some $\sigma \geq 0$ and $\alpha \in (1,2]$ we have $\EE[X] = x\in\R^d$, $\EE[\|X - x\|^\alpha] \leq \sigma^\alpha$, 
    % \begin{equation}
    %     \EE[X] = x\in\R^d,\quad \EE[\|X - x\|^\alpha] \leq \sigma^\alpha \label{eq:UBV_X_main}
    % \end{equation}
    and $\|x\| \leq \nicefrac{\lambda}{2}$, then
    \begin{eqnarray}
        \left\|\EE[\tX] - x\right\| &\leq& \frac{2^\alpha\sigma^\alpha}{\lambda^{\alpha-1}}, \label{eq:bias_X_main}\\
        \EE\left[\left\|\tX - x\right\|^2\right] &\leq& 18 \lambda^{2-\alpha} \sigma^\alpha, \label{eq:distortion_X_main}\\
        \EE\left[\left\|\tX - \EE[\tX]\right\|^2\right] &\leq& 18 \lambda^{2-\alpha} \sigma^\alpha. \label{eq:variance_X_main}
    \end{eqnarray}
\end{lemma}

This lemma can be useful on its own for analyses involving clipping operators. Moreover, our high-probability analysis does not rely on the choice of clipping explicitly: in the proofs, we use only $\|\tX\| \leq \lambda$ and inequalities \eqref{eq:bias_X_main}-\eqref{eq:variance_X_main}. Therefore, our results hold for the methods considered in this work with any other non-linearity $\phi_\lambda(x)$ (not necessary clipping), if it satisfies the conditions from the above lemma for $\tX = \phi_\lambda(X)$.

\section{Discussion}

In this work, we contributed to the stochastic optimization literature via deriving new high-probability results under Assumption~\ref{as:bounded_alpha_moment}. Our results can be extended to the minimization of functions with H\"older continuous gradients using similar ideas to \citep{gorbunov2021near}. Another prominent direction is in obtaining new high-probability results for other types of non-linearities, e.g., like in \citep{polyak1980optimal,jakovetic2022nonlinear}.

\section*{Acknowledgements}

This work was partially supported by a grant for research centers in the field of artificial intelligence, provided by the Analytical Center for the Government of the Russian Federation in accordance with the subsidy agreement (agreement identifier 000000D730321P5Q0002) and the agreement with the Moscow Institute of Physics and Technology dated November 1, 2021 No. 70-2021-00138.

\bibliography{refs}

\begin{thebibliography}{52}
\providecommand{\natexlab}[1]{#1}
\providecommand{\url}[1]{\texttt{#1}}
\expandafter\ifx\csname urlstyle\endcsname\relax
  \providecommand{\doi}[1]{doi: #1}\else
  \providecommand{\doi}{doi: \begingroup \urlstyle{rm}\Url}\fi

\bibitem[Abadi et~al.(2016)Abadi, Chu, Goodfellow, McMahan, Mironov, Talwar,
  and Zhang]{abadi2016deep}
Abadi, M., Chu, A., Goodfellow, I., McMahan, H.~B., Mironov, I., Talwar, K.,
  and Zhang, L.
\newblock Deep learning with differential privacy.
\newblock In \emph{Proceedings of the 2016 ACM SIGSAC conference on computer
  and communications security}, pp.\  308--318, 2016.

\bibitem[Arjevani et~al.(2022)Arjevani, Carmon, Duchi, Foster, Srebro, and
  Woodworth]{arjevani2022lower}
Arjevani, Y., Carmon, Y., Duchi, J.~C., Foster, D.~J., Srebro, N., and
  Woodworth, B.
\newblock Lower bounds for non-convex stochastic optimization.
\newblock \emph{Mathematical Programming}, pp.\  1--50, 2022.

\bibitem[Bennett(1962)]{bennett1962probability}
Bennett, G.
\newblock Probability inequalities for the sum of independent random variables.
\newblock \emph{Journal of the American Statistical Association}, 57\penalty0
  (297):\penalty0 33--45, 1962.

\bibitem[Cutkosky \& Mehta(2021)Cutkosky and Mehta]{cutkosky2021high}
Cutkosky, A. and Mehta, H.
\newblock High-probability bounds for non-convex stochastic optimization with
  heavy tails.
\newblock \emph{Advances in Neural Information Processing Systems}, 34, 2021.

\bibitem[Davis et~al.(2021)Davis, Drusvyatskiy, Xiao, and Zhang]{davis2021low}
Davis, D., Drusvyatskiy, D., Xiao, L., and Zhang, J.
\newblock From low probability to high confidence in stochastic convex
  optimization.
\newblock \emph{Journal of Machine Learning Research}, 22\penalty0
  (49):\penalty0 1--38, 2021.

\bibitem[Dvurechenskii et~al.(2018)Dvurechenskii, Dvinskikh, Gasnikov, Uribe,
  and Nedich]{dvurechenskii2018decentralize}
Dvurechenskii, P., Dvinskikh, D., Gasnikov, A., Uribe, C., and Nedich, A.
\newblock Decentralize and randomize: Faster algorithm for wasserstein
  barycenters.
\newblock \emph{Advances in Neural Information Processing Systems}, 31, 2018.

\bibitem[Dzhaparidze \& Van~Zanten(2001)Dzhaparidze and
  Van~Zanten]{dzhaparidze2001bernstein}
Dzhaparidze, K. and Van~Zanten, J.
\newblock On bernstein-type inequalities for martingales.
\newblock \emph{Stochastic processes and their applications}, 93\penalty0
  (1):\penalty0 109--117, 2001.

\bibitem[Freedman et~al.(1975)]{freedman1975tail}
Freedman, D.~A. et~al.
\newblock On tail probabilities for martingales.
\newblock \emph{the Annals of Probability}, 3\penalty0 (1):\penalty0 100--118,
  1975.

\bibitem[Gasnikov \& Nesterov(2016)Gasnikov and
  Nesterov]{gasnikov2016universal}
Gasnikov, A. and Nesterov, Y.
\newblock Universal fast gradient method for stochastic composit optimization
  problems.
\newblock \emph{arXiv preprint arXiv:1604.05275}, 2016.

\bibitem[Ghadimi \& Lan(2012)Ghadimi and Lan]{ghadimi2012optimal}
Ghadimi, S. and Lan, G.
\newblock Optimal stochastic approximation algorithms for strongly convex
  stochastic composite optimization i: A generic algorithmic framework.
\newblock \emph{SIAM Journal on Optimization}, 22\penalty0 (4):\penalty0
  1469--1492, 2012.

\bibitem[Ghadimi \& Lan(2013)Ghadimi and Lan]{ghadimi2013stochastic}
Ghadimi, S. and Lan, G.
\newblock Stochastic first-and zeroth-order methods for nonconvex stochastic
  programming.
\newblock \emph{SIAM Journal on Optimization}, 23\penalty0 (4):\penalty0
  2341--2368, 2013.

\bibitem[Gidel et~al.(2019)Gidel, Berard, Vignoud, Vincent, and
  Lacoste-Julien]{gidel2019variational}
Gidel, G., Berard, H., Vignoud, G., Vincent, P., and Lacoste-Julien, S.
\newblock A variational inequality perspective on generative adversarial
  networks.
\newblock \emph{International Conference on Learning Representations}, 2019.

\bibitem[Goodfellow et~al.(2014)Goodfellow, Pouget-Abadie, Mirza, Xu,
  Warde-Farley, Ozair, Courville, and Bengio]{goodfellow2014generative}
Goodfellow, I., Pouget-Abadie, J., Mirza, M., Xu, B., Warde-Farley, D., Ozair,
  S., Courville, A., and Bengio, Y.
\newblock Generative adversarial nets.
\newblock In Ghahramani, Z., Welling, M., Cortes, C., Lawrence, N., and
  Weinberger, K.~Q. (eds.), \emph{Advances in Neural Information Processing
  Systems}, volume~27. Curran Associates, Inc., 2014.

\bibitem[Goodfellow et~al.(2016)Goodfellow, Bengio, and
  Courville]{goodfellow2016deep}
Goodfellow, I., Bengio, Y., and Courville, A.
\newblock \emph{Deep learning}.
\newblock MIT press, 2016.

\bibitem[Gorbunov et~al.(2020)Gorbunov, Danilova, and
  Gasnikov]{gorbunov2020stochastic}
Gorbunov, E., Danilova, M., and Gasnikov, A.
\newblock Stochastic optimization with heavy-tailed noise via accelerated
  gradient clipping.
\newblock \emph{Advances in Neural Information Processing Systems},
  33:\penalty0 15042--15053, 2020.

\bibitem[Gorbunov et~al.(2021)Gorbunov, Danilova, Shibaev, Dvurechensky, and
  Gasnikov]{gorbunov2021near}
Gorbunov, E., Danilova, M., Shibaev, I., Dvurechensky, P., and Gasnikov, A.
\newblock Near-optimal high probability complexity bounds for non-smooth
  stochastic optimization with heavy-tailed noise.
\newblock \emph{arXiv preprint arXiv:2106.05958}, 2021.

\bibitem[Gorbunov et~al.(2022{\natexlab{a}})Gorbunov, Danilova, Dobre,
  Dvurechensky, Gasnikov, and Gidel]{gorbunov2022clipped}
Gorbunov, E., Danilova, M., Dobre, D., Dvurechensky, P., Gasnikov, A., and
  Gidel, G.
\newblock Clipped stochastic methods for variational inequalities with
  heavy-tailed noise.
\newblock \emph{arXiv preprint arXiv:2206.01095}, 2022{\natexlab{a}}.

\bibitem[Gorbunov et~al.(2022{\natexlab{b}})Gorbunov, Loizou, and
  Gidel]{gorbunov2022extragradient}
Gorbunov, E., Loizou, N., and Gidel, G.
\newblock Extragradient method: ${\cO}(\nicefrac{1}{K})$ last-iterate
  convergence for monotone variational inequalities and connections with
  cocoercivity.
\newblock In \emph{International Conference on Artificial Intelligence and
  Statistics}, pp.\  366--402. PMLR, 2022{\natexlab{b}}.

\bibitem[Harker \& Pang(1990)Harker and Pang]{harker1990finite}
Harker, P.~T. and Pang, J.-S.
\newblock Finite-dimensional variational inequality and nonlinear
  complementarity problems: a survey of theory, algorithms and applications.
\newblock \emph{Mathematical programming}, 48\penalty0 (1):\penalty0 161--220,
  1990.

\bibitem[Jakovetic et~al.(2022)Jakovetic, Bajovic, Sahu, Kar, Milosevic, and
  Stamenkovic]{jakovetic2022nonlinear}
Jakovetic, D., Bajovic, D., Sahu, A.~K., Kar, S., Milosevic, N., and
  Stamenkovic, D.
\newblock Nonlinear gradient mappings and stochastic optimization: A general
  framework with applications to heavy-tail noise.
\newblock \emph{arXiv preprint arXiv:2204.02593}, 2022.

\bibitem[Juditsky et~al.(2011)Juditsky, Nemirovski, and
  Tauvel]{juditsky2011solving}
Juditsky, A., Nemirovski, A., and Tauvel, C.
\newblock Solving variational inequalities with stochastic mirror-prox
  algorithm.
\newblock \emph{Stochastic Systems}, 1\penalty0 (1):\penalty0 17--58, 2011.

\bibitem[Karimi et~al.(2016)Karimi, Nutini, and Schmidt]{karimi2016linear}
Karimi, H., Nutini, J., and Schmidt, M.
\newblock Linear convergence of gradient and proximal-gradient methods under
  the {P}olyak-{L}ojasiewicz condition.
\newblock In \emph{Joint European conference on machine learning and knowledge
  discovery in databases}, pp.\  795--811. Springer, 2016.

\bibitem[Karimireddy et~al.(2021)Karimireddy, He, and
  Jaggi]{karimireddy2021learning}
Karimireddy, S.~P., He, L., and Jaggi, M.
\newblock Learning from history for byzantine robust optimization.
\newblock In \emph{International Conference on Machine Learning}, pp.\
  5311--5319. PMLR, 2021.

\bibitem[Khaled \& Richt{\'a}rik(2020)Khaled and
  Richt{\'a}rik]{khaled2020better}
Khaled, A. and Richt{\'a}rik, P.
\newblock Better theory for sgd in the nonconvex world.
\newblock \emph{arXiv preprint arXiv:2002.03329}, 2020.

\bibitem[Korpelevich(1976)]{korpelevich1976extragradient}
Korpelevich, G.~M.
\newblock The extragradient method for finding saddle points and other
  problems.
\newblock \emph{Matecon}, 12:\penalty0 747--756, 1976.

\bibitem[Li \& Orabona(2020)Li and Orabona]{li2020high}
Li, X. and Orabona, F.
\newblock A high probability analysis of adaptive sgd with momentum.
\newblock \emph{arXiv preprint arXiv:2007.14294}, 2020.

\bibitem[Liu et~al.(2022)Liu, Zhu, and Belkin]{liu2022loss}
Liu, C., Zhu, L., and Belkin, M.
\newblock Loss landscapes and optimization in over-parameterized non-linear
  systems and neural networks.
\newblock \emph{Applied and Computational Harmonic Analysis}, 59:\penalty0
  85--116, 2022.

\bibitem[Loizou et~al.(2021)Loizou, Berard, Gidel, Mitliagkas, and
  Lacoste-Julien]{loizou2021stochastic}
Loizou, N., Berard, H., Gidel, G., Mitliagkas, I., and Lacoste-Julien, S.
\newblock Stochastic gradient descent-ascent and consensus optimization for
  smooth games: Convergence analysis under expected co-coercivity.
\newblock \emph{Advances in Neural Information Processing Systems}, 34, 2021.

\bibitem[Lojasiewicz(1963)]{lojasiewicz1963topological}
Lojasiewicz, S.
\newblock A topological property of real analytic subsets.
\newblock \emph{Coll. du CNRS, Les {\'e}quations aux d{\'e}riv{\'e}es
  partielles}, 117\penalty0 (87-89):\penalty0 2, 1963.

\bibitem[Mertikopoulos \& Zhou(2019)Mertikopoulos and
  Zhou]{mertikopoulos2019learning}
Mertikopoulos, P. and Zhou, Z.
\newblock Learning in games with continuous action sets and unknown payoff
  functions.
\newblock \emph{Mathematical Programming}, 173\penalty0 (1):\penalty0 465--507,
  2019.

\bibitem[Nazin et~al.(2019)Nazin, Nemirovsky, Tsybakov, and
  Juditsky]{nazin2019algorithms}
Nazin, A.~V., Nemirovsky, A., Tsybakov, A.~B., and Juditsky, A.
\newblock Algorithms of robust stochastic optimization based on mirror descent
  method.
\newblock \emph{Automation and Remote Control}, 80\penalty0 (9):\penalty0
  1607--1627, 2019.

\bibitem[Necoara et~al.(2019)Necoara, Nesterov, and Glineur]{necoara2019linear}
Necoara, I., Nesterov, Y., and Glineur, F.
\newblock Linear convergence of first order methods for non-strongly convex
  optimization.
\newblock \emph{Mathematical Programming}, 175\penalty0 (1):\penalty0 69--107,
  2019.

\bibitem[Nemirovski et~al.(2009)Nemirovski, Juditsky, Lan, and
  Shapiro]{nemirovski2009robust}
Nemirovski, A., Juditsky, A., Lan, G., and Shapiro, A.
\newblock Robust stochastic approximation approach to stochastic programming.
\newblock \emph{SIAM Journal on optimization}, 19\penalty0 (4):\penalty0
  1574--1609, 2009.

\bibitem[Nemirovskij \& Yudin(1983)Nemirovskij and
  Yudin]{nemirovskij1983problem}
Nemirovskij, A.~S. and Yudin, D.~B.
\newblock Problem complexity and method efficiency in optimization.
\newblock 1983.

\bibitem[Nesterov(2007)]{nesterov2007dual}
Nesterov, Y.
\newblock Dual extrapolation and its applications to solving variational
  inequalities and related problems.
\newblock \emph{Mathematical Programming}, 109\penalty0 (2):\penalty0 319--344,
  2007.

\bibitem[Nesterov et~al.(2018)]{nesterov2018lectures}
Nesterov, Y. et~al.
\newblock \emph{Lectures on convex optimization}, volume 137.
\newblock Springer, 2018.

\bibitem[Ouyang \& Xu(2021)Ouyang and Xu]{ouyang2021lower}
Ouyang, Y. and Xu, Y.
\newblock Lower complexity bounds of first-order methods for convex-concave
  bilinear saddle-point problems.
\newblock \emph{Mathematical Programming}, 185\penalty0 (1):\penalty0 1--35,
  2021.

\bibitem[Pascanu et~al.(2013)Pascanu, Mikolov, and
  Bengio]{pascanu2013difficulty}
Pascanu, R., Mikolov, T., and Bengio, Y.
\newblock On the difficulty of training recurrent neural networks.
\newblock In \emph{International conference on machine learning}, pp.\
  1310--1318, 2013.

\bibitem[Patel \& Berahas(2022)Patel and Berahas]{patel2022gradient}
Patel, V. and Berahas, A.~S.
\newblock Gradient descent in the absence of global lipschitz continuity of the
  gradients: Convergence, divergence and limitations of its continuous
  approximation.
\newblock \emph{arXiv preprint arXiv:2210.02418}, 2022.

\bibitem[Patel et~al.(2022)Patel, Zhang, and Tian]{patel2022global}
Patel, V., Zhang, S., and Tian, B.
\newblock Global convergence and stability of stochastic gradient descent.
\newblock \emph{Advances in Neural Information Processing Systems},
  35:\penalty0 36014--36025, 2022.

\bibitem[Polyak(1963)]{polyak1963gradient}
Polyak, B.~T.
\newblock Gradient methods for the minimisation of functionals.
\newblock \emph{USSR Computational Mathematics and Mathematical Physics},
  3\penalty0 (4):\penalty0 864--878, 1963.

\bibitem[Polyak \& Tsypkin(1980)Polyak and Tsypkin]{polyak1980optimal}
Polyak, B.~T. and Tsypkin, Y.~Z.
\newblock Optimal pseudogradient adaptation algorithms.
\newblock \emph{Avtomatika i Telemekhanika}, \penalty0 (8):\penalty0 74--84,
  1980.

\bibitem[Robbins \& Monro(1951)Robbins and Monro]{robbins1951stochastic}
Robbins, H. and Monro, S.
\newblock A stochastic approximation method.
\newblock \emph{The annals of mathematical statistics}, pp.\  400--407, 1951.

\bibitem[Ryu \& Yin(2021)Ryu and Yin]{ryu2021large}
Ryu, E.~K. and Yin, W.
\newblock Large-scale convex optimization via monotone operators, 2021.

\bibitem[Shalev-Shwartz \& Ben-David(2014)Shalev-Shwartz and
  Ben-David]{shalev2014understanding}
Shalev-Shwartz, S. and Ben-David, S.
\newblock \emph{Understanding machine learning: From theory to algorithms}.
\newblock Cambridge university press, 2014.

\bibitem[Song et~al.(2020)Song, Zhou, Zhou, Jiang, and Ma]{song2020optimistic}
Song, C., Zhou, Z., Zhou, Y., Jiang, Y., and Ma, Y.
\newblock Optimistic dual extrapolation for coherent non-monotone variational
  inequalities.
\newblock \emph{Advances in Neural Information Processing Systems},
  33:\penalty0 14303--14314, 2020.

\bibitem[Vural et~al.(2022)Vural, Yu, Balasubramanian, Volgushev, and
  Erdogdu]{vural2022mirror}
Vural, N.~M., Yu, L., Balasubramanian, K., Volgushev, S., and Erdogdu, M.~A.
\newblock Mirror descent strikes again: Optimal stochastic convex optimization
  under infinite noise variance.
\newblock In \emph{Conference on Learning Theory}, pp.\  65--102. PMLR, 2022.

\bibitem[Yue et~al.(2022)Yue, Fang, and Lin]{yue2022lower}
Yue, P., Fang, C., and Lin, Z.
\newblock On the lower bound of minimizing {P}olyak-{L}ojasiewicz functions.
\newblock \emph{arXiv preprint arXiv:2212.13551}, 2022.

\bibitem[Zhang \& Cutkosky(2022)Zhang and Cutkosky]{zhang2022parameter}
Zhang, J. and Cutkosky, A.
\newblock Parameter-free regret in high probability with heavy tails.
\newblock \emph{arXiv preprint arXiv:2210.14355}, 2022.

\bibitem[Zhang et~al.(2020{\natexlab{a}})Zhang, He, Sra, and
  Jadbabaie]{zhang2020gradient}
Zhang, J., He, T., Sra, S., and Jadbabaie, A.
\newblock Why gradient clipping accelerates training: A theoretical
  justification for adaptivity.
\newblock In \emph{International Conference on Learning Representations},
  2020{\natexlab{a}}.
\newblock URL \url{https://openreview.net/forum?id=BJgnXpVYwS}.

\bibitem[Zhang et~al.(2020{\natexlab{b}})Zhang, Karimireddy, Veit, Kim, Reddi,
  Kumar, and Sra]{zhang2020adaptive}
Zhang, J., Karimireddy, S.~P., Veit, A., Kim, S., Reddi, S.~J., Kumar, S., and
  Sra, S.
\newblock Why are adaptive methods good for attention models?
\newblock \emph{Advances in Neural Information Processing Systems}, 33,
  2020{\natexlab{b}}.

\bibitem[Zhang et~al.(2022)Zhang, Hong, and Zhang]{zhang2022lower}
Zhang, J., Hong, M., and Zhang, S.
\newblock On lower iteration complexity bounds for the convex concave saddle
  point problems.
\newblock \emph{Mathematical Programming}, 194\penalty0 (1-2):\penalty0
  901--935, 2022.

\end{thebibliography}
\bibliographystyle{icml2023}

%%%%%%%%%%%%%%%%%%%%%%%%%%%%%%%%%%%%%%%%%%%%%%%%%%%%%%%%%%%%%%%%%%%%%%%%%%%%%%%
%%%%%%%%%%%%%%%%%%%%%%%%%%%%%%%%%%%%%%%%%%%%%%%%%%%%%%%%%%%%%%%%%%%%%%%%%%%%%%%
% APPENDIX
%%%%%%%%%%%%%%%%%%%%%%%%%%%%%%%%%%%%%%%%%%%%%%%%%%%%%%%%%%%%%%%%%%%%%%%%%%%%%%%
%%%%%%%%%%%%%%%%%%%%%%%%%%%%%%%%%%%%%%%%%%%%%%%%%%%%%%%%%%%%%%%%%%%%%%%%%%%%%%%
\newpage
\appendix
\onecolumn
{\small\tableofcontents}

\clearpage

\section{Additional Related Work}\label{appendix:additional_related}

In this section, we provide an overview of the existing in-expectation convergence results under Assumption~\ref{as:bounded_alpha_moment}.

\paragraph{Convex minimization.} The first in-expectation result under Assumption~\ref{as:bounded_alpha_moment} is given by \citet{nemirovskij1983problem}, who derive\footnote{In this section, we hide in $\cO(\cdot)$ all dependencies except the dependency on $\varepsilon$.} $\cO(\varepsilon^{-\nicefrac{\alpha}{(\alpha-1)}})$ complexity for Mirror Descent applied to the minimization of convex functions with bounded gradients.  This result was recently extended by \citet{vural2022mirror} to the uniformly convex functions, and matching lower bounds were derived. 
In the strongly convex case, \citet{zhang2020adaptive} prove $\cO(\varepsilon^{-\nicefrac{\alpha}{2(\alpha-1)}})$ complexity for \algname{clipped-SGD}. However, all these results rely on the boundedness of the gradient. To the best of our knowledge, there are no results for smooth convex problems under Assumption~\ref{as:bounded_alpha_moment} without assuming that the gradient is bounded even in terms of expectation.

\paragraph{Non-convex minimization.} In the non-convex smooth case, \citet{zhang2020adaptive} prove $\cO(\varepsilon^{-\nicefrac{(3\alpha-2)}{(\alpha-1)}})$ complexity for \algname{clipped-SGD} to produce a point $x$ such that $\EE\|\nabla f(x)\| \leq \varepsilon$. In the same work, the authors derive the matching lower bound. However, both upper and lower bounds are derived for $\EE\|\nabla f(x)\|$ which is smaller than $\sqrt{\EE\|\nabla f(x)\|^2}$. The later one is stronger and is more standard performance metric for stochastic non-convex optimization. Therefore, the question of deriving lower and matching upper bounds for the standard metric remains open.

\clearpage

\section{Useful Facts}\label{appendix:useful_facts}

\paragraph{Smoothness.} If $f$ is $L$-smooth on convex set $Q \subseteq \R^d$, then for all $x, y \in Q$ \citep{nesterov2018lectures}
\begin{eqnarray}
    f(y) \leq f(x) + \langle \nabla f(x), y-x \rangle + \frac{L}{2}\|y - x\|^2. \label{eq:L_smoothness_cor_1}
\end{eqnarray}
In particular, if $x$ and $y = x - \frac{1}{L}\nabla f(x)$ lie in $Q$, then the above inequality gives
\begin{eqnarray*}
    f(y) \leq f(x) - \frac{1}{L}\|\nabla f(x)\|^2 + \frac{1}{2L}\|\nabla f(x)\|^2 = f(x) - \frac{1}{2L}\|\nabla f(x)\|^2
\end{eqnarray*}
and
\begin{eqnarray*}
    \|\nabla f(x)\|^2 \leq 2L\left(f(x) - f(y)\right) \leq 2L\left(f(x) - f_*\right)
\end{eqnarray*}
under the assumption that $f_* = \inf_{x \in Q}f(x) > -\infty$. In other words, \eqref{eq:L_smoothness_cor_2} holds for any $x \in Q$ such that $(x - \frac{1}{L}\nabla f(x)) \in Q$. For example, if $x^*$ is an optimum of $f$, then $L$-smoothness on $B_{2R}(x^*)$ implies that \eqref{eq:L_smoothness_cor_2} holds on $B_R(x^*)$: indeed, for any $x \in B_R(x^*)$ we have
\begin{equation}
    \left\|x - \frac{1}{L}\nabla f(x) - x^* \right\| \leq \|x - x^*\| + \frac{1}{L}\|\nabla f(x)\| \overset{\eqref{eq:L_smoothness}}{\leq} 2\|x - x^*\| \leq 2R. \notag
\end{equation}
This derivation means that, in the worst case, to have \eqref{eq:L_smoothness_cor_2} on a set $Q$ we need to assume smoothness on a slightly larger set.

\paragraph{Parameters in \algname{clipped-SSTM}.} To analyze \algname{clipped-SSTM} we use the following lemma about its parameters $\alpha_k$ and $A_k$.

\begin{lemma}[Lemma~E.1 from \citep{gorbunov2020stochastic}]\label{lem:alpha_k_A_K_lemma}
    Let sequences $\{\alpha_k\}_{k\ge0}$ and $\{A_k\}_{k\ge 0}$ satisfy
    \begin{equation}
        \alpha_{0} = A_0 = 0,\quad A_{k+1} = A_k + \alpha_{k+1},\quad \alpha_{k+1} = \frac{k+2}{2aL}\quad \forall k \geq 0, \label{eq:alpha_k_A_k_def}
    \end{equation}
    where $a > 0$, $L > 0$. Then for all $k\ge 0$
    \begin{eqnarray}
        A_{k+1} &=& \frac{(k+1)(k+4)}{4aL}, \label{eq:A_k+1_explicit}\\
        A_{k+1} &\geq& a L \alpha_{k+1}^2. \label{eq:A_k+1_lower_bound}
    \end{eqnarray}
\end{lemma}

\paragraph{Bernstein inequality.} One of the final steps in our proofs is in the proper application of the following lemma known as {\it Bernstein inequality for martingale differences} \citep{bennett1962probability,dzhaparidze2001bernstein,freedman1975tail}.
\begin{lemma}\label{lem:Bernstein_ineq}
    Let the sequence of random variables $\{X_i\}_{i\ge 1}$ form a martingale difference sequence, i.e.\ $\EE\left[X_i\mid X_{i-1},\ldots, X_1\right] = 0$ for all $i \ge 1$. Assume that conditional variances $\sigma_i^2\eqdef\EE\left[X_i^2\mid X_{i-1},\ldots, X_1\right]$ exist and are bounded and assume also that there exists deterministic constant $c>0$ such that $|X_i| \le c$ almost surely for all $i\ge 1$. Then for all $b > 0$, $G > 0$ and $n\ge 1$
    \begin{equation}
        \PP\left\{\Big|\sum\limits_{i=1}^nX_i\Big| > b \text{ and } \sum\limits_{i=1}^n\sigma_i^2 \le G\right\} \le 2\exp\left(-\frac{b^2}{2G + \nicefrac{2cb}{3}}\right).
    \end{equation}
\end{lemma}

\clearpage

\section{Proof of Lemma~\ref{lem:bias_and_variance_clip}}

\begin{lemma}[Lemma~\ref{lem:bias_and_variance_clip}]\label{lem:bias_variance}
    Let $X$ be a random vector in $\R^d$ and $\tX = \clip(X,\lambda)$. Then,
    \begin{equation}
        \left\|\tX - \EE[\tX]\right\| \leq 2\lambda. \label{eq:bound_X}
    \end{equation} 
    Moreover, if for some $\sigma \geq 0$ and $\alpha \in [1,2)$
    \begin{equation}
        \EE[X] = x\in\R^d,\quad \EE[\|X - x\|^{\alpha}] \leq \sigma^{\alpha} \label{eq:UBV_X}
    \end{equation}
    and $\|x\| \leq \nicefrac{\lambda}{2}$, then
    \begin{eqnarray}
        \left\|\EE[\tX] - x\right\| &\leq& \frac{2^{\alpha}\sigma^{\alpha}}{\lambda^{\alpha -1}}, \label{eq:bias_X}\\
        \EE\left[\left\|\tX - x\right\|^2\right] &\leq& 18\lambda^{2-\alpha}\sigma^{\alpha}, \label{eq:distortion_X}\\
        \EE\left[\left\|\tX - \EE[\tX]\right\|^2\right] &\leq& 18\lambda^{2-\alpha}\sigma^{\alpha}. \label{eq:variance_X}
    \end{eqnarray}
\end{lemma}
\begin{proof}
    \textbf{Proof of \eqref{eq:bound_X}:} by definition of a clipping operator, we have 
    \begin{eqnarray*}
        \left\|\tX -\EE\left[\tX\right]\right\| &\leq& \left\|\tX\right\| + \left\|\EE\left[\tX\right]\right\|\\
        &=& \left\|\texttt{clip}(X, \lambda)\right\| + \left\|\EE\left[\texttt{clip}(X,\lambda)\right]\right\|\\
        &\leq& \left\|\min\left\{1,\frac{\lambda}{\left\|X\right\|}\right\}X\right\| + \EE\left[\left\|\min\left\{1,\frac{\lambda}{\left\|X\right\|}\right\}X\right\|\right]\\
        &=& \min\left\{\left\|X\right\|,\lambda\right\} + \EE\left[\min\left\{\left\|X\right\|,\lambda\right\}\right]\\
        &\leq& \lambda + \lambda = 2\lambda.
    \end{eqnarray*}

    \textbf{Proof of \eqref{eq:bias_X}}: To start the proof, we introduce two indicator random variables. Let 
    \begin{equation}
        \chi = \mathbb{I}_{\left\{X: \left\|X\right\| > \lambda\right\}} = \begin{cases} 1, & \text{if } \|X\| > \lambda,\\ 0, & \text{otherwise}\end{cases},~~\eta = \mathbb{I}_{\left\{X: \left\|X-x\right\| > \frac{\lambda}{2}\right\}} = \begin{cases} 1, & \text{if } \|X - x\| > \frac{\lambda}{2},\\ 0, & \text{otherwise}\end{cases}.
    \end{equation}
    Moreover, since $\|X\| \leq \|x\| + \|X-x\| \overset{\|x\| \leq \nicefrac{\lambda}{2}}{\leq} \frac{\lambda}{2} + \|X-x\|$, we have $\chi \leq \eta$. We are now in a position to show \eqref{eq:bias_X}. Using that
    \begin{equation*}
        \tX = \min\left\{1,\frac{\lambda}{\left\|X\right\|}\right\}X = \chi \frac{\lambda}{\left\|X\right\|} X + (1-\chi)X,
    \end{equation*}
    we obtain 
    \begin{eqnarray*}
        \left\|\EE\left[\tX\right] -x\right\| &=& \left\|\EE\left[X + \chi \left(\frac{\lambda}{\left\|X\right\|}-1\right) X\right] -x\right\|\\
        &=&\left\|\EE\left[\chi \left(\frac{\lambda}{\left\|X\right\|}-1\right) X\right] \right\|\\
        &\leq&\EE\left[\Big|\chi \left(\frac{\lambda}{\left\|X\right\|}-1\right)\Big|\left\|X\right\|\right]\\
        &=&\EE\left[\chi \left(1-\frac{\lambda}{\left\|X\right\|}\right)\left\|X\right\|\right].
    \end{eqnarray*}
    Since $1-\nicefrac{\lambda}{\left\|X\right\|}\in (0,1)$ when $\chi \ne 0$, we derive
    \begin{eqnarray*}
        \left\|\EE\left[\tX\right] -x\right\| &\leq& \EE\left[\chi \|X\|\right]\\
        &\overset{\chi\leq \eta}{\leq}& \EE\left[\eta \|X\|\right]\\
        &\leq& \EE\left[\eta \|X-x\| + \eta \|x\|\right]\\
        &\overset{(\ast)}{\leq}& \left(\EE\left[\left\|X-x\right\|^{\alpha}\right]\right)^{\nicefrac{1}{\alpha}}\left(\EE\left[\eta^{\frac{\alpha}{\alpha - 1}}\right]\right)^{\frac{\alpha - 1}{\alpha}} + \|x\|\EE\left[\eta\right]\\
        &\overset{\eqref{eq:UBV_X}}{\leq}&\sigma\left(\EE\left[\eta^{\frac{\alpha}{1-\alpha}}\right]\right)^{\frac{1-\alpha}{\alpha}} + \frac{\lambda}{2} \EE\left[\eta\right],
    \end{eqnarray*}
    where in $(\ast)$ we applied Hölder's inequality. By Markov's inequality,
    \begin{eqnarray}
    \label{eq:bound_eta}
        \EE\left[\eta^{\frac{\alpha}{1-\alpha}}\right] &=& \EE\left[\eta\right] = \mathbb{P}\left\{\|X-x\|>\frac{\lambda}{2}\right\} = \mathbb{P}\left\{\|X-x\|^{\alpha}>\frac{\lambda^{\alpha}}{2^{\alpha}}\right\}\notag\\
        &\leq& \frac{2^{\alpha}}{\lambda^{\alpha}}\EE\left[\|X-x\|^{\alpha}\right]\notag\\
        &\leq& \left(\frac{2\sigma}{\lambda}\right)^{\alpha}.
    \end{eqnarray}
    Thus, in combination with the previous chain of inequalities, we finally have
    \begin{eqnarray*}
        \left\|\EE\left[\tX\right] -x\right\| &\leq& \sigma \left(\frac{2\sigma}{\lambda}\right)^{\alpha-1} + \frac{\lambda}{2}\left(\frac{2\sigma}{\lambda}\right)^{\alpha} = \frac{2^{\alpha}\sigma^{\alpha}}{\lambda^{\alpha-1}}.
    \end{eqnarray*}

    \textbf{Proof of \eqref{eq:distortion_X}}: Using $\|\tX - x\| \leq \|\tX\| +\|x\| \leq \lambda + \frac{\lambda}{2} = \frac{3\lambda}{2}$, we have 
    \begin{eqnarray*}
        \EE\left[\|\tX - x\|^2\right] &=& \EE\left[\|\tX - x\|^{\alpha} \|\tX - x\|^{2-\alpha}\right]\\
        &\leq& \left(\frac{3\lambda}{2}\right)^{2-\alpha}\EE\left[\|\tX - x\|^{\alpha}\chi + \|\tX - x\|^{\alpha}(1-\chi)\right]\\
        &=& \left(\frac{3\lambda}{2}\right)^{2-\alpha}\EE\left[\chi\left\|\frac{\lambda}{\|X\|}X - x\right\|^{\alpha} + \|X - x\|^{\alpha}(1-\chi)\right]\\
        &\leq& \left(\frac{3\lambda}{2}\right)^{2-\alpha}\EE\left[\chi\left(\left\|\frac{\lambda}{\|X\|}X\right\| + \left\|x\right\|\right)^{\alpha} + \|X - x\|^{\alpha}(1-\chi)\right]\\
        &\overset{\|x\|\leq \frac{\lambda}{2}}{\leq}& \left(\frac{3\lambda}{2}\right)^{2-\alpha}\left(\EE\left[\chi\left(\frac{3\lambda}{2}\right)^{\alpha} + \sigma^{\alpha}\right]\right),
    \end{eqnarray*}
    where in the last inequality we applied \eqref{eq:UBV_X} and $1 - \chi \leq 1$. By \eqref{eq:bound_eta} and $\chi \leq \eta$, we obtain
    \begin{eqnarray*}
        \EE\left[\|\tX - x\|^2\right] &\leq&  \frac{9\lambda^2}{4}\left(\frac{2\sigma}{\lambda}\right)^{\alpha} + \left(\frac{3\lambda}{2}\right)^{2-\alpha}\sigma^{\alpha}\\
        &\leq& \frac{9\lambda^2}{4}\left(2^{\alpha} +\frac{2^{\alpha}}{3^{\alpha}}\right)\frac{\sigma^{\alpha}}{\lambda^{\alpha}}\\
        &\leq& 18 \lambda^{2-\alpha}\sigma^{\alpha}.
    \end{eqnarray*}

    \textbf{Proof of \eqref{eq:variance_X}}: Using variance decomposition and \eqref{eq:distortion_X}, we have 
    \begin{equation*}
        \EE\left[\left\|\tX - \EE[\tX]\right\|^2\right] \leq \EE\left[\|\tX - x\|^2\right] \leq 18 \lambda^{2-\alpha}\sigma^{\alpha}.
    \end{equation*}
\end{proof}

\clearpage

\section{Proof of Theorem~\ref{thm:SGD_high-prob_conv}}\label{appendix:failure_of_SGD}

In this section, we give an example of the problem for which \algname{SGD} without clipping leads to a weak high-probability convergence guarantee even under the strong assumption of bounded variance. Theorem below formally states our result, showing that, in the worst-case, the bound for \algname{SGD} scales worse than that of \algname{clipped-SGD} in terms of the probability $\beta$.

\begin{theorem}\label{thm:SGD_high-prob_conv_app}
    For any $\varepsilon > 0$, $\beta \in (0,1)$, and \algname{SGD} parameterized by the number of steps $K$ and stepsize $\gamma$, there exists problem \eqref{eq:min_problem} such that Assumptions~\ref{as:bounded_alpha_moment}, \ref{as:L_smoothness}, and \ref{as:str_cvx} hold with $\alpha = 2$, $0 < \mu \leq L$ and for the iterates produced by \algname{SGD} with any stepsize $0 < \gamma \leq \nicefrac{1}{\mu}$
    \begin{equation}
        \PP\left\{\|x^K - x^*\|^2 \geq \varepsilon\right\} \leq \beta \;\; \Longrightarrow\;\; K = \Omega\left(\frac{\sigma}{\mu\sqrt{\beta \varepsilon}}\right). \label{eq:SGD_high-prob_conv}
    \end{equation}
\end{theorem}

\begin{proof}
    To prove the above theorem, we consider the simple one-dimensional problem $f(x) = \nicefrac{\mu x^2}{2}$. It is easy to see that the considered problem is $\mu$-strongly convex, $\mu$-smooth, and has optimum at $x^* = 0$. We construct the noise in an adversarial way with respect to the parameters of the \algname{SGD}. Concretely, the noise depends on the number of iterates $N$, stepsize $\gamma$, target precision $\varepsilon$, the starting point $x^0$, and bound on the variance $\sigma^2$ such that 
    \begin{align*}
        \nabla f_{\xi_k} (x^k) = \mu x^k - \sigma z_k, 
    \end{align*}
    where 
    \begin{align}
    \label{eq:sgd_noise}
    z_k = 
        \begin{cases}
            0, & \text{if } k < K-1 \text{ or } (1 - \gamma \mu)^K |x^0| > \sqrt{\varepsilon} ,\\ 
            \begin{cases}
                - A,  & \text{with probability } \frac{1}{2A^2}, \\
                 0 & \text{with probability } 1 - \frac{1}{A^2},\\
                  A,  & \text{with probability } \frac{1}{2A^2}, \\
            \end{cases}
            &\text{otherwise}
        \end{cases}
        \quad \forall k \in \{0, 1, \ldots, K-1\},
    \end{align}
    where $A = \max\left\{\frac{2\sqrt{\varepsilon}}{\gamma\sigma} , 1\right\}$. We note that $\EE\left[z^k\right] = 0$. Therefore, $\EE\left[\nabla f_{\xi_k} (x^k)\right] = \mu x^k = \nabla f (x^k).$ Furthermore, 
    \begin{align*}
        \Var\left[z^k\right] = \EE\left[(z^k)^2\right] \leq \frac{1}{2A^2}  A^2 + \frac{1}{2A^2}  A^2 = 1, 
    \end{align*}
    which implies that Assumption~\ref{as:bounded_alpha_moment} holds for  $\alpha = 2$. We note that our construction depends on the parameters of the algorithm and the target value $\varepsilon$. However, our analysis of the methods with clipping works in such generality. 

    Let us now analyze the properties of the introduced problem. We are interested in the situation when  
    \begin{align*}
        \PP\left\{\|x^K - x^*\|^2 \geq \varepsilon\right\} \leq \beta
    \end{align*}
    for $\beta \in (0,1)$. We first prove that this implies that $(1 - \gamma\mu)^K |x^0| \leq \sqrt{\varepsilon}$. To do that we proceed by contradiction and assume that 
    \begin{equation}
        \label{eq:counterexample_proof_1}
        (1 - \gamma\mu)^K |x^0| > \sqrt{\varepsilon}.
    \end{equation}
    By construction, this implies that  $z_k = 0, \; \forall k \in \{0, 1, \ldots, K\}$. This, in turn, implies that $x^K = (1-\gamma\mu)^K x^0$, and, further, by \eqref{eq:counterexample_proof_1} and since $x^*=0$, that
    \begin{align*}
        \PP\left\{\|x^K - x^*\|^2 \geq \varepsilon\right\} = \PP\left\{\|x^K\|^2 \geq \varepsilon\right\} = 1.
    \end{align*}
    Thus, the contradiction shows that $(1 - \gamma\mu)^K |x^0| \leq \sqrt{\varepsilon}$, which yields $K \geq \frac{\ln\frac{\sqrt{\varepsilon}}{|x^0|}}{\ln(1-\gamma\mu)} \geq  \frac{\ln\frac{\sqrt{\varepsilon}}{|x^0|}}{-\gamma\mu} \geq \frac{1}{\gamma \mu} = \ln\frac{|x^0|}{\sqrt{\varepsilon}}$. Using \eqref{eq:sgd_noise} with $K \geq \frac{1}{\gamma\mu} \log\frac{|x^0|}{\sqrt{\varepsilon}}$, we obtain
    \begin{align*}
        \|x^K - x^*\|^2 = ((1 - \gamma\mu)^K x^0 + \gamma \sigma z_K)^2.
    \end{align*}
    % Let us now analyze the properties of the introduced problem. We are interested in the situation when  
    % \begin{align*}
    %     \PP\left\{\|x^K - x^*\|^2 \geq \varepsilon\right\} \leq \beta
    % \end{align*}
    % for $\beta \in (0,1)$, which implies that $(1 - \gamma\mu)^K |x^0| \leq \sqrt{\varepsilon}$, because otherwise we have $z_k = 0, \; \forall k \in \{0, 1, \ldots, K\}$, which would imply 
    % \begin{align*}
    %     \PP\left\{\|x^K - x^*\|^2 \geq \varepsilon\right\} = 1,
    % \end{align*}
    % since $x^K = (1-\gamma\mu)^K x^0$ in this case ($x^* = 0$). Therefore, it has to be the case $(1 - \gamma\mu)^K |x^0| \leq \sqrt{\varepsilon}$, which yields $K \geq \frac{1}{\gamma \mu} \ln\frac{|x^0|}{\sqrt{\varepsilon}}$. Using \eqref{eq:sgd_noise} with $K \geq \frac{1}{\gamma\mu} \log\frac{|x^0|}{\sqrt{\varepsilon}}$, we obtain
    % \begin{align*}
    %     \|x^K - x^*\|^2 = ((1 - \gamma\mu)^K x^0 + \gamma \sigma z_K)^2.
    % \end{align*}
    Furthermore,
    \begin{align*}
        \PP\left\{\|x^K - x^*\|^2 \geq \varepsilon\right\} &= \PP\left\{\left|(1 - \gamma\mu)^K  x^0 + \gamma \sigma z_K \right| \geq \sqrt{\varepsilon}\right\} \\
        &= \PP\left\{\gamma \sigma z_K \geq \sqrt{\varepsilon} - (1 - \gamma\mu)^K  x^0 \right\} + \PP\left\{\gamma \sigma z_K  \leq -\sqrt{\varepsilon} - (1 - \gamma\mu)^K  x^0\right\} \\
        &\geq \PP\left\{\gamma \sigma z_K \geq \sqrt{\varepsilon} + (1 - \gamma\mu)^K  x^0 \right\} + \PP\left\{\gamma \sigma z_K  \leq -\sqrt{\varepsilon} - (1 - \gamma\mu)^K  x^0\right\} \\
        &=  \PP\left\{\left|\gamma \sigma z_K\right| \geq \sqrt{\varepsilon} + (1 - \gamma\mu)^K  x^0 \right\} \\
        &\geq \PP\left\{\left|\gamma \sigma z_K\right| \geq 2\sqrt{\varepsilon} \right\} = \PP\left\{\left|z_K\right| \geq \frac{2\sqrt{\varepsilon}}{\gamma \sigma} \right\}.
    \end{align*}
    Now if $\frac{2\sqrt{\varepsilon}}{\gamma \sigma} < 1$ then $A = 1$. Therefore, 
    \begin{align*}
        1 = \PP\left\{\left|z_K\right| \geq \frac{2\sqrt{\varepsilon}}{\gamma \sigma} \right\} \leq \PP\left\{\|x^K - x^*\|^2 > \varepsilon\right\} < \beta, 
    \end{align*}
    yielding contradiction, which implies that if $\PP\left\{\|x^K - x^*\|^2 > \varepsilon\right\} < \beta$ for our constructed problem, then $\frac{2\sqrt{\varepsilon}}{\gamma \sigma} \geq 1$, i.e., $\gamma \leq \frac{2\sqrt{\varepsilon}}{ \sigma}$. For $\gamma \leq \frac{2\sqrt{\varepsilon}}{ \sigma}$, we have 
    \begin{align*}
         \beta \geq  \PP\left\{\|x^K - x^*\|^2 \geq \varepsilon\right\} \geq \PP\left\{\left|z_K\right| \geq \frac{2\sqrt{\varepsilon}}{\gamma \sigma} \right\} = \frac{1}{A^2} =  \frac{\gamma^2 \sigma^2}{4\varepsilon}.
    \end{align*}
    This implies that $\gamma \leq \frac{2\sqrt{\beta \varepsilon}}{\sigma}$. Combining this inequality with $K \geq \frac{1}{\gamma\mu} \log\frac{|x^0|}{\sqrt{\varepsilon}}$ yields
    \begin{align*}
        K \geq \frac{\sigma}{2\mu\sqrt{\beta \varepsilon}}\log\frac{|x^0|}{\sqrt{\varepsilon}}
    \end{align*}
    and concludes the proof.
    \end{proof}

\clearpage

\section{Missing Proofs for \algname{clipped-SGD}}\label{appendix:clipped_SGD}
In this section, we provide all the missing details and proofs of the results for \algname{clipped-SGD}. For brevity, we will use the following notation: $\tnabla f_{\xi^k}(x^k) = \clip(\nabla f_{\xi^k}(x^k), \lambda_k)$.

\begin{algorithm}[h]
\caption{Clipped Stochastic Gradient Descent (\algname{clipped-SGD}) \citep{pascanu2013difficulty}}
\label{alg:clipped-SGD}   
\begin{algorithmic}[1]
\REQUIRE starting point $x^0$, number of iterations $K$, stepsize $\gamma > 0$, clipping levels $\{\lambda_k\}_{k=0}^{K-1}$.
\FOR{$k=0,\ldots, K-1$}
\STATE Compute $\tnabla f_{\xi^k}(x^{k}) = \clip\left(\nabla f_{\xi^k}(x^{k}), \lambda_k\right)$ using a fresh sample $\xi^k \sim \cD_k$
\STATE $x^{k+1} = x^k - \gamma \tnabla f_{\xi^k}(x^{k})$
\ENDFOR
\ENSURE $x^K$ 
\end{algorithmic}
\end{algorithm}

\subsection{Non-Convex Functions}

We start the analysis of \algname{clipped-SGD} in the non-convex case with the following lemma that follows the proof of deterministic \algname{GD} and separates the stochasticity from the determinisitc part of the method.

\begin{lemma}\label{lem:main_opt_lemma_clipped_SGD_non_convex}
    Let Assumptions~\ref{as:lower_boundedness} and \ref{as:L_smoothness} hold on $Q = \{x \in \R^d\mid \exists y\in \R^d:\; f(y) \leq f_* + 2\Delta \text{ and } \|x-y\|\leq \nicefrac{\sqrt{\Delta}}{20\sqrt{L}}\}$, where $\Delta \geq \Delta_0 = f(x^0) - f_*$, and let stepsize  $\gamma$ satisfy $\gamma < \frac{2}{L}$. If $x^{k} \in Q$ for all $k = 0,1,\ldots,K$, $K \ge 0$, then after $K$ iterations of \algname{clipped-SGD} we have
    \begin{eqnarray}
        \gamma\left(1-\frac{L\gamma}{2}\right)\sum\limits_{k=0}^{K-1}\|\nabla f(x^k)\|^2 &\leq& (f(x^0) - f_*) - (f(x^K) - f_*) - \gamma(1-L\gamma)\sum\limits_{k=0}^{K-1}\langle \nabla f(x^k), \theta_k \rangle \notag\\
        && + \frac{L\gamma^2}{2} \sum\limits_{k=0}^{K-1} \|\theta_k\|^2,\label{eq:main_opt_lemma_clipped_SGD_non_convex}\\
        \theta_{k} &\eqdef& \tnabla f_{\xi^k}(x^{k}) - \nabla f(x^{k}).\label{eq:theta_k_def_clipped_SGD_non_convex}
    \end{eqnarray}
\end{lemma}
\begin{proof}
    Using $x^{k+1} = x^k - \gamma \tnabla f_{\xi^k}(x^{k})$ and smoothness of $f$ (\ref{as:L_smoothness}) we get that for all $k = 0,1,\ldots, K-1$
\begin{eqnarray*}
        f(x^{k+1}) &\leq& f(x^k) + \langle \nabla f(x^k), x^{k+1}-x^k\rangle + \frac{L}{2}\|x^{k+1} - x^k\|^2\notag\\
        &=& f(x^k) - \gamma\langle \nabla f(x^k), \tnabla f_{\xi^k}(x^{k}) \rangle + \frac{L\gamma^2}{2}\|\tnabla f_{\xi^k}(x^{k})\|^2\notag\\ 
        &\overset{\eqref{eq:theta_k_def_clipped_SGD_non_convex}}{=}& f(x^k) - \gamma\|\nabla f(x^k)\|^2 - \gamma \langle\nabla f(x^k), \theta_k\rangle + \frac{L\gamma^2}{2}\|\theta_k\|^2\notag\\ 
        && + \frac{L\gamma^2}{2}\|\nabla f(x^k)\|^2 + L\gamma^2 \langle \nabla f(x^k), \theta_k\rangle\notag\\ 
        &=& f(x^k) - \gamma\left(1-\frac{L\gamma}{2}\right)\|\nabla f(x^k)\|^2 - \gamma(1-L\gamma)\langle \nabla f(x^k), \theta_k\rangle + \frac{L\gamma^2}{2}\|\theta_k\|^2.\label{eq:main_clipped_SGD_non_convex_technical_1}
\end{eqnarray*}
We rearrange the terms and get
\begin{eqnarray*}
        \gamma\left(1-\frac{L\gamma}{2}\right)\|\nabla f(x^k)\|^2 &\leq& f(x^{k}) - f(x^{k+1}) -  \gamma(1-L\gamma)\langle \nabla f(x^k), \theta_k\rangle + \frac{L\gamma^2}{2}\|\theta_k\|^2.\label{eq:main_clipped_SGD_non_convex_technical_2}
\end{eqnarray*}
Finally, summing up these inequalities for $k=0,\ldots,K-1$, we get
\begin{eqnarray*}
        \gamma\left(1-\frac{L\gamma}{2}\right)\sum\limits_{k=0}^{K-1}\|\nabla f(x^k)\|^2 &\leq& \sum\limits_{k=0}^{K-1}\left(f(x^{k}) -  f(x^{k+1})\right) -  \gamma(1-L\gamma)\sum\limits_{k=0}^{K-1}\langle \nabla f(x^k), \theta_k\rangle\notag\\ && + \frac{L\gamma^2}{2}\sum\limits_{k=0}^{K-1}\|\theta_k\|^2\\
        &=& (f(x^0) - f_*) - (f(x^K) - f_*) - \gamma(1-L\gamma)\sum\limits_{k=0}^{K-1}\langle \nabla f(x^k), \theta_k \rangle \notag\\
        && + \frac{L\gamma^2}{2} \sum\limits_{k=0}^{K-1} \|\theta_k\|^2,
\end{eqnarray*}
which concludes the proof.
\end{proof}

Using this lemma, we prove the main convergence result for \algname{clipped-SGD} in the non-convex case.
\begin{theorem}\label{thm:clipped_SGD_non_convex_main}
Let Assumptions~\ref{as:bounded_alpha_moment}, \ref{as:lower_boundedness}, \ref{as:L_smoothness} hold on $Q = \{x \in \R^d\mid \exists y\in \R^d:\; f(y) \leq f_* + 2\Delta \text{ and } \|x-y\|\leq \nicefrac{\sqrt{\Delta}}{20\sqrt{L}}\}$, where $\Delta \geq \Delta_0 = f(x^0) - f_*$, stepsize
\begin{gather}
        \gamma \leq \min\left\{\frac{1}{80L \ln \frac{4(K+1)}{\beta}}, \; \frac{\sqrt{\Delta}}{27^{\frac{1}{\alpha}}20 \sigma \sqrt{L} K^{\frac{1}{\alpha}}\left(\ln \frac{4(K+1)}{\beta}\right)^{\frac{\alpha-1}{\alpha}}}\right\},\label{eq:clipped_SGD_non_convex_step_size}
\end{gather}
and clipping level
 \begin{gather}
        \lambda_k = \lambda = \frac{\sqrt{\Delta}}{20\sqrt{L}\gamma\ln\frac{4(K+1)}{\beta}}, \label{eq:clipped_SGD_non_convex_clipping_level}
\end{gather}
for some $K \geq 0$ and $\beta \in (0,1]$ such that $\ln \frac{4(K+1)}{\beta} \geq 1$. Then, after $K$ iterations of \algname{clipped-SGD} the iterates with probability at least $1 - \beta$ satisfy
\begin{equation}
       \frac{1}{K+1}\sum\limits_{k=0}^{K}\|\nabla f(x^k)\|^2 \leq \frac{2\Delta}{\gamma\left(1-\frac{L\gamma}{2}\right)(K+1)}. \label{eq:clipped_SGD_non_convex_case_appendix}
\end{equation}
In particular, when $\gamma$ equals the minimum from \eqref{eq:clipped_SGD_non_convex_step_size}, then the iterates produced by \algname{clipped-SGD} after $K$ iterations with probability at least $1-\beta$ satisfy
    \begin{equation}
        \frac{1}{K+1}\sum\limits_{k=0}^{K}\|\nabla f(x^k)\|^2 = \cO\left(\max\left\{\frac{L\Delta \ln \frac{K}{\beta}}{K}, \frac{\sqrt{L\Delta} \sigma\ln^{\frac{\alpha-1}{\alpha}}\frac{K}{\beta}}{K^{\frac{\alpha-1}{\alpha}}}\right\}\right), \label{eq:clipped_SSTM_non_convex_case_2_appendix}
    \end{equation}
    meaning that to achieve $\frac{1}{K+1}\sum\limits_{k=0}^{K}\|\nabla f(x^k)\|^2 \leq \varepsilon$ with probability at least $1 - \beta$ \algname{clipped-SGD} requires
    \begin{equation}
        K = \cO\left(\max\left\{\frac{L\Delta}{\varepsilon} \ln \frac{L\Delta}{\beta\varepsilon}, \left(\frac{\sqrt{L\Delta}\sigma}{\varepsilon}\right)^{\frac{\alpha}{\alpha-1}} \ln \left(\frac{1}{\beta} \left(\frac{\sqrt{L\Delta}\sigma}{\varepsilon}\right)^{\frac{\alpha}{\alpha-1}}\right) \right\}\right)\quad \text{iterations/oracle calls.} \label{eq:clipped_SSTM_non_convex_case_complexity_appendix}
    \end{equation}
\end{theorem}

\begin{proof}
    Let $\Delta_k = f(x^k) - f_*$ for all $k\geq 0$. %We first show by induction that for all $k\geq 0$ the iterates $x^{k+1}$ lie in $Q$. The induction base is trivial $x^0 \in Q$ due to the definition $Q$. Next, assume that $x^{l} \in Q$ for some $l\ge 1$. 
    Next, our goal is to show by induction that $\Delta_{l} \leq 2\Delta$ with high probability, which allows to apply the result of Lemma~\ref{lem:main_opt_lemma_clipped_SGD_non_convex} and then use Bernstein's inequality to estimate the stochastic part of the upper-bound. More precisely, for each $k = 0,\ldots, K+1$ we consider probability event $E_k$ defined as follows: inequalities
    \begin{eqnarray}
        \frac{L\gamma^2}{2}\sum\limits_{l=0}^{t-1}\|\theta_l\|^2 - \gamma(1-L\gamma)\sum\limits_{l=1}^{t-1}\langle \nabla f(x^l), \theta_l\rangle &\leq& \Delta, \label{eq:clipped_SGD_non_convex_induction_inequality_1}\\
        \Delta_t &\leq& 2\Delta \label{eq:clipped_SGD_non_convex_induction_inequality_2}
    \end{eqnarray}
    hold for all $t = 0,1,\ldots, k$ simultaneously. We want to prove via induction that $\PP\{E_k\} \geq 1 - \nicefrac{k\beta}{(K+1)}$ for all $k = 0,1,\ldots, K+1$. For $k = 0$ the statement is trivial. Assume that the statement is true for some $k = T - 1 \leq K$: $\PP\{E_{T-1}\} \geq 1 - \nicefrac{(T-1)\beta}{(K+1)}$. One needs to prove that $\PP\{E_{T}\} \geq 1 - \nicefrac{T\beta}{(K+1)}$. First, we notice that probability event $E_{T-1}$ implies that $\Delta_t \leq 2\Delta$ for all $t = 0,1,\ldots, T-1$, i.e., $x^t \in \{y \in \R^d\mid f(y) \leq f_* + 2\Delta\}$ for $t = 0,1,\ldots, T-1$. Moreover, due to the choice of clipping level $\lambda$ we have
    \begin{equation*}
        \|x^{T} - x^{T-1}\| = \gamma \|\tnabla f_{\xi^{T-1}}(x^{T-1})\| \leq \gamma \lambda \overset{\eqref{eq:clipped_SGD_non_convex_clipping_level}}{=} \frac{\sqrt{\Delta}}{20\sqrt{L}\ln\frac{4K}{\beta}} \leq \frac{\sqrt{\Delta}}{20\sqrt{L}}.
    \end{equation*}

    Therefore, $E_{T-1}$ implies $\{x^k\}_{k=0}^{T} \subseteq Q$, meaning that the assumptions of Lemma~\ref{lem:main_opt_lemma_clipped_SGD_non_convex} are satisfied and we have
    \begin{eqnarray}
        A\sum\limits_{l=0}^{t-1}\|\nabla f(x^l)\|^2 &\leq& \Delta_0 - \Delta_t - \gamma(1-L\gamma)\sum\limits_{k=0}^{t-1}\langle \nabla f(x^l), \theta_l \rangle + \frac{L\gamma^2}{2} \sum\limits_{l=0}^{t-1} \|\theta_l\|^2,\label{eq:clipped_SGD_non_convex_technical_1}\\
        A &\eqdef& \gamma\left(1-\frac{L \gamma}{2}\right) \overset{\eqref{eq:clipped_SGD_non_convex_step_size}}{\geq} 0
    \end{eqnarray}
    for all $t = 0,1,\ldots, T$ simultaneously and for all $t = 1,\ldots, T-1$ this probability event also implies that
    \begin{eqnarray}
        \sum\limits_{l=0}^{t-1}\|\nabla f(x^l)\|^2 \overset{ \eqref{eq:clipped_SGD_non_convex_technical_1}}{\leq} \frac{1}{A} \left(\Delta - \gamma(1-L\gamma)\sum\limits_{k=0}^{t-1}\langle \nabla f(x^l), \theta_l \rangle + \frac{L\gamma^2}{2} \sum\limits_{l=0}^{t-1} \|\theta_l\|^2 \right) \overset{\eqref{eq:clipped_SGD_non_convex_induction_inequality_1}, }{\leq} \frac{2\Delta}{A} .\label{eq:clipped_SGD_non_convex_technical_1_1}
    \end{eqnarray}
    Taking into account that $A\sum\limits_{t=0}^{T-1}\|\nabla f(x^t)\|^2 \geq 0$, we also derive that $E_{T-1}$ implies
    \begin{eqnarray}
        \Delta_T \leq \Delta + \frac{L\gamma^2}{2}\sum\limits_{t=0}^{T-1}\|\theta_t\|^2 - \gamma(1-L\gamma)\sum\limits_{t=0}^{T-1}\langle \nabla f(x^t), \theta_t\rangle. \label{eq:clipped_SGD_non_convex_technical_2}
    \end{eqnarray}
    Next, we define random vectors
    \begin{equation}
        \eta_t = \begin{cases} \nabla f(x^t),& \text{if } \|\nabla f(x^t)\| \leq 2\sqrt{L\Delta},\\ 0,&\text{otherwise}, \end{cases} \notag
    \end{equation}
    for all $t = 0,1,\ldots, T-1$. By definition these random vectors are bounded with probability 1
    \begin{equation}
        \|\eta_t\| \leq 2\sqrt{L\Delta}. \label{eq:clipped_SGD_non_convex_technical_6}
    \end{equation}
    Moreover, for $t = 1,\ldots, T-1$ event $E_{T-1}$ implies
    \begin{eqnarray}
        \|\nabla f(x^{l})\| \overset{\eqref{eq:L_smoothness_cor_2}}{\leq} \sqrt{2L(f(x^l) - f_*)} = \sqrt{2L\Delta_l} \leq  2\sqrt{L\Delta} \overset{\eqref{eq:clipped_SGD_non_convex_step_size},\eqref{eq:clipped_SGD_non_convex_clipping_level}}{\leq} \frac{\lambda}{2},\label{eq:clipped_SGD_non_convex_technical_4} 
    \end{eqnarray}
    meaning that $E_{T-1}$ implies that $\eta_t = \nabla f(x^t)$ for all $t = 0,1,\ldots, T-1$. Next, we define the unbiased part and the bias of $\theta_{t}$ as $\theta_{t}^u$ and $\theta_{t}^b$, respectively:
    \begin{equation}
        \theta_{t}^u = \tnabla f_{\xi^t}(x^{t}) - \EE_{\xi^t}\left[\tnabla f_{\xi^t}(x^{t})\right],\quad \theta_{t}^b = \EE_{\xi^t}\left[\tnabla f_{\xi^t}(x^{t})\right] - \nabla f(x^{t}). \label{eq:clipped_SGD_non_convex_theta_u_b}
    \end{equation}
    We notice that $\theta_{t} = \theta_{t}^u + \theta_{t}^b$. Using new notation, we get that $E_{T-1}$ implies
    \begin{eqnarray}
        \Delta_T 
        &\leq& \Delta \underbrace{-\gamma(1-L\gamma)\sum\limits_{t=0}^{T-1}\langle \theta_t^u, \eta_t\rangle}_{\circledOne}  \underbrace{-\gamma(1-L\gamma)\sum\limits_{t=0}^{T-1}\langle \theta_t^b, \eta_t\rangle}_{\circledTwo} + \underbrace{L\gamma^2\sum\limits_{t=0}^{T-1}\left(\left\|\theta_{t}^u\right\|^2 - \EE_{\xi^t}\left[\left\|\theta_{t}^u\right\|^2\right]\right)}_{\circledThree}\notag\\
        &&\quad + \underbrace{L\gamma^2\sum\limits_{t=0}^{T-1}\EE_{\xi^t}\left[\left\|\theta_{t}^u\right\|^2\right]}_{\circledFour} + \underbrace{L\gamma^2\sum\limits_{t=0}^{T-1}\left\|\theta_{t}^b\right\|^2}_{\circledFive}. \label{eq:clipped_SGD_non_convex_technical_7}
    \end{eqnarray}
    It remains to derive good enough high-probability upper-bounds for the terms $\circledOne, \circledTwo, \circledThree, \circledFour, \circledFive$, i.e., to finish our inductive proof we need to show that $\circledOne + \circledTwo + \circledThree + \circledFour + \circledFive \leq \Delta$ with high probability. In the subsequent parts of the proof, we will need to use many times the bounds for the norm and second moments of $\theta_{t}^u$ and $\theta_{t}^b$. First, by definition of clipping operator, we have with probability $1$ that
     \begin{equation}
        \|\theta_{t}^u\| \leq 2\lambda. \label{eq:clipped_SGD_non_convex_norm_theta_u_bound}
    \end{equation}
    Moreover, since $E_{T-1}$ implies that $\|\nabla f(x^{t})\| \leq \nicefrac{\lambda}{2}$ for $t = 0,1,\ldots,T-1$ (see \eqref{eq:clipped_SGD_non_convex_technical_4}), then, in view of Lemma~\ref{lem:bias_and_variance_clip}, we have that $E_{T-1}$ implies
    \begin{eqnarray}
        \|\theta_{t}^b\| &\leq& \frac{2^\alpha\sigma^\alpha}{\lambda^{\alpha-1}}, \label{eq:clipped_SGD_non_convex_norm_theta_b_bound} \\
        \EE_{\xi^t}\left[\|\theta_{t}^u\|^2\right] &\leq& 18 \lambda^{2-\alpha}\sigma^\alpha. \label{eq:clipped_SGD_non_convex_second_moment_theta_u_bound}
    \end{eqnarray}

\textbf{Upper bound for $\circledOne$.} By definition of $\theta_{t}^u$, we have $\EE_{\xi^t}[\theta_{t}^u] = 0$ and
    \begin{equation}
        \EE_{\xi^t}\left[-\gamma(1-L\gamma)\langle\theta_t^u, \eta_t\rangle\right] = 0. \notag
    \end{equation}
    Next, sum $\circledOne$ has bounded with probability $1$ terms:
    \begin{equation}
        |\gamma(1-L\gamma)\left\la \theta_{t}^u, \eta_t\right\ra| \overset{\eqref{eq:clipped_SGD_non_convex_step_size}}{\leq} \gamma \|\theta_{t}^u\| \cdot \|\eta_t\| \overset{\eqref{eq:clipped_SGD_non_convex_technical_6},\eqref{eq:clipped_SGD_non_convex_norm_theta_u_bound}}{\leq} 4\gamma \lambda \sqrt{L \Delta}\overset{\eqref{eq:clipped_SGD_non_convex_clipping_level}}{=} \frac{\Delta}{5\ln\frac{4(K+1)}{\beta}} \eqdef c. \label{eq:clipped_SGD_non_convex_technical_8} 
    \end{equation}
    The summands also have bounded conditional variances $\sigma_t^2 \eqdef \EE_{\xi^t}[\gamma^2(1-L\gamma)^2\langle\theta_t^u, \eta_t\rangle^2]$:
    \begin{equation}
        \sigma_t^2 \leq \EE_{\xi^t}\left[\gamma^2(1-L\gamma)^2\|\theta_{t}^u\|^2\cdot \|\eta_t\|^2\right] \overset{\eqref{eq:clipped_SGD_non_convex_technical_6}}{\leq} 4\gamma^2(1-L\gamma)^2L\Delta \EE_{\xi^t}\left[\|\theta_{t}^u\|^2\right]\overset{\eqref{eq:clipped_SGD_non_convex_step_size}}{\leq} 4\gamma^2L\Delta \EE_{\xi^t}\left[\|\theta_{t}^u\|^2\right]. \label{eq:clipped_SGD_non_convex_technical_9}
    \end{equation}
    In other words, we showed that $\{-\gamma(1-L\gamma)\left\la \theta_{t}^u, \eta_t\right\ra\}_{t=0}^{T-1}$ is a bounded martingale difference sequence with bounded conditional variances $\{\sigma_t^2\}_{t=0}^{T-1}$. Next, we apply Bernstein's inequality (Lemma~\ref{lem:Bernstein_ineq}) with $X_t = -\gamma(1-L\gamma)\left\la \theta_{t}^u, \eta_t\right\ra$, parameter $c$ as in \eqref{eq:clipped_SGD_non_convex_technical_8}, $b = \frac{\Delta}{5}$, $G = \frac{\Delta^2}{150\ln\frac{4(K+1)}{\beta}}$:
    \begin{equation*}
        \PP\left\{|\circledOne| > \frac{\Delta}{5}\quad \text{and}\quad \sum\limits_{t=0}^{T-1} \sigma_{t}^2 \leq \frac{\Delta^2}{150\ln\frac{4(K+1)}{\beta}}\right\} \leq 2\exp\left(- \frac{b^2}{2G + \nicefrac{2cb}{3}}\right) = \frac{\beta}{2(K+1)}.
    \end{equation*}
    Equivalently, we have
    \begin{equation}
        \PP\left\{ E_{\circledOne} \right\} \geq 1 - \frac{\beta}{2(K+1)},\quad \text{for}\quad E_{\circledOne} = \left\{ \text{either} \quad  \sum\limits_{t=0}^{T-1} \sigma_{t}^2 > \frac{\Delta^2}{150\ln\frac{4(K+1)}{\beta}} \quad \text{or}\quad |\circledOne| \leq \frac{\Delta}{5}\right\}. \label{eq:clipped_SGD_non_convex_sum_1_upper_bound}
    \end{equation}
    In addition, $E_{T-1}$ implies that
    \begin{eqnarray}
        \sum\limits_{t=0}^{T-1} \sigma_{t}^2 &\overset{\eqref{eq:clipped_SGD_non_convex_technical_9}}{\leq}& 4\gamma^2L\Delta \sum\limits_{t=0}^{T-1}  \EE_{\xi^t}\left[\|\theta_{t}^u\|^2\right] \overset{\eqref{eq:clipped_SGD_non_convex_second_moment_theta_u_bound}}{\leq} 72\gamma^2 L \Delta \sigma^{\alpha}T \lambda^{2-\alpha}\notag\\
        &\overset{\eqref{eq:clipped_SGD_non_convex_clipping_level}}{=}& \frac{9 \cdot 20^{\alpha}\sqrt{\Delta}^{4-\alpha}\sigma^\alpha T\sqrt{L}^\alpha\gamma^{\alpha}}{50 \ln^{2-\alpha}\frac{4(K+1)}{\beta}} \overset{\eqref{eq:clipped_SGD_non_convex_step_size}}{\leq} \frac{\Delta^2}{150 \ln\frac{4(K+1)}{\beta}}. \label{eq:clipped_SGD_non_convex_sum_1_variance_bound}
    \end{eqnarray}

\textbf{Upper bound for $\circledTwo$.} From $E_{T-1}$ it follows that
\begin{eqnarray}
        \circledTwo &=& -\gamma(1-L\gamma)\sum\limits_{t=0}^{T-1}\langle \theta_t^b, \eta_t \rangle \overset{\eqref{eq:clipped_SGD_non_convex_step_size}}{\leq} \gamma\sum\limits_{t=0}^{T-1}\|\theta_{t}^b\|\cdot \|\eta_t\| \overset{\eqref{eq:clipped_SGD_non_convex_technical_6},\eqref{eq:clipped_SGD_non_convex_norm_theta_b_bound}}{\leq}  \frac{2\cdot 2^\alpha \gamma \sigma^\alpha T \sqrt{L \Delta}}{\lambda^{\alpha-1}} \notag\\ &\overset{\eqref{eq:clipped_SGD_non_convex_clipping_level}}{=}& \frac{40^\alpha}{10} \cdot \frac{\sigma^\alpha T \sqrt{\Delta}^{2-\alpha}\sqrt{L}^\alpha\gamma^\alpha}{\ln^{1-\alpha}\frac{4(K+1)}{\beta}}\overset{\eqref{eq:clipped_SGD_non_convex_step_size}}{\leq} \frac{\Delta}{5}. \label{eq:clipped_SGD_non_convex_sum_2_upper_bound}
\end{eqnarray}

\textbf{Upper bound for $\circledThree$.} First, we have
    \begin{equation}
        \EE_{\xi^t}\left[L\gamma^2\left(\left\|\theta_{t}^u\right\|^2 - \EE_{\xi^t}\left[\left\|\theta_{t}^u\right\|^2\right]\right)\right] = 0. \notag
    \end{equation}
    Next, sum $\circledThree$ has bounded with probability $1$ terms:
    \begin{eqnarray}
        \left|L\gamma^2\left(\left\|\theta_{t}^u\right\|^2 - \EE_{\xi^t}\left[\left\|\theta_{t}^u\right\|^2\right]\right)\right| &\leq& L\gamma^2\left( \|\theta_{t}^u\|^2 +   \EE_{\xi^t}\left[\left\|\theta_{t}^u\right\|^2\right]\right)\notag\\
        &\overset{\eqref{eq:clipped_SGD_non_convex_norm_theta_u_bound}}{\leq}& 8L\gamma^2\lambda^2\overset{\eqref{eq:clipped_SGD_non_convex_clipping_level}}{=} \frac{\Delta}{50\ln^2\frac{4(K+1)}{\beta}} \le \frac{\Delta}{5 \ln\frac{4(K+1)}{\beta}}\eqdef c. \label{eq:clipped_SGD_non_convex_technical_10}
    \end{eqnarray}
    The summands also have bounded conditional variances $\widetilde\sigma_t^2 \eqdef \EE_{\xi^t}\left[L^2\gamma^4\left(\left\|\theta_{t}^u\right\|^2 - \EE_{\xi^t}\left[\left\|\theta_{t}^u\right\|^2\right]\right)^2\right]$:
    \begin{eqnarray}
        \widetilde\sigma_t^2 &\overset{\eqref{eq:clipped_SGD_non_convex_technical_10}}{\leq}& \frac{\Delta}{5 \ln\frac{4(K+1)}{\beta}} \EE_{\xi^t}\left[L\gamma^2\left|\left\|\theta_{t}^u\right\|^2 - \EE_{\xi^t}\left[\left\|\theta_{t}^u\right\|^2\right]\right|\right] \leq \frac{2L\gamma^2 \Delta}{5\ln\frac{4(K+1)}{\beta}} \EE_{\xi^t}\left[\|\theta_{t}^u\|^2\right], \label{eq:clipped_SGD_non_convex_technical_11}
    \end{eqnarray}
    since $\ln\frac{4K}{\beta} \geq 1$. In other words, we showed that $\left\{L\gamma^2\left(\left\|\theta_{t}^u\right\|^2 - \EE_{\xi^t}\left[\left\|\theta_{t}^u\right\|^2\right]\right)\right\}_{t=0}^{T-1}$ is a bounded martingale difference sequence with bounded conditional variances $\{\widetilde\sigma_t^2\}_{t=0}^{T-1}$. Next, we apply Bernstein's inequality (Lemma~\ref{lem:Bernstein_ineq}) with $X_t = L\gamma^2\left(\left\|\theta_{t}^u\right\|^2 - \EE_{\xi^t}\left[\left\|\theta_{t}^u\right\|^2\right]\right)$, parameter $c$ as in \eqref{eq:clipped_SGD_non_convex_technical_10}, $b = \frac{\Delta}{5}$, $G = \frac{\Delta^2}{150\ln\frac{4(K+1)}{\beta}}$:
    \begin{equation*}
        \PP\left\{|\circledThree| > \frac{\Delta}{5}\quad \text{and}\quad \sum\limits_{t=0}^{T-1} \widetilde\sigma_{t}^2 \leq \frac{\Delta^2}{150\ln\frac{4(K+1)}{\beta}}\right\} \leq 2\exp\left(- \frac{b^2}{2G + \nicefrac{2cb}{3}}\right) = \frac{\beta}{2(K+1)}.
    \end{equation*}
    Equivalently, we have
    \begin{equation}
        \PP\left\{ E_{\circledThree} \right\} \geq 1 - \frac{\beta}{2(K+1)},\quad \text{for}\quad E_{\circledThree} = \left\{ \text{either} \quad  \sum\limits_{t=0}^{T-1} \widetilde\sigma_{t}^2 > \frac{\Delta^2}{150\ln\frac{4(K+1)}{\beta}} \quad \text{or}\quad |\circledThree| \leq \frac{\Delta}{5}\right\}. \label{eq:clipped_SGD_non_convex_sum_3_upper_bound}
    \end{equation}
    In addition, $E_{T-1}$ implies that
    \begin{eqnarray}
        \sum\limits_{t=0}^{T-1} \widetilde\sigma_{t}^2 &\overset{\eqref{eq:clipped_SGD_non_convex_technical_11}}{\leq}& \frac{2L\gamma^2\Delta}{5\ln\frac{4(K+1)}{\beta}} \sum\limits_{t=0}^{T-1}  \EE_{\xi^t}\left[\|\theta_{t}^u\|^2\right] \overset{\eqref{eq:clipped_SGD_non_convex_second_moment_theta_u_bound}}{\leq}  
        \frac{36L\gamma^2\Delta\lambda^{2-\alpha}\sigma^{\alpha}T}{5\ln\frac{4(K+1)}{\beta}}\notag\\
        &\overset{\eqref{eq:clipped_SGD_non_convex_clipping_level}}{=}&
        \frac{9\cdot 20^\alpha}{500}\cdot\frac{\sigma^\alpha T \sqrt{\Delta}^{4-\alpha}\sqrt{L}^{\alpha}\gamma^{\alpha}}{\ln^{3-\alpha}\frac{4(K+1)}{\beta}} \overset{\eqref{eq:clipped_SGD_non_convex_step_size}}{\leq} \frac{\Delta^2}{150 \ln\frac{4(K+1)}{\beta}}. \label{eq:clipped_SGD_non_convex_sum_3_variance_bound}
    \end{eqnarray}

\textbf{Upper bound for $\circledFour$.} From $E_{T-1}$ it follows that
\begin{eqnarray}
        \circledFour &=& L\gamma^2\sum\limits_{t=0}^{T-1}\EE_{\xi^t}\left[\left\|\theta_{t}^u\right\|^2\right] \overset{\eqref{eq:clipped_SGD_non_convex_second_moment_theta_u_bound}}{\leq} 18L\gamma^2\lambda^{2-\alpha}\sigma^{\alpha}T \overset{\eqref{eq:clipped_SGD_non_convex_clipping_level}}{=}
        \frac{9\cdot 20^{\alpha}}{200}\cdot\frac{\sqrt{L}^{\alpha}\gamma^{\alpha}\sigma^{\alpha}T\sqrt{\Delta}^{2-\alpha}}{\ln^{2-\alpha}\frac{4(K+1)}{\beta}}\overset{\eqref{eq:clipped_SGD_non_convex_step_size}}{\leq} \frac{\Delta}{5}.\label{eq:clipped_SGD_non_convex_sum_4_upper_bound}
\end{eqnarray}

\textbf{Upper bound for $\circledFive$.} From $E_{T-1}$ it follows that
\begin{eqnarray}
        \circledFive &=& L\gamma^2\sum\limits_{t=0}^{T-1}\left\|\theta_{t}^b\right\|^2 \overset{\eqref{eq:clipped_SGD_non_convex_norm_theta_b_bound}}{\leq} \frac{4^\alpha \sigma^{2\alpha}TL\gamma^2}{\lambda^{2(\alpha-1)}} \overset{\eqref{eq:clipped_SGD_non_convex_clipping_level}}{=}   \frac{1600^\alpha}{400}\cdot \frac{\sigma^{2\alpha}TL^\alpha\gamma^{2\alpha}\Delta^{1-\alpha}}{\ln^{2(1-\alpha)}\frac{4(K+1)}{\beta}}\overset{\eqref{eq:clipped_SGD_non_convex_step_size}}{\leq} \frac{\Delta}{5}.\label{eq:clipped_SGD_non_convex_sum_5_upper_bound}
\end{eqnarray}

Now, we have the upper bounds for  $\circledOne, \circledTwo, \circledThree, \circledFour, \circledFive$. In particular, probability event $E_{T-1}$ implies
\begin{gather*}
        \Delta_T \overset{\eqref{eq:clipped_SGD_non_convex_technical_7}}{\leq} \Delta + \circledOne + \circledTwo + \circledThree + \circledFour + \circledFive,\\
        \circledTwo \overset{\eqref{eq:clipped_SGD_non_convex_sum_2_upper_bound}}{\leq} \frac{\Delta}{5},\quad \circledFour \overset{\eqref{eq:clipped_SGD_non_convex_sum_4_upper_bound}}{\leq} \frac{\Delta}{5}, \quad \circledFive \overset{\eqref{eq:clipped_SGD_non_convex_sum_5_upper_bound}}{\leq} \frac{\Delta}{5},\\
        \sum\limits_{t=0}^{T-1} \sigma_t^2 \overset{\eqref{eq:clipped_SGD_non_convex_sum_1_variance_bound}}{\leq} \frac{\Delta^2}{150 \ln\frac{4(K+1)}{\beta}},\quad \sum\limits_{t=0}^{T-1} \widetilde\sigma_t^2 \overset{\eqref{eq:clipped_SGD_non_convex_sum_3_variance_bound}}{\leq} \frac{\Delta^2}{150 \ln\frac{4(K+1)}{\beta}}.
\end{gather*}
    Moreover, we also have (see \eqref{eq:clipped_SGD_non_convex_sum_1_upper_bound}, \eqref{eq:clipped_SGD_non_convex_sum_3_upper_bound} and our induction assumption)
\begin{equation*}
        \PP\{E_{T-1}\} \geq 1 - \frac{(T-1)\beta}{K+1},\quad \PP\{E_{\circledOne}\} \geq 1 - \frac{\beta}{2(K+1)},\quad \PP\{E_{\circledThree}\} \geq 1 - \frac{\beta}{2(K+1)},
\end{equation*}
    where 
\begin{eqnarray*}
        E_{\circledOne} &=& \left\{ \text{either} \quad  \sum\limits_{t=0}^{T-1} \sigma_{t}^2 > \frac{\Delta^2}{150\ln\frac{4(K+1)}{\beta}} \quad \text{or}\quad |\circledOne| \leq \frac{\Delta}{5}\right\},\\
        E_{\circledThree} &=& \left\{ \text{either} \quad  \sum\limits_{t=0}^{T-1} \widetilde\sigma_{t}^2 > \frac{\Delta^2}{150\ln\frac{4(K+1)}{\beta}} \quad \text{or}\quad |\circledThree| \leq \frac{\Delta}{5}\right\}.
\end{eqnarray*}
    Thus, probability event $E_{T-1} \cap E_{\circledOne} \cap E_{\circledThree}$ implies
\begin{eqnarray}
        \Delta_T &\leq& \Delta + \frac{\Delta}{5} + \frac{\Delta}{5} + \frac{\Delta}{5} + \frac{\Delta}{5} + \frac{\Delta}{5} = 2\Delta, \notag
\end{eqnarray}
    which is equivalent to \eqref{eq:clipped_SGD_non_convex_induction_inequality_1} and \eqref{eq:clipped_SGD_non_convex_induction_inequality_2} for $t = T$, and 
\begin{equation*}
        \PP\{E_T\} \geq \PP\left\{E_{T-1} \cap E_{\circledOne} \cap E_{\circledThree}\right\} = 1 - \PP\left\{\overline{E}_{T-1} \cup \overline{E}_{\circledOne} \cup \overline{E}_{\circledThree}\right\} \geq 1 - \PP\{\overline{E}_{T-1}\} - \PP\{\overline{E}_{\circledOne}\} - \PP\{\overline{E}_{\circledThree}\} \geq 1 - \frac{T\beta}{K+1}.
\end{equation*}
    This finishes the inductive part of our proof, i.e., for all $k = 0,1,\ldots,K+1$ we have $\PP\{E_k\} \geq 1 - \nicefrac{k\beta}{(K+1)}$. In particular, for $k = K+1$ we have that with probability at least $1 - \beta$
\begin{equation*}
        \frac{1}{K+1}\sum\limits_{k=0}^{K}\|\nabla f(x^k)\|^2 \overset{\eqref{eq:clipped_SGD_non_convex_technical_1_1}}{\leq}\frac{2\Delta}{A(K+1)} \overset{\eqref{eq:clipped_SGD_non_convex_step_size}}{=} \frac{2\Delta}{\gamma\left(1-\frac{L\gamma}{2}\right)(K+1)}
\end{equation*}
    and $\{x^k\}_{k=0}^{K} \subseteq Q$, which follows from \eqref{eq:clipped_SGD_non_convex_induction_inequality_2}.
    
Finally, if
\begin{equation*}
        \gamma \leq \min\left\{\frac{1}{80L \ln \frac{4(K+1)}{\beta}}, \; \frac{\sqrt{\Delta}}{27^{\frac{1}{\alpha}}20 \sigma \sqrt{L} K^{\frac{1}{\alpha}}\left(\ln \frac{4(K+1)}{\beta}\right)^{\frac{\alpha-1}{\alpha}}}\right\},
\end{equation*}
then with probability at least $1-\beta$
\begin{eqnarray*}
        \frac{1}{K+1}\sum\limits_{k=0}^{K}\|\nabla f(x^k)\|^2 &\leq& \frac{2\Delta}{\gamma\left(1-\frac{L\gamma}{2}\right)(K+1)} \leq \frac{4\Delta}{\gamma(K+1)}\\
        &=& \max\left\{\frac{320\Delta L \ln \frac{4(K+1)}{\beta}}{K+1}, \frac{80\sqrt{\Delta}27^{\frac{1}{\alpha}}\sigma\sqrt{L}K^{\frac{1}{\alpha}}\left(\ln\frac{4(K+1)}{\beta}\right)^{\frac{
        \alpha-1
        }{\alpha}}}{K+1}\right\}\notag\\
        &=& \cO\left(\max\left\{\frac{L\Delta \ln \frac{K}{\beta}}{K}, \frac{\sqrt{L\Delta} \sigma\ln^{\frac{\alpha-1}{\alpha}}\frac{K}{\beta}}{K^{\frac{\alpha-1}{\alpha}}}\right\}\right).
\end{eqnarray*}
To get $\frac{1}{K+1}\sum\limits_{k=0}^{K}\|\nabla f(x^k)\|^2 \leq \varepsilon$ with probability at least $1-\beta$ it is sufficient to choose $K$ such that both terms in the maximum above are $\cO(\varepsilon)$. This leads to
\begin{equation*}
         K = \cO\left(\max\left\{\frac{L\Delta}{\varepsilon} \ln\frac{L\Delta}{\varepsilon\beta}, \left(\frac{\sqrt{L\Delta}\sigma}{\varepsilon}\right)^{\frac{\alpha}{\alpha-1}} \ln \left(\frac{1}{\beta} \left(\frac{\sqrt{L\Delta}\sigma}{\varepsilon}\right)^{\frac{\alpha}{\alpha-1}}\right) \right\}\right),
\end{equation*}
which concludes the proof.
\end{proof}

\subsection{Polyak-{\L}ojasiewicz Functions}

In this subsection, we provide a high-probability analysis of \algname{clipped-SGD} in the case of Polyak-{\L}ojasiewicz functions. As in the non-convex case, we start with the lemma that handles optimization part of the algorithm and separates it from the stochastic one.

\begin{lemma}\label{lem:main_opt_lemma_clipped_SGD_PL}
    Let Assumptions~\ref{as:L_smoothness} and \ref{as:PL} hold on $Q = \{x \in \R^d\mid \exists y\in \R^d:\; f(y) \leq f_* + 2\Delta \text{ and } \|x-y\|\leq \nicefrac{\sqrt{\Delta}}{20\sqrt{L}}\}$, where $\Delta = f(x^0) - f_*$, and let stepsize  $\gamma$ satisfy $\gamma \leq \frac{1}{L}$. If $x^{k} \in Q$ for all $k = 0,1,\ldots,K+1$, $K \ge 0$, then after $K$ iterations of \algname{clipped-SGD} for all $x \in Q$ we have
    \begin{eqnarray}
       f(x^{K+1}) -f_* &\leq& (1-\gamma\mu)^{K+1}(f(x^0) - f_*) - \gamma(1-L\gamma)\sum\limits_{k=0}^K(1-\gamma\mu)^{K-k} \langle \nabla f(x^k), \theta_k \rangle \notag\\
        && + \frac{L\gamma^2}{2} \sum\limits_{k=0}^{K} (1-\gamma\mu)^{K-k} \|\theta_k\|^2,\label{eq:main_opt_lemma_clipped_SGD_PL}
    \end{eqnarray}
    where $\theta_k$ is defined in \eqref{eq:theta_k_def_clipped_SGD_non_convex}.
\end{lemma}
\begin{proof}
    Using $x^{k+1} = x^k - \gamma \tnabla f_{\xi^k}(x^{k})$ and smoothness of $f$ \eqref{as:L_smoothness} we get that for all $k = 0,1,\ldots, K$
\begin{eqnarray*}
        f(x^{k+1}) &\leq& f(x^k) + \langle \nabla f(x^k), x^{k+1}-x^k\rangle + \frac{L}{2}\|x^{k+1} - x^k\|^2\notag\\
        &\leq& f(x^k) - \gamma\langle \nabla f(x^k), \tnabla f_{\xi^k}(x^{k}) \rangle + \frac{L\gamma^2}{2}\|\tnabla f_{\xi^k}(x^{k})\|^2\notag\\ &\overset{\eqref{eq:theta_k_def_clipped_SGD_non_convex}}{=}& f(x^k) - \gamma\left(1-\frac{L\gamma}{2}\right)\|\nabla f(x^k)\|^2 - \gamma(1-L\gamma)\langle \nabla f(x^k), \theta_k\rangle + \frac{L\gamma^2}{2}\|\theta_k\|^2\notag\\ 
        &\overset{\gamma \leq \frac{1}{L}}{\leq}&
        f(x^k) - \frac{\gamma}{2}\|\nabla f(x^k)\|^2 - \gamma(1-L\gamma)\langle \nabla f(x^k), \theta_k\rangle + \frac{L\gamma^2}{2}\|\theta_k\|^2\notag\\ 
        &\overset{\eqref{eq:PL}}{\leq}&
        f(x^k) - \gamma \mu (f(x^k) - f_*) - \gamma(1-L\gamma)\langle \nabla f(x^k), \theta_k\rangle + \frac{L\gamma^2}{2}\|\theta_k\|^2.\label{eq:main_clipped_SGD_PL_technical_1}
\end{eqnarray*}
By rearranging the terms and subtracting $f_*$, we obtain 
\begin{eqnarray*}
        f(x^{k+1}) - f_* &\leq& (1-\gamma\mu)(f(x^k)-f_*)- \gamma(1-L\gamma)\langle \nabla f(x^k), \theta_k\rangle + \frac{L\gamma^2}{2}\|\theta_k\|^2.\label{eq:main_clipped_SGD_PL_technical_3}
\end{eqnarray*}
Unrolling the recurrence, we obtain \eqref{eq:main_opt_lemma_clipped_SGD_PL}.
\end{proof}
\begin{theorem}\label{thm:clipped_SGD_PL_main}
Let Assumptions~\ref{as:bounded_alpha_moment}, \ref{as:L_smoothness}, \ref{as:PL} hold on $Q = \{x \in \R^d\mid \exists y\in \R^d:\; f(y) \leq f_* + 2\Delta \text{ and } \|x-y\|\leq \nicefrac{\sqrt{\Delta}}{20\sqrt{L}}\}$, where $\Delta \geq \Delta_0 = f(x^0) - f_*$, stepsize
\begin{eqnarray}
        0 < \gamma &\leq& \min\left\{\frac{1}{250 L \ln\frac{4(K+1)}{\beta}}, \; \frac{\ln(B_K)}{\mu(K+1)}\right\},\label{eq:clipped_SGD_PL_step_size}\\
        B_K &=& \max\left\{2, \frac{(K+1)^{\frac{2(\alpha-1)}{\alpha}}\mu^2\Delta}{264600^{\frac{2}{\alpha}}L\sigma^2\ln^{\frac{2(\alpha-1)}{\alpha}}\left(\frac{4(K+1)}{\beta}\right)\ln^2(B_K)} \right\} \label{eq:B_K_SGD_PL_1} \\
        &=& \Theta\left(\max\left\{1, \frac{K^{\frac{2(\alpha-1)}{\alpha}}\mu^2 \Delta}{L\sigma^2\ln^{\frac{2(\alpha-1)}{\alpha}}\left(\frac{K}{\beta}\right)\ln^2\left(\max\left\{2, \frac{K^{\frac{2(\alpha-1)}{\alpha}}\mu^2 \Delta}{L\sigma^2\ln^{\frac{2(\alpha-1)}{\alpha}}\left(\frac{K}{\beta}\right)} \right\}\right)} \right\}\right), \label{eq:B_K_SGD_PL_2}
\end{eqnarray}
and clipping level
 \begin{gather}
        \lambda_k = \frac{\exp(-\gamma\mu(1 + \nicefrac{k}{2}))\sqrt{\Delta}}{120\sqrt{L}\gamma \ln \tfrac{4(K+1)}{\beta}}, \label{eq:clipped_SGD_PL_clipping_level}
\end{gather}
for some $K > 0$ and $\beta \in (0,1]$ such that $\ln \frac{4(K+1)}{\beta} \geq 1$. Then, after $K$ iterations of \algname{clipped-SGD} the iterates with probability at least $1 - \beta$ satisfy
\begin{equation}
       f\left(x^{K+1}\right) - f_* \leq 2 \exp(-\gamma \mu (K+1))\Delta. \label{eq:clipped_SGD_PL_case_appendix}
\end{equation}
In particular, when $\gamma$ equals the minimum from \eqref{eq:clipped_SGD_PL_step_size}, then the iterates produced by \algname{clipped-SGD} after $K$ iterations with probability at least $1-\beta$ satisfy
    \begin{equation}
        f\left( x^{K}\right) - f_* = \cO\left(\max\left\{\Delta\exp\left(- \frac{\mu K}{L \ln \tfrac{K}{\beta}}\right), \frac{L\sigma^2\ln^{\frac{2(\alpha-1)}{\alpha}}\left(\frac{K}{\beta}\right)\ln^2\left(\max\left\{2, \frac{K^{\frac{2(\alpha-1)}{\alpha}}\mu^2\Delta}{L\sigma^2\ln^{\frac{2(\alpha-1)}{\alpha}}\left(\frac{K}{\beta}\right)} \right\}\right)}{K^{\frac{2(\alpha-1)}{\alpha}}\mu^2}\right\}\right), \label{eq:clipped_SGD_PL_case_2_appendix}
    \end{equation}
    meaning that to achieve $f\left( x^{K}\right) - f_* \leq \varepsilon$ with probability at least $1 - \beta$ \algname{clipped-SGD} requires
    \begin{equation}
        K = \cO\left(\frac{L}{\mu}\ln\left(\frac{\Delta}{\varepsilon}\right)\ln\left(\frac{L}{\mu \beta}\ln\frac{\Delta}{\varepsilon}\right), \left(\frac{L\sigma^2}{\mu^2\varepsilon}\right)^{\frac{\alpha}{2(\alpha-1)}}\ln \left(\frac{1}{\beta} \left(\frac{L\sigma^2}{\mu^2\varepsilon}\right)^{\frac{\alpha}{2(\alpha-1)}}\right)\ln^{\frac{\alpha}{\alpha-1}}\left(B_\varepsilon\right)\right),\label{eq:clipped_SGD_PL_case_complexity_appendix}
    \end{equation}
    iterations/oracle calls, where
    \begin{equation*}
        B_\varepsilon = \max\left\{2, \frac{\Delta}{\varepsilon \ln \left(\frac{1}{\beta} \left(\frac{L\sigma^2}{\mu^2\varepsilon}\right)^{\frac{\alpha}{2(\alpha-1)}}\right)}\right\}.
    \end{equation*}
\end{theorem}
\begin{proof}
    As in the previous results, the proof is based on the induction argument and shows that the iterates do not leave some set with high probability. More precisely, for each $k = 0,1,\ldots,K+1$ we consider probability event $E_k$ as follows: inequalities
    \begin{equation}
        \Delta_t \leq 2 \exp(-\gamma\mu t) \Delta \label{eq:induction_inequality_SGD_PL}
    \end{equation}
    hold for $t = 0,1,\ldots,k$ simultaneously, where $\Delta_t = f(x^t) - f_*$. We want to prove $\PP\{E_k\} \geq  1 - \nicefrac{k\beta}{(K+1)}$ for all $k = 0,1,\ldots,K+1$ by induction. The base of the induction is trivial: for $k=0$ we have $\Delta_0 \leq \Delta < 2\Delta$ by definition. Next, assume that for $k = T-1 \leq K$ the statement holds: $\PP\{E_{T-1}\} \geq  1 - \nicefrac{(T-1)\beta}{(K+1)}$. Given this, we need to prove $\PP\{E_{T}\} \geq  1 - \nicefrac{T\beta}{(K+1)}$. Since $\Delta_t \leq 2\exp(-\gamma\mu t) \Delta \leq 2\Delta$, we have $x^t \in \{y\in \R^d\mid f(y) \leq f_* + 2\Delta\}$ for $t = 0,1,\ldots,T-1$, where function $f$ is $L$-smooth. Thus, $E_{T-1}$ implies
    \begin{eqnarray}
        \|\nabla f(x^t)\| &\overset{\eqref{eq:L_smoothness_cor_2}}{\leq}& \sqrt{2L(f(x^t) - f_*)} \overset{\eqref{eq:induction_inequality_SGD_PL}}{\leq} 2\sqrt{L\exp(- \gamma\mu t)\Delta} \overset{\eqref{eq:clipped_SGD_PL_step_size},\eqref{eq:clipped_SGD_PL_clipping_level}}{\leq} \frac{\lambda_t}{2} \label{eq:operator_bound_x_t_SGD_PL}
    \end{eqnarray}
    for all $t = 0, 1, \ldots, T-1$. Moreover
    \begin{equation*}
        \|x^{T} - x^{T-1}\| = \gamma \|\tnabla f_{\xi^{T-1}}(x^{T-1})\| \leq \gamma\lambda_{T-1} \overset{\eqref{eq:clipped_SGD_PL_clipping_level}}{\leq} \frac{\sqrt{\Delta}}{20\sqrt{L}},
    \end{equation*}
    meaning that $E_{T-1}$ implies $x^T \in \{x \in \R^d\mid \exists y\in \R^d:\; f(y) \leq f_* + 2\Delta \text{ and } \|x-y\|\leq \nicefrac{\sqrt{\Delta}}{20\sqrt{L}}\}$. Using Lemma~\ref{lem:main_opt_lemma_clipped_SGD_PL} and $(1 - \gamma\mu)^T \leq \exp(-\gamma\mu T)$, we obtain that $E_{T-1}$ implies
    \begin{eqnarray}
        \Delta_T &\leq& \exp(-\gamma\mu T)\Delta - \gamma (1-L\gamma) \sum\limits_{l=0}^{T-1} (1-\gamma\mu)^{T-1-l} \langle \nabla f(x^l), \theta_l \rangle\notag\\
        &&\quad + \frac{L\gamma^2}{2} \sum\limits_{l=0}^{T-1} (1-\gamma\mu)^{T-1-l} \|\theta_l\|^2. \notag
    \end{eqnarray}
    To handle the sums above, we introduce a new notation:
    \begin{gather}
        \eta_t = \begin{cases} \nabla f(x^t),& \text{if } \|\nabla f(x^t)\| \leq 2\sqrt{L}\exp(-\nicefrac{\gamma\mu t}{2})\sqrt{\Delta},\\ 0,& \text{otherwise}, \end{cases} \label{eq:eta_t_SGD_PL}
    \end{gather}
    for $t = 0, 1, \ldots, T-1$. These vectors are bounded almost surely:
    \begin{equation}
        \|\eta_t\| \leq 2\sqrt{L}\exp(-\nicefrac{\gamma\mu t}{2})\sqrt{\Delta} \label{eq:zeta_t_eta_t_bound_SGD_PL} 
    \end{equation}
    for all $t = 0, 1, \ldots, T-1$. In other words, $E_{T-1}$ implies $\eta_t = \nabla f(x^t)$ for all $t = 0,1,\ldots,T-1$, meaning that from $E_{T-1}$ it follows that 
    \begin{eqnarray}
        \Delta_T &\leq& \exp(-\gamma\mu T)\Delta - \gamma(1-L\gamma) \sum\limits_{l=0}^{T-1} (1-\gamma\mu)^{T-1-l} \langle \eta_l, \theta_l \rangle \notag\\
        && + \frac{L\gamma^2}{2} \sum\limits_{l=0}^{T-1} (1-\gamma\mu)^{T-1-l} \|\theta_l\|^2. \notag
    \end{eqnarray}
    To handle the sums appeared on the right-hand side of the previous inequality we consider unbiased and biased parts of $\theta_l$:
    \begin{equation}
        \theta_{t}^u = \tnabla f_{\xi^t}(x^{t}) - \EE_{\xi^t}\left[\tnabla f_{\xi^t}(x^{t})\right],\quad \theta_{t}^b = \EE_{\xi^t}\left[\tnabla f_{\xi^t}(x^{t})\right] - \nabla f(x^{t}). \label{eq:theta_unbias_bias_SGD_PL}
    \end{equation}
    for all $l = 0,\ldots, T-1$. By definition we have $\theta_l = \theta_l^u + \theta_l^b$ for all $l = 0,\ldots, T-1$. Therefore, $E_{T-1}$ implies
    \begin{eqnarray}
        \Delta_T &\leq& \exp(-\gamma\mu T) \Delta \underbrace{- \gamma (1-L\gamma) \sum\limits_{l=0}^{T-1} (1-\gamma\mu)^{T-1-l} \langle \eta_l, \theta_l^u \rangle}_{\circledOne} \underbrace{-\gamma(1-L\gamma) \sum\limits_{l=0}^{T-1} (1-\gamma\mu)^{T-1-l} \langle \eta_l, \theta_l^b \rangle}_{\circledTwo} \notag\\
        &&\quad + \underbrace{L\gamma^2 \sum\limits_{l=0}^{T-1} (1-\gamma\mu)^{T-1-l} \EE_{\xi^l}\left[\|\theta_l^u\|^2\right]}_{\circledThree} + \underbrace{L\gamma^2 \sum\limits_{l=0}^{T-1} (1-\gamma\mu)^{T-1-l} \left(\|\theta_l^u\|^2 -\EE_{\xi^l}\left[\|\theta_l^u\|^2\right] \right)}_{\circledFour}\notag\\
        &&\quad + \underbrace{L\gamma^2 \sum\limits_{l=0}^{T-1} (1-\gamma\mu)^{T-1-l} \|\theta_l^b\|^2}_{\circledFive}. \label{eq:SGD_PL_12345_bound}
    \end{eqnarray}
    where we also apply inequality $\|a+b\|^2 \leq 2\|a\|^2 + 2\|b\|^2$ holding for all $a,b \in \R^d$ to upper bound $\|\theta_l\|^2$. It remains to derive good enough high-probability upper-bounds for the terms $\circledOne, \circledTwo, \circledThree, \circledFour, \circledFive$, i.e., to finish our inductive proof we need to show that $\circledOne + \circledTwo + \circledThree + \circledFour + \circledFive \leq \exp(-\gamma\mu T) \Delta$ with high probability. In the subsequent parts of the proof, we will need to use many times the bounds for the norm and second moments of $\theta_{t}^u$ and $\theta_{t}^b$. First, by definition of clipping operator, we have with probability $1$ that
    \begin{equation}
        \|\theta_l^u\| \leq 2\lambda_l.\label{eq:theta_omega_magnitude_SGD_PL}
    \end{equation}
    Moreover, since $E_{T-1}$ implies that $\|\nabla f(x^l)\|^2 \leq \nicefrac{\lambda_l}{2}$ for all $l = 0,1, \ldots, T-1$ (see \eqref{eq:operator_bound_x_t_SGD_PL}), from Lemma~\ref{lem:bias_and_variance_clip} we also have that $E_{T-1}$ implies
    \begin{gather}
        \left\|\theta_l^b\right\| \leq \frac{2^{\alpha}\sigma^\alpha}{\lambda_l^{\alpha-1}}, \label{eq:bias_theta_omega_SGD_PL}\\
        \EE_{\xi^l}\left[\left\|\theta_l^u\right\|^2\right] \leq 18 \lambda_l^{2-\alpha}\sigma^\alpha, \label{eq:distortion_theta_omega_SGD_PL}
    \end{gather}
    for all $l = 0,1, \ldots, T-1$.
    
\paragraph{Upper bound for $\circledOne$.} By definition of $\theta_l^u$, we have $\EE_{\xi^l}[\theta_{l}^u] = 0$ and
    \begin{equation*}
        \EE_{\xi^l}\left[-\gamma (1-L\gamma) (1-\gamma\mu)^{T-1-l} \langle \eta_l, \theta_l^u \rangle\right] = 0.
    \end{equation*}
    Next, sum $\circledOne$ has bounded with probability $1$ terms:
    \begin{eqnarray}
        |-\gamma (1-L\gamma)(1-\gamma\mu)^{T-1-l} \langle \eta_l, \theta_l^u \rangle | &\overset{\eqref{eq:clipped_SGD_PL_step_size}}{\leq}& \gamma\exp(-\gamma\mu (T - 1 - l)) \|\eta_l\|\cdot \|\theta_l^u\|\notag\\
        &\overset{\eqref{eq:zeta_t_eta_t_bound_SGD_PL},\eqref{eq:theta_omega_magnitude_SGD_PL}}{\leq}& 4\sqrt{L\Delta}\gamma \exp(-\gamma\mu (T - 1 - \nicefrac{l}{2}))  \lambda_l\notag\\
        &\overset{\eqref{eq:clipped_SGD_PL_step_size},\eqref{eq:clipped_SGD_PL_clipping_level}}{\leq}& \frac{\exp(-\gamma\mu T)\Delta}{5\ln\tfrac{4(K+1)}{\beta}} \eqdef c. \label{eq:SGD_PL_technical_3_1}
    \end{eqnarray}
    The summands also have bounded conditional variances $\sigma_l^2 \eqdef \EE_{\xi^l}\left[\gamma^2 (1-L\gamma)^2(1-\gamma\mu)^{2T-2-2l} \langle \eta_l, \theta_l^u \rangle^2\right]$:
    \begin{eqnarray}
        \sigma_l^2 &\leq& \EE_{\xi^l}\left[\gamma^2(1-L\gamma)^2\exp(-\gamma\mu (2T - 2 - 2l)) \|\eta_l\|^2\cdot \|\theta_l^u\|^2\right]\notag\\
        &\overset{\eqref{eq:zeta_t_eta_t_bound_SGD_PL}, \eqref{eq:clipped_SGD_PL_step_size}}{\leq}& 4\gamma^2 L\Delta \exp(-\gamma\mu (2T - 2 - l))  \EE_{\xi^l}\left[\|\theta_l^u\|^2\right]\notag\\
        &\overset{\eqref{eq:clipped_SGD_PL_step_size}}{\leq}& 10\gamma^2L\Delta\exp(-\gamma\mu (2T - l))R^2 \EE_{\xi^l}\left[\|\theta_l^u\|^2\right]. \label{eq:SGD_PL_technical_3_2}
    \end{eqnarray}
    In other words, we showed that $\{-\gamma(1-L\gamma) (1-\gamma\mu)^{T-1-l} \langle \eta_l, \theta_l^u \rangle\}_{l = 0}^{T-1}$ is a bounded martingale difference sequence with bounded conditional variances $\{\sigma_l^2\}_{l = 0}^{T-1}$. Next, we apply Bernstein's inequality (Lemma~\ref{lem:Bernstein_ineq}) with $X_l = -\gamma(1-L\gamma) (1-\gamma\mu)^{T-1-l} \langle \eta_l, \theta_l^u \rangle$, parameter $c$ as in \eqref{eq:SGD_PL_technical_3_1}, $b = \tfrac{1}{5}\exp(-\gamma\mu T) \Delta$, $G = \tfrac{\exp(-2 \gamma\mu T) \Delta^2}{150\ln\frac{4(K+1)}{\beta}}$:
    \begin{equation*}
        \PP\left\{|\circledOne| > \frac{1}{5}\exp(-\gamma\mu T) \Delta \text{ and } \sum\limits_{l=0}^{T-1}\sigma_l^2 \leq \frac{\exp(- 2\gamma\mu T) \Delta^2}{150\ln\tfrac{4(K+1)}{\beta}}\right\} \leq 2\exp\left(- \frac{b^2}{2G + \nicefrac{2cb}{3}}\right) = \frac{\beta}{2(K+1)}.
    \end{equation*}
    Equivalently, we have
    \begin{equation}
        \PP\{E_{\circledOne}\} \geq 1 - \frac{\beta}{2(K+1)},\quad \text{for}\quad E_{\circledOne} = \left\{\text{either} \quad \sum\limits_{l=0}^{T-1}\sigma_l^2 > \frac{\exp(- 2\gamma\mu T) \Delta^2}{150\ln\tfrac{4(K+1)}{\beta}}\quad \text{or}\quad |\circledOne| \leq \frac{1}{5}\exp(-\gamma\mu T) \Delta\right\}. \label{eq:bound_1_SGD_PL}
    \end{equation}
    In addition, $E_{T-1}$ implies that
    \begin{eqnarray}
        \sum\limits_{l=0}^{T-1}\sigma_l^2 &\overset{\eqref{eq:SGD_PL_technical_3_2}}{\leq}& 10\gamma^2L\Delta\exp(- 2\gamma\mu T)\sum\limits_{l=0}^{T-1} \frac{\EE_{\xi^l}\left[\|\theta_l^u\|^2\right]}{\exp(-\gamma\mu l)}\notag\\ 
        &\overset{\eqref{eq:distortion_theta_omega_SGD_PL}, T \leq K+1}{\leq}& 180\gamma^2L\Delta\exp(-2\gamma\mu T)  \sigma^\alpha \sum\limits_{l=0}^{K} \frac{\lambda_l^{2-\alpha}}{\exp(-\gamma\mu l)}\notag\\
        &\overset{\eqref{eq:clipped_SGD_PL_clipping_level}}{=}& \frac{180\gamma^{\alpha}\sqrt{L}^{\alpha}\sqrt{\Delta}^{4-\alpha}\exp(-2\gamma\mu T)  \sigma^\alpha}{120^{2-\alpha} \ln^{2-\alpha}\tfrac{4(K+1)}{\beta}} \sum\limits_{l=0}^{K} \frac{1}{\exp(-\gamma\mu l)} \cdot \left(\exp(-\gamma\mu(1 + \nicefrac{l}{2}))\right)^{2-\alpha} \notag\\
        &=& \frac{180\gamma^{\alpha}\sqrt{L}^{\alpha}\sqrt{\Delta}^{4-\alpha}\exp(-2\gamma\mu T)  \sigma^\alpha}{120^{2-\alpha} \ln^{2-\alpha}\tfrac{4(K+1)}{\beta}} \sum\limits_{l=0}^{K} \exp(\gamma\mu(\alpha-2)) \cdot \exp\left(\frac{\gamma\mu\alpha l}{2}\right) \notag\\
        &\leq& \frac{180\gamma^{\alpha}\sqrt{L}^{\alpha}\sqrt{\Delta}^{4-\alpha}\exp(-2\gamma\mu T)  \sigma^\alpha(K+1)\exp(\frac{\gamma\mu\alpha K}{2})}{120^{2-\alpha} \ln^{2-\alpha}\tfrac{4(K+1)}{\beta}} \notag\\
        &\overset{\eqref{eq:clipped_SGD_PL_step_size}}{\leq}& \frac{\exp(-2\gamma\mu T)\Delta^2}{150\ln\tfrac{4(K+1)}{\beta}}, \label{eq:bound_3_variances_SGD_PL}
    \end{eqnarray} 
where we also show that $E_{T-1}$ implies
\begin{equation}
        \gamma^2 L \Delta\sum\limits_{l=0}^{K} \frac{\lambda_{l}^{2-\alpha}}{\exp(-\gamma\mu l)} \leq \frac{\gamma^\alpha \sqrt{L}^\alpha \sqrt{\Delta}^{4-\alpha}(K+1)\exp(\frac{\gamma\mu\alpha K}{2})}{120^{2-\alpha}\ln^{2-\alpha} \tfrac{4(K+1)}{\beta}}. \label{eq:useful_inequality_on_lambda_SGD_PL}
\end{equation}

\paragraph{Upper bound for $\circledTwo$.} From $E_{T-1}$ it follows that
    \begin{eqnarray}
        \circledTwo &\overset{\eqref{eq:clipped_SGD_PL_step_size}}{\leq}& \gamma \exp(-\gamma\mu (T-1)) \sum\limits_{l=0}^{T-1} \frac{\|\eta_l\|\cdot \|\theta_l^b\|}{\exp(-\gamma\mu l)}\notag\\
        &\overset{\eqref{eq:zeta_t_eta_t_bound_SGD_PL}, \eqref{eq:bias_theta_omega_SGD_PL}}{\leq}& 2^{1+\alpha} \gamma \exp(-\gamma\mu (T-1)) \sqrt{\Delta} \sigma^\alpha \sum\limits_{l=0}^{T-1} \frac{1}{\lambda_l^{\alpha-1} \exp(-\nicefrac{\gamma\mu l}{2})}\notag\\
        &\overset{\eqref{eq:clipped_SGD_PL_clipping_level}}{=}& 2^{1+\alpha}\cdot 120^{\alpha-1}\sqrt{L}^{1-\alpha} \sqrt{\Delta}^{2-\alpha} \exp(-\gamma\mu (T-1)) \gamma^{\alpha} \sigma^\alpha \ln^{\alpha-1}\tfrac{4(K+1)}{\beta} \sum\limits_{l=0}^{T-1} \frac{\exp(\nicefrac{\gamma\mu l}{2})}{\exp\left(-\gamma\mu(1 + \nicefrac{l}{2})\right)^{\alpha-1}} \notag\\
        &\overset{T \leq K+1}{\leq}& 2^{1+\alpha}\cdot 120^{\alpha-1}\sqrt{L}^{1-\alpha} \sqrt{\Delta}^{2-\alpha} \exp(-\gamma\mu (T-1)) \gamma^{\alpha} \sigma^\alpha \ln^{\alpha-1}\tfrac{4(K+1)}{\beta} \sum\limits_{l=0}^{K} \exp\left(\frac{\gamma\mu \alpha l}{2}\right) \notag\\
        &\leq& 2^{1+\alpha}\cdot 120^{\alpha-1}\sqrt{L}^{1-\alpha} \sqrt{\Delta}^{2-\alpha} \exp(-\gamma\mu (T-1)) \gamma^{\alpha} \sigma^\alpha \ln^{\alpha-1}\tfrac{4(K+1)}{\beta}(K+1)\exp\left(\frac{\gamma\mu \alpha K}{2}\right) \notag\\
        &\overset{\eqref{eq:clipped_SGD_PL_step_size}}{\leq}& \frac{1}{5}\exp(-\gamma\mu T) \Delta. \label{eq:bound_4_SGD_PL}
\end{eqnarray}

\paragraph{Upper bound for $\circledThree$.} From $E_{T-1}$ it follows that
    \begin{eqnarray}
        \circledThree &=& L\gamma^2 \exp(-\gamma\mu (T-1)) \sum\limits_{l=0}^{T-1} \frac{\EE_{\xi^l}\left[\|\theta_l^u\|^2\right]}{\exp(-\gamma\mu l)} \notag\\
        &\overset{\eqref{eq:distortion_theta_omega_SGD_PL}}{\leq}& 18 L \gamma^2\exp(-\gamma\mu (T-1)) \sigma^\alpha\sum\limits_{l=0}^{T-1} \frac{\lambda_l^{2-\alpha}}{\exp(-\gamma\mu l)} \notag\\
        &\overset{\eqref{eq:useful_inequality_on_lambda_SGD_PL}}{\leq}&  \frac{18\gamma^\alpha \sqrt{L}^{\alpha}\sqrt{\Delta}^{2-\alpha}\exp(-\gamma\mu (T-1)) \sigma^\alpha (K+1)\exp(\frac{\gamma\mu\alpha K}{2})}{120^{2-\alpha} \ln^{2-\alpha}\tfrac{4(K+1)}{\beta}} \notag\\
        &\overset{\eqref{eq:clipped_SGD_PL_step_size}}{\leq}& \frac{1}{5} \exp(-\gamma\mu T) \Delta. \label{eq:bound_5_SGD_PL}
    \end{eqnarray}

\paragraph{Upper bound for $\circledFour$.} First, we have
    \begin{equation*}
        L\gamma^2 (1-\gamma\mu)^{T-1-l}\EE_{\xi^l}\left[\|\theta_l^u\|^2 -\EE_{\xi_2^l}\left[\|\theta_l^u\|^2\right] \right] = 0.
    \end{equation*}
    Next, sum $\circledFour$ has bounded with probability $1$ terms:
    \begin{eqnarray}
        L\gamma^2 (1-\gamma\mu)^{T-1-l}\left| \|\theta_l^u\|^2 -\EE_{\xi^l}\left[\|\theta_l^u\|^2\right] \right| 
        &\overset{\eqref{eq:theta_omega_magnitude_SGD_PL}}{\leq}& \frac{8 L \gamma^2 \exp(-\gamma\mu T) \lambda_l^2}{\exp(-\gamma\mu (1+l))}\notag\\
        &\overset{\eqref{eq:clipped_SGD_PL_clipping_level}}{=}& \frac{\exp(-\gamma\mu (T+1))\Delta}{1800\ln^2\tfrac{4(K+1)}{\beta}}\notag\\
        &\leq& \frac{\exp(-\gamma\mu T)\Delta}{5\ln\tfrac{4(K+1)}{\beta}}\notag\\
        &\eqdef& c. \label{eq:SGD_PL_technical_6_1}
    \end{eqnarray}
    The summands also have conditional variances
    \begin{equation*}
        \widehat\sigma_l^2 \eqdef \EE_{\xi^l}\left[L^2\gamma^4 (1-\gamma\mu)^{2T-2-2l} \left| \|\theta_l^u\|^2 -\EE_{\xi^l}\left[\|\theta_l^u\|^2\right]  \right|^2\right]
    \end{equation*}
    that are bounded
    \begin{eqnarray}
        \widehat\sigma_l^2 &\overset{\eqref{eq:SGD_PL_technical_6_1}}{\leq}& \frac{L\gamma^2\exp(-2\gamma\mu T)\Delta}{5\exp(-\gamma\mu (1+l))\ln\tfrac{4(K+1)}{\beta}} \EE_{\xi^l}\left[\left| \|\theta_l^u\|^2 -\EE_{\xi^l}\left[\|\theta_l^u\|^2\right] \right|\right]\notag\\
        &\leq& \frac{2L\gamma^2\exp(-2\gamma\mu T)\Delta}{5\exp(-\gamma\mu (1+l))\ln\tfrac{4(K+1)}{\beta}} \EE_{\xi^l}\left[\|\theta_l^u\|^2\right]. \label{eq:SGD_PL_technical_6_2}
    \end{eqnarray}
    In other words, we showed that $\left\{L\gamma^2 (1-\gamma\mu)^{T-1-l}\left( \|\theta_l^u\|^2 -\EE_{\xi^l}\left[\|\theta_l^u\|^2\right] \right)\right\}_{l = 0}^{T-1}$ is a bounded martingale difference sequence with bounded conditional variances $\{\widehat\sigma_l^2\}_{l = 0}^{T-1}$. Next, we apply Bernstein's inequality (Lemma~\ref{lem:Bernstein_ineq}) with $X_l = L\gamma^2 (1-\gamma\mu)^{T-1-l}\left( \|\theta_l^u\|^2 -\EE_{\xi^l}\left[\|\theta_l^u\|^2\right] \right)$, parameter $c$ as in \eqref{eq:SGD_PL_technical_6_1}, $b = \tfrac{1}{5}\exp(-\gamma\mu T) \Delta$, $G = \tfrac{\exp(-2 \gamma\mu T) \Delta^2}{150\ln\frac{4(K+1)}{\beta}}$:
    \begin{equation*}
        \PP\left\{|\circledFour| > \frac{1}{5}\exp(-\gamma\mu T) \Delta \text{ and } \sum\limits_{l=0}^{T-1}\widehat\sigma_l^2 \leq \frac{\exp(-2\gamma\mu T) \Delta^2}{150\ln\frac{4(K+1)}{\beta}}\right\} \leq 2\exp\left(- \frac{b^2}{2G + \nicefrac{2cb}{3}}\right) = \frac{\beta}{2(K+1)}.
    \end{equation*}
    Equivalently, we have
    \begin{equation}
        \PP\{E_{\circledFour}\} \geq 1 - \frac{\beta}{2(K+1)},\quad \text{for}\quad E_{\circledFour} = \left\{\text{either} \quad \sum\limits_{l=0}^{T-1}\widehat\sigma_l^2 > \frac{\exp(-2\gamma\mu T) \Delta^2}{150\ln\tfrac{4(K+1)}{\beta}}\quad \text{or}\quad |\circledFour| \leq \frac{1}{5}\exp(-\gamma\mu T) \Delta\right\}. \label{eq:bound_6_SGD_PL}
    \end{equation}
    In addition, $E_{T-1}$ implies that
    \begin{eqnarray}
        \sum\limits_{l=0}^{T-1}\widehat\sigma_l^2 &\overset{\eqref{eq:SGD_PL_technical_6_2}}{\leq}& \frac{2L\gamma^2\exp(-\gamma\mu (2T-1))\Delta}{5\ln\tfrac{4(K+1)}{\beta}} \sum\limits_{l=0}^{T-1} \frac{\EE_{\xi^l}\left[\|\theta_l^u\|^2 \right]}{\exp(-\gamma\mu l)}\notag\\ &\overset{\eqref{eq:distortion_theta_omega_SGD_PL}, T \leq K+1}{\leq}& \frac{36L\gamma^2\exp(-\gamma\mu (2T-1)) \Delta \sigma^\alpha}{5\ln\tfrac{4(K+1)}{\beta}} \sum\limits_{l=0}^{K} \frac{\lambda_l^{2-\alpha}}{\exp(-\gamma\mu l)}\notag\\
        &\overset{\eqref{eq:useful_inequality_on_lambda_SGD_PL}}{\leq}& \frac{36\sqrt{L}^\alpha\gamma^\alpha\exp(-\gamma\mu (2T-1)) \sqrt{\Delta}^{4-\alpha} \sigma^\alpha (K+1)\exp(\frac{\gamma\mu\alpha K}{2})}{5\cdot 120^{2-\alpha}\ln^{3-\alpha}\tfrac{4(K+1)}{\beta}} \notag\\
        &\overset{\eqref{eq:clipped_SGD_non_convex_step_size}}{\leq}& \frac{\exp(-2\gamma\mu T)\Delta^2}{150\ln\tfrac{4(K+1)}{\beta}}. \label{eq:bound_6_variances_SGD_PL}
\end{eqnarray}

\paragraph{Upper bound for $\circledFive$.} From $E_{T-1}$ it follows that
\begin{eqnarray}
        \circledFive &=&  L\gamma^2 \sum\limits_{l=0}^{T-1} \exp(-\gamma\mu (T-1-l)) \|\theta_l^b\|^2\notag\\
        &\overset{\eqref{eq:bias_theta_omega_SGD_PL}}{\leq}& 2^{2\alpha}L\gamma^2 \exp(-\gamma\mu (T-1)) \sigma^{2\alpha} \sum\limits_{l=0}^{T-1} \frac{1}{\lambda_l^{2\alpha-2} \exp(-\gamma\mu l)} \notag\\
        &\overset{\eqref{eq:clipped_SGD_PL_clipping_level}, T \leq K+1}{\leq}& \frac{2\cdot 2^{2\alpha}\cdot 120^{2\alpha-2}\gamma^{2\alpha}\sqrt{L}^{2\alpha} \exp(-\gamma\mu T) \sigma^{2\alpha} \ln^{2\alpha-2}\tfrac{4(K+1)}{\beta}}{\sqrt{\Delta}^{2\alpha-2}} \sum\limits_{l=0}^{K} \exp\left(\gamma\mu(2\alpha-2)\left(1 + \frac{l}{2}\right)\right)\exp(\gamma\mu l)\notag\\
        &\leq& \frac{4\cdot 2^{2\alpha}\cdot 120^{2\alpha-2}\gamma^{2\alpha}\sqrt{L}^{2\alpha} \exp(-\gamma\mu T) \sigma^{2\alpha} \ln^{2\alpha-2}\tfrac{4(K+1)}{\beta}}{\sqrt{\Delta}^{2\alpha-2}} \sum\limits_{l=0}^{K} \exp(\gamma\mu \alpha l)\notag\\
        &\leq& \frac{4\cdot 2^{2\alpha}\cdot 120^{2\alpha-2}\gamma^{2\alpha} \sqrt{L}^{2\alpha} \exp(-\gamma\mu T) \sigma^{2\alpha} \ln^{2\alpha-2}\tfrac{4(K+1)}{\beta} (K+1) \exp(\gamma\mu \alpha K)}{\sqrt{\Delta}^{2\alpha-2}}\notag\\
        &\overset{\eqref{eq:clipped_SGD_non_convex_step_size}}{\leq}& \frac{1}{5}\exp(-\gamma\mu T) \Delta. \label{eq:bound_7_SGD_PL}
\end{eqnarray}

 Now, we have the upper bounds for  $\circledOne, \circledTwo, \circledThree, \circledFour, \circledFive$. In particular, probability event $E_{T-1}$ implies
    \begin{gather*}
        \Delta_T \overset{\eqref{eq:SGD_PL_12345_bound}}{\leq} \exp(-\gamma\mu T) \Delta + \circledOne + \circledTwo + \circledThree + \circledFour + \circledFive,\\
        \circledTwo \overset{\eqref{eq:bound_4_SGD_PL}}{\leq} \frac{1}{5}\exp(-\gamma\mu T)\Delta,\quad \circledThree \overset{\eqref{eq:bound_5_SGD_PL}}{\leq} \frac{1}{5}\exp(-\gamma\mu T)\Delta,\quad \circledFive \overset{\eqref{eq:bound_7_SGD_PL}}{\leq} \frac{1}{5}\exp(-\gamma\mu T)\Delta,\\
        \sum\limits_{l=0}^{T-1}\sigma_l^2 \overset{\eqref{eq:bound_3_variances_SGD_PL}}{\leq} \frac{\exp(-2\gamma\mu T)\Delta^2}{150\ln\tfrac{4(K+1)}{\beta}},\quad \sum\limits_{l=0}^{T-1}\widehat\sigma_l^2 \overset{\eqref{eq:bound_6_variances_SGD_PL}}{\leq}  \frac{\exp(-2\gamma\mu T)\Delta^2}{150\ln\tfrac{4(K+1)}{\beta}}.
    \end{gather*}
    Moreover, we also have (see \eqref{eq:bound_1_SGD_PL}, \eqref{eq:bound_6_SGD_PL} and our induction assumption)
     \begin{gather*}
        \PP\{E_{T-1}\} \geq 1 - \frac{(T-1)\beta}{K+1},\\
        \PP\{E_{\circledOne}\} \geq 1 - \frac{\beta}{2(K+1)}, \quad \PP\{E_{\circledFour}\} \geq 1 - \frac{\beta}{2(K+1)},
    \end{gather*}
    where
    \begin{eqnarray}
        E_{\circledOne}&=& \left\{\text{either} \quad \sum\limits_{l=0}^{T-1}\sigma_l^2 > \frac{\exp(-2\gamma\mu T) \Delta^2}{150\ln\tfrac{4(K+1)}{\beta}}\quad \text{or}\quad |\circledOne| \leq \frac{1}{5}\exp(-\gamma\mu T) \Delta\right\},\notag\\
        E_{\circledFour}&=& \left\{\text{either} \quad \sum\limits_{l=0}^{T-1}\widehat\sigma_l^2 > \frac{\exp(-2\gamma\mu T) \Delta^2}{150\ln\tfrac{4(K+1)}{\beta}}\quad \text{or}\quad |\circledFour| \leq \frac{1}{5}\exp(-\gamma\mu T) \Delta\right\}.\notag
    \end{eqnarray}
    Thus, probability event $E_{T-1} \cap E_{\circledOne} \cap E_{\circledFour}$ implies
    \begin{eqnarray*}
        \Delta_T &\overset{\eqref{eq:SGD_PL_12345_bound}}{\leq}& \exp(-\gamma\mu T) \Delta + \circledOne + \circledTwo + \circledThree + \circledFour + \circledFive\\
        &\leq& 2\exp(-\gamma\mu T) \Delta,
    \end{eqnarray*}
    which is equivalent to \eqref{eq:induction_inequality_SGD_PL} for $t = T$, and
    \begin{equation}
        \PP\{E_T\} \geq \PP\{E_{T-1} \cap E_{\circledOne} \cap E_{\circledFour}\} = 1 - \PP\{\overline{E}_{T-1} \cup \overline{E}_{\circledOne} \cup \overline{E}_{\circledFour}\} \geq 1 - \frac{T\beta}{K+1}. \notag
    \end{equation}
    This finishes the inductive part of our proof, i.e., for all $k = 0,1,\ldots,K+1$ we have $\PP\{E_k\} \geq 1 - \nicefrac{k\beta}{(K+1)}$. In particular, for $k = K+1$ we have that with probability at least $1 - \beta$
    \begin{equation}
        f(x^{K+1}) - f_* \leq 2\exp(-\gamma\mu (K+1))\Delta. \notag
    \end{equation}
    Finally, if 
    \begin{eqnarray*}
        \gamma &=& \min\left\{\frac{1}{250 L \ln \tfrac{4(K+1)}{\beta}}, \frac{\ln(B_K)}{\mu(K+1)}\right\}, \notag\\
        B_K &=& \max\left\{2, \frac{(K+1)^{\frac{2(\alpha-1)}{\alpha}}\mu^2\Delta}{264600^{\frac{2}{\alpha}}L\sigma^2\ln^{\frac{2(\alpha-1)}{\alpha}}\left(\frac{6(K+1)}{\beta}\right)\ln^2(B_K)} \right\} \\
        &=& \cO\left(\max\left\{2, \frac{K^{\frac{2(\alpha-1)}{\alpha}}\mu^2 \Delta}{L\sigma^2\ln^{\frac{2(\alpha-1)}{\alpha}}\left(\frac{K}{\beta}\right)\ln^2\left(\max\left\{2, \frac{K^{\frac{2(\alpha-1)}{\alpha}}\mu^2 \Delta}{L\sigma^2\ln^{\frac{2(\alpha-1)}{\alpha}}\left(\frac{K}{\beta}\right)} \right\}\right)} \right\}\right)
    \end{eqnarray*}
    then with probability at least $1-\beta$
    \begin{eqnarray*}
        f(x^{K+1}) - f_* &\leq& 2\exp(-\gamma\mu (K+1))\Delta\\
        &=& 2\Delta\max\left\{\exp\left(-\frac{\mu(K+1)}{250 L \ln \tfrac{4(K+1)}{\beta}}\right), \frac{1}{B_K} \right\}\\
        &=& \cO\left(\max\left\{\Delta\exp\left(- \frac{\mu K}{L \ln \tfrac{K}{\beta}}\right), \frac{L\sigma^2\ln^{\frac{2(\alpha-1)}{\alpha}}\left(\frac{K}{\beta}\right)\ln^2\left(\max\left\{2, \frac{K^{\frac{2(\alpha-1)}{\alpha}}\mu^2\Delta}{L\sigma^2\ln^{\frac{2(\alpha-1)}{\alpha}}\left(\frac{K}{\beta}\right)} \right\}\right)}{K^{\frac{2(\alpha-1)}{\alpha}}\mu^2}\right\}\right).
    \end{eqnarray*}
    To get $\|x^{K+1} - x^*\|^2 \leq \varepsilon$ with probability at least $1-\beta$ it is sufficient to choose $K$ such that both terms in the maximum above are $\cO(\varepsilon)$. This leads to
    \begin{equation*}
         K = \cO\left(\frac{L}{\mu}\ln\left(\frac{\Delta}{\varepsilon}\right)\ln\left(\frac{L}{\mu \beta}\ln\frac{\Delta}{\varepsilon}\right), \left(\frac{L\sigma^2}{\mu^2\varepsilon}\right)^{\frac{\alpha}{2(\alpha-1)}}\ln \left(\frac{1}{\beta} \left(\frac{L\sigma^2}{\mu^2\varepsilon}\right)^{\frac{\alpha}{2(\alpha-1)}}\right)\ln^{\frac{\alpha}{\alpha-1}}\left(B_\varepsilon\right)\right),
    \end{equation*}
    where
    \begin{equation*}
        B_\varepsilon = \max\left\{2, \frac{\Delta}{\varepsilon \ln \left(\frac{1}{\beta} \left(\frac{L\sigma^2}{\mu^2\varepsilon}\right)^{\frac{\alpha}{2(\alpha-1)}}\right)}\right\}.
    \end{equation*}
    This concludes the proof.
\end{proof}

\subsection{Convex Functions}

Now, we focus on the case of convex functions. We start with the following lemma.
\begin{lemma}\label{lem:main_opt_lemma_clipped_SGD_convex}
    Let Assumptions~\ref{as:L_smoothness} and \ref{as:str_cvx} with $\mu = 0$ hold on $Q = B_{2R}(x^*)$, where $R \geq \|x^0 - x^*\|$, and let stepsize  $\gamma$ satisfy $\gamma \leq \frac{1}{L}$. If $x^{k} \in Q$ for all $k = 0,1,\ldots,K+1$, $K \ge 0$, then after $K$ iterations of \algname{clipped-SGD} we have
    \begin{eqnarray}
       \gamma\left(f(\overline{x}^K) -f(x^*)\right) &\leq& \frac{\|x^0 - x^*\|^2 - \|x^{K+1} - x^*\|^2}{K+1}\notag\\
       &&\quad - \frac{2\gamma}{K+1}\sum\limits_{k=0}^K \langle x^k - x^*  - \gamma \nabla f(x^k), \theta_k \rangle  + \frac{\gamma^2}{K+1}\sum\limits_{k=0}^K \|\theta_k\|^2 ,\label{eq:main_opt_lemma_clipped_SGD_star_convex}\\
       \overline{x}^K &=& \frac{1}{K+1}\sum\limits_{k=0}^K x^k, \label{eq:x_avg_clipped_SGD}
    \end{eqnarray}
    where $\theta_k$ is defined in \eqref{eq:theta_k_def_clipped_SGD_non_convex}.
\end{lemma}
\begin{proof}
    Using $x^{k+1} = x^k - \gamma \tnabla f_{\xi^k}(x^{k})$, we derive for all $k = 0,1, \ldots, K$ that
    \begin{eqnarray*}
        \|x^{k+1} - x^*\|^2 &=& \|x^k - x^*\|^2 -2\gamma \langle x^k - x^*,  \tnabla f_{\xi^k}(x^{k})\rangle + \gamma^2\|\tnabla f_{\xi^k}(x^k)\|^2\\
        &=& \|x^k - x^*\|^2 -2\gamma \langle x^k - x^*,  \nabla f(x^{k})\rangle - 2\gamma \langle x^k - x^*,  \theta_k\rangle  + \gamma^2\|\nabla f(x^k) + \theta_k\|^2\\
        &\overset{\eqref{eq:str_cvx}, \mu=0}{\leq}&  \|x^k - x^*\|^2 -2\gamma\left(f(x^k) - f(x^*)\right) - 2\gamma\langle x^k - x^* - \gamma \nabla f(x^k), \theta_k \rangle\\
        &&\quad + \gamma^2\|\nabla f(x^k)\|^2 + \gamma^2 \|\theta_k\|^2\\
        &\overset{\eqref{eq:L_smoothness_cor_2}}{\leq}& \|x^k - x^*\|^2 -2\gamma\left(1 - \gamma L\right)\left(f(x^k) - f(x^*)\right) - 2\gamma\langle x^k - x^* - \gamma \nabla f(x^k), \theta_k \rangle\\
        &&\quad + \gamma^2 \|\theta_k\|^2\\
        &\overset{\gamma\leq\nicefrac{1}{L}}{\leq}& \|x^k - x^*\|^2 -\gamma \left(f(x^k) - f(x^*)\right) - 2\gamma\langle x^k - x^* - \gamma \nabla f(x^k), \theta_k \rangle + \gamma^2 \|\theta_k\|^2.
    \end{eqnarray*}
    Summing up the above inequalities for $k = 0,1,\ldots, K$ and rearranging the terms, we get
    \begin{eqnarray*}
        \frac{\gamma}{K+1}\sum\limits_{k=0}^K\left(f(x^k) - f(x^*)\right) &\leq& \frac{1}{K+1}\sum\limits_{k=0}^K\left(\|x^k - x^*\|^2 - \|x^{k+1} - x^*\|^2\right) - \frac{2\gamma}{K+1}\sum\limits_{k=0}^K \langle x^k - x^* - \gamma \nabla f(x^k), \theta_k \rangle\\
        &&\quad+ \frac{\gamma^2}{K+1}\sum\limits_{k=0}^K \|\theta_k\|^2\\
        &=& \frac{\|x^0 - x^*\|^2 - \|x^{K+1} - x^*\|^2}{K+1} - \frac{2\gamma}{K+1}\sum\limits_{k=0}^K \langle x^k - x^* - \gamma \nabla f(x^k), \theta_k \rangle \\
        &&\quad + \frac{\gamma^2}{K+1}\sum\limits_{k=0}^K \|\theta_k\|^2.  
    \end{eqnarray*}
    Finally, we use the definition of $\overline{x}^K$ and Jensen's inequality and get the result.
\end{proof}

Using this lemma we prove the main convergence result for \algname{clipped-SGD}.
\begin{theorem}[Case 3 from Theorem~\ref{thm:clipped_SGD_main_theorem}]\label{thm:clipped_SGD_convex_case_appendix}
    Let Assumptions~\ref{as:bounded_alpha_moment}, \ref{as:L_smoothness} and \ref{as:str_cvx} with $\mu = 0$ hold on $Q = B_{2R}(x^*)$, where $R \geq \|x^0 - x^*\|$, and
    \begin{gather}
        \gamma \leq \min\left\{\frac{1}{80L \ln \frac{4(K+1)}{\beta}}, \; \frac{R}{108^{\frac{1}{\alpha}}\cdot 20 \sigma  K^{\frac{1}{\alpha}}\left(\ln \frac{4(K+1)}{\beta}\right)^{\frac{\alpha-1}{\alpha}}}\right\},\label{eq:clipped_SGD_step_size_cvx}\\
        \lambda_{k} \equiv \lambda = \frac{R}{40 \gamma\ln\frac{4(K+1)}{\beta}}, \label{eq:clipped_SGD_clipping_level_cvx}
    \end{gather}
    for some $K > 0$ and $\beta \in (0,1]$ such that $\ln\frac{4K}{\beta} \geq 1$. Then, after $K$ iterations of \algname{clipped-SGD} the iterates with probability at least $1 - \beta$ satisfy
    \begin{equation}
        f(\overline{x}^K) - f(x^*) \leq \frac{2R^2}{\gamma(K+1)} \quad \text{and} \quad  \{x^k\}_{k=0}^{K} \subseteq B_{\sqrt{2}R}(x^*). \label{eq:clipped_SGD_convex_case_appendix}
    \end{equation}
    In particular, when $\gamma$ equals the minimum from \eqref{eq:clipped_SGD_non_convex_step_size}, then the iterates produced by \algname{clipped-SGD} after $K$ iterations with probability at least $1-\beta$ satisfy
    \begin{equation}
        f(\overline{x}^K) - f(x^*) = \cO\left(\max\left\{\frac{LR^2 \ln \frac{K}{\beta}}{K}, \frac{\sigma R\ln^{\frac{\alpha-1}{\alpha}}\frac{K}{\beta}}{K^{\frac{\alpha-1}{\alpha}}}\right\}\right), \label{eq:clipped_SGD_convex_case_2_appendix}
    \end{equation}
    meaning that to achieve $f(\overline{x}^K) - f(x^*) \leq \varepsilon$ with probability at least $1 - \beta$ \algname{clipped-SGD} requires
    \begin{equation}
        K = \cO\left(\max\left\{\frac{LR^2}{\varepsilon}, \left(\frac{\sigma R}{\varepsilon}\right)^{\frac{\alpha}{\alpha-1}} \ln \left(\frac{1}{\beta} \left(\frac{\sigma R}{\varepsilon}\right)^{\frac{\alpha}{\alpha-1}}\right) \right\}\right)\quad \text{iterations/oracle calls.} \label{eq:clipped_SGD_convex_case_complexity_appendix}
    \end{equation}
\end{theorem}
\begin{proof}
    Let $R_k = \|x^k - x^*\|$ for all $k\geq 0$. Next, our goal is to show by induction that $R_{l} \leq 2R$ with high probability, which allows to apply the result of Lemma~\ref{lem:main_opt_lemma_clipped_SGD_convex} and then use Bernstein's inequality to estimate the stochastic part of the upper-bound. More precisely, for each $k = 0,\ldots, K+1$ we consider probability event $E_k$ defined as follows: inequalities
    \begin{eqnarray}
        - 2\gamma\sum\limits_{l=0}^{t-1} \langle x^l - x^* - \gamma \nabla f(x^l), \theta_l \rangle  + \gamma^2\sum\limits_{l=0}^{t-1} \|\theta_l\|^2 &\leq&  R^2, \label{eq:clipped_SGD_convex_induction_inequality_1}\\
        R_t &\leq& \sqrt{2}R \label{eq:clipped_SGD_convex_induction_inequality_2}
    \end{eqnarray}
    hold for all $t = 0,1,\ldots, k$ simultaneously. We want to prove via induction that $\PP\{E_k\} \geq 1 - \nicefrac{k\beta}{(K+1)}$ for all $k = 0,1,\ldots, K+1$. For $k = 0$ the statement is trivial. Assume that the statement is true for some $k = T - 1 \leq K$: $\PP\{E_{T-1}\} \geq 1 - \nicefrac{(T-1)\beta}{(K+1)}$. One needs to prove that $\PP\{E_{T}\} \geq 1 - \nicefrac{T\beta}{(K+1)}$. First, we notice that probability event $E_{T-1}$ implies that $x_t \in B_{\sqrt{2}R}(x^*)$ for all $t = 0,1,\ldots, T-1$. Moreover, $E_{T-1}$ implies
    \begin{eqnarray*}
        \|x^T - x^*\| = \|x^{T-1} - x^* - \gamma \tnabla f_{\xi^{T-1}}(x^{T-1})\| \leq \|x^{T-1} - x^*\| + \gamma\|\tnabla f_{\xi^{T-1}}(x^{T-1})\| \leq \sqrt{2}R + \gamma \lambda \overset{\eqref{eq:clipped_SGD_clipping_level_cvx}}{\leq} 2R,
    \end{eqnarray*}
    i.e., $x^0, x^1, \ldots, x^T \in B_{2R}(x^*)$. Therefore, $E_{T-1}$ implies $\{x^k\}_{k=0}^{T} \subseteq Q$, meaning that the assumptions of Lemma~\ref{lem:main_opt_lemma_clipped_SGD_convex} are satisfied and we have 
    \begin{eqnarray}
        \gamma \left(f(\overline{x}^{t-1}) -f(x^*)\right) &\leq& \frac{\|x^0 - x^*\|^2 - \|x^{t} - x^*\|^2}{t}\notag\\
       &&\quad - \frac{2\gamma}{t}\sum\limits_{l=0}^{t-1} \langle x^l - x^*  - \gamma \nabla f(x^l), \theta_l \rangle  + \frac{\gamma^2}{t}\sum\limits_{l=0}^{t-1} \|\theta_l\|^2 \label{eq:clipped_SGD_convex_technical_1}
    \end{eqnarray}
    for all $t=1,\ldots, T$ simultaneously and for all $t = 1, \ldots, T-1$ this probability event also implies that
    \begin{eqnarray}
        f(\overline{x}^{t-1}) -f(x^*) &\leq& \frac{1}{\gamma t}\left(R^2 - 2\gamma\sum\limits_{l=0}^{t-1} \langle x^l - x^* - \gamma \nabla f(x^l), \theta_l \rangle  + \gamma^2\sum\limits_{l=0}^{t-1} \|\theta_l\|^2\right)  \overset{\eqref{eq:clipped_SGD_convex_induction_inequality_1}}{\leq} \frac{2R^2}{\gamma t}. \label{eq:clipped_SGD_convex_technical_1_1}
    \end{eqnarray}
    Taking into account that $f(\overline{x}^{T-1}) -f(x^*) \geq 0$, we also derive from \eqref{eq:clipped_SGD_convex_technical_1} that $E_{T-1}$ implies
    \begin{eqnarray}
        R_T^2 \leq R^2 - 2\gamma\sum\limits_{l=0}^{t-1} \langle x^l - x^* - \gamma \nabla f(x^l), \theta_l \rangle  + \gamma^2\sum\limits_{l=0}^{t-1} \|\theta_l\|^2. \label{eq:clipped_SGD_convex_technical_2}
    \end{eqnarray}
    Next, we define random vectors
    \begin{equation}
        \eta_t = \begin{cases} x^t - x^* - \gamma \nabla f(x^t),& \text{if } \|x^t - x^* - \gamma \nabla f(x^t)\|^2 \leq 2R,\\ 0,&\text{otherwise}, \end{cases} \notag
    \end{equation}
    for all $t = 0,1,\ldots, T-1$. By definition, these random vectors are bounded with probability 1
    \begin{equation}
        \|\eta_t\| \leq 2R. \label{eq:clipped_SGD_convex_technical_6}
    \end{equation}
    Moreover, for $t = 0,\ldots, T-1$ event $E_{T-1}$ implies
    \begin{eqnarray}
        \|\nabla f(x^{t})\| &\overset{\eqref{eq:L_smoothness}}{\leq}& L\|x^t - x^*\|  \overset{\eqref{eq:clipped_SGD_convex_induction_inequality_2}}{\leq}  \sqrt{2}LR \overset{\eqref{eq:clipped_SGD_step_size_cvx},\eqref{eq:clipped_SGD_clipping_level_cvx}}{\leq} \frac{\lambda}{2},\label{eq:clipped_SGD_convex_technical_4}\\
        \|x^t - x^* - \gamma \nabla f(x^t)\| &\leq& \|x^t - x^*\| + \gamma \|\nabla f(x^t)\| \overset{\eqref{eq:clipped_SGD_convex_technical_4}}{\leq} \sqrt{2}R (1+L\gamma) \overset{\eqref{eq:clipped_SGD_step_size_cvx}}{\leq} 2R.\notag
    \end{eqnarray}
    Next, we define the unbiased part and the bias of $\theta_{t}$ as $\theta_{t}^u$ and $\theta_{t}^b$, respectively:
    \begin{equation}
        \theta_{t}^u = \tnabla f_{\xi^t}(x^{t}) - \EE_{\xi^t}\left[\tnabla f_{\xi^t}(x^{t})\right],\quad \theta_{t}^b = \EE_{\xi^t}\left[\tnabla f_{\xi^t}(x^{t})\right] - \nabla f(x^{t}). \label{eq:clipped_SGD_convex_theta_u_b}
    \end{equation}
    We notice that $\theta_{t} = \theta_{t}^u + \theta_{t}^b$. Using new notation, we get that $E_{T-1}$ implies
    \begin{eqnarray}
        R_T^2 
        &\leq& R^2 \underbrace{-2\gamma\sum\limits_{t=0}^{T-1}\langle \theta_t^u, \eta_t\rangle}_{\circledOne}  \underbrace{-2\gamma\sum\limits_{t=0}^{T-1}\langle \theta_t^b, \eta_t\rangle}_{\circledTwo} + \underbrace{2\gamma^2\sum\limits_{t=0}^{T-1}\left(\left\|\theta_{t}^u\right\|^2 - \EE_{\xi^t}\left[\left\|\theta_{t}^u\right\|^2\right]\right)}_{\circledThree}\notag\\
        &&\quad + \underbrace{2\gamma^2\sum\limits_{t=0}^{T-1}\EE_{\xi^t}\left[\left\|\theta_{t}^u\right\|^2\right]}_{\circledFour} + \underbrace{2\gamma^2\sum\limits_{t=0}^{T-1}\left\|\theta_{t}^b\right\|^2}_{\circledFive}. \label{eq:clipped_SGD_convex_technical_7}
    \end{eqnarray}

    It remains to derive good enough high-probability upper-bounds for the terms $\circledOne, \circledTwo, \circledThree, \circledFour, \circledFive$, i.e., to finish our inductive proof we need to show that $\circledOne + \circledTwo + \circledThree + \circledFour + \circledFive \leq R^2$ with high probability. In the subsequent parts of the proof, we will need to use many times the bounds for the norm and second moments of $\theta_{t}^u$ and $\theta_{t}^b$. First, by definition of clipping operator, we have with probability $1$ that
     \begin{equation}
        \|\theta_{t}^u\| \leq 2\lambda. \label{eq:clipped_SGD_convex_norm_theta_u_bound}
    \end{equation}
    Moreover, since $E_{T-1}$ implies that $\|\nabla f(x^{t})\| \leq \nicefrac{\lambda}{2}$ for $t = 0,1,\ldots,T-1$ (see \eqref{eq:clipped_SGD_convex_technical_4}), then, in view of Lemma~\ref{lem:bias_and_variance_clip}, we have that $E_{T-1}$ implies
    \begin{eqnarray}
        \|\theta_{t}^b\| &\leq& \frac{2^\alpha\sigma^\alpha}{\lambda^{\alpha-1}}, \label{eq:clipped_SGD_convex_norm_theta_b_bound} \\
        \EE_{\xi^t}\left[\|\theta_{t}^u\|^2\right] &\leq& 18 \lambda^{2-\alpha}\sigma^\alpha. \label{eq:clipped_SGD_convex_second_moment_theta_u_bound}
    \end{eqnarray}

    \textbf{Upper bound for $\circledOne$.} By definition of $\theta_{t}^u$, we have $\EE_{\xi^t}[\theta_{t}^u] = 0$ and
    \begin{equation}
        \EE_{\xi^t}\left[-2\gamma\langle\theta_t^u, \eta_t\rangle\right] = 0. \notag
    \end{equation}
    Next, sum $\circledOne$ has bounded with probability $1$ terms:
    \begin{equation}
        |2\gamma\left\la \theta_{t}^u, \eta_t\right\ra| \leq 2\gamma \|\theta_{t}^u\| \cdot \|\eta_t\| \overset{\eqref{eq:clipped_SGD_convex_technical_6},\eqref{eq:clipped_SGD_convex_norm_theta_u_bound}}{\leq} 8\gamma \lambda R\overset{\eqref{eq:clipped_SGD_clipping_level_cvx}}{=} \frac{R^2}{5\ln\frac{4(K+1)}{\beta}} \eqdef c. \label{eq:clipped_SGD_convex_technical_8} 
    \end{equation}
    The summands also have bounded conditional variances $\sigma_t^2 \eqdef \EE_{\xi^t}[4\gamma^2\langle\theta_t^u, \eta_t\rangle^2]$:
    \begin{equation}
        \sigma_t^2 \leq \EE_{\xi^t}\left[4\gamma^2\|\theta_{t}^u\|^2\cdot \|\eta_t\|^2\right] \overset{\eqref{eq:clipped_SGD_convex_technical_6}}{\leq} 16\gamma^2 R^2 \EE_{\xi^t}\left[\|\theta_{t}^u\|^2\right]. \label{eq:clipped_SGD_convex_technical_9}
    \end{equation}
    In other words, we showed that $\{-2\gamma\left\la \theta_{t}^u, \eta_t\right\ra\}_{t=0}^{T-1}$ is a bounded martingale difference sequence with bounded conditional variances $\{\sigma_t^2\}_{t=0}^{T-1}$. Next, we apply Bernstein's inequality (Lemma~\ref{lem:Bernstein_ineq}) with $X_t = -2\gamma \left\la \theta_{t}^u, \eta_t\right\ra$, parameter $c$ as in \eqref{eq:clipped_SGD_convex_technical_8}, $b = \frac{R^2}{5}$, $G = \frac{R^4}{150\ln\frac{4(K+1)}{\beta}}$:
    \begin{equation*}
        \PP\left\{|\circledOne| > \frac{R^2}{5}\quad \text{and}\quad \sum\limits_{t=0}^{T-1} \sigma_{t}^2 \leq \frac{R^4}{150\ln\frac{4(K+1)}{\beta}}\right\} \leq 2\exp\left(- \frac{b^2}{2G + \nicefrac{2cb}{3}}\right) = \frac{\beta}{2(K+1)}.
    \end{equation*}
    Equivalently, we have
    \begin{equation}
        \PP\left\{ E_{\circledOne} \right\} \geq 1 - \frac{\beta}{2(K+1)},\quad \text{for}\quad E_{\circledOne} = \left\{ \text{either} \quad  \sum\limits_{t=0}^{T-1} \sigma_{t}^2 > \frac{R^4}{150\ln\frac{4(K+1)}{\beta}} \quad \text{or}\quad |\circledOne| \leq \frac{R^2}{5}\right\}. \label{eq:clipped_SGD_convex_sum_1_upper_bound}
    \end{equation}
    In addition, $E_{T-1}$ implies that
    \begin{eqnarray}
        \sum\limits_{t=0}^{T-1} \sigma_{t}^2 &\overset{\eqref{eq:clipped_SGD_convex_technical_9}}{\leq}& 16\gamma^2 R^2 \sum\limits_{t=0}^{T-1}  \EE_{\xi^t}\left[\|\theta_{t}^u\|^2\right] \overset{\eqref{eq:clipped_SGD_convex_second_moment_theta_u_bound}}{\leq} 288\gamma^2 R^2 \sigma^{\alpha}T \lambda^{2-\alpha}\notag\\
        &\overset{\eqref{eq:clipped_SGD_clipping_level_cvx}}{=}& \frac{9 \cdot 40^{\alpha}R^{4-\alpha}\sigma^\alpha T\gamma^{\alpha}}{50 \ln^{2-\alpha}\frac{4(K+1)}{\beta}} \overset{\eqref{eq:clipped_SGD_step_size_cvx}}{\leq} \frac{R^4}{150 \ln\frac{4(K+1)}{\beta}}. \label{eq:clipped_SGD_convex_sum_1_variance_bound}
    \end{eqnarray}

    \textbf{Upper bound for $\circledTwo$.} From $E_{T-1}$ it follows that
\begin{eqnarray}
        \circledTwo &=& -2\gamma\sum\limits_{t=0}^{T-1}\langle \theta_t^b, \eta_t \rangle \leq 2\gamma\sum\limits_{t=0}^{T-1}\|\theta_{t}^b\|\cdot \|\eta_t\| \overset{\eqref{eq:clipped_SGD_convex_technical_6},\eqref{eq:clipped_SGD_convex_norm_theta_b_bound}}{\leq}  \frac{4\cdot 2^\alpha \gamma \sigma^\alpha T R}{\lambda^{\alpha-1}} \notag\\ &\overset{\eqref{eq:clipped_SGD_clipping_level_cvx}}{=}& \frac{80^\alpha}{10} \cdot \frac{\sigma^\alpha T R^{2-\alpha}\gamma^\alpha}{\ln^{1-\alpha}\frac{4(K+1)}{\beta}}\overset{\eqref{eq:clipped_SGD_step_size_cvx}}{\leq} \frac{R^2}{5}. \label{eq:clipped_SGD_convex_sum_2_upper_bound}
\end{eqnarray}

\textbf{Upper bound for $\circledThree$.} First, we have
    \begin{equation}
        \EE_{\xi^t}\left[2\gamma^2\left(\left\|\theta_{t}^u\right\|^2 - \EE_{\xi^t}\left[\left\|\theta_{t}^u\right\|^2\right]\right)\right] = 0. \notag
    \end{equation}
    Next, sum $\circledThree$ has bounded with probability $1$ terms:
    \begin{eqnarray}
        \left|2\gamma^2\left(\left\|\theta_{t}^u\right\|^2 - \EE_{\xi^t}\left[\left\|\theta_{t}^u\right\|^2\right]\right)\right| &\leq& 2\gamma^2\left( \|\theta_{t}^u\|^2 +   \EE_{\xi^t}\left[\left\|\theta_{t}^u\right\|^2\right]\right)\notag\\
        &\overset{\eqref{eq:clipped_SGD_convex_norm_theta_u_bound}}{\leq}& 16\gamma^2\lambda^2\overset{\eqref{eq:clipped_SGD_clipping_level_cvx}}{=} \frac{R^2}{100\ln^2\frac{4(K+1)}{\beta}} \le \frac{R^2}{5 \ln\frac{4(K+1)}{\beta}}\eqdef c. \label{eq:clipped_SGD_convex_technical_10}
    \end{eqnarray}
    The summands also have bounded conditional variances $\widetilde\sigma_t^2 \eqdef \EE_{\xi^t}\left[4\gamma^4\left(\left\|\theta_{t}^u\right\|^2 - \EE_{\xi^t}\left[\left\|\theta_{t}^u\right\|^2\right]\right)^2\right]$:
    \begin{eqnarray}
        \widetilde\sigma_t^2 &\overset{\eqref{eq:clipped_SGD_convex_technical_10}}{\leq}& \frac{R^2}{5 \ln\frac{4(K+1)}{\beta}} \EE_{\xi^t}\left[2\gamma^2\left|\left\|\theta_{t}^u\right\|^2 - \EE_{\xi^t}\left[\left\|\theta_{t}^u\right\|^2\right]\right|\right] \leq \frac{4\gamma^2 R^2}{5\ln\frac{4(K+1)}{\beta}} \EE_{\xi^t}\left[\|\theta_{t}^u\|^2\right], \label{eq:clipped_SGD_convex_technical_11}
    \end{eqnarray}
    since $\ln\frac{4(K+1)}{\beta} \geq 1$. In other words, we showed that $\left\{2\gamma^2\left(\left\|\theta_{t}^u\right\|^2 - \EE_{\xi^t}\left[\left\|\theta_{t}^u\right\|^2\right]\right)\right\}_{t=0}^{T-1}$ is a bounded martingale difference sequence with bounded conditional variances $\{\widetilde\sigma_t^2\}_{t=0}^{T-1}$. Next, we apply Bernstein's inequality (Lemma~\ref{lem:Bernstein_ineq}) with $X_t = 2\gamma^2\left(\left\|\theta_{t}^u\right\|^2 - \EE_{\xi^t}\left[\left\|\theta_{t}^u\right\|^2\right]\right)$, parameter $c$ as in \eqref{eq:clipped_SGD_convex_technical_10}, $b = \frac{R^2}{5}$, $G = \frac{R^4}{150\ln\frac{4(K+1)}{\beta}}$:
    \begin{equation*}
        \PP\left\{|\circledThree| > \frac{R^2}{5}\quad \text{and}\quad \sum\limits_{t=0}^{T-1} \widetilde\sigma_{t}^2 \leq \frac{R^4}{150\ln\frac{4(K+1)}{\beta}}\right\} \leq 2\exp\left(- \frac{b^2}{2G + \nicefrac{2cb}{3}}\right) = \frac{\beta}{2(K+1)}.
    \end{equation*}
    Equivalently, we have
    \begin{equation}
        \PP\left\{ E_{\circledThree} \right\} \geq 1 - \frac{\beta}{2(K+1)},\quad \text{for}\quad E_{\circledThree} = \left\{ \text{either} \quad  \sum\limits_{t=0}^{T-1} \widetilde\sigma_{t}^2 > \frac{R^4}{150\ln\frac{4(K+1)}{\beta}} \quad \text{or}\quad |\circledThree| \leq \frac{R^2}{5}\right\}. \label{eq:clipped_SGD_convex_sum_3_upper_bound}
    \end{equation}
    In addition, $E_{T-1}$ implies that
    \begin{eqnarray}
        \sum\limits_{t=0}^{T-1} \widetilde\sigma_{t}^2 &\overset{\eqref{eq:clipped_SGD_convex_technical_11}}{\leq}& \frac{4\gamma^2R^2}{5\ln\frac{4(K+1)}{\beta}} \sum\limits_{t=0}^{T-1}  \EE_{\xi^t}\left[\|\theta_{t}^u\|^2\right] \overset{\eqref{eq:clipped_SGD_convex_second_moment_theta_u_bound}}{\leq}  
        \frac{72\gamma^2R^2\lambda^{2-\alpha}\sigma^{\alpha}T}{5\ln\frac{4(K+1)}{\beta}}\notag\\
        &\overset{\eqref{eq:clipped_SGD_clipping_level_cvx}}{=}&
        \frac{9\cdot 40^\alpha}{1000}\cdot\frac{\sigma^\alpha T R^{4-\alpha}\gamma^{\alpha}}{\ln^{3-\alpha}\frac{4(K+1)}{\beta}} \overset{\eqref{eq:clipped_SGD_step_size_cvx}}{\leq} \frac{R^4}{150 \ln\frac{4(K+1)}{\beta}}. \label{eq:clipped_SGD_convex_sum_3_variance_bound}
    \end{eqnarray}

    \textbf{Upper bound for $\circledFour$.} From $E_{T-1}$ it follows that
\begin{eqnarray}
        \circledFour &=& 2\gamma^2\sum\limits_{t=0}^{T-1}\EE_{\xi^t}\left[\left\|\theta_{t}^u\right\|^2\right] \overset{\eqref{eq:clipped_SGD_convex_second_moment_theta_u_bound}}{\leq} 36\gamma^2\lambda^{2-\alpha}\sigma^{\alpha}T \overset{\eqref{eq:clipped_SGD_clipping_level_cvx}}{=}
        \frac{9\cdot 40^{\alpha}}{400}\cdot\frac{\gamma^{\alpha}\sigma^{\alpha}TR^{2-\alpha}}{\ln^{2-\alpha}\frac{4(K+1)}{\beta}}\overset{\eqref{eq:clipped_SGD_step_size_cvx}}{\leq} \frac{R^2}{5}.\label{eq:clipped_SGD_convex_sum_4_upper_bound}
\end{eqnarray}

\textbf{Upper bound for $\circledFive$.} From $E_{T-1}$ it follows that
\begin{eqnarray}
        \circledFive &=& 2\gamma^2\sum\limits_{t=0}^{T-1}\left\|\theta_{t}^b\right\|^2 \overset{\eqref{eq:clipped_SGD_convex_norm_theta_b_bound}}{\leq} \frac{2 \cdot 4^\alpha \sigma^{2\alpha}T\gamma^2}{\lambda^{2(\alpha-1)}} \overset{\eqref{eq:clipped_SGD_clipping_level_cvx}}{=}   \frac{6400^\alpha}{800}\cdot \frac{\sigma^{2\alpha}T\gamma^{2\alpha}R^{2-2\alpha}}{\ln^{2(1-\alpha)}\frac{4(K+1)}{\beta}}\overset{\eqref{eq:clipped_SGD_step_size_cvx}}{\leq} \frac{R^2}{5}.\label{eq:clipped_SGD_convex_sum_5_upper_bound}
\end{eqnarray}

Now, we have the upper bounds for  $\circledOne, \circledTwo, \circledThree, \circledFour, \circledFive$. In particular, probability event $E_{T-1}$ implies
\begin{gather*}
        R_T^2 \overset{\eqref{eq:clipped_SGD_convex_technical_7}}{\leq} R^2 + \circledOne + \circledTwo + \circledThree + \circledFour + \circledFive,\\
        \circledTwo \overset{\eqref{eq:clipped_SGD_convex_sum_2_upper_bound}}{\leq} \frac{R^2}{5},\quad \circledFour \overset{\eqref{eq:clipped_SGD_convex_sum_4_upper_bound}}{\leq} \frac{R^2}{5}, \quad \circledFive \overset{\eqref{eq:clipped_SGD_convex_sum_5_upper_bound}}{\leq} \frac{R^2}{5},\\
        \sum\limits_{t=0}^{T-1} \sigma_t^2 \overset{\eqref{eq:clipped_SGD_convex_sum_1_variance_bound}}{\leq} \frac{R^4}{150 \ln\frac{4(K+1)}{\beta}},\quad \sum\limits_{t=0}^{T-1} \widetilde\sigma_t^2 \overset{\eqref{eq:clipped_SGD_convex_sum_3_variance_bound}}{\leq} \frac{R^4}{150 \ln\frac{4(K+1)}{\beta}}.
\end{gather*}
    Moreover, we also have (see \eqref{eq:clipped_SGD_convex_sum_1_upper_bound}, \eqref{eq:clipped_SGD_convex_sum_3_upper_bound} and our induction assumption)
\begin{equation*}
        \PP\{E_{T-1}\} \geq 1 - \frac{(T-1)\beta}{K+1},\quad \PP\{E_{\circledOne}\} \geq 1 - \frac{\beta}{2(K+1)},\quad \PP\{E_{\circledThree}\} \geq 1 - \frac{\beta}{2(K+1)},
\end{equation*}
    where 
\begin{eqnarray*}
        E_{\circledOne} &=& \left\{ \text{either} \quad  \sum\limits_{t=0}^{T-1} \sigma_{t}^2 > \frac{R^4}{150\ln\frac{4(K+1)}{\beta}} \quad \text{or}\quad |\circledOne| \leq \frac{R^2}{5}\right\},\\
        E_{\circledThree} &=& \left\{ \text{either} \quad  \sum\limits_{t=0}^{T-1} \widetilde\sigma_{t}^2 > \frac{R^4}{150\ln\frac{4(K+1)}{\beta}} \quad \text{or}\quad |\circledThree| \leq \frac{R^2}{5}\right\}.
\end{eqnarray*}
    Thus, probability event $E_{T-1} \cap E_{\circledOne} \cap E_{\circledThree}$ implies
\begin{eqnarray}
        R_T^2 &\leq& R^2 + \frac{R^2}{5} + \frac{R^2}{5} + \frac{R^2}{5} + \frac{R^2}{5} + \frac{R^2}{5} = 2R^2, \notag
\end{eqnarray}
    which is equivalent to \eqref{eq:clipped_SGD_convex_induction_inequality_1} and \eqref{eq:clipped_SGD_convex_induction_inequality_2} for $t = T$, and 
\begin{equation*}
        \PP\{E_T\} \geq \PP\left\{E_{T-1} \cap E_{\circledOne} \cap E_{\circledThree}\right\} = 1 - \PP\left\{\overline{E}_{T-1} \cup \overline{E}_{\circledOne} \cup \overline{E}_{\circledThree}\right\} \geq 1 - \PP\{\overline{E}_{T-1}\} - \PP\{\overline{E}_{\circledOne}\} - \PP\{\overline{E}_{\circledThree}\} \geq 1 - \frac{T\beta}{K+1}.
\end{equation*}
    This finishes the inductive part of our proof, i.e., for all $k = 0,1,\ldots,K+1$ we have $\PP\{E_k\} \geq 1 - \nicefrac{k\beta}{(K+1)}$. In particular, for $k = K+1$ we have that with probability at least $1 - \beta$
\begin{equation*}
        f(\overline{x}^K) - f(x^*) \overset{\eqref{eq:clipped_SGD_convex_technical_1_1}}{\leq}\frac{2R^2}{\gamma(K+1)}
\end{equation*}
    and $\{x^k\}_{k=0}^{K} \subseteq Q$, which follows from \eqref{eq:clipped_SGD_convex_induction_inequality_2}.
    
Finally, if
\begin{equation*}
        \gamma \leq \min\left\{\frac{1}{80L \ln \frac{4(K+1)}{\beta}}, \; \frac{R}{108^{\frac{1}{\alpha}}\cdot 20 \sigma  K^{\frac{1}{\alpha}}\left(\ln \frac{4(K+1)}{\beta}\right)^{\frac{\alpha-1}{\alpha}}}\right\},
\end{equation*}
then with probability at least $1-\beta$
\begin{eqnarray*}
        f(\overline{x}^K) - f(x^*) &\leq& \frac{2R^2}{\gamma(K+1)}\\
        &=& \max\left\{\frac{160 LR^2 \ln \frac{4(K+1)}{\beta}}{K+1}, \frac{40\cdot 108^{\frac{1}{\alpha}}\sigma RK^{\frac{1}{\alpha}}\left(\ln\frac{4(K+1)}{\beta}\right)^{\frac{
        \alpha-1
        }{\alpha}}}{K+1}\right\}\notag\\
        &=& \cO\left(\max\left\{\frac{L R^2 \ln \frac{K}{\beta}}{K}, \frac{ \sigma R\ln^{\frac{\alpha-1}{\alpha}}\frac{K}{\beta}}{K^{\frac{\alpha-1}{\alpha}}}\right\}\right).
\end{eqnarray*}
To get $f(\overline{x}^K) - f(x^*) \leq \varepsilon$ with probability at least $1-\beta$ it is sufficient to choose $K$ such that both terms in the maximum above are $\cO(\varepsilon)$. This leads to
\begin{equation*}
         K = \cO\left(\max\left\{\frac{LR^2}{\varepsilon} \ln\frac{L R^2}{\varepsilon\beta}, \left(\frac{\sigma R}{\varepsilon}\right)^{\frac{\alpha}{\alpha-1}} \ln \left(\frac{1}{\beta} \left(\frac{\sigma R}{\varepsilon}\right)^{\frac{\alpha}{\alpha-1}}\right) \right\}\right),
\end{equation*}
which concludes the proof.

\end{proof}

\subsection{Quasi-Strongly Convex Functions}

Finally, we consider \algname{clipped-SGD} under smoothness and quasi-strong convexity assumptions. As the next lemma shows, the gradient of such function is quasi-strongly monotone and star-cocoercive operator.

\begin{lemma}
    Consider differentiable function $f: \R^d \to \R$. If $f$ satisfies Assumption~\ref{as:QSC} on some set $Q$ with parameter $\mu$, then operator $F(x) = \nabla f(x)$ satisfies Assumption~\ref{as:QSM} on $Q$ with parameter $\nicefrac{\mu}{2}$. If $f$ satisfies Assumptions~\ref{as:L_smoothness} and \ref{as:QSC} with $\mu = 0$ on some set $Q$, then operator $F(x) = \nabla f(x)$ satisfies Assumption~\ref{as:star-cocoercivity} on $Q$ with $\ell = 2L$.
\end{lemma}
\begin{proof}
    We start with the first part. Assumption~\ref{as:QSC} on set $Q$ means that for any $x \in Q$
    \begin{equation*}
        f(x^*) \geq f(x) + \langle \nabla f(x), x^* - x \rangle + \frac{\mu}{2}\|x- x^*\|^2.
    \end{equation*}
    For $F(x) = \nabla f(x)$ it implies that for all $x \in Q$
    \begin{equation*}
        \langle F(x), x - x^* \rangle \geq f(x) - f(x^*) + \frac{\mu}{2}\|x - x^*\|^2 \geq \frac{\mu}{2}\|x - x^*\|^2,
    \end{equation*}
    i.e., Assumption~\ref{as:QSM} holds on $Q$ with parameter $\nicefrac{\mu}{2}$ for operator $F(x)$.

    Next, we prove the second part. Assume that $f$ satisfies Assumptions~\ref{as:L_smoothness} and \ref{as:QSC} with $\mu = 0$ on some set $Q$. Our goal is to show that $F(x) = \nabla f(x)$ satisfies Assumption~\ref{as:star-cocoercivity} on $Q$. In view of \citep[Lemma C.6]{gorbunov2022extragradient}, this is equivalent to showing that operator $\Id - \frac{1}{L}F$ is non-expansive around $x^*$, i.e., we need to show that $\|(\Id - \frac{1}{L}F)(x) - (\Id - \frac{1}{L}F)(x^*)\| \leq \|x-x^*\|$ for any $x \in Q$. We have
    \begin{eqnarray*}
        \left\|\left(\Id - \frac{1}{L}F\right)(x) - \left(\Id - \frac{1}{L}F\right)(x^*)\right\|^2 &=& \left\|x-x^* - \frac{1}{L}F(x)\right\|^2\\
        &=&\|x - x^*\|^2 - \frac{2}{L}\langle x - x^*, F(x) \rangle + \frac{1}{L^2}\|F(x)\|^2\\
        &=& \|x - x^*\|^2 - \frac{2}{L}\langle x - x^*, \nabla f(x) \rangle + \frac{1}{L^2}\|\nabla f(x)\|^2\\
        &\overset{\eqref{eq:QSC}, \eqref{eq:L_smoothness_cor_2}}{\leq}& \|x - x^*\|^2 - \frac{2}{L}\left(f(x) - f(x^*)\right) + \frac{2}{L}\left(f(x) - f(x^*)\right)\\
        &=& \|x-x^*\|^2.
    \end{eqnarray*}
    This finishes the proof.
\end{proof}

Therefore, using the result of Theorem~\ref{thm:main_result_str_mon_SGDA} with $\ell := 2L$ and $\mu := \nicefrac{\mu}{2}$, we get the convergence result for \algname{clipped-SGD} under smoothness and quasi-strong convexity assumptions.

\begin{theorem}[Case 4 in Theorem~\ref{thm:clipped_SGD_main_theorem}]\label{thm:main_result_QSC_SGD}
    Let Assumptions~\ref{as:bounded_alpha_moment}, \ref{as:L_smoothness}, \ref{as:QSC}, hold for $Q = B_{2R}(x^*) = \{x\in\R^d\mid \|x - x^*\| \leq 2R\}$, where $R \geq \|x^0 - x^*\|$, and
    \begin{eqnarray}
        0< \gamma &\leq& \min\left\{\frac{1}{800 L\ln \tfrac{4(K+1)}{\beta}}, \frac{2\ln(B_K)}{\mu(K+1)}\right\}, \label{eq:gamma_SGD_QSC}\\
        B_K &=& \max\left\{2, \frac{(K+1)^{\frac{2(\alpha-1)}{\alpha}}\mu^2R^2}{4\cdot5400^{\frac{2}{\alpha}}\sigma^2\ln^{\frac{2(\alpha-1)}{\alpha}}\left(\frac{4(K+1)}{\beta}\right)\ln^2(B_K)} \right\} \label{eq:B_K_SGD_QSC_1} \\
        &=& \cO\left(\max\left\{2, \frac{K^{\frac{2(\alpha-1)}{\alpha}}\mu^2R^2}{\sigma^2\ln^{\frac{2(\alpha-1)}{\alpha}}\left(\frac{K}{\beta}\right)\ln^2\left(\max\left\{2, \frac{K^{\frac{2(\alpha-1)}{\alpha}}\mu^2R^2}{\sigma^2\ln^{\frac{2(\alpha-1)}{\alpha}}\left(\frac{K}{\beta}\right)} \right\}\right)} \right\}\right\}, \label{eq:B_K_SGD_QSC_2} \\
        \lambda_k &=& \frac{\exp(-\gamma(\nicefrac{\mu}{2})(1 + \nicefrac{k}{2}))R}{120\gamma \ln \tfrac{4(K+1)}{\beta}}, \label{eq:lambda_SGD_QSC}
    \end{eqnarray}
    for some $K \geq 0$ and $\beta \in (0,1]$ such that $\ln \tfrac{4(K+1)}{\beta} \geq 1$. Then, after $K$ iterations the iterates produced by \algname{clipped-SGD} with probability at least $1 - \beta$ satisfy 
    \begin{equation}
        \|x^{K+1} - x^*\|^2 \leq 2\exp(-\gamma(\nicefrac{\mu}{2})(K+1))R^2. \label{eq:main_result_str_cvx_SGD}
    \end{equation}
    In particular, when $\gamma$ equals the minimum from \eqref{eq:gamma_SGD_QSC}, then the iterates produced by \algname{clipped-SGD} after $K$ iterations with probability at least $1-\beta$ satisfy
    \begin{equation}
       \|x^{K} - x^*\|^2 = \cO\left(\max\left\{R^2\exp\left(- \frac{\mu K}{\ell \ln \tfrac{K}{\beta}}\right), \frac{\sigma^2\ln^{\frac{2(\alpha-1)}{\alpha}}\left(\frac{K}{\beta}\right)\ln^2\left(\max\left\{2, \frac{K^{\frac{2(\alpha-1)}{\alpha}}\mu^2R^2}{\sigma^2\ln^{\frac{2(\alpha-1)}{\alpha}}\left(\frac{K}{\beta}\right)} \right\}\right)}{K^{\frac{2(\alpha-1)}{\alpha}}\mu^2}\right\}\right), \label{eq:clipped_SGD_QSC_case_2_appendix}
    \end{equation}
    meaning that to achieve $\|x^{K} - x^*\|^2 \leq \varepsilon$ with probability at least $1 - \beta$ \algname{clipped-SGD} requires
    \begin{equation}
        K = \cO\left(\frac{L}{\mu}\ln\left(\frac{R^2}{\varepsilon}\right)\ln\left(\frac{L}{\mu \beta}\ln\frac{R^2}{\varepsilon}\right), \left(\frac{\sigma^2}{\mu^2\varepsilon}\right)^{\frac{\alpha}{2(\alpha-1)}}\ln \left(\frac{1}{\beta} \left(\frac{\sigma^2}{\mu^2\varepsilon}\right)^{\frac{\alpha}{2(\alpha-1)}}\right)\ln^{\frac{\alpha}{\alpha-1}}\left(B_\varepsilon\right)\right) \label{eq:clipped_SGD_QSC_case_complexity_appendix}
    \end{equation}
    iterations/oracle calls, where
    \begin{equation*}
        B_\varepsilon = \max\left\{2, \frac{R^2}{\varepsilon \ln \left(\frac{1}{\beta} \left(\frac{4\sigma^2}{\mu^2\varepsilon}\right)^{\frac{\alpha}{2(\alpha-1)}}\right)}\right\}.
    \end{equation*}
\end{theorem}

\clearpage

\section{Missing Proofs for \algname{clipped-SSTM} and \algname{R-clipped-SSTM}}

In this section, we provide the complete formulation of the main results for \algname{clipped-SSTM} and \algname{R-clipped-SSTM} and the missing proofs. For brevity, we will use the following notation: $\tnabla f_{\xi^k}(x^{k+1}) = \clip\left(\nabla f_{\xi^k}(x^{k+1}), \lambda_k\right)$.

\begin{algorithm}[h]
\caption{Clipped Stochastic Similar Triangles Method (\algname{clipped-SSTM}) \citep{gorbunov2020stochastic}}
\label{alg:clipped-SSTM}   
\begin{algorithmic}[1]
\REQUIRE starting point $x^0$, number of iterations $K$, stepsize parameter $a>0$, clipping levels $\{\lambda_k\}_{k=0}^{K-1}$, smoothness constant $L$.
\STATE Set $A_0 = \alpha_0 = 0$, $y^0 = z^0 = x^0$
\FOR{$k=0,\ldots, K-1$}
\STATE Set $\alpha_{k+1} = \frac{k+2}{2aL}$, $A_{k+1} = A_k + \alpha_{k+1}$
\STATE $x^{k+1} = \frac{A_k y^k + \alpha_{k+1} z^k}{A_{k+1}}$
\STATE Compute $\tnabla f_{\xi^k}(x^{k+1}) = \clip\left(\nabla f_{\xi^k}(x^{k+1}), \lambda_k\right)$ using a fresh sample $\xi^k \sim \cD_k$
\STATE $z^{k+1} = z^k - \alpha_{k+1}\tnabla f_{\xi^k}(x^{k+1})$
\STATE $y^{k+1} = \frac{A_k y^k + \alpha_{k+1} z^{k+1}}{A_{k+1}}$
\ENDFOR
\ENSURE $y^K$ 
\end{algorithmic}
\end{algorithm}

\subsection{Convex Functions}\label{appendix:clipped_SSTM}

We start with the following lemma, which is a special case of Lemma~6 from \citep{gorbunov2021near}. This result can be seen the ``optimization'' part of the analysis of \algname{clipped-SSTM}: the proof follows the same steps as the analysis of deterministic Similar Triangles Method \citep{gasnikov2016universal, dvurechenskii2018decentralize} and separates stochasticity from the deterministic part of the method. 

\begin{lemma}[Special case of Lemma~4.1 from \citep{gorbunov2021near}]\label{lem:main_opt_lemma_clipped_SSTM}
    Let Assumptions~\ref{as:L_smoothness} and \ref{as:str_cvx} with $\mu = 0$ hold on $Q = B_{3R}(x^*)$, where $R \ge \|x^0 - x^*\|$, and let stepsize parameter $a$ satisfy $a\ge 1$. If $x^{k}, y^k, z^k \in B_{3R}(x^*)$ for all $k = 0,1,\ldots,N$, $N \ge 0$, then after $N$ iterations of \algname{clipped-SSTM} for all $z\in B_{3R}(x^*)$ we have
    \begin{eqnarray}
        A_N\left(f(y^N) - f(z)\right) &\le& \frac{1}{2}\|z^0 - z\|^2 - \frac{1}{2}\|z^{N} - z\|^2 + \sum\limits_{k=0}^{N-1}\alpha_{k+1}\left\la \theta_{k+1}, z - z^{k} + \alpha_{k+1} \nabla f(x^{k+1})\right\ra\notag\\
        &&\quad + \sum\limits_{k=0}^{N-1}\alpha_{k+1}^2\left\|\theta_{k+1}\right\|^2,\label{eq:main_opt_lemma_clipped_SSTM}\\
        \theta_{k+1} &\eqdef& \tnabla f_{\xi^k}(x^{k+1}) - \nabla f(x^{k+1}).\label{eq:theta_k+1_def_clipped_SSTM}
    \end{eqnarray}
\end{lemma}
\begin{proof}
    For completeness, we provide the full proof. Using $z^{k+1} = z^k - \alpha_{k+1}\tnabla f_{\xi^k}(x^{k+1})$ we get that for all $z\in B_{3R}(x^*)$ and $k = 0,1,\ldots, N-1$
    \begin{eqnarray}
        \alpha_{k+1}\left\la\tnabla f_{\xi^k}(x^{k+1}), z^k - z \right\ra &=& \alpha_{k+1}\left\la \tnabla f_{\xi^k}(x^{k+1}), z^k - z^{k+1}\right\ra + \alpha_{k+1}\left\la \tnabla f_{\xi^k}(x^{k+1}), z^{k+1} - z\right\ra\notag\\
        &=& \alpha_{k+1}\left\la \tnabla f_{\xi^k}(x^{k+1}), z^k - z^{k+1}\right\ra + \left\la z^{k+1}-z^k, z - z^{k+1}\right\ra \notag\\
        &=& \alpha_{k+1}\left\la \tnabla f_{\xi^k}(x^{k+1}), z^k - z^{k+1}\right\ra - \frac{1}{2}\|z^k - z^{k+1}\|^2 \notag\\
        &&\quad+ \frac{1}{2}\|z^k - z\|^2 - \frac{1}{2}\|z^{k+1}-z\|^2,\label{eq:main_olc_clipped_SSTM_technical_2}
    \end{eqnarray}
    where in the last step we apply $2\langle a, b \rangle = \|a + b\|^2 - \|a\|^2 - \|b\|^2$ with $a = z^{k+1} - z^k$ and $b = z - z^{k+1}$. The update rules \eqref{eq:clipped_SSTM_y_update} and \eqref{eq:clipped_SSTM_x_update} give the following formula:
    \begin{eqnarray}
         y^{k+1} &=& \frac{A_k y^k + \alpha_{k+1}z^{k+1}}{A_{k+1}} = \frac{A_k y^k + \alpha_{k+1}z^{k}}{A_{k+1}} + \frac{\alpha_{k+1}}{A_{k+1}}\left(z^{k+1}-z^k\right) = x^{k+1} + \frac{\alpha_{k+1}}{A_{k+1}}\left(z^{k+1}-z^k\right). \label{eq:main_olc_clipped_SSTM_technical_3}
    \end{eqnarray}
    It implies
    \begin{eqnarray}
        \alpha_{k+1}\left\la\tnabla f_{\xi^k}(x^{k+1}), z^k - z \right\ra &\overset{\eqref{eq:theta_k+1_def_clipped_SSTM},\eqref{eq:main_olc_clipped_SSTM_technical_2}}{\le}& \alpha_{k+1}\left\la \nabla f(x^{k+1}), z^{k} - z^{k+1}\right\ra - \frac{1}{2}\|z^k - z^{k+1}\|^2 \notag\\
        &&\quad + \alpha_{k+1}\left\la \theta_{k+1}, z^{k} - z^{k+1}\right\ra  + \frac{1}{2}\|z^k - z\|^2 - \frac{1}{2}\|z^{k+1}-z\|^2\notag\\
        &\overset{\eqref{eq:main_olc_clipped_SSTM_technical_3}}{=}& A_{k+1}\left\la \nabla f(x^{k+1}), x^{k+1} - y^{k+1}\right\ra - \frac{1}{2}\|z^k - z^{k+1}\|^2 \notag\\
        &&\quad + \alpha_{k+1}\left\la \theta_{k+1}, z^{k} - z^{k+1}\right\ra   + \frac{1}{2}\|z^k - z\|^2 - \frac{1}{2}\|z^{k+1}-z\|^2\notag\\
        &\overset{\eqref{eq:L_smoothness_cor_1}}{\le}& A_{k+1}\left(f(x^{k+1}) - f(y^{k+1})\right) + \frac{A_{k+1}L}{2}\|x^{k+1}-y^{k+1}\|^2 - \frac{1}{2}\|z^k - z^{k+1}\|^2\notag\\
        &&\quad  + \alpha_{k+1}\left\la \theta_{k+1}, z^{k} - z^{k+1}\right\ra + \frac{1}{2}\|z^k - z\|^2 - \frac{1}{2}\|z^{k+1}-z\|^2\notag\\
        &\overset{\eqref{eq:main_olc_clipped_SSTM_technical_3}}{=}& A_{k+1}\left(f(x^{k+1}) - f(y^{k+1})\right)  + \frac{1}{2}\left(\frac{\alpha_{k+1}^2L}{A_{k+1}} - 1\right)\|z^k - z^{k+1}\|^2\notag\\
        &&\quad + \alpha_{k+1}\left\la \theta_{k+1}, z^{k} - z^{k+1}\right\ra  + \frac{1}{2}\|z^k - z\|^2 - \frac{1}{2}\|z^{k+1}-z\|^2,\notag
    \end{eqnarray}
    where in the third inequality we use $x^{k+1}, y^{k+1} \in B_{3R}(x^*)$. Since $A_{k+1} \ge aL_{k+1}\alpha_{k+1}^2$ (Lemma~\ref{lem:alpha_k_A_K_lemma}) and $a\ge 1$ we can continue our derivation as follows:
    \begin{eqnarray}
        \alpha_{k+1}\left\la\tnabla f_{\xi^k}(x^{k+1}), z^k - z \right\ra &\le& A_{k+1}\left(f(x^{k+1}) - f(y^{k+1})\right) + \alpha_{k+1}\left\la \theta_{k+1}, z^{k} - z^{k+1}\right\ra\notag\\
        &&\quad + \frac{1}{2}\|z^k - z\|^2 - \frac{1}{2}\|z^{k+1}-z\|^2. \label{eq:main_olc_clipped_SSTM_technical_4}
    \end{eqnarray}
    Convexity of $f$ gives
    \begin{eqnarray}
        \left\la\tnabla f_{\xi^k}(x^{k+1}), y^k - x^{k+1}\right\ra &\overset{\eqref{eq:theta_k+1_def_clipped_SSTM}}{=}& \left\la\nabla f(x^{k+1}), y^k - x^{k+1}\right\ra + \left\la \theta_{k+1}, y^k - x^{k+1}\right\ra\notag\\
        &\le& f(y^k) - f(x^{k+1}) + \left\la \theta_{k+1}, y^k - x^{k+1}\right\ra.\label{eq:main_olc_clipped_SSTM_technical_5}
    \end{eqnarray}
    The definition of $x^{k+1}$ \eqref{eq:clipped_SSTM_x_update} implies
    \begin{equation}
        \alpha_{k+1}\left(x^{k+1} - z^k\right) = A_k\left(y^k - x^{k+1}\right)\label{eq:main_olc_clipped_SSTM_technical_6}
    \end{equation}
    since $A_{k+1} = A_k + \alpha_{k+1}$. Putting all inequalities together, we derive that
    \begin{eqnarray*}
        \alpha_{k+1}\left\la\tnabla f_{\xi^k}(x^{k+1}), x^{k+1} - z \right\ra &=& \alpha_{k+1}\left\la\tnabla f_{\xi^k}(x^{k+1}), x^{k+1} - z^k \right\ra + \alpha_{k+1}\left\la\tnabla f_{\xi^k}(x^{k+1}), z^k - z \right\ra\\
        &\overset{\eqref{eq:main_olc_clipped_SSTM_technical_6}}{=}& A_{k}\left\la\tnabla f_{\xi^k}(x^{k+1}), y^{k} - x^{k+1} \right\ra + \alpha_{k+1}\left\la\tnabla f_{\xi^k}(x^{k+1}), z^k - z \right\ra\\
        &\overset{\eqref{eq:main_olc_clipped_SSTM_technical_5},\eqref{eq:main_olc_clipped_SSTM_technical_4}}{\le}& A_k\left(f(y^k)-f(x^{k+1})\right) + A_k\left\la \theta_{k+1}, y^k - x^{k+1}\right\ra\\
        &&\quad + A_{k+1}\left(f(x^{k+1}) - f(y^{k+1})\right)  + \alpha_{k+1}\left\la \theta_{k+1}, z^{k} - z^{k+1}\right\ra \\
        &&\quad + \frac{1}{2}\|z^k - z\|^2 - \frac{1}{2}\|z^{k+1}-z\|^2\\
        &\overset{\eqref{eq:main_olc_clipped_SSTM_technical_6}}{=}& A_kf(y^k) - A_{k+1}f(y^{k+1}) + \alpha_{k+1}\left\la \theta_{k+1}, x^{k+1}-z^k\right\ra\\
        &&\quad + \alpha_{k+1}f(x^{k+1}) + \alpha_{k+1}\left\la \theta_{k+1}, z^{k} - z^{k+1}\right\ra\\
        &&\quad + \frac{1}{2}\|z^k - z\|^2 - \frac{1}{2}\|z^{k+1}-z\|^2\\
        &=& A_kf(y^k) - A_{k+1}f(y^{k+1}) + \alpha_{k+1}f(x^{k+1})\\
        &&\quad + \alpha_{k+1}\left\la \theta_{k+1}, x^{k+1} - z^{k+1}\right\ra + \frac{1}{2}\|z^k - z\|^2 - \frac{1}{2}\|z^{k+1}-z\|^2.
    \end{eqnarray*}
    Rearranging the terms, we get
    \begin{eqnarray*}
        A_{k+1}f(y^{k+1}) - A_kf(y^k) &\le& \alpha_{k+1}\left(f(x^{k+1}) + \left\la\tnabla f_{\xi^k}(x^{k+1}), z-x^{k+1}\right\ra\right) + \frac{1}{2}\|z^k - z\|^2\\
        &&\quad - \frac{1}{2}\|z^{k+1} - z\|^2 + \alpha_{k+1}\left\la \theta_{k+1}, x^{k+1} - z^{k+1}\right\ra\\
        &\overset{\eqref{eq:theta_k+1_def_clipped_SSTM}}{=}& \alpha_{k+1}\left(f(x^{k+1}) + \left\la\nabla f(x^{k+1}), z-x^{k+1}\right\ra\right)\\
        &&\quad + \alpha_{k+1}\left\la\theta_{k+1}, z-x^{k+1}\right\ra + \frac{1}{2}\|z^k - z\|^2 - \frac{1}{2}\|z^{k+1} - z\|^2\\
        &&\quad + \alpha_{k+1}\left\la \theta_{k+1}, x^{k+1} - z^{k+1}\right\ra\\
        &\le& \alpha_{k+1}f(z) + \frac{1}{2}\|z^k - z\|^2 - \frac{1}{2}\|z^{k+1} - z\|^2  + \alpha_{k+1}\left\la \theta_{k+1}, z - z^{k+1}\right\ra,
    \end{eqnarray*}
    where in the last inequality we use the convexity of $f$. Taking into account $A_0 = \alpha_0 = 0$ and $A_{N} = \sum_{k=0}^{N-1}\alpha_{k+1}$ we sum up these inequalities for $k=0,\ldots,N-1$ and get
    \begin{eqnarray*}
        A_Nf(y^N) &\le& A_N f(z) + \frac{1}{2}\|z^0 - z\|^2 - \frac{1}{2}\|z^{N} - z\|^2 + \sum\limits_{k=0}^{N-1}\alpha_{k+1}\left\la \theta_{k+1}, z - z^{k+1}\right\ra\\
        &=& A_N f(z) + \frac{1}{2}\|z^0 - z\|^2 - \frac{1}{2}\|z^{N} - z\|^2 + \sum\limits_{k=0}^{N-1}\alpha_{k+1}\left\la \theta_{k+1}, z - z^{k} + \alpha_{k+1} \tnabla f_{\xi^k}(x^{k+1})\right\ra \\
        &\overset{\eqref{eq:theta_k+1_def_clipped_SSTM}}{=}& A_N f(z) + \frac{1}{2}\|z^0 - z\|^2 - \frac{1}{2}\|z^{N} - z\|^2 + \sum\limits_{k=0}^{N-1}\alpha_{k+1}\left\la \theta_{k+1}, z - z^{k} + \alpha_{k+1}\nabla f_{\xi^k}(x^{k+1})\right\ra\\
        &&\quad + \sum\limits_{k=0}^{N-1}\alpha_{k+1}^2\left\|\theta_{k+1}\right\|^2,
    \end{eqnarray*}
    which concludes the proof.
\end{proof}

Using this lemma we prove the main convergence result for \algname{clipped-SSTM}.
\begin{theorem}[Full version of Theorem~\ref{thm:clipped_SSTM_main_theorem}]\label{thm:clipped_SSTM_convex_case_appendix}
    Let Assumptions~\ref{as:bounded_alpha_moment}, \ref{as:L_smoothness} and \ref{as:str_cvx} with $\mu = 0$ hold on $Q = B_{3R}(x^*)$, where $R \geq \|x^0 - x^*\|$, and
    \begin{gather}
        a \geq \max\left\{48600\ln^2\frac{4K}{\beta}, \frac{900\sigma(K+1)K^{\frac{1}{\alpha}}\ln^{\frac{\alpha-1}{\alpha}}\frac{4K}{\beta}}{LR}\right\},\label{eq:clipped_SSTM_parameter_a}\\
        \lambda_{k} = \frac{R}{30\alpha_{k+1}\ln\frac{4K}{\beta}}, \label{eq:clipped_SSTM_clipping_level}
    \end{gather}
    for some $K > 0$ and $\beta \in (0,1]$ such that $\ln\frac{4K}{\beta} \geq 1$. Then, after $K$ iterations of \algname{clipped-SSTM} the iterates with probability at least $1 - \beta$ satisfy
    \begin{equation}
        f(y^K) - f(x^*) \leq \frac{6aL R^2}{K(K+3)} \quad \text{and} \quad  \{x^k\}_{k=0}^{K+1}, \{z^k\}_{k=0}^{K}, \{y^k\}_{k=0}^K \subseteq B_{2R}(x^*). \label{eq:clipped_SSTM_convex_case_appendix}
    \end{equation}
    In particular, when parameter $a$ equals the maximum from \eqref{eq:clipped_SSTM_parameter_a}, then the iterates produced by \algname{clipped-SSTM} after $K$ iterations with probability at least $1-\beta$ satisfy
    \begin{equation}
        f(y^K) - f(x^*) = \cO\left(\max\left\{\frac{LR^2\ln^2\frac{K}{\beta}}{K^2}, \frac{\sigma R \ln^{\frac{\alpha-1}{\alpha}}\frac{K}{\beta}}{K^{\frac{\alpha-1}{\alpha}}}\right\}\right), \label{eq:clipped_SSTM_convex_case_2_appendix}
    \end{equation}
    meaning that to achieve $f(y^K) - f(x^*) \leq \varepsilon$ with probability at least $1 - \beta$ \algname{clipped-SSTM} requires
    \begin{equation}
        K = \cO\left(\sqrt{\frac{LR^2}{\varepsilon}}\ln\frac{LR^2}{\varepsilon\beta}, \left(\frac{\sigma R}{\varepsilon}\right)^{\frac{\alpha}{\alpha-1}}\ln\left( \frac{1}{\beta} \left(\frac{\sigma R}{\varepsilon}\right)^{\frac{\alpha}{\alpha-1}}\right)\right)\quad \text{iterations/oracle calls.} \label{eq:clipped_SSTM_convex_case_complexity_appendix}
    \end{equation}
\end{theorem}
\begin{proof}
    The proof starts similarly to the proof of Theorem~4.1 from \citep{gorbunov2021near}. Let $R_k = \|z^k - x^*\|$, $\widetilde{R}_0 = R_0$, $\widetilde{R}_{k+1} = \max\{\widetilde{R}_k, R_{k+1}\}$ for all $k\geq 0$. We first show by induction that for all $k\geq 0$ the iterates $x^{k+1}, z^k, y^k$ lie in $B_{\widetilde{R}_k}(x^*)$. The induction base is trivial since $y^0 = z^0$, $\widetilde{R}_0 = R_0$, and $x^1 = \tfrac{A_0 y^0 + \alpha_1 z^0}{A_1} = z^0$. Next, assume that $x^{l}, z^{l-1}, y^{l-1} \in B_{\widetilde{R}_{l-1}}(x^*)$ for some $l\ge 1$. By definitions of $R_l$ and $\widetilde{R}_l$ we have that $z^l \in B_{R_{l}}(x^*)\subseteq B_{\widetilde{R}_{l}}(x^*)$. Since $y^l$ is a convex combination of $y^{l-1}\in B_{\widetilde{R}_{l-1}}(x^*)\subseteq B_{\widetilde{R}_{l}}(x^*)$, $z^l\in B_{\widetilde{R}_{l}}(x^*)$ and $B_{\widetilde{R}_{l}}(x^*)$ is a convex set we conclude that $y^l \in B_{\widetilde{R}_{l}}(x^*)$. Finally, since $x^{l+1}$ is a convex combination of $y^l$ and $z^l$ we have that $x^{l+1}$ lies in $B_{\widetilde{R}_{l}}(x^*)$ as well.
    
    Next, our goal is to show by induction that $\widetilde{R}_{l} \leq 3R$ with high probability, which allows us to apply the result of Lemma~\ref{lem:main_opt_lemma_clipped_SSTM} and then use Bernstein's inequality to estimate the stochastic part of the upper-bound. More precisely, for each $k = 0,\ldots, K$ we consider probability event $E_k$ defined as follows: inequalities
    \begin{gather}
        \sum\limits_{l=0}^{t-1}\alpha_{l+1}\left\la \theta_{l+1}, x^* - z^{l} + \alpha_{l+1}\nabla f_{\xi^l}(x^{l+1})\right\ra + \sum\limits_{l=0}^{t-1}\alpha_{l+1}^2\left\|\theta_{l+1}\right\|^2 \leq R^2, \label{eq:clipped_SSTM_induction_inequality_1}\\
        R_t \leq 2R \label{eq:clipped_SSTM_induction_inequality_2}
    \end{gather}
    hold for all $t = 0,1,\ldots, k$ simultaneously. We want to prove via induction that $\PP\{E_k\} \geq 1 - \nicefrac{k\beta}{K}$ for all $k = 0,1,\ldots, K$. For $k = 0$ the statement is trivial: the left-hand side of \eqref{eq:clipped_SSTM_induction_inequality_1} equals zero and $R \geq R_0$ by definition. Assume that the statement is true for some $k = T - 1 \leq K-1$: $\PP\{E_{T-1}\} \geq 1 - \nicefrac{(T-1)\beta}{K}$. One needs to prove that $\PP\{E_{T}\} \geq 1 - \nicefrac{T\beta}{K}$. First, we notice that probability event $E_{T-1}$ implies that $\widetilde{R}_t \leq 2R$ for all $t = 0,1,\ldots, T-1$. Moreover, it implies that
    \begin{eqnarray}
        \|z^{T} - x^*\| \overset{\eqref{eq:clipped_SSTM_z_update}}{\leq} \|z^{T} - x^*\| + \alpha_{T}\|\tnabla f_{\xi^{T-1}}(x^{T})\| \leq 2R + \alpha_{T}\lambda_{T-1} \overset{\eqref{eq:clipped_SSTM_clipping_level}}{\leq} 3R.\notag
    \end{eqnarray}
    Therefore, $E_{T-1}$ implies $\{x^k\}_{k=0}^{T}, \{z^k\}_{k=0}^{T}, \{y^k\}_{k=0}^T \subseteq B_{3R}(x^*)$, meaning that the assumptions of Lemma~\ref{lem:main_opt_lemma_clipped_SSTM} are satisfied and we have
    \begin{eqnarray}
        A_t\left(f(y^t) - f(x^*)\right) &\le& \frac{1}{2}R_0^2 - \frac{1}{2}R_t^2 + \sum\limits_{l=0}^{t-1}\alpha_{l+1}\left\la \theta_{l+1}, x^* - z^{l} + \alpha_{l+1} \nabla f(x^{l+1})\right\ra + \sum\limits_{l=0}^{t-1}\alpha_{l+1}^2\left\|\theta_{l+1}\right\|^2 \label{eq:clipped_SSTM_technical_1}
    \end{eqnarray}
    for all $t = 0,1,\ldots, T$ simultaneously and for all $t = 1,\ldots, T-1$ this probability event also implies that
    \begin{eqnarray}
        f(y^t) - f(x^*) \overset{\eqref{eq:clipped_SSTM_induction_inequality_1}, \eqref{eq:clipped_SSTM_technical_1}}{\leq} \frac{\frac{1}{2}R_0^2 - \frac{1}{2}R_t^2 + R^2}{A_t} \leq \frac{3R^2}{2A_t} = \frac{6aLR^2}{t(t+3)}. \label{eq:clipped_SSTM_technical_1_1}
    \end{eqnarray}
    Taking into account that $f(y^T) - f(x^*) \geq 0$, we also derive that $E_{T-1}$ implies
    \begin{eqnarray}
        R_T^2 &\leq& R_0^2 + \underbrace{2\sum\limits_{t=0}^{T-1}\alpha_{t+1}\left\la \theta_{t+1}, x^* - z^{t} + \alpha_{t+1} \nabla f(x^{t+1})\right\ra + 2\sum\limits_{t=0}^{T-1}\alpha_{t+1}^2\left\|\theta_{t+1}\right\|^2}_{2B_T}\notag\\
        &\leq& R^2 + 2B_T. \label{eq:clipped_SSTM_technical_2}
    \end{eqnarray}
    Before we estimate $B_T$, we need to derive a few useful inequalities. We start with showing that $E_{T-1}$ implies $\|\nabla f(x^{t+1})\| \leq \nicefrac{\lambda_t}{2}$ for all $t = 0,1,\ldots, T-1$. For $t = 0$ we have $x^1 = x^0$ and
    \begin{eqnarray}
        \|\nabla f(x^{1})\| = \|\nabla f(x^0)\| \overset{\eqref{eq:L_smoothness}}{\leq} L\|x^0 - x^*\| \leq \frac{R}{a\alpha_1} = \frac{\lambda_0}{2}\cdot \frac{60\ln\frac{4K}{\beta}}{a} \overset{\eqref{eq:clipped_SSTM_parameter_a}}{\leq} \frac{\lambda_0}{2}. \label{eq:clipped_SSTM_technical_3}
    \end{eqnarray}
    Next, for $t = 1,\ldots, T-1$ event $E_{T-1}$ implies
    \begin{eqnarray}
        \|\nabla f(x^{t+1})\| &\leq& \|\nabla f(x^{t+1}) - \nabla f(y^t)\| + \|\nabla f(y^t)\| \notag\\
        &\overset{\eqref{eq:L_smoothness}, \eqref{eq:L_smoothness_cor_2}}{\leq}& L\|x^{t+1} - y^t\| + \sqrt{2L\left(f(y^t) - f(x^*)\right)} \notag\\
        &\overset{\eqref{eq:main_olc_clipped_SSTM_technical_6},\eqref{eq:clipped_SSTM_technical_1_1}}{\leq}& \frac{L\alpha_{t+1}}{A_t}\|x^{t+1} - z^t\| + \sqrt{\frac{12aL^2R^2}{t(t+3)}}\notag\\
        &\leq& \frac{4LR\alpha_{t+1}}{A_t} + \sqrt{\frac{12aL^2R^2}{t(t+3)}} \notag\\
        &=&  \frac{R}{60\alpha_{t+1}\ln\frac{4K}{\beta}}\left(\frac{240L\alpha_{t+1}^2\ln\frac{4K}{\beta}}{A_t} + 60\sqrt{\frac{12aL^2\alpha_{t+1}^2\ln^2\frac{4K}{\beta}}{t(t+3)}}\right)\notag\\
        &\overset{\eqref{eq:A_k+1_explicit},\eqref{eq:clipped_SSTM_clipping_level}}{\leq}& \frac{\lambda_t}{2}\left(\frac{240L\left(\frac{t+2}{2aL}\right)^2\ln\frac{4K}{\beta}}{\frac{t(t+3)}{4aL}} + 60\sqrt{\frac{12aL^2\left(\frac{t+2}{2aL}\right)^2\ln^2\frac{4K}{\beta}}{t(t+3)}}\right) \notag\\
        &=& \frac{\lambda_t}{2} \left(\frac{240(t+2)^2\ln\frac{4K}{\beta}}{t(t+3)a} + 60\sqrt\frac{3(t+2)^2\ln\frac{4K}{\beta}}{t(t+3)a}\right) \notag\\
        &\leq& \frac{\lambda_t}{2} \left(\frac{540\ln\frac{4K}{\beta}}{a} + \frac{90\sqrt{3}\ln\frac{4K}{\beta}}{\sqrt{a}}\right) \overset{\eqref{eq:clipped_SSTM_parameter_a}}{\leq} \frac{\lambda_t}{2}, \label{eq:clipped_SSTM_technical_4}
    \end{eqnarray}
    where in the last row we use $\frac{(t+2)^2}{t(t+3)} \leq \frac{9}{4}$ for all $t \geq 1$. Therefore, probability event $E_{T-1}$ implies that
    \begin{eqnarray}
        \|x^* - z^{t} + \alpha_{t+1} \nabla f(x^{t+1})\| \leq \|x^* - z^{t}\| + \alpha_{t+1}\|\nabla f(x^{t+1})\| \overset{\eqref{eq:clipped_SSTM_induction_inequality_2}, \eqref{eq:clipped_SSTM_technical_3}, \eqref{eq:clipped_SSTM_technical_4}}{\leq} 2R + \frac{R}{60\ln\frac{4K}{\beta}} \leq 3R \label{eq:clipped_SSTM_technical_5}
    \end{eqnarray}
    for all $t = 0,1,\ldots, T-1$. Next, we define random vectors
    \begin{equation}
        \eta_t = \begin{cases} x^* - z^{t} + \alpha_{t+1} \nabla f(x^{t+1}),& \text{if } \|x^* - z^{t} + \alpha_{t+1} \nabla f(x^{t+1})\| \leq 3R,\\ 0,&\text{otherwise}, \end{cases} \notag
    \end{equation}
    for all $t = 0,1,\ldots, T-1$. By definition these random vectors are bounded with probability 1
    \begin{equation}
        \|\eta_t\| \leq 3R \label{eq:clipped_SSTM_technical_6}
    \end{equation}
    and probability event $E_{T-1}$ implies that $\eta_t = x^* - z^{t} + \alpha_{t+1} \nabla f(x^{t+1})$ for all $t = 0,1,\ldots, T-1$. Then, form $E_{T-1}$ it follows that
    \begin{eqnarray}
        B_T &=& \sum\limits_{t=0}^{T-1}\alpha_{t+1}\left\la \theta_{t+1}, \eta_t\right\ra + \sum\limits_{t=0}^{T-1}\alpha_{t+1}^2\left\|\theta_{t+1}\right\|^2. \notag
    \end{eqnarray}
    Next, we define the unbiased part and the bias of $\theta_{t+1}$ as $\theta_{t+1}^u$ and $\theta_{t+1}^b$ respectively:
    \begin{equation}
        \theta_{t+1}^u = \tnabla f_{\xi^t}(x^{t+1}) - \EE_{\xi^t}\left[\tnabla f_{\xi^t}(x^{t+1})\right],\quad \theta_{t+1}^b = \EE_{\xi^t}\left[\tnabla f_{\xi^t}(x^{t+1})\right] - \nabla f(x^{t+1}). \label{eq:clipped_SSTM_theta_u_b}
    \end{equation}
    We notice that $\theta_{t+1} = \theta_{t+1}^u + \theta_{t+1}^b$. Using new notation, we get that $E_{T-1}$ implies
    \begin{eqnarray}
        B_T &=& \sum\limits_{t=0}^{T-1}\alpha_{t+1}\left\la \theta_{t+1}^u + \theta_{t+1}^b, \eta_t\right\ra + \sum\limits_{t=0}^{T-1}\alpha_{t+1}^2\left\|\theta_{t+1}^u + \theta_{t+1}^b\right\|^2 \notag\\
        &\leq& \underbrace{\sum\limits_{t=0}^{T-1}\alpha_{t+1}\left\la \theta_{t+1}^u, \eta_t\right\ra}_{\circledOne} + \underbrace{\sum\limits_{t=0}^{T-1}\alpha_{t+1}\left\la \theta_{t+1}^b, \eta_t\right\ra}_{\circledTwo} + \underbrace{2\sum\limits_{t=0}^{T-1}\alpha_{t+1}^2\left(\left\|\theta_{t+1}^u\right\|^2 - \EE_{\xi^t}\left[\left\|\theta_{t+1}^u\right\|^2\right]\right)}_{\circledThree}\notag\\
        &&\quad + \underbrace{2\sum\limits_{t=0}^{T-1}\alpha_{t+1}^2\EE_{\xi^t}\left[\left\|\theta_{t+1}^u\right\|^2\right]}_{\circledFour} + \underbrace{2\sum\limits_{t=0}^{T-1}\alpha_{t+1}^2\left\|\theta_{t+1}^b\right\|^2}_{\circledFive}. \label{eq:clipped_SSTM_technical_7}
    \end{eqnarray}
    It remains to derive good enough high-probability upper-bounds for the terms $\circledOne, \circledTwo, \circledThree, \circledFour, \circledFive$, i.e., to finish our inductive proof we need to show that $\circledOne + \circledTwo + \circledThree + \circledFour + \circledFive \leq R^2$ with high probability. In the subsequent parts of the proof, we will need to use many times the bounds for the norm and second moments of $\theta_{t+1}^u$ and $\theta_{t+1}^b$. First, by definition of clipping operator, we have with probability $1$ that
    \begin{equation}
        \|\theta_{t+1}^u\| \leq 2\lambda_{t}. \label{eq:clipped_SSTM_norm_theta_u_bound}
    \end{equation}
    Moreover, since $E_{T-1}$ implies that $\|\nabla f(x^{t+1})\| \leq \nicefrac{\lambda_t}{2}$ for $t = 0,1,\ldots,T-1$ (see \eqref{eq:clipped_SSTM_technical_3} and \eqref{eq:clipped_SSTM_technical_4}), then, in view of Lemma~\ref{lem:bias_and_variance_clip}, we have that $E_{T-1}$ implies
    \begin{eqnarray}
        \|\theta_{t+1}^b\| &\leq& \frac{2^\alpha\sigma^\alpha}{\lambda_{t}^{\alpha-1}}, \label{eq:clipped_SSTM_norm_theta_b_bound} \\
        \EE_{\xi^t}\left[\|\theta_{t+1}^u\|^2\right] &\leq& 18 \lambda_t^{2-\alpha}\sigma^\alpha. \label{eq:clipped_SSTM_second_moment_theta_u_bound}
    \end{eqnarray}

    \textbf{Upper bound for $\circledOne$.} By definition of $\theta_{t+1}^u$, we have $\EE_{\xi^t}[\theta_{t+1}^u] = 0$ and
    \begin{equation}
        \EE_{\xi^t}\left[\alpha_{t+1}\left\la \theta_{t+1}^u, \eta_t\right\ra\right] = 0. \notag
    \end{equation}
    Next, sum $\circledOne$ has bounded with probability $1$ terms:
    \begin{equation}
        |\alpha_{t+1}\left\la \theta_{t+1}^u, \eta_t\right\ra| \leq \alpha_{t+1} \|\theta_{t+1}^u\| \cdot \|\eta_t\| \overset{\eqref{eq:clipped_SSTM_technical_6},\eqref{eq:clipped_SSTM_norm_theta_u_bound}}{\leq} 6\alpha_{t+1}\lambda_t R \overset{\eqref{eq:clipped_SSTM_clipping_level}}{=} \frac{R^2}{5\ln\frac{4K}{\beta}} \eqdef c. \label{eq:clipped_SSTM_technical_8} 
    \end{equation}
    The summands also have bounded conditional variances $\sigma_t^2 \eqdef \EE_{\xi^t}[\alpha_{t+1}^2\left\la \theta_{t+1}^u, \eta_t\right\ra^2]$:
    \begin{equation}
        \sigma_t^2 \leq \EE_{\xi^t}\left[\alpha_{t+1}^2\|\theta_{t+1}^u\|^2\cdot \|\eta_t\|^2\right] \overset{\eqref{eq:clipped_SSTM_technical_6}}{\leq} 9\alpha_{t+1}^2R^2 \EE_{\xi^t}\left[\|\theta_{t+1}^u\|^2\right]. \label{eq:clipped_SSTM_technical_9}
    \end{equation}
    In other words, we showed that $\{\alpha_{t+1}\left\la \theta_{t+1}^u, \eta_t\right\ra\}_{t=0}^{T-1}$ is a bounded martingale difference sequence with bounded conditional variances $\{\sigma_t^2\}_{t=0}^{T-1}$. Next, we apply Bernstein's inequality (Lemma~\ref{lem:Bernstein_ineq}) with $X_t = \alpha_{t+1}\left\la \theta_{t+1}^u, \eta_t\right\ra$, parameter $c$ as in \eqref{eq:clipped_SSTM_technical_8}, $b = \frac{R^2}{5}$, $G = \frac{R^4}{150\ln\frac{4K}{\beta}}$:
    \begin{equation*}
        \PP\left\{|\circledOne| > \frac{R^2}{5}\quad \text{and}\quad \sum\limits_{t=0}^{T-1} \sigma_{t}^2 \leq \frac{R^4}{150\ln\frac{4K}{\beta}}\right\} \leq 2\exp\left(- \frac{b^2}{2G + \nicefrac{2cb}{3}}\right) = \frac{\beta}{2K}.
    \end{equation*}
    Equivalently, we have
    \begin{equation}
        \PP\left\{ E_{\circledOne} \right\} \geq 1 - \frac{\beta}{2K},\quad \text{for}\quad E_{\circledOne} = \left\{ \text{either} \quad  \sum\limits_{t=0}^{T-1} \sigma_{t}^2 > \frac{R^4}{150\ln\frac{4K}{\beta}} \quad \text{or}\quad |\circledOne| \leq \frac{R^2}{5}\right\}. \label{eq:clipped_SSTM_sum_1_upper_bound}
    \end{equation}
    In addition, $E_{T-1}$ implies that
    \begin{eqnarray}
        \sum\limits_{t=0}^{T-1} \sigma_{t}^2 &\overset{\eqref{eq:clipped_SSTM_technical_9}}{\leq}& 9R^2 \sum\limits_{t=0}^{T-1} \alpha_{t+1}^2 \EE_{\xi^t}\left[\|\theta_{t+1}^u\|^2\right] \overset{\eqref{eq:clipped_SSTM_second_moment_theta_u_bound}}{\leq} 162\sigma^\alpha R^2 \sum\limits_{t=0}^{T-1} \alpha_{t+1}^2\lambda_t^{2-\alpha}\notag\\
        &\overset{\eqref{eq:clipped_SSTM_clipping_level}}{\leq}& \frac{162\sigma^\alpha R^{4-\alpha}}{30^{2-\alpha}\ln^{2-\alpha}\frac{4K}{\beta}}\sum\limits_{t=0}^{T-1} \alpha_{t+1}^\alpha = \frac{162\sigma^\alpha R^{4-\alpha}}{30^{2-\alpha}\cdot 2^\alpha a^\alpha L^\alpha\ln^{2-\alpha}\frac{4K}{\beta}}\sum\limits_{t=0}^{T-1} (t+2)^\alpha \notag\\
        &\leq& \frac{1}{a^\alpha} \cdot \frac{162\sigma^\alpha R^{4-\alpha} T(T+1)^\alpha}{60 L^\alpha\ln^{2-\alpha}\frac{4K}{\beta}} \overset{\eqref{eq:clipped_SSTM_parameter_a}}{\leq} \frac{R^4}{150 \ln\frac{4K}{\beta}}. \label{eq:clipped_SSTM_sum_1_variance_bound}
    \end{eqnarray}

    \textbf{Upper bound for $\circledTwo$.} From $E_{T-1}$ it follows that
    \begin{eqnarray}
        \circledTwo &\leq& \sum\limits_{t=0}^{T-1}\alpha_{t+1}\|\theta_{t+1}^b\|\cdot \|\eta_t\| \overset{\eqref{eq:clipped_SSTM_technical_6},\eqref{eq:clipped_SSTM_norm_theta_b_bound}}{\leq} 3R\cdot 2^\alpha\sigma^\alpha \sum\limits_{t=0}^{T-1}\frac{\alpha_{t+1}}{\lambda_{t}^{\alpha-1}} \overset{\eqref{eq:clipped_SSTM_clipping_level}}{\leq} 12R\sigma^\alpha \cdot \frac{30^{\alpha-1}\ln^{\alpha-1}\frac{4K}{\beta}}{R^{\alpha-1}} \sum\limits_{t=0}^{T-1}\alpha_{t+1}^{\alpha}\notag\\
        &\leq& \frac{360\sigma^\alpha R^{2 - \alpha}\ln^{\alpha-1}\frac{4K}{\beta}}{2^\alpha a^{\alpha} L^\alpha}\sum\limits_{t=0}^{T-1} (t+2)^\alpha \leq \frac{1}{a^\alpha} \cdot \frac{180\sigma^\alpha R^{2-\alpha} T(T+1)^\alpha\ln^{\alpha-1}\frac{4K}{\beta}}{L^\alpha} \overset{\eqref{eq:clipped_SSTM_parameter_a}}{\leq} \frac{R^2}{5}. \label{eq:clipped_SSTM_sum_2_upper_bound}
    \end{eqnarray}

    \textbf{Upper bound for $\circledThree$.} First, we have
    \begin{equation}
        \EE_{\xi^t}\left[2\alpha_{t+1}^2\left(\left\|\theta_{t+1}^u\right\|^2 - \EE_{\xi^t}\left[\left\|\theta_{t+1}^u\right\|^2\right]\right)\right] = 0. \notag
    \end{equation}
    Next, sum $\circledThree$ has bounded with probability $1$ terms:
    \begin{eqnarray}
        \left|2\alpha_{t+1}^2\left(\left\|\theta_{t+1}^u\right\|^2 - \EE_{\xi^t}\left[\left\|\theta_{t+1}^u\right\|^2\right]\right)\right| &\leq& 2\alpha_{t+1}^2\left( \|\theta_{t+1}^u\|^2 +   \EE_{\xi^t}\left[\left\|\theta_{t+1}^u\right\|^2\right]\right)\notag\\
        &\overset{\eqref{eq:clipped_SSTM_norm_theta_u_bound}}{\leq}& 16\alpha_{t+1}^2\lambda_t^2 \overset{\eqref{eq:clipped_SSTM_clipping_level}}{\leq} \frac{R^2}{5\ln\frac{4K}{\beta}} \eqdef c. \label{eq:clipped_SSTM_technical_10}
    \end{eqnarray}
    The summands also have bounded conditional variances $\widetilde\sigma_t^2 \eqdef \EE_{\xi^t}\left[4\alpha_{t+1}^4\left(\left\|\theta_{t+1}^u\right\|^2 - \EE_{\xi^t}\left[\left\|\theta_{t+1}^u\right\|^2\right]\right)^2\right]$:
    \begin{eqnarray}
        \widetilde\sigma_t^2 &\overset{\eqref{eq:clipped_SSTM_technical_10}}{\leq}& \frac{R^2}{5 \ln\frac{4K}{\beta}} \EE_{\xi^t}\left[2\alpha_{t+1}^2\left|\left\|\theta_{t+1}^u\right\|^2 - \EE_{\xi^k}\left[\left\|\theta_{t+1}^u\right\|^2\right]\right|\right] \leq \alpha_{t+1}^2R^2 \EE_{\xi^t}\left[\|\theta_{t+1}^u\|^2\right], \label{eq:clipped_SSTM_technical_11}
    \end{eqnarray}
    since $\ln\frac{4K}{\beta} \geq 1$. In other words, we showed that $\left\{2\alpha_{t+1}^2\left(\left\|\theta_{t+1}^u\right\|^2 - \EE_{\xi^t}\left[\left\|\theta_{t+1}^u\right\|^2\right]\right)\right\}_{t=0}^{T-1}$ is a bounded martingale difference sequence with bounded conditional variances $\{\widetilde\sigma_t^2\}_{t=0}^{T-1}$. Next, we apply Bernstein's inequality (Lemma~\ref{lem:Bernstein_ineq}) with $X_t = 2\alpha_{t+1}^2\left(\left\|\theta_{t+1}^u\right\|^2 - \EE_{\xi^t}\left[\left\|\theta_{t+1}^u\right\|^2\right]\right)$, parameter $c$ as in \eqref{eq:clipped_SSTM_technical_10}, $b = \frac{R^2}{5}$, $G = \frac{R^4}{150\ln\frac{4K}{\beta}}$:
    \begin{equation*}
        \PP\left\{|\circledThree| > \frac{R^2}{5}\quad \text{and}\quad \sum\limits_{t=0}^{T-1} \widetilde\sigma_{t}^2 \leq \frac{R^4}{150\ln\frac{4K}{\beta}}\right\} \leq 2\exp\left(- \frac{b^2}{2G + \nicefrac{2cb}{3}}\right) = \frac{\beta}{2K}.
    \end{equation*}
    Equivalently, we have
    \begin{equation}
        \PP\left\{ E_{\circledThree} \right\} \geq 1 - \frac{\beta}{2K},\quad \text{for}\quad E_{\circledThree} = \left\{ \text{either} \quad  \sum\limits_{t=0}^{T-1} \widetilde\sigma_{t}^2 > \frac{R^4}{150\ln\frac{4K}{\beta}} \quad \text{or}\quad |\circledThree| \leq \frac{R^2}{5}\right\}. \label{eq:clipped_SSTM_sum_3_upper_bound}
    \end{equation}
    In addition, $E_{T-1}$ implies that
    \begin{eqnarray}
        \sum\limits_{t=0}^{T-1} \widetilde\sigma_{t}^2 &\overset{\eqref{eq:clipped_SSTM_technical_11}}{\leq}& R^2 \sum\limits_{t=0}^{T-1} \alpha_{t+1}^2 \EE_{\xi^t}\left[\|\theta_{t+1}^u\|^2\right] \leq  9R^2 \sum\limits_{t=0}^{T-1} \alpha_{t+1}^2 \EE_{\xi^t}\left[\|\theta_{t+1}^u\|^2\right] \overset{\eqref{eq:clipped_SSTM_sum_1_variance_bound}}{\leq} \frac{R^4}{150 \ln\frac{4K}{\beta}}. \label{eq:clipped_SSTM_sum_3_variance_bound}
    \end{eqnarray}

    \textbf{Upper bound for $\circledFour$.} From $E_{T-1}$ it follows that
    \begin{eqnarray}
        \circledFour &=& 2\sum\limits_{t=0}^{T-1}\alpha_{t+1}^2\EE_{\xi^t}\left[\left\|\theta_{t+1}^u\right\|^2\right] \leq \frac{1}{R^2}\cdot 9R^2\sum\limits_{t=0}^{T-1}\alpha_{t+1}^2\EE_{\xi^t}\left[\left\|\theta_{t+1}^u\right\|^2\right] \overset{\eqref{eq:clipped_SSTM_sum_1_variance_bound}}{\leq} \frac{R^2}{150\ln\frac{4K}{\beta}} \leq \frac{R^2}{5}.\label{eq:clipped_SSTM_sum_4_upper_bound}
    \end{eqnarray}

    \textbf{Upper bound for $\circledFive$.} From $E_{T-1}$ it follows that
    \begin{eqnarray}
        \circledFive &=& 2\sum\limits_{t=0}^{T-1}\alpha_{t+1}^2\left\|\theta_{t+1}^b\right\|^2 \leq 2^{2\alpha + 1}\sigma^{2\alpha} \sum\limits_{t=0}^{T-1}\frac{\alpha_{t+1}^2}{\lambda_{t}^{2\alpha - 2}} \overset{\eqref{eq:clipped_SSTM_clipping_level}}{=} \frac{2^{2\alpha + 1}\cdot 30^{2\alpha-2}\sigma^{2\alpha} \ln^{2\alpha-2}\frac{4K}{\beta}}{R^{2\alpha-2}} \sum\limits_{t=0}^{T-1}\alpha_{t+1}^{2\alpha} \notag\\
        &=&   \frac{2^{2\alpha + 1}\cdot 30^{2\alpha-2}\sigma^{2\alpha} \ln^{2\alpha-2}\frac{4K}{\beta}}{2^{2\alpha}a^{2\alpha} L^{2\alpha} R^{2\alpha-2}} \sum\limits_{t=0}^{T-1}(t+2)^{2\alpha} \leq \frac{1}{a^{2\alpha}} \cdot \frac{1800 \sigma^{2\alpha}T(T+1)^{2\alpha}\ln^{2\alpha-2}\frac{4K}{\beta}}{L^{2\alpha}R^{2\alpha-2}} \overset{\eqref{eq:clipped_SSTM_parameter_a}}{\leq} \frac{R^2}{5}.\label{eq:clipped_SSTM_sum_5_upper_bound}
    \end{eqnarray}

    Now, we have the upper bounds for  $\circledOne, \circledTwo, \circledThree, \circledFour, \circledFive$. In particular, probability event $E_{T-1}$ implies
    \begin{gather*}
        B_T \overset{\eqref{eq:clipped_SSTM_technical_7}}{\leq} R^2 + \circledOne + \circledTwo + \circledThree + \circledFour + \circledFive,\\
        \circledTwo \overset{\eqref{eq:clipped_SSTM_sum_2_upper_bound}}{\leq} \frac{R^2}{5},\quad \circledFour \overset{\eqref{eq:clipped_SSTM_sum_4_upper_bound}}{\leq} \frac{R^2}{5}, \quad \circledFive \overset{\eqref{eq:clipped_SSTM_sum_5_upper_bound}}{\leq} \frac{R^2}{5},\\
        \sum\limits_{t=0}^{T-1} \sigma_t^2 \overset{\eqref{eq:clipped_SSTM_sum_1_variance_bound}}{\leq} \frac{R^4}{150 \ln\frac{4K}{\beta}},\quad \sum\limits_{t=0}^{T-1} \widetilde\sigma_t^2 \overset{\eqref{eq:clipped_SSTM_sum_3_variance_bound}}{\leq} \frac{R^4}{150 \ln\frac{4K}{\beta}}.
    \end{gather*}
    Moreover, we also have (see \eqref{eq:clipped_SSTM_sum_1_upper_bound}, \eqref{eq:clipped_SSTM_sum_3_upper_bound} and our induction assumption)
    \begin{equation*}
        \PP\{E_{T-1}\} \geq 1 - \frac{(T-1)\beta}{K},\quad \PP\{E_{\circledOne}\} \geq 1 - \frac{\beta}{2K},\quad \PP\{E_{\circledThree}\} \geq 1 - \frac{\beta}{2K},
    \end{equation*}
    where 
    \begin{eqnarray*}
        E_{\circledOne} &=& \left\{ \text{either} \quad  \sum\limits_{t=0}^{T-1} \sigma_{t}^2 > \frac{R^4}{150\ln\frac{4K}{\beta}} \quad \text{or}\quad |\circledOne| \leq \frac{R^2}{5}\right\},\\
        E_{\circledThree} &=& \left\{ \text{either} \quad  \sum\limits_{t=0}^{T-1} \widetilde\sigma_{t}^2 > \frac{R^4}{150\ln\frac{4K}{\beta}} \quad \text{or}\quad |\circledThree| \leq \frac{R^2}{5}\right\}.
    \end{eqnarray*}
    Thus, probability event $E_{T-1} \cap E_{\circledOne} \cap E_{\circledThree}$ implies
    \begin{eqnarray}
        B_T &\leq& R^2 + \frac{R^2}{5} + \frac{R^2}{5} + \frac{R^2}{5} + \frac{R^2}{5} + \frac{R^2}{5} = 2R^2, \notag\\
        R_T^2 &\overset{\eqref{eq:clipped_SSTM_technical_2}}{\leq}& R^2 + 2R^2 \leq (2R)^2, \notag
    \end{eqnarray}
    which is equivalent to \eqref{eq:clipped_SSTM_induction_inequality_1} and \eqref{eq:clipped_SSTM_induction_inequality_2} for $t = T$, and 
    \begin{equation*}
        \PP\{E_T\} \geq \PP\left\{E_{T-1} \cap E_{\circledOne} \cap E_{\circledThree}\right\} = 1 - \PP\left\{\overline{E}_{T-1} \cup \overline{E}_{\circledOne} \cup \overline{E}_{\circledThree}\right\} \geq 1 - \PP\{\overline{E}_{T-1}\} - \PP\{\overline{E}_{\circledOne}\} - \PP\{\overline{E}_{\circledThree}\} \geq 1 - \frac{T\beta}{K}.
    \end{equation*}
    This finishes the inductive part of our proof, i.e., for all $k = 0,1,\ldots,K$ we have $\PP\{E_k\} \geq 1 - \nicefrac{k\beta}{K}$. In particular, for $k = K$ we have that with probability at least $1 - \beta$
    \begin{equation*}
        f(y^K) - f(x^*) \overset{\eqref{eq:clipped_SSTM_technical_1_1}}{\leq} \frac{6aLR^2}{K(K+3)}
    \end{equation*}
    and $\{x^k\}_{k=0}^{K+1}, \{z^k\}_{k=0}^{K}, \{y^k\}_{k=0}^K \subseteq B_{2R}(x^*)$, which follows from \eqref{eq:clipped_SSTM_induction_inequality_2}.
    
    Finally, if
    \begin{equation*}
        a = \max\left\{48600\ln^2\frac{4K}{\beta}, \frac{900\sigma(K+1)K^{\frac{1}{\alpha}}\ln^{\frac{\alpha-1}{\alpha}}\frac{4K}{\beta}}{LR}\right\},
    \end{equation*}
    then with probability at least $1-\beta$
    \begin{eqnarray*}
        f(y^K) - f(x^*) &\leq& \frac{6aLR^2}{K(K+3)} = \max\left\{\frac{291600LR^2\ln^2\frac{4K}{\beta}}{K(K+3)}, \frac{5400\sigma R(K+1)K^{\frac{1}{\alpha}}\ln^{\frac{\alpha-1}{\alpha}}\frac{4K}{\beta}}{K(K+3)}\right\}\\
        &=& \cO\left(\max\left\{\frac{LR^2\ln^2\frac{K}{\beta}}{K^2}, \frac{\sigma R \ln^{\frac{\alpha-1}{\alpha}}\frac{K}{\beta}}{K^{\frac{\alpha-1}{\alpha}}}\right\}\right).
    \end{eqnarray*}
    To get $f(y^K) - f(x^*) \leq \varepsilon$ with probability at least $1-\beta$ it is sufficient to choose $K$ such that both terms in the maximum above are $\cO(\varepsilon)$. This leads to
    \begin{equation*}
         K = \cO\left(\sqrt{\frac{LR^2}{\varepsilon}}\ln\frac{LR^2}{\varepsilon\beta}, \left(\frac{\sigma R}{\varepsilon}\right)^{\frac{\alpha}{\alpha-1}}\ln\left( \frac{1}{\beta} \left(\frac{\sigma R}{\varepsilon}\right)^{\frac{\alpha}{\alpha-1}}\right)\right)
    \end{equation*}
    that concludes the proof.
\end{proof}

\subsection{Strongly Convex Functions}\label{appendix:R_clipped_SSTM}
In the strongly convex case, we consider the restarted version of \algname{clipped-SSTM} (\algname{R-clipped-SSTM}). The main result is summarized below.

\begin{algorithm}[h]
\caption{Restarted \algname{clipped-SSTM} (\algname{R-clipped-SSTM}) \citep{gorbunov2020stochastic}}
\label{alg:R-clipped-SSTM}   
\begin{algorithmic}[1]
\REQUIRE starting point $x^0$, number of restarts $\tau$, number of steps of \algname{clipped-SSTM} between restarts $\{K_t\}_{t=1}^{\tau}$, stepsize parameters $\{a_t\}_{t=1}^\tau$, clipping levels $\{\lambda_{k}^1\}_{k=0}^{K_1-1}$, $\{\lambda_{k}^2\}_{k=0}^{K_2-1}$, \ldots, $\{\lambda_{k}^\tau\}_{k=0}^{K_\tau - 1}$, smoothness constant $L$.
\STATE $\hat{x}^0 = x^0$
\FOR{$t=1,\ldots, \tau$}
\STATE Run \algname{clipped-SSTM} (Algorithm~\ref{alg:clipped-SSTM}) for $K_t$ iterations with stepsize parameter $a_{t}$, clipping levels $\{\lambda_{k}^t\}_{k=0}^{K_t-1}$, and starting point $\hat{x}^{t-1}$. Define the output of \algname{clipped-SSTM} by $\hat{x}^{t}$.
\ENDFOR
\ENSURE $\hat{x}^\tau$ 
\end{algorithmic}
\end{algorithm}

\begin{theorem}[Full version of Theorem~\ref{thm:R_clipped_SSTM_main_theorem}]\label{thm:R_clipped_SSTM_main_theorem_appendix}
    Let Assumptions~\ref{as:bounded_alpha_moment}, \ref{as:L_smoothness}, \ref{as:str_cvx} with $\mu > 0$ hold for $Q = B_{3R}(x^*)$, where $R \geq \|x^0 - x^*\|^2$ and \algname{R-clipped-SSTM} runs \algname{clipped-SSTM} $\tau$ times. Let
    \begin{gather}
        K_t = \left\lceil \max\left\{ 1080\sqrt{\frac{LR_{t-1}^2}{\varepsilon_t}}\ln\frac{2160\sqrt{LR_{t-1}^2}\tau}{\sqrt{\varepsilon_t}\beta}, 2\left(\frac{5400\sigma R_{t-1}}{\varepsilon_t}\right)^{\frac{\alpha}{\alpha-1}} \ln\left(\frac{4\tau}{\beta}\left(\frac{5400\sigma R_{t-1}}{\varepsilon_t}\right)^{\frac{\alpha}{\alpha-1}}\right) \right\} \right\rceil, \label{eq:R_clipped_SSTM_K_t} \\
        \varepsilon_t = \frac{\mu R_{t-1}^2}{4},\quad R_{t-1} = \frac{R}{2^{\nicefrac{(t-1)}{2}}}, \quad \tau = \left\lceil \log_2 \frac{\mu R^2}{2\varepsilon} \right\rceil,\quad \ln\frac{4K_t\tau}{\beta} \geq 1, \label{eq:R_clipped_SSTM_epsilon_R_t_tau} \\
        a_t = \max\left\{48600\ln^2\frac{4K_t\tau}{\beta}, \frac{900\sigma(K_t+1)K_t^{\frac{1}{\alpha}}\ln^{\frac{\alpha-1}{\alpha}}\frac{4K_t\tau}{\beta}}{L R_t}\right\}, \label{eq:R_clipped_SSTM_parameter_a}\\
        \lambda_{k}^t = \frac{R_t}{30\alpha_{k+1}^t \ln\frac{4K_t\tau}{\beta}} \label{eq:R_clipped_SSTM_clipping_level}
    \end{gather}
    for $t = 1, \ldots, \tau$. Then to guarantee $f(\hat x^\tau) - f(x^*) \leq \varepsilon$ with probability $\geq 1 - \beta$ \algname{R-clipped-SSTM} requires
    \begin{equation}
        \cO\left(\max\left\{ \sqrt{\frac{L}{\mu}}\ln\left(\frac{\mu R^2}{\varepsilon}\right)\ln\left(\frac{\sqrt{L}}{\sqrt{\mu}\beta}\ln\left(\frac{\mu R^2}{\varepsilon}\right)\right), \left(\frac{\sigma^2}{\mu \varepsilon}\right)^{\frac{\alpha}{2(\alpha-1)}} \ln\left(\frac{1}{\beta}\left(\frac{\sigma^2}{\mu \varepsilon}\right)^{\frac{\alpha}{2(\alpha-1)}}\ln\left(\frac{\mu R^2}{\varepsilon}\right)\right)\right\} \right) \label{eq:R_clipped_SSTM_main_result_appendix}
    \end{equation}
    iterations/oracle calls. Moreover, with probability $\geq 1-\beta$ the iterates of \algname{R-clipped-SSTM} at stage $t$ stay in the ball $B_{2R_{t-1}}(x^*)$.
\end{theorem}
\begin{proof}
    We show by induction that for any $t = 1,\ldots, \tau$ with probability at least $1 - \nicefrac{t\beta}{\tau}$ inequalities
    \begin{equation}
        f(\hat{x}^l) - f(x^*) \leq \varepsilon_l, \quad \|\hat x^l - x^*\|^2 \leq  R_l^2 = \frac{R^2}{2^l}
    \end{equation}
    hold for $l = 1,\ldots, t$ simultaneously. First, we prove the base of the induction. Theorem~\ref{thm:clipped_SSTM_convex_case_appendix} implies that with probability at least $1 - \nicefrac{\beta}{\tau}$
    \begin{eqnarray}
        f(\hat x^1) - f(x^*) &\leq& \frac{6a_1LR^2}{K_1(K_1+3)} \overset{\eqref{eq:R_clipped_SSTM_parameter_a}}{=} \max\left\{\frac{291600LR^2\ln^2\frac{4K_1\tau}{\beta}}{K_1(K_1+3)}, \frac{5400\sigma R(K_1+1)K_1^{\frac{1}{\alpha}}\ln^{\frac{\alpha-1}{\alpha}}\frac{4K_1\tau}{\beta}}{K_1(K_1+3)}\right\} \notag\\
        &\leq& \max\left\{\frac{291600LR^2\ln^2\frac{4K_1\tau}{\beta}}{K_1^2}, \frac{5400\sigma R\ln^{\frac{\alpha-1}{\alpha}}\frac{4K_1\tau}{\beta}}{K_1^{\frac{\alpha-1}{\alpha}}}\right\}\notag\\
        &\overset{\eqref{eq:R_clipped_SSTM_K_t}}{\leq}& \varepsilon_1 = \frac{\mu R^2}{4} \notag
    \end{eqnarray}
    and, due to the strong convexity,
    \begin{equation*}
        \|\hat x^1 - x^*\|^2 \leq \frac{2(f(\hat x^1) - f(x^*))}{\mu} \leq \frac{R^2}{2} = R_1^2.
    \end{equation*}
    The base of the induction is proven. Now, assume that the statement holds for some $t = T < \tau$, i.e., with probability at least $1 - \nicefrac{T\beta}{\tau}$ inequalities
    \begin{equation}
        f(\hat{x}^l) - f(x^*) \leq \varepsilon_l, \quad \|\hat x^l - x^*\|^2 \leq  R_l^2 = \frac{R^2}{2^l}
    \end{equation}
    hold for $l = 1,\ldots, T$ simultaneously. In particular, with probability at least $1 - \nicefrac{T\beta}{\tau}$ we have $\|\hat x^T - x^*\|^2 \leq R_T^2$. Applying Theorem~\ref{thm:clipped_SSTM_convex_case_appendix} and using union bound for probability events, we get that with probability at least $1 - \nicefrac{(T+1)\beta}{\tau}$
    \begin{eqnarray}
        f(\hat x^{T+1}) - f(x^*) &\leq& \frac{6a_{T+1}LR_{T}^2}{K_{T+1}(K_{T+1}+3)}\notag\\
        &\overset{\eqref{eq:R_clipped_SSTM_parameter_a}}{=}& \max\left\{\frac{291600LR_{T}^2\ln^2\frac{4K_{T+1}\tau}{\beta}}{K_{T+1}(K_{T+1}+3)}, \frac{5400\sigma R_T(K_{T+1}+1)K_{T+1}^{\frac{1}{\alpha}}\ln^{\frac{\alpha-1}{\alpha}}\frac{4K_{T+1}\tau}{\beta}}{K_{T+1}(K_{T+1}+3)}\right\} \notag\\
        &\leq& \max\left\{\frac{291600LR_T^2\ln^2\frac{4K_{T+1}\tau}{\beta}}{K_{T+1}^2}, \frac{5400\sigma R_T\ln^{\frac{\alpha-1}{\alpha}}\frac{4K_{T+1}\tau}{\beta}}{K_{T+1}^{\frac{\alpha-1}{\alpha}}}\right\}\notag\\
        &\overset{\eqref{eq:R_clipped_SSTM_K_t}}{\leq}& \varepsilon_{T+1} = \frac{\mu R_T^2}{4} \notag
    \end{eqnarray}
    and, due to the strong convexity,
    \begin{equation*}
        \|\hat x^{T+1} - x^*\|^2 \leq \frac{2(f(\hat x^{T+1}) - f(x^*))}{\mu} \leq \frac{R_T^2}{2} = R_{T+1}^2.
    \end{equation*}
    Thus, we finished the inductive part of the proof. In particular, with probability at least $1 - \beta$ inequalities
    \begin{equation}
        f(\hat{x}^l) - f(x^*) \leq \varepsilon_l, \quad \|\hat x^l - x^*\|^2 \leq  R_l^2 = \frac{R^2}{2^l}\notag
    \end{equation}
    hold for $l = 1,\ldots, \tau$ simultaneously, which gives for $l = \tau$ that with probability at least $1 - \beta$
    \begin{equation*}
        f(\hat{x}^\tau) - f(x^*) \leq \varepsilon_\tau = \frac{\mu R_{\tau-1}^2}{4} = \frac{\mu R^2}{2^{\tau+1}} \overset{\eqref{eq:R_clipped_SSTM_epsilon_R_t_tau}}{\leq} \varepsilon.
    \end{equation*}
    It remains to calculate the overall number of oracle calls during all runs of \algname{clipped-SSTM}. We have
    \begin{eqnarray*}
        \sum\limits_{t=1}^\tau K_t &=& \cO\left(\sum\limits_{t=1}^\tau \max\left\{ \sqrt{\frac{LR_{t-1}^2}{\varepsilon_t}}\ln\left(\frac{\sqrt{LR_{t-1}^2}\tau}{\sqrt{\varepsilon_t}\beta}\right), \left(\frac{\sigma R_{t-1}}{\varepsilon_t}\right)^{\frac{\alpha}{\alpha-1}} \ln\left(\frac{\tau}{\beta}\left(\frac{\sigma R_{t-1}}{\varepsilon_t}\right)^{\frac{\alpha}{\alpha-1}}\right) \right\} \right)\\
        &=& \cO\left(\sum\limits_{t=1}^\tau \max\left\{ \sqrt{\frac{L}{\mu}}\ln\left(\frac{\sqrt{L}\tau}{\sqrt{\mu}\beta}\right), \left(\frac{\sigma}{\mu R_{t-1}}\right)^{\frac{\alpha}{\alpha-1}} \ln\left(\frac{\tau}{\beta}\left(\frac{\sigma }{\mu R_{t-1}}\right)^{\frac{\alpha}{\alpha-1}}\right) \right\} \right)\\
        &=& \cO\left(\max\left\{ \tau \sqrt{\frac{L}{\mu}}\ln\left(\frac{\sqrt{L}\tau}{\sqrt{\mu}\beta}\right), \sum\limits_{t=1}^\tau\left(\frac{\sigma \cdot 2^{\nicefrac{t}{2}}}{\mu R}\right)^{\frac{\alpha}{\alpha-1}} \ln\left(\frac{\tau}{\beta}\left(\frac{\sigma \cdot 2^{\nicefrac{t}{2}}}{\mu R}\right)^{\frac{\alpha}{\alpha-1}}\right) \right\} \right)\\
        &=& \cO\left(\max\left\{ \sqrt{\frac{L}{\mu}}\ln\left(\frac{\mu R^2}{\varepsilon}\right)\ln\left(\frac{\sqrt{L}}{\sqrt{\mu}\beta}\ln\left(\frac{\mu R^2}{\varepsilon}\right)\right), \left(\frac{\sigma}{\mu R}\right)^{\frac{\alpha}{\alpha-1}} \ln\left(\frac{\tau}{\beta}\left(\frac{\sigma \cdot 2^{\nicefrac{\tau}{2}}}{\mu R}\right)^{\frac{\alpha}{\alpha-1}}\right)\sum\limits_{t=1}^\tau 2^\frac{\alpha t}{2(\alpha-1)}\right\} \right)\\
        &=& \cO\left(\max\left\{ \sqrt{\frac{L}{\mu}}\ln\left(\frac{\mu R^2}{\varepsilon}\right)\ln\left(\frac{\sqrt{L}}{\sqrt{\mu}\beta}\ln\left(\frac{\mu R^2}{\varepsilon}\right)\right), \left(\frac{\sigma}{\mu R}\right)^{\frac{\alpha}{\alpha-1}} \ln\left(\frac{\tau}{\beta}\left(\frac{\sigma}{\mu R}\right)^{\frac{\alpha}{\alpha-1}}\cdot 2^{\frac{\alpha}{2(\alpha-1)}}\right)2^\frac{\alpha \tau}{2(\alpha-1)}\right\} \right)\\
        &=& \cO\left(\max\left\{ \sqrt{\frac{L}{\mu}}\ln\left(\frac{\mu R^2}{\varepsilon}\right)\ln\left(\frac{\sqrt{L}}{\sqrt{\mu}\beta}\ln\left(\frac{\mu R^2}{\varepsilon}\right)\right), \left(\frac{\sigma^2}{\mu \varepsilon}\right)^{\frac{\alpha}{2(\alpha-1)}} \ln\left(\frac{1}{\beta}\left(\frac{\sigma^2}{\mu \varepsilon}\right)^{\frac{\alpha}{2(\alpha-1)}}\ln\left(\frac{\mu R^2}{\varepsilon}\right)\right)\right\} \right),
    \end{eqnarray*}
    which concludes the proof.
\end{proof}

\clearpage

\section{Missing Proofs for \algname{clipped-SEG}}\label{appendix:SEG}
In this section, we provide the complete formulation of the main results for \algname{clipped-SSTM} and \algname{R-clipped-SSTM} and the missing proofs. For brevity, we will use the following notation: $\tF_{\xi_1^k}(x^{k}) = \clip\left(F_{\xi_1^k}(x^{k}), \lambda_k\right)$ and $\tF_{\xi_2^k}(\tx^{k}) = \clip\left(F_{\xi_2^k}(\tx^{k}), \lambda_k\right)$.

\begin{algorithm}[h]
\caption{Clipped Stochastic Extragradient (\algname{clipped-SEG}) \citep{gorbunov2022clipped}}
\label{alg:clipped-SEG}   
\begin{algorithmic}[1]
\REQUIRE starting point $x^0$, number of iterations $K$, stepsize $\gamma > 0$, clipping levels $\{\lambda_k\}_{k=0}^{K-1}$.
\FOR{$k=0,\ldots, K$}
\STATE Compute $\tF_{\xi_1^k}(x^{k}) = \clip\left(F_{\xi_1^k}(x^{k}), \lambda_k\right)$ using a fresh sample $\xi_1^k \sim \cD_k$
\STATE $\tx^{k} = x^k - \gamma \tF_{\xi_1^k}(x^{k})$
\STATE Compute $\tF_{\xi_2^k}(\tx^{k}) = \clip\left(F_{\xi_2^k}(\tx^{k}), \lambda_k\right)$ using a fresh sample $\xi_2^k \sim \cD_k$
\STATE $x^{k+1} = x^k - \gamma \tF_{\xi_2^k}(\tx^{k})$
\ENDFOR
\ENSURE $x^{K+1}$ or $\tx_{\avg}^K = \frac{1}{K+1}\sum\limits_{k=0}^K \tx^K$
\end{algorithmic}
\end{algorithm}

\subsection{Monotone Problems}
We start with the following lemma derived by \citet{gorbunov2022extragradient}. Since this lemma handles only deterministic part of the algorithm, the proof is the same as in the original work.
\begin{lemma}[Lemma C.1 from \citep{gorbunov2022extragradient}]\label{lem:optimization_lemma_gap_SEG}
    Let Assumptions~\ref{as:L_Lip} and \ref{as:monotonicity} hold for $Q = B_{4R}(x^*)$, where $R \geq \|x^0 - x^*\|$ and $0 < \gamma \leq \nicefrac{1}{\sqrt{2}L}$. If $x^k$ and $\tx^k$ lie in $B_{4R}(x^*)$ for all $k = 0,1,\ldots, K$ for some $K\geq 0$, then for all $u \in B_{4R}(x^*)$ the iterates produced by \algname{clipped-SEG} satisfy
    \begin{eqnarray}
        \langle F(u), \tx^K_{\avg} - u\rangle &\leq& \frac{\|x^0 - u\|^2 - \|x^{K+1} - u\|^2}{2\gamma(K+1)} + \frac{\gamma}{2(K+1)}\sum\limits_{k=0}^K\left(\|\theta_k\|^2 + 2\|\omega_k\|^2\right)\notag\\
        &&\quad + \frac{1}{K+1}\sum\limits_{k=0}^K\langle x^k - u - \gamma F(\tx^k), \theta_k\rangle, \label{eq:optimization_lemma_SEG}\\
        \tx^K_{\avg} &\eqdef& \frac{1}{K+1}\sum\limits_{k=0}^{K}\tx^k, \label{eq:tx_avg_SEG}\\
        \theta_k &\eqdef& F(\tx^k) - \tF_{\xi_2^k}(\tx^k), \label{eq:theta_k_SEG}\\
        \omega_k &\eqdef& F(x^k) - \tF_{\xi_1^k}(x^k). \label{eq:omega_k_SEG}
    \end{eqnarray}
\end{lemma}

Using this lemma we prove the main convergence result for \algname{clipped-SEG} in the monotone case.
\begin{theorem}[Case 1 in Theorem~\ref{thm:clipped_SEG_main_theorem}]\label{thm:main_result_gap_SEG}
    Let Assumptions~\ref{as:bounded_alpha_moment}, \ref{as:L_Lip}, \ref{as:monotonicity} hold for $Q = B_{4R}(x^*)$, where $R \geq \|x^0 - x^*\|$, and    
    \begin{eqnarray}
        0< \gamma &\leq& \min\left\{\frac{1}{160L \ln \tfrac{6(K+1)}{\beta}}, \frac{20^{\frac{2-\alpha}{\alpha}}R}{10800^{\frac{1}{\alpha}}(K+1)^{\frac{1}{\alpha}}\sigma \ln^{\frac{\alpha-1}{\alpha}} \tfrac{6(K+1)}{\beta}}\right\}, \label{eq:gamma_SEG}\\
        \lambda_{k} \equiv \lambda &=& \frac{R}{20\gamma \ln \tfrac{6(K+1)}{\beta}}, \label{eq:lambda_SEG}
    \end{eqnarray}
    for some $K \geq 0$ and $\beta \in (0,1]$ such that $\ln \tfrac{6(K+1)}{\beta} \geq 1$. Then, after $K$ iterations the iterates produced by \algname{clipped-SEG} with probability at least $1 - \beta$ satisfy 
    \begin{equation}
        \gap_R(\tx_{\avg}^K) \leq \frac{9R^2}{2\gamma(K+1)} \quad \text{and}\quad \{x^k\}_{k=0}^{K+1} \subseteq B_{3R}(x^*), \{\tx^k\}_{k=0}^{K+1} \subseteq B_{4R}(x^*), \label{eq:main_result}
    \end{equation}
    where $\tx_{\avg}^K$ is defined in \eqref{eq:tx_avg_SEG}. In particular, when $\gamma$ equals the minimum from \eqref{eq:gamma_SEG}, then the iterates produced by \algname{clipped-SEG} after $K$ iterations with probability at least $1-\beta$ satisfy
    \begin{equation}
        \gap_R(\tx_{\avg}^K) = \cO\left(\max\left\{\frac{LR^2\ln\frac{K}{\beta}}{K}, \frac{\sigma R \ln^{\frac{\alpha-1}{\alpha}}\frac{K}{\beta}}{K^{\frac{\alpha-1}{\alpha}}}\right\}\right), \label{eq:clipped_SEG_monotone_case_2_appendix}
    \end{equation}
    meaning that to achieve $\gap_R(\tx_{\avg}^K) \leq \varepsilon$ with probability at least $1 - \beta$ \algname{clipped-SEG} requires
    \begin{equation}
        K = \cO\left(\frac{LR^2}{\varepsilon}\ln\frac{LR^2}{\varepsilon\beta}, \left(\frac{\sigma R}{\varepsilon}\right)^{\frac{\alpha}{\alpha-1}}\ln \frac{\sigma R}{\varepsilon\beta}\right)\quad \text{iterations/oracle calls.} \label{eq:clipped_SEG_monotone_case_complexity_appendix}
    \end{equation}
\end{theorem}
\begin{proof}
    The proof follows similar steps as the proof of Theorem~C.1 from \citep{gorbunov2022clipped}. The key difference is related to the application of Bernstein inequality and estimating biases and variances of stochastic terms.
    
    Let $R_k = \|x^k - x^*\|$ for all $k\geq 0$. As in the previous results, the proof is based on the induction argument and showing that the iterates do not leave some ball around the solution with high probability. More precisely, for each $k = 0,1,\ldots,K+1$ we consider probability event $E_k$ as follows: inequalities
    \begin{gather}
        \underbrace{\max\limits_{u \in B_{R}(x^*)}\left\{ \|x^0 - u\|^2 + 2\gamma \sum\limits_{l = 0}^{t-1} \langle x^l - u - \gamma F(\tx^l), \theta_l \rangle + \gamma^2 \sum\limits_{l=0}^{t-1}\left(\|\theta_l\|^2 + 2\|\omega_l\|^2\right)\right\}}_{A_t} \leq 9R^2, \label{eq:induction_inequality_1_SEG}\\
        \left\|\gamma\sum\limits_{l=0}^{t-1}\theta_l\right\| \leq R \label{eq:induction_inequality_2_SEG}
    \end{gather}
    hold for $t = 0,1,\ldots,k$ simultaneously. We want to prove $\PP\{E_k\} \geq  1 - \nicefrac{k\beta}{(K+1)}$ for all $k = 0,1,\ldots,K+1$ by induction. The base of the induction is trivial: for $k=0$ we have $\|x^0 - u\|^2 \leq 2\|x^0 - x^*\|^2 + 2\|x^* - u\|^2 \leq 4R^2 \leq 9R^2$ and $\|\gamma\sum_{l=0}^{k-1}\theta_l\| = 0$ for any $u \in B_R(x^*)$. Next, assume that for $k = T-1 \leq K$ the statement holds: $\PP\{E_{T-1}\} \geq  1 - \nicefrac{(T-1)\beta}{(K+1)}$. Given this, we need to prove $\PP\{E_{T}\} \geq  1 - \nicefrac{T\beta}{(K+1)}$. We start with showing that $E_{T-1}$ implies $R_t \leq 3R$ for all $t = 0,1,\ldots,T$ (also by induction). For $t = 0$ this is already shown. Now, assume that $R_t \leq 3R$ for all $t = 0,1,\ldots,t'$ for some $t' < T$. Then for $t = 0,1,\ldots, t'$
    \begin{eqnarray}
        \|\tx^t - x^*\| &=& \|x^t - x^* - \gamma \tF_{\xi_1^t}(x^t)\| \leq \|x^t - x^*\| + \gamma \|\tF_{\xi_1^t}(x^t)\|\notag\\
        &\leq& \|x^t - x^*\| + \gamma\lambda \overset{\eqref{eq:lambda_SEG}}{\leq} 3R + \frac{R}{20\ln\frac{6(K+1)}{\beta}}  \leq 4R. \label{eq:gap_thm_SEG_technical_1}
    \end{eqnarray}
    Therefore, the conditions of Lemma~\ref{lem:optimization_lemma_gap_SEG} are satisfied and we have that $E_{T-1}$ implies
    \begin{eqnarray*}
        \max\limits_{u \in B_{R}(x^*)}\left\{2\gamma(t' + 1)\langle F(u), \tx^{t'}_{\avg} - u\rangle + \|x^{t'+1} - u\|^2\right\} &\\
        &\hspace{-3cm}\leq \max\limits_{u \in B_{R}(x^*)}\!\left\{\! \|x^0 - u\|^2 + 2\gamma \sum\limits_{l = 0}^{t'} \langle x^l - u - \gamma F(\tx^l), \theta_l \rangle\!\right\}\\
        & + \gamma^2 \sum\limits_{l=0}^{t'}\left(\|\theta_l\|^2 + 2\|\omega_l\|^2\right)\\
        &\hspace{-10.2cm} \overset{\eqref{eq:induction_inequality_1_SEG}}{\leq} 9R^2,
    \end{eqnarray*}
    meaning that
    \begin{eqnarray*}
        \|x^{t' + 1} - x^*\|^2 \leq \max\limits_{u \in B_{R}(x^*)}\left\{2\gamma(t' + 1)\langle F(u), \tx^{t'}_{\avg} - u\rangle + \|x^{t'+1} - u\|^2\right\} \leq 9R^2,
    \end{eqnarray*}
    i.e., $R_{t'+1} \leq 3R$. In other words, we derived that probability event $E_{T-1}$ implies $R_{t} \leq 3R$ and
    \begin{equation}
        \max\limits_{u \in B_{R}(x^*)}\left\{2\gamma(t + 1)\langle F(u), \tx^{t}_{\avg} - u\rangle + \|x^{t+1} - u\|^2\right\} \leq 9R^2 \label{eq:gap_thm_SEG_technical_1_5}
    \end{equation}
    for all $t = 0, 1, \ldots, T$. In addition, due to \eqref{eq:gap_thm_SEG_technical_1} $E_{T-1}$ also implies that $\|\tx^t - x^*\| \leq 4R$ for all $t = 0, 1, \ldots, T$. Thus, $E_{T-1}$ implies
    \begin{eqnarray}
        \|x^t - x^* - \gamma F(\tx^t)\| &\leq& \|x^t - x^*\| + \gamma\|F(\tx^t)\| \overset{\eqref{eq:L_Lip}}{\leq} 3R + \gamma L\|\tx^t - x^*\|\notag\\
        &\overset{\eqref{eq:gap_thm_SEG_technical_1}}{\leq}& 3R + 4R\gamma L \overset{\eqref{eq:gamma_SEG}}{\leq} 5R, \label{eq:gap_thm_SEG_technical_2}
    \end{eqnarray}
    for all $t = 0, 1, \ldots, T$. Next, we introduce random vectors
    \begin{equation*}
        \eta_t = \begin{cases}x^t - x^* - \gamma F(\tx^t),& \text{if } \|x^t - x^* - \gamma F(\tx^t)\| \leq 5R,\\ 0,& \text{otherwise,} \end{cases}
    \end{equation*}
    for all $t = 0, 1, \ldots, T$. These vectors are bounded almost surely:
     \begin{equation}
        \|\eta_t\| \leq 5R  \label{eq:gap_thm_SEG_technical_3}
    \end{equation}
    for all $t = 0, 1, \ldots, T$. Moreover, due to \eqref{eq:gap_thm_SEG_technical_2}, probability event $E_{T-1}$ implies $\eta_t = x^t - x^* - \gamma F(\tx^t)$ for all $t = 0, 1, \ldots, T$ and
    \begin{eqnarray}
        A_T &=& \max\limits_{u \in B_{R}(x^*)}\left\{ \|x^0 - u\|^2 + 2\gamma \sum\limits_{l = 0}^{T-1} \langle x^* - u, \theta_l \rangle \right\}  + 2\gamma \sum\limits_{l = 0}^{T-1} \langle x^l - x^* - \gamma F(\tx^l), \theta_l \rangle + \gamma^2 \sum\limits_{l=0}^{T-1}\left(\|\theta_l\|^2 + 2\|\omega_l\|^2\right)\notag\\
        &\leq& 4R^2 + 2\gamma\max\limits_{u \in B_{R}(x^*)}\left\{\left\langle x^* - u, \sum\limits_{l = 0}^{T-1}\theta_l \right\rangle \right\} + 2\gamma \sum\limits_{l = 0}^{T-1} \langle \eta_l, \theta_l \rangle + \gamma^2 \sum\limits_{l=0}^{T-1}\left(\|\theta_l\|^2 + 2\|\omega_l\|^2\right)\notag\\
        &=& 4R^2 + 2\gamma R \left\|\sum\limits_{l = 0}^{T-1}\theta_l\right\| + 2\gamma \sum\limits_{l = 0}^{T-1} \langle \eta_l, \theta_l \rangle + \gamma^2 \sum\limits_{l=0}^{T-1}\left(\|\theta_l\|^2 + 2\|\omega_l\|^2\right), \notag
    \end{eqnarray}
    where $A_T$ is defined in \eqref{eq:induction_inequality_1_SEG}.
    
    To handle the sums appeared in the right-hand side of the previous inequality we consider unbiased and biased parts of $\theta_l, \omega_l$:
     \begin{gather}
        \theta_l^u \eqdef \EE_{\xi_2^l}\left[\tF_{\xi_2^l}(\tx^l)\right] - \tF_{\xi_2^l}(\tx^l),\quad \theta_l^b \eqdef F(\tx^l) - \EE_{\xi_2^l}\left[\tF_{\xi_2^l}(\tx^l)\right], \label{eq:gap_thm_SEG_technical_4}\\
        \omega_l^u \eqdef \EE_{\xi_1^l}\left[\tF_{\xi_1^l}(x^l)\right] - \tF_{\xi_1^l}(x^l),\quad \omega_l^b \eqdef F(x^l) - \EE_{\xi_1^l}\left[\tF_{\xi_1^l}(x^l)\right], \label{eq:gap_thm_SEG_technical_5}
    \end{gather}
    for all $l = 0,\ldots, T-1$. By definition we have $\theta_l = \theta_l^u + \theta_l^b$, $\omega_l = \omega_l^u + \omega_l^b$ for all $l = 0,\ldots, T-1$. Therefore, $E_{T-1}$ implies
    \begin{eqnarray}
        A_T &\leq& 4R^2 + 2\gamma R \left\|\sum\limits_{l = 0}^{T-1}\theta_l\right\| + \underbrace{2\gamma \sum\limits_{l = 0}^{T-1} \langle \eta_l, \theta_l^u \rangle}_{\circledOne} + \underbrace{2\gamma \sum\limits_{l = 0}^{T-1} \langle \eta_l, \theta_l^b \rangle}_{\circledTwo}\notag\\
        &&\quad + \underbrace{2\gamma^2 \sum\limits_{l=0}^{T-1}\left(\EE_{\xi_2^l}\left[\|\theta_l^u\|^2\right] + 2\EE_{\xi_1^l}\left[\|\omega_l^u\|^2\right]\right)}_{\circledThree} \notag\\
        &&\quad + \underbrace{2\gamma^2 \sum\limits_{l=0}^{T-1}\left(\|\theta_l^u\|^2 + 2\|\omega_l^u\|^2 - \EE_{\xi_2^l}\left[\|\theta_l^u\|^2\right] - 2\EE_{\xi_1^l}\left[\|\omega_l^u\|^2\right]\right)}_{\circledFour}\notag\\
        &&\quad + \underbrace{2\gamma^2 \sum\limits_{l=0}^{T-1}\left(\|\theta_l^b\|^2 + 2\|\omega_l^b\|^2\right)}_{\circledFive},\label{eq:gap_thm_SEG_technical_6}
    \end{eqnarray}
    where we also apply inequality $\|a+b\|^2 \leq 2\|a\|^2 + 2\|b\|^2$ holding for all $a,b \in \R^d$ to upper bound $\|\theta_l\|^2$ and $\|\omega_l\|^2$. It remains to derive good enough high-probability upper-bounds for the terms $2\gamma R \left\|\sum_{l = 0}^{T-1}\theta_l\right\|, \circledOne, \circledTwo, \circledThree, \circledFour, \circledFive$, i.e., to finish our inductive proof we need to show that $2\gamma R \left\|\sum_{l = 0}^{T-1}\theta_l\right\| + \circledOne + \circledTwo + \circledThree + \circledFour + \circledFive \leq 5R^2$ with high probability. In the subsequent parts of the proof, we will need use many times the bounds for the norm and second moments of $\theta_{t+1}^u$ and $\theta_{t+1}^b$. First, by definition of clipping operator we have with probability $1$ that 
    \begin{equation}
        \|\theta_l^u\| \leq 2\lambda,\quad \|\omega_l^u\| \leq 2\lambda. \label{eq:theta_omega_magnitude}
    \end{equation}
    Moreover, since $E_{T-1}$ implies that
    \begin{gather*}
        \|F(x^l)\| \overset{\eqref{eq:L_Lip}}{\leq} L\|x^l - x^*\| \leq 3LR \overset{\eqref{eq:gamma_SEG}}{\leq} \frac{R}{40\gamma \ln\tfrac{6(K+1)}{\beta}} \overset{\eqref{eq:lambda_SEG}}{=} \frac{\lambda}{2},\\
        \|F(\tx^l)\| \overset{\eqref{eq:L_Lip}}{\leq} L\|\tx^l - x^*\| \overset{\eqref{eq:gap_thm_SEG_technical_1}}{\leq} 4LR \overset{\eqref{eq:gamma_SEG}}{\leq} \frac{R}{40\gamma \ln\tfrac{6(K+1)}{\beta}} \overset{\eqref{eq:lambda_SEG}}{=} \frac{\lambda}{2}
    \end{gather*}
    for $t = 0,1,\ldots,T-1$. Then, in view of Lemma~\ref{lem:bias_and_variance_clip}, we have that $E_{T-1}$ implies
    \begin{gather}
        \left\|\theta_l^b\right\| \leq \frac{2^\alpha\sigma^\alpha}{\lambda^{\alpha-1}},\quad \left\|\omega_l^b\right\| \leq \frac{2^\alpha\sigma^\alpha}{\lambda^{\alpha-1}}, \label{eq:bias_theta_omega}\\
        \EE_{\xi_2^l}\left[\left\|\theta_l\right\|^2\right] \leq 18 \lambda^{2-\alpha}\sigma^\alpha,\quad \EE_{\xi_1^l}\left[\left\|\omega_l\right\|^2\right] \leq 18 \lambda^{2-\alpha}\sigma^\alpha, \label{eq:distortion_theta_omega}\\
        \EE_{\xi_2^l}\left[\left\|\theta_l^u\right\|^2\right] \leq 18 \lambda^{2-\alpha}\sigma^\alpha,\quad \EE_{\xi_1^l}\left[\left\|\omega_l^u\right\|^2\right] \leq 18 \lambda^{2-\alpha}\sigma^\alpha, \label{eq:variance_theta_omega}
    \end{gather}
    for all $l = 0,1, \ldots, T-1$.

    \paragraph{Upper bound for $\circledOne$.} By definition of $\theta_{l}^u$, we have $\EE_{\xi_2^l}[\theta_{l}^u] = 0$ and
    \begin{equation*}
        \EE_{\xi_2^l}\left[2\gamma \langle \eta_l, \theta_l^u \rangle\right] = 0.
    \end{equation*}
    Next, sum $\circledOne$ has bounded with probability $1$ terms:
    \begin{equation}
        |2\gamma\langle \eta_l, \theta_l^u \rangle| \leq 2\gamma \|\eta_l\|\cdot \|\theta_l^u\| \overset{\eqref{eq:gap_thm_SEG_technical_3},\eqref{eq:theta_omega_magnitude}}{\leq} 20\gamma R\lambda \overset{\eqref{eq:lambda_SEG}}{=} \frac{R^2}{\ln\tfrac{6(K+1)}{\beta}} \eqdef c. \label{eq:gap_thm_SEG_technical_6_5}
    \end{equation}
    The summands also have bounded conditional variances $\sigma_l^2 \eqdef \EE_{\xi_2^l}\left[4\gamma^2 \langle \eta_l, \theta_l^u \rangle^2\right]$:
    \begin{equation}
        \sigma_l^2 \leq \EE_{\xi_2^l}\left[4\gamma^2 \|\eta_l\|^2\cdot \|\theta_l^u\|^2\right] \overset{\eqref{eq:gap_thm_SEG_technical_3}}{\leq} 100\gamma^2 R^2 \EE_{\xi_2^l}\left[\|\theta_l^u\|^2\right]. \label{eq:gap_thm_SEG_technical_7}
    \end{equation}
    In other words, we showed that $\{2\gamma \langle \eta_l, \theta_l^u \rangle\}_{l = 0}^{T-1}$ is a bounded martingale difference sequence with bounded conditional variances $\{\sigma_l^2\}_{l = 0}^{T-1}$. Next, we apply Bernstein's inequality (Lemma~\ref{lem:Bernstein_ineq}) with $X_l = 2\gamma \langle \eta_l, \theta_l^u \rangle$, parameter $c$ as in \eqref{eq:gap_thm_SEG_technical_6_5}, $b = R^2$, $G = \tfrac{R^4}{6\ln\tfrac{6(K+1)}{\beta}}$:
    \begin{equation*}
        \PP\left\{|\circledOne| > R^2 \text{ and } \sum\limits_{l=0}^{T-1}\sigma_l^2 \leq \frac{R^4}{6\ln\tfrac{6(K+1)}{\beta}}\right\} \leq 2\exp\left(- \frac{b^2}{2G + \nicefrac{2cb}{3}}\right) = \frac{\beta}{3(K+1)}.
    \end{equation*}
    Equivalently, we have
    \begin{equation}
        \PP\{E_{\circledOne}\} \geq 1 - \frac{\beta}{3(K+1)},\quad \text{for}\quad E_{\circledOne} = \left\{\text{either} \quad \sum\limits_{l=0}^{T-1}\sigma_l^2 > \frac{R^4}{6\ln\tfrac{6(K+1)}{\beta}}\quad \text{or}\quad |\circledOne| \leq R^2\right\}. \label{eq:bound_1_gap_SEG}
    \end{equation}
    In addition, $E_{T-1}$ implies that
    \begin{eqnarray}
        \sum\limits_{l=0}^{T-1}\sigma_l^2 &\overset{\eqref{eq:gap_thm_SEG_technical_7}}{\leq}& 100\gamma^2 R^2 \sum\limits_{l=0}^{T-1} \EE_{\xi_2^l}\left[\|\theta_l^u\|^2\right] \overset{\eqref{eq:variance_theta_omega}, T \leq K+1}{\leq} 1800(K+1)\gamma^2 R^2 \lambda^{2-\alpha} \sigma^\alpha \notag \\
        &\overset{\eqref{eq:lambda_SEG}}{\leq}& \frac{1800 (K+1)\gamma^{\alpha}\sigma^\alpha R^{4-\alpha}}{20^{2-\alpha} \ln^{2-\alpha}\frac{6(K+1)}{\beta}} \overset{\eqref{eq:gamma_SEG}}{\leq} \frac{R^4}{6\ln\tfrac{6(K+1)}{\beta}}. \label{eq:bound_1_variances_gap_SEG}
    \end{eqnarray}

    \textbf{Upper bound for $\circledTwo$.} From $E_{T-1}$ it follows that
    \begin{eqnarray}
        \circledTwo &\leq& 2\gamma \sum\limits_{l=0}^{T-1}\|\eta_l\| \cdot \|\theta_l^b\| \overset{\eqref{eq:gap_thm_SEG_technical_3}, \eqref{eq:bias_theta_omega}, T \leq K+1}{\leq} \frac{10\cdot 2^\alpha(K+1)\gamma R\sigma^\alpha}{\lambda^{\alpha-1}} \notag \\
        &\overset{\eqref{eq:lambda_SEG}}{=}& \frac{10\cdot 2^\alpha\cdot 20^{\alpha-1} (K+1)\gamma^\alpha \sigma^\alpha \ln^{\alpha-1}\tfrac{6(K+1)}{\beta}}{R^{\alpha-2}} \overset{\eqref{eq:gamma_SEG}}{\leq} R^2. \label{eq:bound_2_variances_gap_SEG}
    \end{eqnarray}

    \textbf{Upper bound for $\circledThree$.} From $E_{T-1}$ it follows that
    \begin{eqnarray}
        2\gamma^2 \sum\limits_{l=0}^{T-1}\EE_{\xi_2^l}[\|\theta_l^u\|^2] &\overset{\eqref{eq:distortion_theta_omega}, T \leq K+1}{\leq}& 36\gamma^2 (K+1)\lambda^{2-\alpha}\sigma^\alpha \overset{\eqref{eq:lambda_SEG}}{=} \frac{36\gamma^\alpha (K+1)\sigma^\alpha}{20^{2-\alpha} \ln^{2-\alpha} \tfrac{6(K+1)}{\beta}}  \overset{\eqref{eq:gamma_SEG}}{\leq} \frac{1}{12} R^2, \label{eq:sum_theta_squared_bound_gap_SEG}\\
        4\gamma^2 \sum\limits_{l=0}^{T-1}\EE_{\xi_1^l}[\|\omega_l^u\|^2] &\overset{\eqref{eq:distortion_theta_omega}, T \leq K+1}{\leq}& 72\gamma^2 (K+1)\lambda^{2-\alpha}\sigma^\alpha \overset{\eqref{eq:lambda_SEG}}{=} \frac{72\gamma^\alpha (K+1)\sigma^\alpha}{20^{2-\alpha} \ln^{2-\alpha} \tfrac{6(K+1)}{\beta}}  \overset{\eqref{eq:gamma_SEG}}{\leq} \frac{1}{12} R^2, \label{eq:sum_omega_squared_bound_gap_SEG}\\
        \circledThree &\overset{\eqref{eq:sum_theta_squared_bound_gap_SEG}, \eqref{eq:sum_omega_squared_bound_gap_SEG}}{\leq}& \frac{1}{6}R^2. \label{eq:bound_3_variances_gap_SEG}
    \end{eqnarray}

    \paragraph{Upper bound for $\circledFour$.} By the construction we have
    \begin{equation*}
        2\gamma^2\EE_{\xi_1^l,\xi_2^l}\left[\|\theta_l^u\|^2 + 2\|\omega_l^u\|^2 - \EE_{\xi_2^l}\left[\|\theta_l^u\|^2\right] - 2\EE_{\xi_1^l}\left[\|\omega_l^u\|^2\right]\right] = 0.
    \end{equation*}
    Next, sum $\circledOne$ has bounded with probability $1$ terms:
    \begin{eqnarray}
        2\gamma^2\left|\|\theta_l^u\|^2 + 2\|\omega_l^u\|^2 - \EE_{\xi_2^l}\left[\|\theta_l^u\|^2\right] - 2\EE_{\xi_1^l}\left[\|\omega_l^u\|^2\right] \right| &\leq& 2\gamma^2 \|\theta_l^u\|^2 + 2\gamma^2\EE_{\xi_2^l}\left[\|\theta_l^u\|^2\right]\notag\\
        &&\quad + 4\gamma^2 \|\omega_l^u\|^2 + 4\gamma^2\EE_{\xi_1^l}\left[\|\omega_l^u\|^2\right] \notag\\
        &\overset{\eqref{eq:theta_omega_magnitude}}{\leq}& 48\gamma^2 \lambda^2 \notag\\
        &\overset{\eqref{eq:lambda_SEG}}{\leq}& \frac{R^2}{6\ln\tfrac{6(K+1)}{\beta}} \eqdef c.\label{eq:gap_thm_SEG_technical_6_5_1}
    \end{eqnarray}
    The summands also have bounded conditional variances\newline $\widetilde\sigma_l^2 \eqdef 4\gamma^4\EE_{\xi_1^l, \xi_2^l}\left[\left|\|\theta_l^u\|^2 + 2\|\omega_l^u\|^2 - \EE_{\xi_2^l}\left[\|\theta_l^u\|^2\right] - 2\EE_{\xi_1^l}\left[\|\omega_l^u\|^2\right] \right|^2\right]$:
    \begin{eqnarray}
        \widetilde\sigma_l^2 &\overset{\eqref{eq:gap_thm_SEG_technical_6_5_1}}{\leq}& \frac{\gamma^2R^2}{3\ln\tfrac{6(K+1)}{\beta}}\EE_{\xi_1^l,\xi_2^l}\left[\left|\|\theta_l^u\|^2 + 2\|\omega_l^u\|^2 - \EE_{\xi_2^l}\left[\|\theta_l^u\|^2\right] - 2\EE_{\xi_1^l}\left[\|\omega_l^u\|^2\right] \right|\right]\notag\\
        &\leq& \frac{2\gamma^2R^2}{3\ln\tfrac{6(K+1)}{\beta}} \EE_{\xi_1^l,\xi_2^l}\left[\|\theta_l^u\|^2 + 2\|\omega_l^u\|^2 \right]. \label{eq:gap_thm_SEG_technical_7_1}
    \end{eqnarray}
    In other words, we showed that $\left\{2\gamma^2\left(\|\theta_l^u\|^2 + 2\|\omega_l^u\|^2 - \EE_{\xi_2^l}\left[\|\theta_l^u\|^2\right] - 2\EE_{\xi_1^l}\left[\|\omega_l^u\|^2\right]\right)\right\}_{l = 0}^{T-1}$ is a bounded martingale difference sequence with bounded conditional variances $\{\sigma_l^2\}_{l = 0}^{T-1}$. Next, we apply Bernstein's inequality (Lemma~\ref{lem:Bernstein_ineq}) with $X_l = 2\gamma^2\left(\|\theta_l^u\|^2 + 2\|\omega_l^u\|^2 - \EE_{\xi_2^l}\left[\|\theta_l^u\|^2\right] - 2\EE_{\xi_1^l}\left[\|\omega_l^u\|^2\right]\right)$, parameter $c$ as in \eqref{eq:gap_thm_SEG_technical_6_5_1}, $b = \tfrac{1}{6}R^2$, $G = \tfrac{R^4}{216\ln\tfrac{6(K+1)}{\beta}}$:
    \begin{equation*}
        \PP\left\{|\circledFour| > \frac{1}{6}R^2 \text{ and } \sum\limits_{l=0}^{T-1}\widetilde\sigma_l^2 \leq \frac{R^4}{216\ln\tfrac{6(K+1)}{\beta}}\right\} \leq 2\exp\left(- \frac{b^2}{2G + \nicefrac{2cb}{3}}\right) = \frac{\beta}{3(K+1)}.
    \end{equation*}
    Equivalently, we have
    \begin{equation}
        \PP\{E_{\circledFour}\} \geq 1 - \frac{\beta}{3(K+1)},\quad \text{for}\quad E_{\circledFour} = \left\{\text{either} \quad \sum\limits_{l=0}^{T-1}\widetilde\sigma_l^2 > \frac{R^4}{216\ln\tfrac{6(K+1)}{\beta}}\quad \text{or}\quad |\circledFour| \leq \frac{1}{6}R^2\right\}. \label{eq:bound_1_gap_SEG_1}
    \end{equation}
    In addition, $E_{T-1}$ implies that
    \begin{eqnarray}
        \sum\limits_{l=0}^{T-1}\widetilde\sigma_l^2 &\overset{\eqref{eq:gap_thm_SEG_technical_7_1}}{\leq}& \frac{2\gamma^2R^2}{3\ln\tfrac{6(K+1)}{\beta}} \sum\limits_{l=0}^{T-1} \EE_{\xi_1^l,\xi_2^l}\left[\|\theta_l^u\|^2 + 2\|\omega_l^u\|^2 \right]\notag\\
        &\overset{\eqref{eq:variance_theta_omega}, T \leq K+1}{\leq}& \frac{36(K+1)\gamma^2 R^2 \lambda^{2-\alpha} \sigma^\alpha}{\ln\tfrac{6(K+1)}{\beta}} \notag\\
        &\overset{\eqref{eq:lambda_SEG}}{\leq}& \frac{36(K+1)\gamma^\alpha R^{4-\alpha} \sigma^\alpha}{20^{2-\alpha}\ln^{3-\alpha}\tfrac{6(K+1)}{\beta}}  \overset{\eqref{eq:gamma_SEG}}{\leq} \frac{R^4}{216\ln\tfrac{6(K+1)}{\beta}}. \label{eq:bound_1_variances_gap_SEG_1}
    \end{eqnarray}

    \textbf{Upper bound for $\circledFive$.} From $E_{T-1}$ it follows that
    \begin{eqnarray}
        \circledFive &=& 2\gamma^2 \sum\limits_{l=0}^{T-1}\left(\|\theta_l^b\|^2 + 2\|\omega_l^b\|^2\right) \overset{\eqref{eq:bias_theta_omega}, T\leq K+1}{\leq} \frac{6\cdot 2^{2\alpha}\gamma^2\sigma^{2\alpha} (K+1)}{\lambda^{2\alpha-2}}  \notag\\
        &\overset{\eqref{eq:lambda_SEG}}{=}&  \frac{6\cdot 2^{2\alpha}\cdot 20^{2\alpha - 2}\gamma^{2\alpha}\sigma^{2\alpha} (K+1) \ln^{2\alpha-2}\tfrac{6(K+1)}{\beta}}{R^{2\alpha-2}} \overset{\eqref{eq:gamma_SEG}}{\leq} \frac{1}{6}R^2. \label{eq:nsjcbdjhcbfjdhfbj}
    \end{eqnarray}

    \paragraph{Upper bound for $2\gamma R \left\|\sum_{l=0}^{T-1} \theta_l\right\|$.} To upper-bound this sum, we introduce new random vectors:
    \begin{equation*}
        \zeta_l = \begin{cases} \gamma \sum\limits_{r=0}^{l-1}\theta_r,& \text{if } \left\|\gamma \sum\limits_{r=0}^{l-1}\theta_r\right\| \leq R,\\ 0, & \text{otherwise} \end{cases}
    \end{equation*}
    for $l = 1, 2, \ldots, T-1$. These vectors are bounded with probability $1$:
    \begin{equation}
        \|\zeta_l\| \leq R.  \label{eq:gap_thm_SEG_technical_8}
    \end{equation}
    Therefore, taking into account \eqref{eq:induction_inequality_2_SEG}, we derive that $E_{T-1}$ implies
    \begin{eqnarray}
        2\gamma R\left\|\sum\limits_{l = 0}^{T-1}\theta_l\right\| &=& 2R\sqrt{\gamma^2\left\|\sum\limits_{l = 0}^{T-1}\theta_l\right\|^2}\notag\\
        &=& 2R\sqrt{\gamma^2\sum\limits_{l=0}^{T-1}\|\theta_l\|^2 + 2\gamma\sum\limits_{l=0}^{T-1}\left\langle \gamma\sum\limits_{r=0}^{l-1} \theta_r, \theta_l \right\rangle} \notag\\
        &=& 2R \sqrt{\gamma^2\sum\limits_{l=0}^{T-1}\|\theta_l\|^2 + 2\gamma \sum\limits_{l=0}^{T-1} \langle \zeta_l, \theta_l\rangle} \notag\\
        &\overset{\eqref{eq:gap_thm_SEG_technical_4}}{\leq}& 2R \sqrt{\circledThree + \circledFour + \circledFive + \underbrace{2\gamma \sum\limits_{l=0}^{T-1} \langle \zeta_l, \theta_l^u\rangle}_{\circledSix} + \underbrace{2\gamma \sum\limits_{l=0}^{T-1} \langle \zeta_l, \theta_l^b}_{\circledSeven}\rangle}. \label{eq:norm_sum_theta_bound_gap_SEG}
    \end{eqnarray}
    Similarly to the previous parts of the proof, we bound $\circledSix$ and $\circledSeven$.

    \paragraph{Upper bound for $\circledSix$.} By definition of $\theta_{l}^u$, we have $\EE_{\xi_2^l}[\theta_{l}^u] = 0$ and
    \begin{equation*}
        \EE_{\xi_2^l}\left[2\gamma \langle \zeta_l, \theta_l^u \rangle\right] = 0.
    \end{equation*}
    Next, sum $\circledSix$ has bounded with probability $1$ terms:
    \begin{equation}
        |2\gamma\langle \zeta_l, \theta_l^u \rangle| \leq 2\gamma \|\eta_l\|\cdot \|\theta_l^u\| \overset{\eqref{eq:gap_thm_SEG_technical_8},\eqref{eq:theta_omega_magnitude}}{\leq} 4\gamma R\lambda \overset{\eqref{eq:lambda_SEG}}{\leq} \frac{R^2}{4\ln\tfrac{6(K+1)}{\beta}} \eqdef c. \label{eq:gap_thm_SEG_technical_8_5}
    \end{equation}
    The summands also have bounded conditional variances $\hat\sigma_l^2 \eqdef \EE_{\xi_2^l}\left[4\gamma^2 \langle \zeta_l, \theta_l^u \rangle^2\right]$:
    \begin{equation}
        \hat\sigma_l^2 \leq \EE_{\xi_2^l}\left[4\gamma^2 \|\zeta_l\|^2\cdot \|\theta_l^u\|^2\right] \overset{\eqref{eq:gap_thm_SEG_technical_8}}{\leq} 4\gamma^2 R^2 \EE_{\xi_2^l}\left[\|\theta_l^u\|^2\right]. \label{eq:gap_thm_SEG_technical_9}
    \end{equation}
    In other words, we showed that $\{2\gamma \langle \zeta_l, \theta_l^u \rangle\}_{l = 0}^{T-1}$ is a bounded martingale difference sequence with bounded conditional variances $\{\hat \sigma_l^2\}_{l = 0}^{T-1}$. Next, we apply Bernstein's inequality (Lemma~\ref{lem:Bernstein_ineq}) with $X_l = 2\gamma \langle \zeta_l, \theta_l^u \rangle$, parameter $c$ as in \eqref{eq:gap_thm_SEG_technical_8_5}, $b = \tfrac{R^2}{4}$, $G = \tfrac{R^4}{96\ln\tfrac{6(K+1)}{\beta}}$:
    \begin{equation*}
        \PP\left\{|\circledFive| > \frac{1}{4}R^2 \text{ and } \sum\limits_{l=0}^{T-1}\hat\sigma_l^2 \leq \frac{R^4}{96\ln\tfrac{4(K+1)}{\beta}}\right\} \leq 2\exp\left(- \frac{b^2}{2G + \nicefrac{2cb}{3}}\right) = \frac{\beta}{3(K+1)}.
    \end{equation*}
    Equivalently, we have
    \begin{equation}
        E_{\circledSix} = \left\{ \text{either} \quad \sum\limits_{l=0}^{T-1}\hat\sigma_l^2 > \frac{R^4}{96\ln\tfrac{6(K+1)}{\beta}}\quad \text{or}\quad |\circledFive| \leq \frac{1}{4}R^2\right\}. \label{eq:bound_4_gap_SEG}
    \end{equation}
    In addition, $E_{T-1}$ implies that
    \begin{eqnarray}
        \sum\limits_{l=0}^{T-1}\hat\sigma_l^2 &\overset{\eqref{eq:gap_thm_SEG_technical_9}}{\leq}& 4\gamma^2 R^2 \sum\limits_{l=0}^{T-1} \EE_{\xi_2^l}\left[\|\theta_l^u\|^2\right] \overset{\eqref{eq:variance_theta_omega}, T \leq K+1}{\leq} 72(K+1)\gamma^2 R^2 \lambda^{2-\alpha} \sigma^\alpha\notag\\
        &\overset{\eqref{eq:lambda_SEG}}{=}& \frac{72(K+1)\gamma^\alpha R^{4-\alpha} \sigma^\alpha}{20^{2-\alpha} \ln^{2-\alpha}\tfrac{6(K+1)}{\beta}} \overset{\eqref{eq:gamma_SEG}}{\leq} \frac{R^4}{96\ln\tfrac{6(K+1)}{\beta}}. \label{eq:bound_4_variances_gap_SEG}
    \end{eqnarray}

    \textbf{Upper bound for $\circledSeven$.} From $E_{T-1}$ it follows that
    \begin{eqnarray}
        \circledSeven &\leq& 2\gamma \sum\limits_{l=0}^{T-1}\|\zeta_l\| \cdot \|\theta_l^b\| \overset{\eqref{eq:gap_thm_SEG_technical_8}, \eqref{eq:bias_theta_omega}, T \leq K+1}{\leq} \frac{2^{\alpha+1}(K+1)\gamma R\sigma^\alpha}{\lambda^{\alpha-1}} \notag \\
        &\overset{\eqref{eq:lambda_SEG}}{=}& \frac{2^{\alpha+1}\cdot 20^{\alpha-1}(K+1)\gamma^\alpha \sigma^\alpha \ln^{\alpha-1}\tfrac{6(K+1)}{\beta}}{R^{\alpha-2}} \overset{\eqref{eq:gamma_SEG}}{\leq} \frac{1}{4}R^2. \label{eq:bound_5_variances_gap_SEG}
    \end{eqnarray}

    Now, we have the upper bounds for  $2\gamma R \left\|\sum_{l = 0}^{T-1}\theta_l\right\|, \circledOne, \circledTwo, \circledThree, \circledFour, \circledFive$. In particular, probability event $E_{T-1}$ implies
    \begin{gather*}
        A_T \overset{\eqref{eq:gap_thm_SEG_technical_6}}{\leq} 4R^2 + 2\gamma R\left\|\sum\limits_{l=0}^{T-1} \theta_l\right\| + \circledOne + \circledTwo + \circledThree + \circledFour + \circledFive,\\
        2\gamma R\left\|\sum\limits_{l=0}^{T-1} \theta_l\right\| \overset{\eqref{eq:norm_sum_theta_bound_gap_SEG}}{\leq} 2R \sqrt{\circledThree + \circledFour + \circledFive + \circledSix + \circledSeven},\\
        \circledTwo \overset{\eqref{eq:bound_2_variances_gap_SEG}}{\leq} R^2,\quad \circledThree \overset{\eqref{eq:bound_3_variances_gap_SEG}}{\leq} \frac{1}{6}R^2,\quad \circledFive \overset{\eqref{eq:nsjcbdjhcbfjdhfbj}}{\leq} \frac{1}{6}R^2,\quad \circledSeven \overset{\eqref{eq:bound_5_variances_gap_SEG}}{\leq} \frac{1}{4}R^2,\\
        \sum\limits_{l=0}^{T-1}\sigma_l^2 \overset{\eqref{eq:bound_1_variances_gap_SEG}}{\leq}  \frac{R^4}{6\ln\tfrac{6(K+1)}{\beta}},\quad \sum\limits_{l=0}^{T-1}\widetilde\sigma_l^2 \overset{\eqref{eq:bound_1_variances_gap_SEG_1}}{\leq} \frac{R^4}{216\ln\tfrac{6(K+1)}{\beta}},\quad \sum\limits_{l=0}^{T-1}\hat\sigma_l^2 \overset{\eqref{eq:bound_4_variances_gap_SEG}}{\leq}  \frac{R^4}{96\ln\tfrac{6(K+1)}{\beta}}.
    \end{gather*}
    Moreover, we also have (see \eqref{eq:bound_1_gap_SEG}, \eqref{eq:bound_1_gap_SEG_1}, \eqref{eq:bound_4_gap_SEG} and our induction assumption)
    \begin{gather*}
        \PP\{E_{T-1}\} \geq 1 - \frac{(T-1)\beta}{K+1},\\
        \PP\{E_{\circledOne}\} \geq 1 - \frac{\beta}{3(K+1)}, \quad \PP\{E_{\circledFour}\} \geq 1 - \frac{\beta}{3(K+1)}, \quad \PP\{E_{\circledSix}\} \geq 1 - \frac{\beta}{3(K+1)} ,
    \end{gather*}
    where
    \begin{eqnarray}
        E_{\circledOne}&=& \left\{\text{either} \quad \sum\limits_{l=0}^{T-1}\sigma_l^2 > \frac{R^4}{6\ln\tfrac{6(K+1)}{\beta}}\quad \text{or}\quad |\circledOne| \leq R^2\right\},\notag\\
        E_{\circledFour}&=& \left\{\text{either} \quad \sum\limits_{l=0}^{T-1}\widetilde\sigma_l^2 > \frac{R^4}{216\ln\tfrac{6(K+1)}{\beta}}\quad \text{or}\quad |\circledFour| \leq \frac{1}{6}R^2\right\},\notag\\
        E_{\circledSix}&=& \left\{\text{either} \quad \sum\limits_{l=0}^{T-1}\hat\sigma_l^2 > \frac{R^4}{96\ln\tfrac{6(K+1)}{\beta}}\quad \text{or}\quad |\circledSix| \leq \frac{1}{4}R^2\right\}.\notag
    \end{eqnarray}
    Thus, probability event $E_{T-1} \cap E_{\circledOne} \cap E_{\circledFour} \cap E_{\circledSix}$ implies
    \begin{eqnarray}
        \left\|\gamma\sum\limits_{l=0}^{T-1} \theta_l\right\| &\leq& \sqrt{\frac{1}{6}R^2 + \frac{1}{6}R^2 + \frac{1}{6}R^2 + \frac{1}{4}R^2 + \frac{1}{4}R^2} = R, \label{eq:gap_thm_SEG_technical_10} \\
        A_T &\leq& 4R^2 + 2R\sqrt{\frac{1}{6}R^2 + \frac{1}{6}R^2 + \frac{1}{6}R^2 + \frac{1}{4}R^2 + \frac{1}{4}R^2}\notag\\
        &&\quad + R^2 + R^2 + \frac{1}{6}R^2 + \frac{1}{6}R^2 + \frac{1}{6}R^2\notag\\
        &\leq& 9R^2, \label{eq:gap_thm_SEG_technical_11}
    \end{eqnarray}
    which is equivalent to \eqref{eq:induction_inequality_1_SEG} and \eqref{eq:induction_inequality_2_SEG} for $t = T$, and 
    \begin{equation*}
        \PP\{E_T\} \geq \PP\{E_{T-1} \cap E_{\circledOne} \cap E_{\circledFour} \cap E_{\circledSix}\} = 1 - \PP\{\overline{E}_{T-1} \cup \overline{E}_{\circledOne} \cup \overline{E}_{\circledFour} \cup \overline{E}_{\circledSix}\} \geq 1 - \frac{T\beta}{K+1}.
    \end{equation*}
    This finishes the inductive part of our proof, i.e., for all $k = 0,1,\ldots,K+1$ we have $\PP\{E_k\} \geq 1 - \nicefrac{k\beta}{(K+1)}$. In particular, for $k = K+1$ we have that with probability at least $1 - \beta$
    \begin{eqnarray*}
        \gap_{R}(\tx^{K}_{\avg}) &=& \max\limits_{u \in B_{R}(x^*)}\left\{\langle F(u), \tx^{K}_{\avg} - u\rangle\right\}\\
        &\leq& \frac{1}{2\gamma(K + 1)}\max\limits_{u \in B_{R}(x^*)}\left\{2\gamma(K + 1)\langle F(u), \tx^{t}_{\avg} - u\rangle + \|x^{K+1} - u\|^2\right\}\\
        &\overset{\eqref{eq:gap_thm_SEG_technical_1_5}}{\leq}& \frac{9R^2}{2\gamma(K+1)}.
    \end{eqnarray*}
    Finally, if 
    \begin{equation*}
        \gamma = \min\left\{\frac{1}{160L \ln \tfrac{6(K+1)}{\beta}}, \frac{20^{\frac{2-\alpha}{\alpha}}R}{10800^{\frac{1}{\alpha}}(K+1)^{\frac{1}{\alpha}}\sigma \ln^{\frac{\alpha-1}{\alpha}} \tfrac{6(K+1)}{\beta}}\right\}
    \end{equation*}
    then with probability at least $1-\beta$
    \begin{eqnarray*}
        \gap_{R}(\tx^{K}_{\avg}) &\leq& \frac{9R^2}{2\gamma(K+1)} = \max\left\{\frac{720 LR^2 \ln \tfrac{6(K+1)}{\beta}}{K+1}, \frac{9 \sigma R \ln^{\frac{\alpha-1}{\alpha}} \tfrac{6(K+1)}{\beta}}{2\cdot 20^{\frac{2-\alpha}{\alpha}} (K+1)^{\frac{\alpha-1}{\alpha}}}\right\}\\
        &=& \cO\left(\max\left\{\frac{LR^2\ln\frac{K}{\beta}}{K}, \frac{\sigma R \ln^{\frac{\alpha-1}{\alpha}}\frac{K}{\beta}}{K^{\frac{\alpha-1}{\alpha}}}\right\}\right).
    \end{eqnarray*}
    To get $\gap_{R}(\tx^{K}_{\avg}) \leq \varepsilon$ with probability at least $1-\beta$ it is sufficient to choose $K$ such that both terms in the maximum above are $\cO(\varepsilon)$. This leads to
    \begin{equation*}
         K = \cO\left(\frac{LR^2}{\varepsilon}\ln\frac{LR^2}{\varepsilon\beta}, \left(\frac{\sigma R}{\varepsilon}\right)^{\frac{\alpha}{\alpha-1}}\ln \frac{\sigma R}{\varepsilon\beta}\right)
    \end{equation*}
    that concludes the proof.
\end{proof}

\subsection{Quasi-Strongly Monotone Problems}
As in the monotone case, we use another lemma from \citep{gorbunov2022clipped} that handles the deterministic part of \algname{clipped-SEG} in the quasi-strongly monotone case.
\begin{lemma}[Lemma~C.3 from \citep{gorbunov2022clipped}]\label{lem:optimization_lemma_str_mon_SEG}
    Let Assumptions~\ref{as:L_Lip}, \ref{as:QSM} hold for $Q = B_{3R}(x^*) = \{x\in\R^d\mid \|x - x^*\| \leq 3R\}$, where $R \geq \|x^0 - x^*\|$, and $0 < \gamma \leq \nicefrac{1}{2(L+2\mu)}$. If $x^k$ and $\tx^k$ lie in $B_{3R}(x^*)$ for all $k = 0,1,\ldots, K$ for some $K\geq 0$, then the iterates produced by \algname{clipped-SEG} satisfy
    \begin{eqnarray}
        \|x^{K+1} - x^*\|^2 &\leq& (1 - \gamma \mu)^{K+1}\|x^0 - x^*\|^2 - 4\gamma^3 \mu \sum\limits_{k=0}^K (1-\gamma\mu)^{K-k}\langle F(x^k), \omega_k \rangle\notag\\
        &&\quad + 2\gamma \sum\limits_{k=0}^K (1-\gamma\mu)^{K-k} \langle x^k - x^* - \gamma F(\tx^k), \theta_k \rangle\notag\\
        &&\quad + \gamma^2 \sum\limits_{k=0}^K (1-\gamma\mu)^{K-k} \left(\|\theta_k\|^2 + 4\|\omega_k\|^2\right), \label{eq:optimization_lemma_SEG_str_mon}
    \end{eqnarray}
    where $\theta_k, \omega_k$ are defined in \eqref{eq:theta_k_SEG}, \eqref{eq:omega_k_SEG}.
\end{lemma}

Using this lemma we prove the main convergence result for \algname{clipped-SEG} in the quasi-strongly monotone case.

\begin{theorem}[Case 2 in Theorem~\ref{thm:clipped_SEG_main_theorem}]\label{thm:main_result_str_mon_SEG}
    Let Assumptions~\ref{as:bounded_alpha_moment}, \ref{as:L_Lip}, \ref{as:QSM}, hold for $Q = B_{3R}(x^*) = \{x\in\R^d\mid \|x - x^*\| \leq 3R\}$, where $R \geq \|x^0 - x^*\|$, and
    \begin{eqnarray}
        0< \gamma &\leq& \min\left\{\frac{1}{650 L \ln \tfrac{6(K+1)}{\beta}}, \frac{\ln(B_K)}{\mu(K+1)}\right\}, \label{eq:gamma_SEG_str_mon}\\
        B_K &=& \max\left\{2, \frac{(K+1)^{\frac{2(\alpha-1)}{\alpha}}\mu^2R^2}{264600^{\frac{2}{\alpha}}\sigma^2\ln^{\frac{2(\alpha-1)}{\alpha}}\left(\frac{6(K+1)}{\beta}\right)\ln^2(B_K)} \right\} \label{eq:B_K_SEG_str_mon_1} \\
        &=& \cO\left(\max\left\{2, \frac{K^{\frac{2(\alpha-1)}{\alpha}}\mu^2R^2}{\sigma^2\ln^{\frac{2(\alpha-1)}{\alpha}}\left(\frac{K}{\beta}\right)\ln^2\left(\max\left\{2, \frac{K^{\frac{2(\alpha-1)}{\alpha}}\mu^2R^2}{\sigma^2\ln^{\frac{2(\alpha-1)}{\alpha}}\left(\frac{K}{\beta}\right)} \right\}\right)} \right\}\right\}, \label{eq:B_K_SEG_str_mon_2} \\
        \lambda_k &=& \frac{\exp(-\gamma\mu(1 + \nicefrac{k}{2}))R}{120\gamma \ln \tfrac{6(K+1)}{\beta}}, \label{eq:lambda_SEG_str_mon}
    \end{eqnarray}
    for some $K \geq 0$ and $\beta \in (0,1]$ such that $\ln \tfrac{6(K+1)}{\beta} \geq 1$. Then, after $K$ iterations the iterates produced by \algname{clipped-SEG} with probability at least $1 - \beta$ satisfy 
    \begin{equation}
        \|x^{K+1} - x^*\|^2 \leq 2\exp(-\gamma\mu(K+1))R^2. \label{eq:main_result_str_mon}
    \end{equation}
    In particular, when $\gamma$ equals the minimum from \eqref{eq:gamma_SEG_str_mon}, then the iterates produced by \algname{clipped-SEG} after $K$ iterations with probability at least $1-\beta$ satisfy
    \begin{equation}
        \|x^{K} - x^*\|^2 = \cO\left(\max\left\{R^2\exp\left(- \frac{\mu K}{L \ln \tfrac{K}{\beta}}\right), \frac{\sigma^2\ln^{\frac{2(\alpha-1)}{\alpha}}\left(\frac{K}{\beta}\right)\ln^2\left(\max\left\{2, \frac{K^{\frac{2(\alpha-1)}{\alpha}}\mu^2R^2}{\sigma^2\ln^{\frac{2(\alpha-1)}{\alpha}}\left(\frac{K}{\beta}\right)} \right\}\right)}{K^{\frac{2(\alpha-1)}{\alpha}}\mu^2}\right\}\right), \label{eq:clipped_SEG_str_monotone_case_2_appendix}
    \end{equation}
    meaning that to achieve $\|x^{K} - x^*\|^2 \leq \varepsilon$ with probability at least $1 - \beta$ \algname{clipped-SEG} requires
    \begin{equation}
        K = \cO\left(\frac{L}{\mu}\ln\left(\frac{R^2}{\varepsilon}\right)\ln\left(\frac{L}{\mu \beta}\ln\frac{R^2}{\varepsilon}\right), \left(\frac{\sigma^2}{\mu^2\varepsilon}\right)^{\frac{\alpha}{2(\alpha-1)}}\ln \left(\frac{1}{\beta} \left(\frac{\sigma^2}{\mu^2\varepsilon}\right)^{\frac{\alpha}{2(\alpha-1)}}\right)\ln^{\frac{\alpha}{\alpha-1}}\left(B_\varepsilon\right)\right) \label{eq:clipped_SEG_str_monotone_case_complexity_appendix}
    \end{equation}
    iterations/oracle calls, where
    \begin{equation*}
        B_\varepsilon = \max\left\{2, \frac{R^2}{\varepsilon \ln \left(\frac{1}{\beta} \left(\frac{\sigma^2}{\mu^2\varepsilon}\right)^{\frac{\alpha}{2(\alpha-1)}}\right)}\right\}.
    \end{equation*}
\end{theorem}
\begin{proof}
    Again, we will closely follow the proof of Theorem~C.3 from \citep{gorbunov2022clipped} and the main difference will be reflected in the application of Bernstein inequality and estimating biases and variances of stochastic terms.
    
    Let $R_k = \|x^k - x^*\|$ for all $k\geq 0$. As in the previous results, the proof is based on the induction argument and showing that the iterates do not leave some ball around the solution with high probability. More precisely, for each $k = 0,1,\ldots,K+1$ we consider probability event $E_k$ as follows: inequalities
    \begin{equation}
        R_t^2 \leq 2 \exp(-\gamma\mu t) R^2 \label{eq:induction_inequality_str_mon_SEG}
    \end{equation}
    hold for $t = 0,1,\ldots,k$ simultaneously. We want to prove $\PP\{E_k\} \geq  1 - \nicefrac{k\beta}{(K+1)}$ for all $k = 0,1,\ldots,K+1$ by induction. The base of the induction is trivial: for $k=0$ we have $R_0^2 \leq R^2 < 2R^2$ by definition. Next, assume that for $k = T-1 \leq K$ the statement holds: $\PP\{E_{T-1}\} \geq  1 - \nicefrac{(T-1)\beta}{(K+1)}$. Given this, we need to prove $\PP\{E_{T}\} \geq  1 - \nicefrac{T\beta}{(K+1)}$. Since $R_t^2 \leq 2\exp(-\gamma\mu t) R^2 \leq 9R^2$, we have $x^t \in B_{3R}(x^*)$, where operator $F$ is $L$-Lipschitz. Thus, $E_{T-1}$ implies
    \begin{eqnarray}
        \|F(x^t)\| &\leq& L\|x^t - x^*\| \overset{\eqref{eq:induction_inequality_str_mon_SEG}}{\leq} \sqrt{2}L\exp(- \nicefrac{\gamma\mu t}{2})R \overset{\eqref{eq:gamma_SEG_str_mon},\eqref{eq:lambda_SEG_str_mon}}{\leq} \frac{\lambda_t}{2} \label{eq:operator_bound_x_t_SEG_str_mon}
    \end{eqnarray}
    and
    \begin{eqnarray}
        \|\omega_t\|^2 &\leq& 2\|\tF_{\xi_1}(x^t)\|^2 + 2\|F(x^t)\|^2 \overset{\eqref{eq:operator_bound_x_t_SEG_str_mon}}{\leq} \frac{5}{2}\lambda_t^2 \overset{\eqref{eq:lambda_SEG_str_mon}}{\leq} \frac{\exp(-\gamma\mu t)R^2}{4\gamma^2} \label{eq:omega_bound_x_t_SEG_str_mon}
    \end{eqnarray}
    for all $t = 0, 1, \ldots, T-1$, where we use that $\|a+b\|^2 \leq 2\|a\|^2 + 2\|b\|^2$ holding for all $a,b \in \R^d$. 
    
    Next, we need to prove that $E_{T-1}$ implies $\|\tx^t - x^*\| \leq 3R$ and show several useful inequalities related to $\theta_t$. Lipschitzness of $F$ probability event $E_{T-1}$ implies
    \begin{eqnarray}
        \|\tx^t - x^*\|^2 &=& \|x^t - x^* - \gamma \tF_{\xi_1}(x^t)\|^2 \leq  2\|x^t - x^*\|^2 + 2\gamma^2\|\tF_{\xi_1}(x^t)\|^2 \notag \\
        &\leq& 2R_t^2 + 4\gamma^2\|F(x^t)\|^2 + 4\gamma^2\|\omega_t\|^2 \notag \\
        &\overset{\eqref{eq:L_Lip}}{\leq}& 2(1 + 2\gamma^2 L^2)R_t^2 + 4\gamma^2\|\omega_t\|^2 \notag \\
        &\overset{\eqref{eq:gamma_SEG_str_mon},\eqref{eq:omega_bound_x_t_SEG_str_mon}}{\leq}& 7\exp(-\gamma\mu t) R^2 \leq 9R^2 \label{eq:tilde_x_distance_SEG_str_mon}
    \end{eqnarray}
    and
    \begin{eqnarray}
        \|F(\tx^t)\| &\leq& L\|\tx^t - x^*\| \leq \sqrt{7}L\exp(- \nicefrac{\gamma\mu t}{2})R \overset{\eqref{eq:gamma_SEG_str_mon},\eqref{eq:lambda_SEG_str_mon}}{\leq} \frac{\lambda_t}{2} \label{eq:operator_bound_tx_t_SEG_str_mon}
    \end{eqnarray}
    for all $t = 0, 1, \ldots, T-1$. Therefore, $E_{T-1}$ implies that $x^t, \tx^t \in B_{3R}(x^*)$ for all $t = 0, 1, \ldots, T-1$. Using Lemma~\ref{lem:optimization_lemma_str_mon_SEG} and $(1 - \gamma\mu)^T \leq \exp(-\gamma\mu T)$, we obtain that $E_{T-1}$ implies
    \begin{eqnarray}
        R_T^2 &\leq& \exp(-\gamma\mu T)R^2 - 4\gamma^3 \mu \sum\limits_{l=0}^{T-1} (1-\gamma\mu)^{T-1-l}\langle F(x^l), \omega_l \rangle\notag\\
        &&\quad + 2\gamma \sum\limits_{l=0}^{T-1} (1-\gamma\mu)^{T-1-l} \langle x^l - x^* - \gamma F(\tx^l), \theta_l \rangle\notag\\
        &&\quad + \gamma^2 \sum\limits_{l=0}^{T-1} (1-\gamma\mu)^{T-1-l} \left(\|\theta_l\|^2 + 4\|\omega_l\|^2\right). \notag
    \end{eqnarray}
    To handle the sums above, we introduce a new notation:
    \begin{gather}
        \zeta_t = \begin{cases} F(x^t),& \text{if } \|F(x^t)\| \leq \sqrt{2}L\exp(-\nicefrac{\gamma\mu t}{2})R,\\ 0,& \text{otherwise}, \end{cases} \label{eq:zeta_t_SEG_str_mon}\\
        \eta_t = \begin{cases} x^t - x^* - \gamma F(\tx^t),& \text{if } \|x^t - x^* - \gamma F(\tx^t)\| \leq \sqrt{7}(1 + \gamma L) \exp(- \nicefrac{\gamma\mu t}{2})R,\\ 0,& \text{otherwise}, \end{cases} \label{eq:eta_t_SEG_str_mon}
    \end{gather}
    for $t = 0, 1, \ldots, T-1$. These vectors are bounded almost surely:
    \begin{equation}
        \|\zeta_t\| \leq \sqrt{2}L\exp(-\nicefrac{\gamma\mu t}{2})R,\quad \|\eta_t\| \leq \sqrt{7}(1 + \gamma L)\exp(-\nicefrac{\gamma\mu t}{2})R \label{eq:zeta_t_eta_t_bound_SEG_str_mon} 
    \end{equation}
    for all $t = 0, 1, \ldots, T-1$. We also notice that $E_{T-1}$ implies $\|F(x^t)\| \leq \sqrt{2}L\exp(-\nicefrac{\gamma\mu t}{2})R$ (due to \eqref{eq:operator_bound_x_t_SEG_str_mon}) and
    \begin{eqnarray*}
        \|x^t - x^* - \gamma F(\tx^t)\| &\leq& \|x^t - x^*\| + \gamma \|F(\tx^t)\|\\
        &\overset{\eqref{eq:tilde_x_distance_SEG_str_mon},\eqref{eq:operator_bound_tx_t_SEG_str_mon}}{\leq}& \sqrt{7}(1 + \gamma L)\exp(-\nicefrac{\gamma\mu t}{2})R
    \end{eqnarray*}
    for $t = 0, 1, \ldots, T-1$. In other words, $E_{T-1}$ implies $\zeta_t = F(x^t)$ and $\eta_t = x^t - x^* - \gamma F(\tx^t)$ for all $t = 0,1,\ldots,T-1$, meaning that from $E_{T-1}$ it follows that
    \begin{eqnarray}
        R_T^2 &\leq& \exp(-\gamma\mu T)R^2 - 4\gamma^3 \mu \sum\limits_{l=0}^{T-1} (1-\gamma\mu)^{T-1-l}\langle \zeta_l, \omega_l \rangle\notag\\
        &&\quad + 2\gamma \sum\limits_{l=0}^{T-1} (1-\gamma\mu)^{T-1-l} \langle \eta_l, \theta_l \rangle + \gamma^2 \sum\limits_{l=0}^{T-1} (1-\gamma\mu)^{T-1-l} \left(\|\theta_l\|^2 + 4\|\omega_l\|^2\right). \notag
    \end{eqnarray}

    To handle the sums appeared on the right-hand side of the previous inequality we consider unbiased and biased parts of $\theta_l, \omega_l$:
    \begin{gather}
        \theta_l^u \eqdef \EE_{\xi_2^l}\left[\tF_{\xi_2^l}(\tx^l)\right] - \tF_{\xi_2^l}(\tx^l),\quad \theta_l^b \eqdef F(\tx^l) - \EE_{\xi_2^l}\left[\tF_{\xi_2^l}(\tx^l)\right], \label{eq:theta_unbias_bias_SEG_str_mon}\\
        \omega_l^u \eqdef \EE_{\xi_1^l}\left[\tF_{\xi_1^l}(x^l)\right] - \tF_{\xi_1^l}(x^l),\quad \omega_l^b \eqdef F(x^l) - \EE_{\xi_1^l}\left[\tF_{\xi_1^l}(x^l)\right], \label{eq:omega_unbias_bias_SEG_str_mon}
    \end{gather}
    for all $l = 0,\ldots, T-1$. By definition we have $\theta_l = \theta_l^u + \theta_l^b$, $\omega_l = \omega_l^u + \omega_l^b$ for all $l = 0,\ldots, T-1$. Therefore, $E_{T-1}$ implies
    \begin{eqnarray}
        R_T^2 &\leq& \exp(-\gamma\mu T) R^2 \underbrace{- 4\gamma^3 \mu \sum\limits_{l=0}^{T-1} (1-\gamma\mu)^{T-1-l}\langle \zeta_l, \omega_l^u \rangle}_{\circledOne} \underbrace{- 4\gamma^3 \mu \sum\limits_{l=0}^{T-1} (1-\gamma\mu)^{T-1-l}\langle \zeta_l, \omega_l^b \rangle}_{\circledTwo}\notag\\
        &&\quad + \underbrace{2\gamma \sum\limits_{l=0}^{T-1} (1-\gamma\mu)^{T-1-l} \langle \eta_l, \theta_l^u \rangle}_{\circledThree} + \underbrace{2\gamma \sum\limits_{l=0}^{T-1} (1-\gamma\mu)^{T-1-l} \langle \eta_l, \theta_l^b \rangle}_{\circledFour} \notag\\
        &&\quad + \underbrace{2\gamma^2 \sum\limits_{l=0}^{T-1} (1-\gamma\mu)^{T-1-l} \left(\EE_{\xi_2^l}\left[\|\theta_l^u\|^2\right] + 4\EE_{\xi_1^l}\left[\|\omega_l^u\|^2\right]\right)}_{\circledFive} \notag\\
        &&\quad + \underbrace{2\gamma^2 \sum\limits_{l=0}^{T-1} (1-\gamma\mu)^{T-1-l} \left(\|\theta_l^u\|^2 + 4\|\omega_l^u\|^2 -\EE_{\xi_2^l}\left[\|\theta_l^u\|^2\right] - 4\EE_{\xi_1^l}\left[\|\omega_l^u\|^2\right]\right)}_{\circledSix}\notag\\
        &&\quad + \underbrace{2\gamma^2 \sum\limits_{l=0}^{T-1} (1-\gamma\mu)^{T-1-l} \left(\|\theta_l^b\|^2 + 4\|\omega_l^b\|^2\right)}_{\circledSeven}. \label{eq:SEG_str_mon_1234567_bound}
    \end{eqnarray}
    where we also apply inequality $\|a+b\|^2 \leq 2\|a\|^2 + 2\|b\|^2$ holding for all $a,b \in \R^d$ to upper bound $\|\theta_l\|^2$ and $\|\omega_l\|^2$. It remains to derive good enough high-probability upper-bounds for the terms $\circledOne, \circledTwo, \circledThree, \circledFour, \circledFive, \circledSix, \circledSeven$, i.e., to finish our inductive proof we need to show that $\circledOne + \circledTwo + \circledThree + \circledFour + \circledFive + \circledSix + \circledSeven \leq \exp(-\gamma\mu T) R^2$ with high probability. In the subsequent parts of the proof, we will need to use many times the bounds for the norm and second moments of $\theta_{t+1}^u$ and $\theta_{t+1}^b$. First, by definition of clipping operator, we have with probability $1$ that
    \begin{equation}
        \|\theta_l^u\| \leq 2\lambda_l,\quad \|\omega_l^u\| \leq 2\lambda_l.\label{eq:theta_omega_magnitude_str_mon}
    \end{equation}
    Moreover, since $E_{T-1}$ implies that $\|F(x^l)\| \leq \nicefrac{\lambda_l}{2}$ and $\|F(\tx^l)\| \leq \nicefrac{\lambda_l}{2}$ for all $l = 0,1, \ldots, T-1$ (see \eqref{eq:operator_bound_x_t_SEG_str_mon} and \eqref{eq:operator_bound_tx_t_SEG_str_mon}), from Lemma~\ref{lem:bias_and_variance_clip} we also have that $E_{T-1}$ implies
    \begin{gather}
        \left\|\theta_l^b\right\| \leq \frac{2^{\alpha}\sigma^\alpha}{\lambda_l^{\alpha-1}},\quad \left\|\omega_l^b\right\| \leq \frac{2^{\alpha}\sigma^\alpha}{\lambda_l^{\alpha-1}}, \label{eq:bias_theta_omega_str_mon}\\
        \EE_{\xi_2^l}\left[\left\|\theta_l\right\|^2\right] \leq 18 \lambda_l^{2-\alpha}\sigma^\alpha,\quad \EE_{\xi_1^l}\left[\left\|\omega_l\right\|^2\right] \leq 18 \lambda_l^{2-\alpha}\sigma^\alpha, \label{eq:distortion_theta_omega_str_mon}\\
        \EE_{\xi_2^l}\left[\left\|\theta_l^u\right\|^2\right] \leq 18 \lambda_l^{2-\alpha}\sigma^\alpha,\quad \EE_{\xi_1^l}\left[\left\|\omega_l^u\right\|^2\right] \leq 18 \lambda_l^{2-\alpha}\sigma^\alpha, \label{eq:variance_theta_omega_str_mon}
    \end{gather}
    for all $l = 0,1, \ldots, T-1$.

    \paragraph{Upper bound for $\circledOne$.} By definition of $\omega_{l}^u$, we have $\EE_{\xi_1^l}[\omega_{l}^u] = 0$ and
    \begin{equation*}
        \EE_{\xi_1^l}\left[-4\gamma^3\mu (1-\gamma\mu)^{T-1-l} \langle \zeta_l, \omega_l^u \rangle\right] = 0.
    \end{equation*}
    Next, sum $\circledOne$ has bounded with probability $1$ terms:
    \begin{eqnarray}
        |-4\gamma^3\mu (1-\gamma\mu)^{T-1-l} \langle \zeta_l, \omega_l^u \rangle | &\leq& 4\gamma^3\mu \exp(-\gamma\mu (T - 1 - l)) \|\zeta_l\|\cdot \|\omega_l^u\|\notag\\
        &\overset{\eqref{eq:zeta_t_eta_t_bound_SEG_str_mon},\eqref{eq:theta_omega_magnitude_str_mon}}{\leq}& 8\sqrt{2}\gamma^3\mu L \exp(-\gamma\mu (T - 1 - \nicefrac{l}{2})) R \lambda_l\notag\\
        &\overset{\eqref{eq:gamma_SEG_str_mon},\eqref{eq:lambda_SEG_str_mon}}{\leq}& \frac{\exp(-\gamma\mu T)R^2}{7\ln\tfrac{6(K+1)}{\beta}}  \eqdef c. \label{eq:SEG_str_mon_technical_1_1}
    \end{eqnarray}
    The summands also have bounded conditional variances $\sigma_l^2 \eqdef \EE_{\xi_1^l}\left[16\gamma^6\mu^2 (1-\gamma\mu)^{2T-2-2l} \langle \zeta_l, \omega_l^u \rangle^2\right]$:
    \begin{eqnarray}
        \sigma_l^2 &\leq& \EE_{\xi_1^l}\left[16\gamma^6\mu^2 \exp(-\gamma\mu (2T - 2 - 2l)) \|\zeta_l\|^2\cdot \|\omega_l^u\|^2\right]\notag\\
        &\overset{\eqref{eq:zeta_t_eta_t_bound_SEG_str_mon}}{\leq}& 36\gamma^6 \mu^2L^2 \exp(-\gamma\mu (2T - 2 - l)) R^2 \EE_{\xi_1^l}\left[\|\omega_l^u\|^2\right]\notag\\
        &\overset{\eqref{eq:gamma_SEG_str_mon}}{\leq}& \frac{4\gamma^2\exp(-\gamma\mu(2T - l))R^2}{2809\ln\tfrac{6(K+1)}{\beta}} \EE_{\xi_1^l}\left[\|\omega_l^u\|^2\right]. \label{eq:SEG_str_mon_technical_1_2}
    \end{eqnarray}
    In other words, we showed that $\{-4\gamma^3\mu (1-\gamma\mu)^{T-1-l} \langle \zeta_l, \omega_l^u \rangle\}_{l = 0}^{T-1}$ is a bounded martingale difference sequence with bounded conditional variances $\{\sigma_l^2\}_{l = 0}^{T-1}$. Next, we apply Bernstein's inequality (Lemma~\ref{lem:Bernstein_ineq}) with $X_l = -4\gamma^3\mu (1-\gamma\mu)^{T-1-l} \langle \zeta_l, \omega_l^u \rangle$, parameter $c$ as in \eqref{eq:SEG_str_mon_technical_1_1}, $b = \tfrac{1}{7}\exp(-\gamma\mu T) R^2$, $G = \tfrac{\exp(-2 \gamma\mu T) R^4}{294\ln\frac{6(K+1)}{\beta}}$:
    \begin{equation*}
        \PP\left\{|\circledOne| > \frac{1}{7}\exp(-\gamma\mu T) R^2 \text{ and } \sum\limits_{l=0}^{T-1}\sigma_l^2 \leq \frac{\exp(-2\gamma\mu T) R^4}{294\ln\tfrac{6(K+1)}{\beta}}\right\} \leq 2\exp\left(- \frac{b^2}{2G + \nicefrac{2cb}{3}}\right) = \frac{\beta}{3(K+1)}.
    \end{equation*}
    Equivalently, we have
    \begin{equation}
        \PP\{E_{\circledOne}\} \geq 1 - \frac{\beta}{3(K+1)},\quad \text{for}\quad E_{\circledOne} = \left\{\text{either} \quad \sum\limits_{l=0}^{T-1}\sigma_l^2 > \frac{\exp(-2\gamma\mu T) R^4}{294\ln\tfrac{6(K+1)}{\beta}}\quad \text{or}\quad |\circledOne| \leq \frac{1}{7}\exp(-\gamma\mu T) R^2\right\}. \label{eq:bound_1_SEG_str_mon}
    \end{equation}
    In addition, $E_{T-1}$ implies that
    \begin{eqnarray}
        \sum\limits_{l=0}^{T-1}\sigma_l^2 &\overset{\eqref{eq:SEG_str_mon_technical_1_2}}{\leq}& \frac{4\gamma^2\exp(-2\gamma\mu T)R^2}{2809 \ln\tfrac{6(K+1)}{\beta}} \sum\limits_{l=0}^{T-1} \frac{\EE_{\xi_1^l}\left[\|\omega_l^u\|^2\right]}{\exp(-\gamma\mu l)}\notag\\ &\overset{\eqref{eq:variance_theta_omega_str_mon}, T \leq K+1}{\leq}& \frac{72\gamma^2\exp(-2\gamma\mu T) R^2 \sigma^\alpha}{2809 \ln\tfrac{6(K+1)}{\beta}} \sum\limits_{l=0}^{K} \frac{\lambda_{l}^{2-\alpha}}{\exp(-\gamma\mu l)}\notag\\
        &\overset{\eqref{eq:lambda_SEG_str_mon}}{\leq}& \frac{72\gamma^\alpha\exp(-2\gamma\mu T) R^{4-\alpha} \sigma^\alpha}{2809\cdot 120^{2-\alpha} \ln^{3-\alpha}\tfrac{6(K+1)}{\beta}} \sum\limits_{l=0}^{K} \frac{1}{\exp(-\gamma\mu l)} \cdot \left(\exp(-\gamma\mu(1 + \nicefrac{l}{2}))\right)^{2-\alpha} \notag\\
        &\leq& \frac{72\gamma^\alpha\exp(-2\gamma\mu T) R^{4-\alpha} \sigma^\alpha}{2809\cdot 120^{2-\alpha} \ln^{3-\alpha}\tfrac{6(K+1)}{\beta}} \sum\limits_{l=0}^{K} \exp(\gamma\mu(\alpha-2)) \cdot \exp\left(\frac{\gamma\mu\alpha l}{2}\right) \notag\\
        &\leq& \frac{72\gamma^\alpha\exp(-2\gamma\mu T) R^{4-\alpha} \sigma^\alpha (K+1) \exp\left(\frac{\gamma\mu\alpha K}{2}\right)}{2809\cdot 120^{2-\alpha} \ln^{3-\alpha}\tfrac{6(K+1)}{\beta}}  \notag\\
        &\overset{\eqref{eq:gamma_SEG_str_mon}}{\leq}& \frac{\exp(-2\gamma\mu T)R^4}{294\ln\tfrac{6(K+1)}{\beta}}, \label{eq:bound_1_variances_SEG_str_mon}
    \end{eqnarray}
    where we also show that $E_{T-1}$ implies
    \begin{equation}
        \gamma^2 R^2\sum\limits_{l=0}^{K} \frac{\lambda_{l}^{2-\alpha}}{\exp(-\gamma\mu l)} \leq \frac{\gamma^\alpha R^{4-\alpha}(K+1)\exp(\frac{\gamma\mu\alpha K}{2})}{120^{2-\alpha}\ln^{2-\alpha} \tfrac{6(K+1)}{\beta}}. \label{eq:useful_inequality_on_lambda_SEG_str_mon}
    \end{equation}

    \textbf{Upper bound for $\circledTwo$.} From $E_{T-1}$ it follows that
    \begin{eqnarray}
        \circledTwo &\leq& 4\gamma^3 \mu \sum\limits_{l=0}^{T-1} \exp(-\gamma\mu (T-1-l)) \|\zeta_l\| \cdot \|\omega_l^b\| \notag\\
        &\overset{\eqref{eq:zeta_t_eta_t_bound_SEG_str_mon},\eqref{eq:bias_theta_omega_str_mon}}{\leq}& 2^{2+\alpha}\cdot\sqrt{2} \exp(-\gamma\mu (T-1)) \gamma^3 \mu L R \sum\limits_{l=0}^{T-1} \frac{\sigma^\alpha}{\lambda_l^{\alpha-1} \exp(-\nicefrac{\gamma\mu l}{2})} \notag\\
        &\overset{\eqref{eq:lambda_SEG_str_mon}}{=}& \frac{2^{2+\alpha}\cdot 120^{\alpha-1}\sqrt{2} \exp(-\gamma\mu (T-1)) \gamma^{2+\alpha} \mu L \sigma^\alpha \ln^{\alpha-1}\tfrac{6(K+1)}{\beta}}{R^{\alpha-2}} \sum\limits_{l=0}^{T-1} \frac{1}{\exp\left(-\gamma\mu(1 + \nicefrac{l}{2})\right)^{\alpha-1}\cdot\exp(-\nicefrac{\gamma\mu l}{2})} \notag\\
        &\overset{T \leq K+1}{\leq}& \frac{2^{3+\alpha}\cdot 120^{\alpha-1}\sqrt{2} \exp(-\gamma\mu (T-1)) \gamma^{2+\alpha} \mu L \sigma^\alpha \ln^{\alpha-1}\tfrac{6(K+1)}{\beta}}{R^{\alpha-2}} \sum\limits_{l=0}^{K} \exp\left(\frac{\gamma\mu \alpha l}{2}\right) \notag\\
        &\leq& \frac{2^{3+\alpha}\cdot 120^{\alpha-1}\sqrt{2} \exp(-\gamma\mu (T-1)) \gamma^{2+\alpha} \mu L \sigma^\alpha \ln^{\alpha-1}\tfrac{6(K+1)}{\beta} (K+1)\exp\left(\frac{\gamma\mu \alpha K}{2}\right)}{R^{\alpha-2}} \notag\\
        &\overset{\eqref{eq:gamma_SEG_str_mon}}{\leq}& \frac{1}{7}\exp(-\gamma\mu T) R^2, \label{eq:bound_2_SEG_str_mon}
    \end{eqnarray}
    where we also show that $E_{T-1}$ implies
    \begin{equation}
        \gamma R\sum\limits_{l=0}^{T-1} \frac{1}{\lambda_l^{\alpha-1}\exp(-\nicefrac{\gamma\mu l}{2})} \leq \frac{120^{\alpha-1}\gamma^\alpha (K+1)\exp(\frac{\gamma\mu \alpha K}{2})\ln^{\alpha-1} \tfrac{6(K+1)}{\beta}}{R^{\alpha-2}}. \label{eq:useful_inequality_on_lambda_SEG_str_mon_1}
    \end{equation}

    \paragraph{Upper bound for $\circledThree$.} By definition of $\theta_l^u$, we have $\EE_{\xi_2^l}[\theta_{l}^u] = 0$ and
    \begin{equation*}
        \EE_{\xi_2^l}\left[2\gamma (1-\gamma\mu)^{T-1-l} \langle \eta_l, \theta_l^u \rangle\right] = 0.
    \end{equation*}
    Next, sum $\circledThree$ has bounded with probability $1$ terms:
    \begin{eqnarray}
        |2\gamma (1-\gamma\mu)^{T-1-l} \langle \eta_l, \theta_l^u \rangle | &\leq& 2\gamma\exp(-\gamma\mu (T - 1 - l)) \|\eta_l\|\cdot \|\theta_l^u\|\notag\\
        &\overset{\eqref{eq:zeta_t_eta_t_bound_SEG_str_mon},\eqref{eq:theta_omega_magnitude_str_mon}}{\leq}& 4\sqrt{7}\gamma (1 + \gamma L) \exp(-\gamma\mu (T - 1 - \nicefrac{l}{2})) R \lambda_l\notag\\
        &\overset{\eqref{eq:gamma_SEG_str_mon},\eqref{eq:lambda_SEG_str_mon}}{\leq}& \frac{\exp(-\gamma\mu T)R^2}{7\ln\tfrac{6(K+1)}{\beta}} \eqdef c. \label{eq:SEG_str_mon_technical_3_1}
    \end{eqnarray}
    The summands also have bounded conditional variances $\widetilde\sigma_l^2 \eqdef \EE_{\xi_2^l}\left[4\gamma^2 (1-\gamma\mu)^{2T-2-2l} \langle \eta_l, \theta_l^u \rangle^2\right]$:
    \begin{eqnarray}
        \widetilde\sigma_l^2 &\leq& \EE_{\xi_2^l}\left[4\gamma^2\exp(-\gamma\mu (2T - 2 - 2l)) \|\eta_l\|^2\cdot \|\theta_l^u\|^2\right]\notag\\
        &\overset{\eqref{eq:zeta_t_eta_t_bound_SEG_str_mon}}{\leq}& 49\gamma^2 (1 + \gamma L)^2 \exp(-\gamma\mu (2T - 2 - l)) R^2 \EE_{\xi_2^l}\left[\|\theta_l^u\|^2\right]\notag\\
        &\overset{\eqref{eq:gamma_SEG_str_mon}}{\leq}& 50\gamma^2\exp(-\gamma\mu (2T - l))R^2 \EE_{\xi_2^l}\left[\|\theta_l^u\|^2\right]. \label{eq:SEG_str_mon_technical_3_2}
    \end{eqnarray}
    In other words, we showed that $\{2\gamma (1-\gamma\mu)^{T-1-l} \langle \eta_l, \theta_l^u \rangle\}_{l = 0}^{T-1}$ is a bounded martingale difference sequence with bounded conditional variances $\{\widetilde\sigma_l^2\}_{l = 0}^{T-1}$. Next, we apply Bernstein's inequality (Lemma~\ref{lem:Bernstein_ineq}) with $X_l = 2\gamma (1-\gamma\mu)^{T-1-l} \langle \eta_l, \theta_l^u \rangle$, parameter $c$ as in \eqref{eq:SEG_str_mon_technical_3_1}, $b = \tfrac{1}{7}\exp(-\gamma\mu T) R^2$, $G = \tfrac{\exp(-2 \gamma\mu T) R^4}{294\ln\frac{6(K+1)}{\beta}}$:
    \begin{equation*}
        \PP\left\{|\circledThree| > \frac{1}{7}\exp(-\gamma\mu T) R^2 \text{ and } \sum\limits_{l=0}^{T-1}\widetilde\sigma_l^2 \leq \frac{\exp(- 2\gamma\mu T) R^4}{294\ln\tfrac{6(K+1)}{\beta}}\right\} \leq 2\exp\left(- \frac{b^2}{2G + \nicefrac{2cb}{3}}\right) = \frac{\beta}{3(K+1)}.
    \end{equation*}
    Equivalently, we have
    \begin{equation}
        \PP\{E_{\circledThree}\} \geq 1 - \frac{\beta}{3(K+1)},\quad \text{for}\quad E_{\circledThree} = \left\{\text{either} \quad \sum\limits_{l=0}^{T-1}\widetilde\sigma_l^2 > \frac{\exp(- 2\gamma\mu T) R^4}{294\ln\tfrac{6(K+1)}{\beta}}\quad \text{or}\quad |\circledThree| \leq \frac{1}{7}\exp(-\gamma\mu T) R^2\right\}. \label{eq:bound_3_SEG_str_mon}
    \end{equation}
    In addition, $E_{T-1}$ implies that
    \begin{eqnarray}
        \sum\limits_{l=0}^{T-1}\widetilde\sigma_l^2 &\overset{\eqref{eq:SEG_str_mon_technical_3_2}}{\leq}& 50\gamma^2\exp(- 2\gamma\mu T)R^2\sum\limits_{l=0}^{T-1} \frac{\EE_{\xi_2^l}\left[\|\theta_l^u\|^2\right]}{\exp(-\gamma\mu l)}\notag\\ 
        &\overset{\eqref{eq:variance_theta_omega_str_mon}, T \leq K+1}{\leq}& 900\gamma^2\exp(-2\gamma\mu T) R^2 \sigma^\alpha \sum\limits_{l=0}^{K} \frac{\lambda_l^{2-\alpha}}{\exp(-\gamma\mu l)}\notag\\
        &\overset{\eqref{eq:useful_inequality_on_lambda_SEG_str_mon}}{\leq}& \frac{900\gamma^\alpha\exp(-2\gamma\mu T) R^{4-\alpha} \sigma^\alpha (K+1)\exp(\frac{\gamma\mu\alpha K}{2})}{120^{2-\alpha} \ln^{2-\alpha}\tfrac{6(K+1)}{\beta}}\notag\\
        &\overset{\eqref{eq:gamma_SEG_str_mon}}{\leq}& \frac{\exp(-2\gamma\mu T)R^4}{294\ln\tfrac{6(K+1)}{\beta}}. \label{eq:bound_3_variances_SEG_str_mon}
    \end{eqnarray}

    \paragraph{Upper bound for $\circledFour$.} From $E_{T-1}$ it follows that
    \begin{eqnarray}
        \circledFour &\leq& 2\gamma \exp(-\gamma\mu (T-1)) \sum\limits_{l=0}^{T-1} \frac{\|\eta_l\|\cdot \|\theta_l^b\|}{\exp(-\gamma\mu l)}\notag\\
        &\overset{\eqref{eq:zeta_t_eta_t_bound_SEG_str_mon}, \eqref{eq:bias_theta_omega_str_mon}}{\leq}& 2^{1+\alpha}\sqrt{7} \gamma (1+\gamma L) \exp(-\gamma\mu (T-1)) R \sigma^\alpha \sum\limits_{l=0}^{T-1} \frac{1}{\lambda_l^{\alpha-1} \exp(-\nicefrac{\gamma\mu l}{2})}\notag\\
        &\overset{\eqref{eq:useful_inequality_on_lambda_SEG_str_mon_1}}{\leq}& \frac{2^{3+\alpha}\cdot 120^{\alpha-1}\sqrt{7} \gamma^\alpha(1+\gamma L) \exp(-\gamma\mu T) (K+1) \exp\left(\frac{\gamma\mu\alpha K}{2}\right) \ln^{\alpha-1}\tfrac{6(K+1)}{\beta}}{R^{\alpha-2}} \notag \\
        &\overset{\eqref{eq:gamma_SEG_str_mon}}{\leq}& \frac{1}{7}\exp(-\gamma\mu T) R^2. \label{eq:bound_4_SEG_str_mon}
    \end{eqnarray}

    \paragraph{Upper bound for $\circledFive$.} From $E_{T-1}$ it follows that
    \begin{eqnarray}
        \circledFive &=& 2\gamma^2 \exp(-\gamma\mu (T-1)) \sum\limits_{l=0}^{T-1} \frac{\EE_{\xi_2^l}\left[\|\theta_l^u\|^2\right] + 4\EE_{\xi_1^l}\left[\|\omega_l^u\|^2\right]}{\exp(-\gamma\mu l)} \notag\\
        &\overset{\eqref{eq:variance_theta_omega_str_mon}}{\leq}& 180\gamma^2\exp(-\gamma\mu (T-1)) \sigma^\alpha\sum\limits_{l=0}^{T-1} \frac{\lambda_l^{2-\alpha}}{\exp(-\gamma\mu l)} \notag\\
        &\overset{\eqref{eq:useful_inequality_on_lambda_SEG_str_mon}}{\leq}&  \frac{180\gamma^\alpha R^{2-\alpha}\exp(-\gamma\mu (T-1)) \sigma^\alpha (K+1)\exp(\frac{\gamma\mu\alpha K}{2})}{120^{2-\alpha} \ln^{2-\alpha}\tfrac{6(K+1)}{\beta}} \notag\\
        &\overset{\eqref{eq:gamma_SEG_str_mon}}{\leq}& \frac{1}{7} \exp(-\gamma\mu T) R^2. \label{eq:bound_5_SEG_str_mon}
    \end{eqnarray}

    \paragraph{Upper bound for $\circledSix$.} First, we have
    \begin{equation*}
        2\gamma^2 (1-\gamma\mu)^{T-1-l}\EE_{\xi_1^l, \xi_2^l}\left[\|\theta_l^u\|^2 + 4\|\omega_l^u\|^2 -\EE_{\xi_2^l}\left[\|\theta_l^u\|^2\right] - 4\EE_{\xi_1^l}\left[\|\omega_l^u\|^2\right]\right] = 0.
    \end{equation*}
    Next, sum $\circledSix$ has bounded with probability $1$ terms:
    \begin{eqnarray}
        2\gamma^2 (1-\gamma\mu)^{T-1-l}\left| \|\theta_l^u\|^2 + 4\|\omega_l^u\|^2 -\EE_{\xi_2^l}\left[\|\theta_l^u\|^2\right] - 4\EE_{\xi_1^l}\left[\|\omega_l^u\|^2\right] \right| 
        &\overset{\eqref{eq:theta_omega_magnitude_str_mon}}{\leq}& \frac{80\gamma^2 \exp(-\gamma\mu T) \lambda_l^2}{\exp(-\gamma\mu (1+l))}\notag\\
        &\overset{\eqref{eq:lambda_SEG_str_mon}}{\leq}& \frac{\exp(-\gamma\mu T)R^2}{7\ln\tfrac{6(K+1)}{\beta}}\notag\\
        &\eqdef& c. \label{eq:SEG_str_mon_technical_6_1}
    \end{eqnarray}
    The summands also have conditional variances
    \begin{equation*}
        \widehat\sigma_l^2 \eqdef \EE_{\xi_1^l,\xi_2^l}\left[4\gamma^4 (1-\gamma\mu)^{2T-2-2l} \left| \|\theta_l^u\|^2 + 4\|\omega_l^u\|^2 -\EE_{\xi_2^l}\left[\|\theta_l^u\|^2\right] - 4\EE_{\xi_1^l}\left[\|\omega_l^u\|^2\right] \right|^2\right]
    \end{equation*}
    that are bounded
    \begin{eqnarray}
        \widehat\sigma_l^2 &\overset{\eqref{eq:SEG_str_mon_technical_6_1}}{\leq}& \frac{2\gamma^2\exp(-2\gamma\mu T)R^2}{7\exp(-\gamma\mu (1+l))\ln\tfrac{6(K+1)}{\beta}} \EE_{\xi_1^l,\xi_2^l}\left[\left| \|\theta_l^u\|^2 + 4\|\omega_l^u\|^2 -\EE_{\xi_2^l}\left[\|\theta_l^u\|^2\right] - 4\EE_{\xi_1^l}\left[\|\omega_l^u\|^2\right] \right|\right]\notag\\
        &\leq& \frac{4\gamma^2\exp(-2\gamma\mu T)R^2}{7\exp(-\gamma\mu (1+l))\ln\tfrac{6(K+1)}{\beta}} \EE_{\xi_1^l,\xi_2^l}\left[\|\theta_l^u\|^2 + 4\|\omega_l^u\|^2\right]. \label{eq:SEG_str_mon_technical_6_2}
    \end{eqnarray}
    In other words, we showed that $\left\{2\gamma^2 (1-\gamma\mu)^{T-1-l}\left( \|\theta_l^u\|^2 + 4\|\omega_l^u\|^2 -\EE_{\xi_2^l}\left[\|\theta_l^u\|^2\right] - 4\EE_{\xi_1^l}\left[\|\omega_l^u\|^2\right]\right)\right\}_{l = 0}^{T-1}$ is a bounded martingale difference sequence with bounded conditional variances $\{\widehat\sigma_l^2\}_{l = 0}^{T-1}$. Next, we apply Bernstein's inequality (Lemma~\ref{lem:Bernstein_ineq}) with $X_l = 2\gamma^2 (1-\gamma\mu)^{T-1-l}\left( \|\theta_l^u\|^2 + 4\|\omega_l^u\|^2 -\EE_{\xi_2^l}\left[\|\theta_l^u\|^2\right] - 4\EE_{\xi_1^l}\left[\|\omega_l^u\|^2\right]\right)$, parameter $c$ as in \eqref{eq:SEG_str_mon_technical_6_1}, $b = \tfrac{1}{7}\exp(-\gamma\mu T) R^2$, $G = \tfrac{\exp(-2 \gamma\mu T) R^4}{294\ln\frac{6(K+1)}{\beta}}$:
    \begin{equation*}
        \PP\left\{|\circledSix| > \frac{1}{7}\exp(-\gamma\mu T) R^2 \text{ and } \sum\limits_{l=0}^{T-1}\widehat\sigma_l^2 \leq \frac{\exp(-2\gamma\mu T) R^4}{294\ln\frac{6(K+1)}{\beta}}\right\} \leq 2\exp\left(- \frac{b^2}{2G + \nicefrac{2cb}{3}}\right) = \frac{\beta}{3(K+1)}.
    \end{equation*}
    Equivalently, we have
    \begin{equation}
        \PP\{E_{\circledSix}\} \geq 1 - \frac{\beta}{3(K+1)},\quad \text{for}\quad E_{\circledSix} = \left\{\text{either} \quad \sum\limits_{l=0}^{T-1}\widehat\sigma_l^2 > \frac{\exp(-2\gamma\mu T) R^4}{294\ln\tfrac{6(K+1)}{\beta}}\quad \text{or}\quad |\circledSix| \leq \frac{1}{7}\exp(-\gamma\mu T) R^2\right\}. \label{eq:bound_6_SEG_str_mon}
    \end{equation}
    In addition, $E_{T-1}$ implies that
    \begin{eqnarray}
        \sum\limits_{l=0}^{T-1}\widehat\sigma_l^2 &\overset{\eqref{eq:SEG_str_mon_technical_6_2}}{\leq}& \frac{4\gamma^2\exp(-\gamma\mu (2T-1))R^2}{7\ln\tfrac{6(K+1)}{\beta}} \sum\limits_{l=0}^{T-1} \frac{\EE_{\xi_1^l,\xi_2^l}\left[\|\theta_l^u\|^2 + 4\|\omega_l^u\|^2\right]}{\exp(-\gamma\mu l)}\notag\\ &\overset{\eqref{eq:variance_theta_omega_str_mon}, T \leq K+1}{\leq}& \frac{360\gamma^2\exp(-\gamma\mu (2T-1)) R^2 \sigma^\alpha}{7\ln\tfrac{6(K+1)}{\beta}} \sum\limits_{l=0}^{K} \frac{\lambda_l^{2-\alpha}}{\exp(-\gamma\mu l)}\notag\\
        &\overset{\eqref{eq:useful_inequality_on_lambda_SEG_str_mon}}{\leq}& \frac{360\gamma^\alpha\exp(-\gamma\mu (2T-1)) R^{4-\alpha} \sigma^\alpha (K+1)\exp(\frac{\gamma\mu\alpha K}{2})}{7\cdot 120^{2-\alpha}\ln^{3-\alpha}\tfrac{6(K+1)}{\beta}} \notag\\
        &\overset{\eqref{eq:gamma_SEG_str_mon}}{\leq}& \frac{\exp(-2\gamma\mu T)R^4}{294\ln\tfrac{6(K+1)}{\beta}}. \label{eq:bound_6_variances_SEG_str_mon}
    \end{eqnarray}

    \paragraph{Upper bound for $\circledSeven$.} From $E_{T-1}$ it follows that
    \begin{eqnarray}
        \circledSeven &=&  2\gamma^2 \sum\limits_{l=0}^{T-1} \exp(-\gamma\mu (T-1-l)) \left(\|\theta_l^b\|^2 + 4\|\omega_l^b\|^2\right)\notag\\
        &\overset{\eqref{eq:bias_theta_omega_str_mon}}{\leq}& 10\cdot 2^{2\alpha}\gamma^2 \exp(-\gamma\mu (T-1)) \sigma^{2\alpha} \sum\limits_{l=0}^{T-1} \frac{1}{\lambda_l^{2\alpha-2} \exp(-\gamma\mu l)} \notag\\
        &\overset{\eqref{eq:lambda_SEG_str_mon}, T \leq K+1}{\leq}& \frac{20\cdot 2^{2\alpha}\cdot 120^{2\alpha-2}\gamma^{2\alpha} \exp(-\gamma\mu T) \sigma^{2\alpha} \ln^{2\alpha-2}\tfrac{6(K+1)}{\beta}}{R^{2\alpha-2}} \sum\limits_{l=0}^{K} \exp\left(\gamma\mu(2\alpha-2)\left(1 + \frac{l}{2}\right)\right)\exp(\gamma\mu l)\notag\\
        &\leq& \frac{40\cdot 2^{2\alpha}\cdot 120^{2\alpha-2}\gamma^{2\alpha} \exp(-\gamma\mu T) \sigma^{2\alpha} \ln^{2\alpha-2}\tfrac{6(K+1)}{\beta}}{R^{2\alpha-2}} \sum\limits_{l=0}^{K} \exp(\gamma\mu \alpha l)\notag\\
        &\leq& \frac{40\cdot 2^{2\alpha}\cdot 120^{2\alpha-2}\gamma^{2\alpha} \exp(-\gamma\mu T) \sigma^{2\alpha} \ln^{2\alpha-2}\tfrac{6(K+1)}{\beta} (K+1) \exp(\gamma\mu \alpha K)}{R^{2\alpha-2}}\notag\\
        &\overset{\eqref{eq:gamma_SEG_str_mon}}{\leq}& \frac{1}{7}\exp(-\gamma\mu T) R^2. \label{eq:bound_7_SEG_str_mon}
    \end{eqnarray}

    Now, we have the upper bounds for  $\circledOne, \circledTwo, \circledThree, \circledFour, \circledFive, \circledSix, \circledSeven$. In particular, probability event $E_{T-1}$ implies
    \begin{gather*}
        R_T^2 \overset{\eqref{eq:SEG_str_mon_1234567_bound}}{\leq} \exp(-\gamma\mu T) R^2 + \circledOne + \circledTwo + \circledThree + \circledFour + \circledFive + \circledSix + \circledSeven,\\
        \circledTwo \overset{\eqref{eq:bound_2_SEG_str_mon}}{\leq} \frac{1}{7}\exp(-\gamma\mu T)R^2,\quad \circledFour \overset{\eqref{eq:bound_4_SEG_str_mon}}{\leq} \frac{1}{7}\exp(-\gamma\mu T)R^2,\\ \circledFive \overset{\eqref{eq:bound_5_SEG_str_mon}}{\leq} \frac{1}{7}\exp(-\gamma\mu T)R^2,\quad \circledSeven \overset{\eqref{eq:bound_7_SEG_str_mon}}{\leq} \frac{1}{7}\exp(-\gamma\mu T)R^2,\\
        \sum\limits_{l=0}^{T-1}\sigma_l^2 \overset{\eqref{eq:bound_1_variances_SEG_str_mon}}{\leq}  \frac{\exp(-2\gamma\mu T)R^4}{294\ln\tfrac{6(K+1)}{\beta}},\quad \sum\limits_{l=0}^{T-1}\widetilde\sigma_l^2 \overset{\eqref{eq:bound_3_variances_SEG_str_mon}}{\leq} \frac{\exp(-2\gamma\mu T)R^4}{294\ln\tfrac{6(K+1)}{\beta}},\quad \sum\limits_{l=0}^{T-1}\widehat\sigma_l^2 \overset{\eqref{eq:bound_6_variances_SEG_str_mon}}{\leq}  \frac{\exp(-2\gamma\mu T)R^4}{294\ln\tfrac{6(K+1)}{\beta}}.
    \end{gather*}
     Moreover, we also have (see \eqref{eq:bound_1_SEG_str_mon}, \eqref{eq:bound_3_SEG_str_mon}, \eqref{eq:bound_6_SEG_str_mon} and our induction assumption)
     \begin{gather*}
        \PP\{E_{T-1}\} \geq 1 - \frac{(T-1)\beta}{K+1},\\
        \PP\{E_{\circledOne}\} \geq 1 - \frac{\beta}{3(K+1)}, \quad \PP\{E_{\circledThree}\} \geq 1 - \frac{\beta}{3(K+1)}, \quad \PP\{E_{\circledSix}\} \geq 1 - \frac{\beta}{3(K+1)} ,
    \end{gather*}
    where
    \begin{eqnarray}
        E_{\circledOne}&=& \left\{\text{either} \quad \sum\limits_{l=0}^{T-1}\sigma_l^2 > \frac{\exp(-2\gamma\mu T) R^4}{294\ln\tfrac{6(K+1)}{\beta}}\quad \text{or}\quad |\circledOne| \leq \frac{1}{7}\exp(-\gamma\mu T) R^2\right\},\notag\\
        E_{\circledThree}&=& \left\{\text{either} \quad \sum\limits_{l=0}^{T-1}\widetilde\sigma_l^2 > \frac{\exp(-2\gamma\mu T) R^4}{294\ln\tfrac{6(K+1)}{\beta}}\quad \text{or}\quad |\circledThree| \leq \frac{1}{7}\exp(-\gamma\mu T) R^2\right\},\notag\\
        E_{\circledSix}&=& \left\{\text{either} \quad \sum\limits_{l=0}^{T-1}\widehat\sigma_l^2 > \frac{\exp(-2\gamma\mu T) R^4}{294\ln\tfrac{6(K+1)}{\beta}}\quad \text{or}\quad |\circledSix| \leq \frac{1}{7}\exp(-\gamma\mu T) R^2\right\}.\notag
    \end{eqnarray}
    Thus, probability event $E_{T-1} \cap E_{\circledOne} \cap E_{\circledThree} \cap E_{\circledSix}$ implies
    \begin{eqnarray*}
        R_T^2 &\overset{\eqref{eq:SEG_str_mon_1234567_bound}}{\leq}& \exp(-\gamma\mu T) R^2 + \circledOne + \circledTwo + \circledThree + \circledFour + \circledFive + \circledSix + \circledSeven\\
        &\leq& 2\exp(-\gamma\mu T) R^2,
    \end{eqnarray*}
    which is equivalent to \eqref{eq:induction_inequality_str_mon_SEG} for $t = T$, and
    \begin{equation}
        \PP\{E_T\} \geq \PP\{E_{T-1} \cap E_{\circledOne} \cap E_{\circledThree} \cap E_{\circledSix}\} = 1 - \PP\{\overline{E}_{T-1} \cup \overline{E}_{\circledOne} \cup \overline{E}_{\circledThree} \cup \overline{E}_{\circledSix}\} \geq 1 - \frac{T\beta}{K+1}. \notag
    \end{equation}
    This finishes the inductive part of our proof, i.e., for all $k = 0,1,\ldots,K+1$ we have $\PP\{E_k\} \geq 1 - \nicefrac{k\beta}{(K+1)}$. In particular, for $k = K+1$ we have that with probability at least $1 - \beta$
    \begin{equation}
        \|x^{K+1} - x^*\|^2 \leq 2\exp(-\gamma\mu (K+1))R^2. \notag
    \end{equation}
    Finally, if 
    \begin{eqnarray*}
        \gamma &=& \min\left\{\frac{1}{650 L \ln \tfrac{6(K+1)}{\beta}}, \frac{\ln(B_K)}{\mu(K+1)}\right\}, \notag\\
        B_K &=& \max\left\{2, \frac{(K+1)^{\frac{2(\alpha-1)}{\alpha}}\mu^2R^2}{264600^{\frac{2}{\alpha}}\sigma^2\ln^{\frac{2(\alpha-1)}{\alpha}}\left(\frac{6(K+1)}{\beta}\right)\ln^2(B_K)} \right\}  \\
        &=& \cO\left(\max\left\{2, \frac{K^{\frac{2(\alpha-1)}{\alpha}}\mu^2R^2}{\sigma^2\ln^{\frac{2(\alpha-1)}{\alpha}}\left(\frac{K}{\beta}\right)\ln^2\left(\max\left\{2, \frac{K^{\frac{2(\alpha-1)}{\alpha}}\mu^2R^2}{\sigma^2\ln^{\frac{2(\alpha-1)}{\alpha}}\left(\frac{K}{\beta}\right)} \right\}\right)} \right\}\right\}
    \end{eqnarray*}
    then with probability at least $1-\beta$
    \begin{eqnarray*}
        \|x^{K+1} - x^*\|^2 &\leq& 2\exp(-\gamma\mu (K+1))R^2\\
        &=& 2R^2\max\left\{\exp\left(-\frac{\mu(K+1)}{650 L \ln \tfrac{6(K+1)}{\beta}}\right), \frac{1}{B_K} \right\}\\
        &=& \cO\left(\max\left\{R^2\exp\left(- \frac{\mu K}{L \ln \tfrac{K}{\beta}}\right), \frac{\sigma^2\ln^{\frac{2(\alpha-1)}{\alpha}}\left(\frac{K}{\beta}\right)\ln^2\left(\max\left\{2, \frac{K^{\frac{2(\alpha-1)}{\alpha}}\mu^2R^2}{\sigma^2\ln^{\frac{2(\alpha-1)}{\alpha}}\left(\frac{K}{\beta}\right)} \right\}\right)}{K^{\frac{2(\alpha-1)}{\alpha}}\mu^2}\right\}\right).
    \end{eqnarray*}
    To get $\|x^{K+1} - x^*\|^2 \leq \varepsilon$ with probability at least $1-\beta$ it is sufficient to choose $K$ such that both terms in the maximum above are $\cO(\varepsilon)$. This leads to
    \begin{equation*}
         K = \cO\left(\frac{L}{\mu}\ln\left(\frac{R^2}{\varepsilon}\right)\ln\left(\frac{L}{\mu \beta}\ln\frac{R^2}{\varepsilon}\right), \left(\frac{\sigma^2}{\mu^2\varepsilon}\right)^{\frac{\alpha}{2(\alpha-1)}}\ln \left(\frac{1}{\beta} \left(\frac{\sigma^2}{\mu^2\varepsilon}\right)^{\frac{\alpha}{2(\alpha-1)}}\right)\ln^{\frac{\alpha}{\alpha-1}}\left(B_\varepsilon\right)\right),
    \end{equation*}
    where
    \begin{equation*}
        B_\varepsilon = \max\left\{2, \frac{R^2}{\varepsilon \ln \left(\frac{1}{\beta} \left(\frac{\sigma^2}{\mu^2\varepsilon}\right)^{\frac{\alpha}{2(\alpha-1)}}\right)}\right\}.
    \end{equation*}
    This concludes the proof.
\end{proof}

\clearpage

\section{Missing Proofs for \algname{clipped-SGDA}}\label{appendix:SGDA}

In this section, we provide the complete formulation of the main results for \algname{clipped-SGDA}  and the missing proofs. For brevity, we will use the following notation: $\tF_{\xi^k}(x^{k}) = \clip\left(F_{\xi^k}(x^{k}), \lambda_k\right)$ .

\begin{algorithm}[h]
\caption{Clipped Stochastic Gradient Descent Ascent (\algname{clipped-SGDA}) \citep{gorbunov2022clipped}}
\label{alg:clipped-SGDA}   
\begin{algorithmic}[1]
\REQUIRE starting point $x^0$, number of iterations $K$, stepsize $\gamma > 0$, clipping levels $\{\lambda_k\}_{k=0}^{K-1}$.
\FOR{$k=0,\ldots, K$}
\STATE Compute $\tF_{\xi^k}(x^{k}) = \clip\left(F_{\xi^k}(x^{k}), \lambda_k\right)$ using a fresh sample $\xi^k \sim \cD_k$
\STATE $x^{k+1} = x^k - \gamma \tF_{\xi^k}(x^{k})$
\ENDFOR
\ENSURE $x^{K+1}$ or $x_{\avg}^K = \frac{1}{K+1}\sum\limits_{k=0}^K x^K$
\end{algorithmic}
\end{algorithm}

\subsection{Monotone Star-Cocoercive Problems}

We start with the following lemma derived by \citet{gorbunov2022extragradient}. Since this lemma handles only deterministic part of the algorithm, the proof is the same as in the original work.
\begin{lemma}[Lemma D.1 from \citep{gorbunov2022extragradient}]\label{lem:optimization_lemma_gap_SGDA}
    Let Assumptions~\ref{as:monotonicity} and \ref{as:star-cocoercivity} hold for $Q = B_{3R}(x^*)$, where $R \geq \|x^0 - x^*\|$ and $0 < \gamma \leq \nicefrac{2}{\ell}$. If $x^k$ lies in $B_{3R}(x^*)$ for all $k = 0,1,\ldots, K$ for some $K\geq 0$, then for all $u \in B_{3R}(x^*)$ the iterates produced by \algname{clipped-SGDA} satisfy
    \begin{eqnarray}
        \langle F(u), x^K_{\avg} - u\rangle &\leq& \frac{\|x^0 - u\|^2 - \|x^{K+1} - u\|^2}{2\gamma(K+1)} + \frac{\gamma}{2(K+1)}\sum\limits_{k=0}^K\left(\|F(x^k)\|^2 + \|\omega_k\|^2\right)\notag\\
        &&\quad + \frac{1}{K+1}\sum\limits_{k=0}^K\langle x^k - u - \gamma F(x^k), \omega_k\rangle, \label{eq:optimization_lemma_SGDA}\\
        x^K_{\avg} &\eqdef& \frac{1}{K+1}\sum\limits_{k=0}^{K}x^k, \label{eq:x_avg_SGDA}\\
        \omega_k &\eqdef& F(x^k) - \tF_{\xi^k}(x^k). \label{eq:omega_k_SGDA}
    \end{eqnarray}
\end{lemma}

Also we need to use the following lemma to estimate the term $\sum\limits_{k=0}^K\|F(x^k)\|^2$ from the right hand side of \eqref{eq:optimization_lemma_SGDA} in the proof of the main theorem.

\begin{lemma}[Lemma D.2 from \citep{gorbunov2022extragradient}]\label{lem:optimization_lemma_SGDA}
    Let Assumption~\ref{as:star-cocoercivity} hold for $Q = B_{3R}(x^*)$, where $R \geq R_0 \eqdef \|x^0 - x^*\|$ and $0 < \gamma \leq \nicefrac{2}{\ell}$. If $x^k$ lies in $B_{3R}(x^*)$ for all $k = 0,1,\ldots, K$ for some $K\geq 0$, then the iterates produced by \algname{clipped-SGDA} satisfy
    \begin{eqnarray}
        \frac{\gamma}{K+1}\left(\frac{2}{\ell} - \gamma\right)\sum\limits_{k=0}^K\|F(x^k)\|^2 &\leq& \frac{\|x^0 - x^*\|^2 - \|x^{K+1} - x^*\|^2}{K+1}+ \frac{2\gamma}{K+1}\sum\limits_{k=0}^K\langle x^k - x^* - \gamma F(x^k), \omega_k\rangle\notag\\
        &&\quad + \frac{\gamma^2}{K+1}\sum\limits_{k=0}^K\|\omega_k\|^2,  \label{eq:optimization_lemma_norm_SGDA}
    \end{eqnarray}
    where $\omega_k$ is defined in \eqref{eq:omega_k_SGDA}.
\end{lemma}

Using those lemmas, we prove the main convergence result for \algname{clipped-SGDA} in the monotone star-cocoercive case.

\begin{theorem}[Case 1 in Theorem~\ref{thm:clipped_SGDA_main_theorem}]\label{thm:main_result_gap_SGDA}
    Let Assumptions~\ref{as:bounded_alpha_moment}, \ref{as:monotonicity}, \ref{as:star-cocoercivity} hold for $Q = B_{3R}(x^*)$, where $R \geq \|x^0 - x^*\|$, and    
    \begin{eqnarray}
        0< \gamma &\leq& \min\left\{\frac{1}{170\ell \ln \tfrac{6(K+1)}{\beta}}, \frac{R}{97200^{\tfrac{1}{\alpha}}(K+1)^{\frac{1}{\alpha}}\sigma \ln^{\frac{\alpha-1}{\alpha}} \tfrac{6(K+1)}{\beta}}\right\}, \label{eq:gamma_SGDA}\\
        \lambda_{k} \equiv \lambda &=& \frac{R}{60\gamma \ln \tfrac{6(K+1)}{\beta}}, \label{eq:lambda_SGDA}
    \end{eqnarray}
    for some $K \geq 0$ and $\beta \in (0,1]$ such that $\ln \tfrac{6(K+1)}{\beta} \geq 1$. Then, after $K$ iterations the iterates produced by \algname{clipped-SGDA} with probability at least $1 - \beta$ satisfy 
    \begin{equation}
        \gap_R(x_{\avg}^K) \leq \frac{5R^2}{\gamma(K+1)} \quad \text{and}\quad \{x^k\}_{k=0}^{K+1} \subseteq B_{3R}(x^*), \label{eq:main_result_SGDA}
    \end{equation}
    where $x_{\avg}^K$ is defined in \eqref{eq:x_avg_SGDA}. In particular, when $\gamma$ equals the minimum from \eqref{eq:gamma_SGDA}, then the iterates produced by \algname{clipped-SGDA} after $K$ iterations with probability at least $1-\beta$ satisfy
    \begin{equation}
        \gap_R(\tx_{\avg}^K) = \cO\left(\max\left\{\frac{\ell R^2\ln\frac{K}{\beta}}{K}, \frac{\sigma R \ln^{\frac{\alpha-1}{\alpha}}\frac{K}{\beta}}{K^{\frac{\alpha-1}{\alpha}}}\right\}\right), \label{eq:clipped_SGDA_monotone_case_2_appendix}
    \end{equation}
    meaning that to achieve $\gap_R(\tx_{\avg}^K) \leq \varepsilon$ with probability at least $1 - \beta$ \algname{clipped-SGDA} requires
    \begin{equation}
        K = \cO\left(\frac{\ell R^2}{\varepsilon}\ln\frac{\ell R^2}{\varepsilon\beta}, \left(\frac{\sigma R}{\varepsilon}\right)^{\frac{\alpha}{\alpha-1}}\ln\left(\frac{1}{\beta}\left(\frac{\sigma R}{\varepsilon}\right)^{\frac{\alpha}{\alpha-1}}\right)\right)\quad \text{iterations/oracle calls.} \label{eq:clipped_SGDA_monotone_case_complexity_appendix}
    \end{equation}
\end{theorem}

\begin{proof}
    The proof follows similar steps as the proof of Theorem~D.1 from \citep{gorbunov2022clipped}. The key difference is related to the application of Bernstein inequality and estimating biases and variances of stochastic terms.
    
    Let $R_k = \|x^k - x^*\|$ for all $k\geq 0$. As in the previous results, the proof is based on the induction argument and showing that the iterates do not leave some ball around the solution with high probability. More precisely, for each $k = 0,1,\ldots,K+1$ we consider probability event $E_k$ as follows: inequalities
    \begin{gather}
        \|x^t - x^*\|^2 \leq 2R^2 \quad \text{and}\quad \gamma\left\|\sum\limits_{l=0}^{t-1} \omega_l\right\| \leq R \label{eq:induction_inequality_SGDA_gap}
    \end{gather}
    hold for $t = 0,1,\ldots,k$ simultaneously. We want to prove that $\PP\{E_k\} \geq  1 - \nicefrac{k\beta}{(K+1)}$ for all $k = 0,1,\ldots,K+1$ by induction. The base of the induction is trivial: for $k = 0$ we have $R_0^2 \leq 2R^2$ by definition and $\sum_{l=0}^{-1} \omega_l = 0$. Next, assume that the statement holds for $k = T \leq K$, i.e., we have $\PP\{E_{T}\} \geq 1 - \nicefrac{T\beta}{(K+1)}$. Given this, we need to prove that $\PP\{E_{T+1}\} \geq 1 - \nicefrac{(T+1)\beta}{(K+1)}$. Since  probability event $E_{T}$ implies $R^2_t \leq 2R^2$, we have $x^t \in B_{2R}(x^*)$ for all $t = 0, 1, \ldots, T$. According to this, the assumptions of Lemma~\ref{lem:optimization_lemma_SGDA} hold and  $E_{T}$ implies ($\gamma < \nicefrac{1}{\ell}$)
    \begin{eqnarray}
        \frac{\gamma}{\ell(T+1)}\sum\limits_{t=0}^T\|F(x^t)\|^2 &\leq& \frac{\|x^0 - x^*\|^2 - \|x^{T+1} - x^*\|^2}{T+1}\notag\\
        &&\quad + \frac{2\gamma}{T+1}\sum\limits_{t=0}^T\langle x^t - x^* - \gamma F(x^t), \omega_t\rangle + \frac{\gamma^2}{T+1}\sum\limits_{t=0}^T\|\omega_t\|^2
        \label{eq:thm_SGDA_technical_gap_0}
    \end{eqnarray}
    and by $\ell$-star-cocoersivity we have 
    \begin{eqnarray}
        \|F(x^t)\| &\leq& \ell\|x^t - x^*\| \overset{\eqref{eq:induction_inequality_SGDA_gap}}{\leq} \sqrt{2}\ell R \overset{\eqref{eq:gamma_SGDA},\eqref{eq:lambda_SGDA}}{\leq} \frac{\lambda}{2} \label{eq:operator_bound_x_t_SGDA_gap}
    \end{eqnarray}
    for all $t = 0, 1, \ldots, T$.
    Using \eqref{eq:thm_SGDA_technical_gap_0}, we obtain  
    \begin{eqnarray*}
       R_{T+1}^2 \leq R_0^2 + 2\gamma \sum\limits_{t=0}^T\langle x^t - x^* - \gamma F(x^t), \omega_t\rangle + \gamma^2\sum\limits_{t=0}^T\|\omega_t\|^2.
    \end{eqnarray*}
    Due to \eqref{eq:operator_bound_x_t_SGDA_gap}, we have 
    \begin{eqnarray}
        \|x^t - x^* - \gamma F(x^t)\| &\leq& \|x^t - x^*\| + \gamma\|F(x^t)\| \overset{\eqref{eq:star-cocoercivity},\eqref{eq:induction_inequality_SGDA_gap}}{\leq} 2R + \gamma \ell\|x^t - x^*\|\notag\\
        &\overset{\eqref{eq:induction_inequality_SGDA_gap}}{\leq}& 2R + 2R\gamma \ell \overset{\eqref{eq:gamma_SGDA}}{\leq} 3R, \label{eq:thm_SGDA_technical_gap_2}
    \end{eqnarray}
    for all $t = 0, 1, \ldots, T$. To handle the sum above, we introduce a new vector
    \begin{equation*}
        \eta_t = \begin{cases}x^t - x^* - \gamma F(x^t),& \text{if } \|x^t - x^* - \gamma F(x^t)\| \leq 3R,\\ 0,& \text{otherwise,} \end{cases}
    \end{equation*}
    for all $t = 0, 1, \ldots, T$. This vector $\eta_t$ is bounded with probability $1$:
    \begin{equation}
        \|\eta_t\| \leq 3R  \label{eq:thm_SGDA_technical_gap_3}
    \end{equation}
    for all $t = 0, 1, \ldots, T$. We also notice that  probability event $E_{T}$ implies $\eta_t = x^t - x^* - \gamma F(x^t)$ for all $t = 0, 1, \ldots, T$ Thus, thanks to \eqref{eq:thm_SGDA_technical_gap_2}, $E_{T}$ implies
    \begin{eqnarray*}
       R_{T+1}^2 \leq R^2 + 2\gamma\sum\limits_{t=0}^T\langle \eta_t, \omega_t\rangle + \gamma^2\sum\limits_{t=0}^T\|\omega_t\|^2.
    \end{eqnarray*}
    To handle the sums appeared on the right-hand side of the previous inequality we consider unbiased and biased parts of $\omega_t$:
    \begin{gather}
        \omega_t^u \eqdef \EE_{\xi^t}\left[\tF_{\xi^t}(x^t)\right] - \tF_{\xi^t}(x^t),\quad \omega_t^b \eqdef F(x^t) - \EE_{\xi^t}\left[\tF_{\xi^t}(x^t)\right] \label{eq:thm_SGDA_technical_gap_4}
    \end{gather}
    for all $t = 0,\ldots, T$. Also, by definition we have $\omega_t = \omega_t^u + \omega_t^b$ for all $t = 0,\ldots, T$. Therefore, $E_{T}$ implies
    \begin{eqnarray}
        R_{T+1}^2 &\leq& R^2 + \underbrace{2\gamma \sum\limits_{t = 0}^{T} \langle \eta_t, \omega_t^u \rangle}_{\circledOne} + \underbrace{2\gamma \sum\limits_{t = 0}^{T} \langle \eta_t, \omega_t^b \rangle}_{\circledTwo} + \underbrace{2\gamma^2 \sum\limits_{t=0}^{T}\left(\EE_{\xi^t}\left[\|\omega_t^u\|^2\right] \right)}_{\circledThree} \notag\\
        &&\quad + \underbrace{2\gamma^2 \sum\limits_{t=0}^{T}\left(\|\omega_t^u\|^2 - \EE_{\xi^t}\left[\|\omega_t^u\|^2\right]\right)}_{\circledFour} + \underbrace{2\gamma^2 \sum\limits_{t=0}^{T}\left(\|\omega_t^b\|^2\right)}_{\circledFive}.\label{eq:thm_SGDA_technical_gap_5}
    \end{eqnarray}
    We notice that the above inequality does not rely on monotonicity of $F$.
    
    According to the induction assumption, from probability event $E_T$ we have $x^t \in B_{2R}(x^*)$ for all $t = 0, 1, \ldots, T$. Thus, the assumptions of Lemma~\ref{lem:optimization_lemma_gap_SGDA} hold and  probability event $E_{T}$ implies
    \begin{eqnarray}
        2\gamma(T+1)\gap_R(x^T_{\avg}) &\leq& \max\limits_{u\in B_R(x^*)}\left\{\|x^0 - u\|^2 + 2\gamma \sum\limits_{t=0}^T\langle x^t - u - \gamma F(x^t), \omega_t\rangle\right\}\notag\\
        &&\quad + \gamma^2\sum\limits_{t=0}^T\left(\|F(x^t)\|^2 + \|\omega_t\|^2\right),  \notag\\
        &=& \max\limits_{u\in B_R(x^*)}\left\{\|x^0 - u\|^2 + 2\gamma \sum\limits_{t=0}^T\langle x^* - u, \omega_t\rangle\right\}\notag\\
        &&\quad + 2\gamma \sum\limits_{t=0}^T\langle x^t - x^* - \gamma F(x^t), \omega_t\rangle\notag\\
        &&\quad + \gamma^2\sum\limits_{t=0}^T\left(\|F(x^t)\|^2 + \|\omega_t\|^2\right).  \notag
    \end{eqnarray}
    As we mentioned before, $E_T$ implies $\eta_t = x^t - x^* - \gamma F(x^t)$ for all $t = 0, 1, \ldots, T$ as well as \eqref{eq:thm_SGDA_technical_gap_0} and $\gamma < \nicefrac{1}{\ell}$. Due to that, probability event $E_T$ implies
    \begin{eqnarray}
        2\gamma(T+1)\gap_R(x^T_{\avg}) &\leq& \max\limits_{u\in B_R(x^*)}\left\{\|x^0 - u\|^2\right\} + 2\gamma \max\limits_{u\in B_R(x^*)}\left\{\sum\limits_{t=0}^T\langle x^* - u, \omega_t\rangle\right\}\notag\\
        &&\quad + 2\gamma \sum\limits_{t=0}^T\langle \eta_t, \omega_t\rangle + \frac{\gamma}{\ell}\sum\limits_{t=0}^T\|F(x^t)\|^2 + \gamma^2\sum\limits_{t=0}^T\|\omega_t\|^2 \notag\\
        &\leq& 4R^2 + 2\gamma \max\limits_{u\in B_R(x^*)}\left\{\left\langle x^* - u, \sum\limits_{t=0}^T\omega_t\right\rangle\right\} \notag\\
        &&\quad + R^2 + 4\gamma \sum\limits_{t=0}^T\langle \eta_t, \omega_t\rangle + 2\gamma^2\sum\limits_{t=0}^T\|\omega_t\|^2\notag\\
        &\leq& 5R^2 + 2\gamma R\left\|\sum\limits_{t=0}^T\omega_t\right\| + 2\cdot\left(\circledOne + \circledTwo + \circledThree + \circledFour + \circledFive\right),\label{eq:thm_SGDA_technical_gap_5_1}
    \end{eqnarray}
    where we also aplly inequality $\|a+b\|^2 \leq 2\|a\|^2 + 2\|b\|^2$ holding for all $a,b \in \R^d$ to upper bound $\|\omega_t\|^2$.
    
    It remains to derive good enough high-probability upper-bounds for the terms $\circledOne, \circledTwo, \circledThree, \circledFour, \circledFive$  and $2\gamma R\left\|\sum_{t=0}^T\omega_t\right\| $, i.e., to finish our inductive proof we need to show  that $\circledOne + \circledTwo + \circledThree + \circledFour + \circledFive \leq R^2$ and $2\gamma R\left\|\sum_{t=0}^T\omega_t\right\| \leq 2R^2$ with high probability.In the subsequent parts of the proof, we will need to use many times the bounds for the norm and second moments of $\omega_t^u, \omega_t^b$. First, by Lemma~\ref{lem:bias_variance}, we have with probability $1$ that
    \begin{equation}
        \|\omega_t^u\| \leq 2\lambda \label{eq:omega_magnitude_SGDA_gap}
    \end{equation}
    for all $t = 0,1, \ldots, T$. Moreover, due to Lemma~\ref{lem:bias_variance}, we also have that $E_{T}$ implies 
    \begin{gather}
         \left\|\omega_t^b\right\| \leq \frac{2^{\alpha}\sigma^{\alpha}}{\lambda^{\alpha-1}}, \label{eq:bias_omega_SGDA_gap}\\
         \EE_{\xi^t}\left[\left\|\omega^b_t\right\|^2\right] \leq 18\lambda^{2-\alpha}\sigma^{\alpha}, \label{eq:distortion_omega_SGDA_gap}\\
         \EE_{\xi^t}\left[\left\|\omega_t^u\right\|^2\right] \leq 18\lambda^{2-\alpha}\sigma^{\alpha} \label{eq:variance_omega_SGDA_gap}
    \end{gather}
    for all $t = 0,1, \ldots, T$.
    
    \paragraph{Upper bound for $\circledOne$.} By definition of $\omega^u_t$, we have $\EE_{\xi^t}[\omega_t^u] = 0$ and 
    \begin{equation*}
        \EE_{\xi^t}\left[2\gamma\langle \eta_t, \omega_t^u \rangle\right] = 0.
    \end{equation*}
    Next, the sum $\circledOne$ has bounded with probability $1$ term:
    \begin{eqnarray}
        |2\gamma\langle \eta_t, \omega_t^u \rangle | \leq 2\gamma  \|\eta_t\|\cdot \|\omega_t^u\| \overset{\eqref{eq:thm_SGDA_technical_gap_3},\eqref{eq:omega_magnitude_SGDA_gap}}{\leq} 12 \gamma R \lambda \overset{\eqref{eq:lambda_SGDA}}{\leq} \frac{R^2}{5\ln\tfrac{6(K+1)}{\beta}} \eqdef c. \label{eq:SGDA_neg_mon_technical_1_1}
    \end{eqnarray}
    Moreover, these summands also have bounded conditional variances $\sigma_t^2 \eqdef \EE_{\xi^t}\left[4\gamma^2 \langle \eta_t, \omega_t^u \rangle^2\right]$:
    \begin{eqnarray}
        \sigma_t^2 \leq \EE_{\xi^t}\left[4\gamma^2 \|\eta_t\|^2\cdot \|\omega_t^u\|^2\right] \overset{\eqref{eq:thm_SGDA_technical_gap_3}}{\leq} 36\gamma^2 R^2 \EE_{\xi^t}\left[\|\omega_t^u\|^2\right]. \label{eq:SGDA_neg_mon_technical_1_2}
    \end{eqnarray}
    In other words, we showed that  $\{2\gamma \langle \eta_t, \omega_t^u \rangle\}_{t\geq 0}$ is a bounded martingale difference sequence with bounded conditional variances $\{\sigma_t^2\}_{t \geq 0}$. Next, we apply Bernstein's inequality (Lemma~\ref{lem:Bernstein_ineq}) with $X_t = 2\gamma \langle \eta_t, \omega_t^u \rangle$,  parameter $c$ as in \eqref{eq:SGDA_neg_mon_technical_1_1}, $b = \frac{R^2}{5}$, $G = \tfrac{R^4}{150\ln\frac{6(K+1)}{\beta}}$:
    \begin{equation*}
        \PP\left\{|\circledOne| > \frac{R^2}{5} \text{ and } \sum\limits_{t=0}^{T}\sigma_t^2 \leq \frac{R^4}{150\ln\tfrac{6(K+1)}{\beta}}\right\} \leq 2\exp\left(- \frac{b^2}{2G + \nicefrac{2cb}{3}}\right) = \frac{\beta}{3(K+1)}.
    \end{equation*}
    Equivalently, we have 
    \begin{equation}
        \PP\{E_{\circledOne}\} \geq 1 - \tfrac{\beta}{3(K+1)}, \text{ for }E_{\circledOne} = \left\{\text{either} \quad \sum\limits_{t=0}^{T}\sigma_t^2 > \frac{R^4}{150\ln\tfrac{6(K+1)}{\beta}}\quad \text{or}\quad |\circledOne| \leq \frac{R^2}{5}\right\}. \label{eq:bound_1_SGDA_neg_mon}
    \end{equation}
    In addition, $E_{T}$ implies that
    \begin{eqnarray}
        \sum\limits_{t=0}^{T}\sigma_t^2 &\overset{\eqref{eq:SGDA_neg_mon_technical_1_2}}{\leq}& 36\gamma^2R^2\sum\limits_{t=0}^{T} \EE_{\xi^t}\left[\|\omega_t^u\|^2\right] \notag\\ &\overset{\eqref{eq:variance_omega_SGDA_gap}, T \leq K+1}{\leq}& 648\gamma^2 R^2 \sigma^{\alpha} (K+1) \lambda^{2-\alpha} \notag\\
        &\overset{\eqref{eq:lambda_SGDA}}{\leq}& 648 \gamma^{\alpha}R^{4-\alpha}\sigma^{\alpha} (K+1)\ln^{\alpha-2}\frac{6(K+1)}{\beta} \notag\\
        &\overset{\eqref{eq:gamma_SGDA}}{\leq}& \frac{R^4}{150\ln\tfrac{6(K+1)}{\beta}}. \label{eq:bound_1_variances_SGDA_neg_mon}
    \end{eqnarray}
    
    \paragraph{Upper bound for $\circledTwo$.} From $E_{T}$ it follows that
    \begin{eqnarray}
        \circledTwo &\leq& 2\gamma \sum\limits_{t=0}^{T}\|\eta_l\| \cdot \|\omega_t^b\| \overset{\eqref{eq:thm_SGDA_technical_gap_3},\eqref{eq:bias_omega_SGDA_gap}, T \leq K+1}{\leq} 6\cdot 2^{\alpha}\gamma R(K+1) \frac{\sigma^{\alpha}}{\lambda^{\alpha-1}}\notag\\
        &\overset{\eqref{eq:lambda_SGDA}}{=}& 12 \cdot 120^{\alpha-1}\gamma^{\alpha}\sigma^{\alpha}R^{2-\alpha} (K+1)\ln^{\alpha-1}\frac{6(K+1)}{\beta} \overset{\eqref{eq:gamma_SGDA}}{\leq} \frac{R^2}{5}. \label{eq:bound_2_SGDA_neg_mon}
    \end{eqnarray}

    \paragraph{Upper bound for $\circledThree$.}  From $E_{T}$ it follows that
    \begin{eqnarray}
        \circledThree &=&  2\gamma^2 \sum\limits_{t = 0}^{T}\EE_{\xi^t}\left[\|\omega_t^u\|^2\right] \overset{\eqref{eq:variance_omega_SGDA_gap}, T \leq K+1}{\leq} 36\gamma^2 \lambda^{2-\alpha} \sigma^{\alpha} (K+1) \notag\\
        &\overset{\eqref{eq:lambda_SGDA}}{\leq}& 36 \gamma^{\alpha} R^{2-\alpha} \sigma^{\alpha} (K+1)  \ln^{\alpha-2}\frac{6(K+1)}{\beta}
        \overset{\eqref{eq:gamma_SGDA}}{\leq} \frac{R^2}{5}. \label{eq:bound_3_SGDA_neg_mon}
    \end{eqnarray}
    
    \paragraph{Upper bound for $\circledFour$.} First, we have
    \begin{equation*}
        2\gamma^2\EE_{\xi^t}\left[\|\omega_t^u\|^2 - \EE_{\xi^t}\left[\|\omega_t^u\|^2\right]\right] = 0.
    \end{equation*}
    Next, the sum $\circledFour$ has bounded with probability $1$ terms:
    \begin{eqnarray}
        2\gamma^2 \left|\|\omega_t^u\|^2 - \EE_{\xi^t}\left[\|\omega_t^u\|^2\right] \right| &\leq& 2\gamma^2\left( \|\omega_t^u\|^2 + \EE_{\xi^t}\left[\|\omega_t^u\|^2\right] \right) 
        \overset{\eqref{eq:omega_magnitude_SGDA_gap}}{\leq} 16\gamma^2 \lambda^2 \notag\\
        &\overset{\eqref{eq:lambda_SGDA}}{\leq}& \frac{R^2}{225\ln\tfrac{6(K+1)}{\beta}} \leq \frac{R^2}{5\ln\tfrac{6(K+1)}{\beta}} \eqdef c. \label{eq:SGDA_neg_mon_technical_4_1}
    \end{eqnarray}
    The summands also have conditional variances $\widetilde\sigma_t^2 \eqdef 4\gamma^4\EE_{\xi^t}\left[\left(\|\omega_t^u\|^2 - \EE_{\xi^t}\left[\|\omega_t^u\|^2\right]\right)^2\right]$ that are bounded
    \begin{eqnarray}
        \widetilde\sigma_t^2 \overset{\eqref{eq:SGDA_neg_mon_technical_4_1}}{\leq} \frac{2\gamma^2 R^2}{225\ln \frac{6(K+1)}{\beta}} \EE_{\xi^t}\left[\left| \|\omega_t^u\|^2 - \EE_{\xi^t}\left[\|\omega_t^u\|^2\right] \right|\right]  \leq \frac{4\gamma^2 R^2}{225\ln \frac{6(K+1)}{\beta}} \EE_{\xi^t}\left[\|\omega_t^u\|^2\right]. \label{eq:SGDA_neg_mon_technical_4_2}
    \end{eqnarray}
    In other words, we showed that  $\{\|\omega_t^u\|^2 - \EE_{\xi^t}[\|\omega_t^u\|^2]\}_{t\geq 0}$ is a bounded martingale difference sequence with bounded conditional variances $\{\widetilde\sigma_t^2\}_{t \geq 0}$.Next, we apply Bernstein's inequality (Lemma~\ref{lem:Bernstein_ineq}) with $X_t = \|\omega_t^u\|^2 - \EE_{\xi^t}[\|\omega_t^u\|^2]$, parameter $c$ as in \eqref{eq:SGDA_neg_mon_technical_4_1}, $b = \frac{R^2}{5}$, $G = \tfrac{R^4}{150\ln\frac{6(K+1)}{\beta}}$:
    \begin{equation*}
        \PP\left\{|\circledFour| > \frac{R^2}{5} \text{ and } \sum\limits_{t=0}^{T}\widetilde\sigma_t^2 \leq \frac{R^4}{150\ln\tfrac{6(K+1)}{\beta}}\right\} \leq 2\exp\left(- \frac{b^2}{2G + \nicefrac{2cb}{3}}\right) = \frac{\beta}{3(K+1)}.
    \end{equation*}
    Equivalently, we have
    \begin{equation}
        \PP\{E_{\circledFour}\} \geq 1 - \tfrac{\beta}{3(K+1)}, \text{ for }E_{\circledFour} = \left\{\text{either} \quad \sum\limits_{t=0}^{T}\widetilde\sigma_t^2 > \frac{R^4}{150\ln\tfrac{6(K+1)}{\beta}}\quad \text{or}\quad |\circledFour| \leq \frac{R^2}{5}\right\}. \label{eq:bound_4_SGDA_neg_mon}
    \end{equation}
    In addition, $E_{T}$ implies that
    \begin{eqnarray}
        \sum\limits_{t=0}^{T}\widetilde\sigma_t^2 &\overset{\eqref{eq:SGDA_neg_mon_technical_4_2}}{\leq}& \frac{4\gamma^2 R^2}{225\ln \frac{6(K+1)}{\beta}}\sum\limits_{t=0}^{T} \EE_{\xi^t}\left[\|\omega_t^u\|^2\right] \overset{\eqref{eq:variance_omega_SGDA_gap}, T \leq K+1}{\leq} \frac{8\gamma^2 R^2 (K+1)}{25 \ln\tfrac{6(K+1)}{\beta}} \lambda^{2-\alpha} \sigma^{\alpha} \notag\\
        &\overset{\eqref{eq:lambda_SGDA}}{\leq}& \frac{8}{25}\gamma^{\alpha}
         R^{4-\alpha} (K+1) \sigma^{\alpha}\ln^{\alpha-3}\frac{6(K+1)}{\beta}\notag\\
        &\overset{\eqref{eq:gamma_SGDA}}{\leq}&\frac{R^4}{150\ln\tfrac{6(K+1)}{\beta}}. \label{eq:bound_4_variances_SGDA_neg_mon}
    \end{eqnarray}

    \paragraph{Upper bound for $\circledFive$.} From $E_{T}$ it follows that
    \begin{eqnarray}
        \circledFive &=&  2\gamma^2 \sum\limits_{t = 0}^{T}\|\omega_t^b\|^2  \overset{\eqref{eq:bias_omega_SGDA_gap}, T \leq K+1}{\leq} 2^{2\alpha+1}\cdot 60^{2\alpha-2}\gamma^2 (K+1) \frac{\sigma^{2\alpha}}{\lambda^{2\alpha-2}}\notag\\
        &\overset{\eqref{eq:lambda_SGDA}}{=}& 2^{2\alpha+1}\cdot 60^{2\alpha-2} \gamma^{2\alpha} (K+1) \frac{\sigma^{2\alpha}}{R^{2\alpha-2}}\ln^{2\alpha-2}\frac{6(K+1)}{\beta} \notag\\ &\overset{\eqref{eq:gamma_SGDA}}{\leq}& \frac{R^2}{5}. \label{eq:bound_5_SGDA_neg_mon}
    \end{eqnarray}

    \paragraph{Upper bound for $\gamma \left\|\sum_{t=0}^T\omega_t\right\|$.} To estimate this term from above, we consider a new vector:
    \begin{equation*}
        \zeta_l = \begin{cases} \gamma \sum\limits_{r=0}^{l-1}\omega_r,& \text{if } \left\|\gamma \sum\limits_{r=0}^{l-1}\omega_r\right\| \leq R,\\ 0, & \text{otherwise} \end{cases}
    \end{equation*}
    for $l = 1, 2, \ldots, T-1$.This vector is bounded almost surely:
    \begin{equation}
        \|\zeta_l\| \leq R.  \label{eq:gap_thm_SGDA_technical_8}
    \end{equation}
    Thus, by \eqref{eq:induction_inequality_SGDA_gap}, probability event $E_{T}$ implies
    \begin{eqnarray}
        \gamma\left\|\sum\limits_{l = 0}^{T}\omega_l\right\| &=& \sqrt{\gamma^2\left\|\sum\limits_{l = 0}^{T}\omega_l\right\|^2}\notag\\
        &=& \sqrt{\gamma^2\sum\limits_{l=0}^{T}\|\omega_l\|^2 + 2\gamma\sum\limits_{l=0}^{T}\left\langle \gamma\sum\limits_{r=0}^{l-1} \omega_r, \omega_l \right\rangle} \notag\\
        &=& \sqrt{\gamma^2\sum\limits_{l=0}^{T}\|\omega_l\|^2 + 2\gamma \sum\limits_{l=0}^{T} \langle \zeta_l, \omega_l\rangle} \notag\\
        &\overset{\eqref{eq:thm_SGDA_technical_gap_5}}{\leq}& \sqrt{\circledThree + \circledFour + \circledFive + \underbrace{2\gamma \sum\limits_{l=0}^{T} \langle \zeta_l, \omega_l^u\rangle}_{\circledSix} + \underbrace{2\gamma \sum\limits_{l=0}^{T} \langle \zeta_l, \omega_l^b}_{\circledSeven}\rangle}. \label{eq:norm_sum_omega_bound_gap_SGDA}
    \end{eqnarray}
    Following similar steps as before, we bound $\circledSix$ and $\circledSeven$.

    \paragraph{Upper bound for $\circledSix$.} By definition of $\omega^t_u$, we have  $\EE_{\xi^t}[\omega_t^u] = 0$ and 
    \begin{equation*}
        \EE_{\xi^t}\left[2\gamma\langle \zeta_t, \omega_t^u \rangle\right] = 0.
    \end{equation*}
    Next, sum $\circledSix$ has bounded with probability $1$ terms:
    \begin{eqnarray}
        |2\gamma\langle \zeta_t, \omega_t^u \rangle | \leq 2\gamma  \|\zeta_t\|\cdot \|\omega_t^u\| 
        \overset{\eqref{eq:gap_thm_SGDA_technical_8},\eqref{eq:omega_magnitude_SGDA_gap}}{\leq} 4 \gamma R \lambda \overset{\eqref{eq:lambda_SGDA}}{\leq} \frac{R^2}{5\ln\tfrac{6(K+1)}{\beta}} \eqdef c. \label{eq:SGDA_neg_mon_technical_6_1}
    \end{eqnarray}
    The summands also have bounded conditional variances $\widehat\sigma_t^2 \eqdef \EE_{\xi^t}\left[4\gamma^2 \langle \zeta_t, \omega_t^u \rangle^2\right]$:
    \begin{eqnarray}
        \widehat\sigma_t^2 \leq \EE_{\xi^t}\left[4\gamma^2 \|\zeta_t\|^2\cdot \|\omega_t^u\|^2\right] \overset{\eqref{eq:gap_thm_SGDA_technical_8}}{\leq} 4\gamma^2 R^2 \EE_{\xi^t}\left[\|\omega_t^u\|^2\right]. \label{eq:SGDA_neg_mon_technical_6_2}
    \end{eqnarray}
    In other words, we showed that $\{2\gamma \langle \zeta_t, \omega_t^u \rangle\}_{t\geq 0}$ is a bounded martingale difference sequence with bounded conditional variances $\{\widehat\sigma_t^2\}_{t \geq 0}$. Applying Bernstein's inequality (Lemma~\ref{lem:Bernstein_ineq}) with $X_t = 2\gamma \langle \zeta_t, \omega_t^u \rangle$, parameter $c$ as in \eqref{eq:SGDA_neg_mon_technical_6_1}, $b = \frac{R^2}{5}$, $G = \tfrac{R^4}{150\ln\frac{6(K+1)}{\beta}}$:
    \begin{equation*}
        \PP\left\{|\circledSix| > \frac{R^2}{5} \text{ and } \sum\limits_{t=0}^{T}\widehat\sigma_t^2 \leq \frac{R^4}{150\ln\tfrac{6(K+1)}{\beta}}\right\} \leq 2\exp\left(- \frac{b^2}{2G + \nicefrac{2cb}{3}}\right) = \frac{\beta}{3(K+1)}.
    \end{equation*}
    Equivalently, we have 
    \begin{equation}
        \PP\{E_{\circledSix}\} \geq 1 - \tfrac{\beta}{3(K+1)} \text{ for }E_{\circledSix} = \left\{\text{either} \quad \sum\limits_{t=0}^{T}\widehat\sigma_t^2 > \frac{R^4}{150\ln\tfrac{6(K+1)}{\beta}}\quad \text{or}\quad |\circledSix| \leq \frac{R^2}{5}\right\}. \label{eq:bound_6_SGDA_neg_mon}
    \end{equation}
    In addition, $E_{T}$ implies that
    \begin{eqnarray}
        \sum\limits_{t=0}^{T}\widehat\sigma_t^2 &\overset{\eqref{eq:SGDA_neg_mon_technical_6_2}}{\leq}& 4\gamma^2R^2\sum\limits_{t=0}^{T} \EE_{\xi^t}\left[\|\omega_t^u\|^2\right] \notag\\ &\overset{\eqref{eq:variance_omega_SGDA_gap}, T \leq K+1}{\leq}& 72\gamma^2 R^2 \sigma^{\alpha} (K+1) \lambda^{2-\alpha} \notag\\
        &\overset{\eqref{eq:lambda_SGDA}}{\leq}& 72 \gamma^{\alpha}R^{4-\alpha}\sigma^{\alpha} (K+1)\ln^{\alpha-2}\frac{6(K+1)}{\beta} \notag\\
        &\overset{\eqref{eq:gamma_SGDA}}{\leq}& \frac{R^4}{150\ln\tfrac{6(K+1)}{\beta}} . \label{eq:bound_6_variances_SGDA_neg_mon}
    \end{eqnarray}
    
    \paragraph{Upper bound for $\circledSeven$.} From $E_{T}$ it follows that
    \begin{eqnarray}
        \circledSeven &\leq& 2\gamma \sum\limits_{t=0}^{T}\|\zeta_t\| \cdot \|\omega_t^b\| \overset{\eqref{eq:gap_thm_SGDA_technical_8},\eqref{eq:bias_omega_SGDA_gap}, T \leq K+1}{\leq} 8\cdot 2^{\alpha}\gamma R(K+1) \frac{\sigma^{\alpha}}{\lambda^{\alpha-1}}\notag\\
        &\overset{\eqref{eq:lambda_SGDA}}{=}& 16 \cdot 120^{\alpha-1}\gamma^{\alpha}\sigma^{\alpha}R^{2-\alpha}(K+1)\ln^{\alpha-1}\frac{6(K+1)}{\beta} \overset{\eqref{eq:gamma_SGDA}}{\leq} \frac{R^2}{5}. \label{eq:bound_7_SGDA_neg_mon}
    \end{eqnarray}

    Now, we have the upper bounds for  $\circledOne, \circledTwo, \circledThree, \circledFour, \circledFive, \circledSix, \circledSeven$. In particular, probability event $E_{T-1}$ implies
    \begin{gather*}
        R_{T+1}^2 \overset{\eqref{eq:thm_SGDA_technical_gap_5}}{\leq} R^2 + \circledOne + \circledTwo + \circledThree + \circledFour + \circledFive ,\\
        2\gamma(T+1)\gap_R(x^T_{\avg}) \overset{\eqref{eq:thm_SGDA_technical_gap_5_1}}{\leq} 5R^2 + 2\gamma R\left\|\sum\limits_{t=0}^T\omega_t\right\| + 2\cdot\left(\circledOne + \circledTwo + \circledThree + \circledFour + \circledFive\right), \\
        \gamma\left\|\sum\limits_{l = 0}^{T}\omega_l\right\| \overset{\eqref{eq:norm_sum_omega_bound_gap_SGDA}}{\leq} \sqrt{\circledThree + \circledFour + \circledFive + \circledSix + \circledSeven},\\
        \circledTwo \overset{\eqref{eq:bound_2_SGDA_neg_mon}}{\leq} \frac{R^2}{5},\quad \circledThree \overset{\eqref{eq:bound_3_SGDA_neg_mon}}{\leq} \frac{R^2}{5},\quad \circledFive \overset{\eqref{eq:bound_5_SGDA_neg_mon}}{\leq} \frac{R^2}{5},\quad \circledSeven \overset{\eqref{eq:bound_7_SGDA_neg_mon}}{\leq} \frac{R^2}{5},\\
        \sum\limits_{t=0}^{T}\sigma_t^2 \overset{\eqref{eq:bound_1_variances_SGDA_neg_mon}}{\leq}  \frac{R^4}{150\ln\tfrac{6(K+1)}{\beta}},\quad \sum\limits_{t=0}^{T}\widetilde\sigma_t^2 \overset{\eqref{eq:bound_4_variances_SGDA_neg_mon}}{\leq} \frac{R^4}{150\ln\tfrac{6(K+1)}{\beta}},\quad \sum\limits_{t=0}^{T}\widehat\sigma_t^2 \overset{\eqref{eq:bound_6_variances_SGDA_neg_mon}}{\leq} \frac{R^4}{150\ln\tfrac{6(K+1)}{\beta}}.
    \end{gather*}
    Moreover, we also have (see  \eqref{eq:bound_1_SGDA_neg_mon}, \eqref{eq:bound_4_SGDA_neg_mon}, \eqref{eq:bound_7_SGDA_neg_mon} and our induction assumption)
    \begin{gather*}
        \PP\{E_{T}\} \geq 1 - \frac{T\beta}{K+1},\\
        \PP\{E_{\circledOne}\} \geq 1 - \frac{\beta}{3(K+1)}, \quad \PP\{E_{\circledFour}\} \geq 1 - \frac{\beta}{3(K+1)}, \quad \PP\{E_{\circledSix}\} \geq 1 - \frac{\beta}{3(K+1)},
    \end{gather*}
    where
    \begin{eqnarray}
        E_{\circledOne}&=& \left\{\text{either} \quad \sum\limits_{t=0}^{T}\sigma_t^2 > \frac{R^4}{150\ln\tfrac{6(K+1)}{\beta}}\quad \text{or}\quad |\circledOne| \leq \frac{R^2}{5}\right\},\notag\\
        E_{\circledFour}&=& \left\{\text{either} \quad \sum\limits_{t=0}^{T}\widetilde\sigma_t^2 > \frac{R^4}{150\ln\tfrac{6(K+1)}{\beta}}\quad \text{or}\quad |\circledFour| \leq \frac{R^2}{5}\right\},\notag\\
        E_{\circledSix}&=& \left\{\text{either} \quad \sum\limits_{t=0}^{T}\widehat\sigma_t^2 > \frac{R^4}{150\ln\tfrac{6(K+1)}{\beta}}\quad \text{or}\quad |\circledSix| \leq \frac{R^2}{5}\right\}.\notag
    \end{eqnarray}
    Thus, probability event $E_{T} \cap E_{\circledOne} \cap E_{\circledFour} \cap E_{\circledSix}$ implies
    \begin{eqnarray*}
        R_{T+1}^2 &\leq& R^2 + \circledOne + \circledTwo + \circledThree + \circledFour + \circledFive \leq 2R^2,\\
        \gamma\left\|\sum\limits_{l = 0}^{T}\omega_l\right\| &\leq& \sqrt{\circledThree + \circledFour + \circledFive + \circledSix + \circledSeven} \leq R,\\
        2\gamma(T+1)\gap_R(x^T_{\avg}) &\leq& 6R^2 + 2\gamma R\left\|\sum\limits_{t=0}^T\omega_t\right\| + 2\cdot\left(\circledOne + \circledTwo + \circledThree + \circledFour + \circledFive\right)\\
        &\leq& 10R^2,
    \end{eqnarray*}
    which gives \eqref{eq:induction_inequality_SGDA_gap} for  $t = T$, and 
    \begin{equation}
        \PP\{E_{T+1}\} \geq \PP\{E_{T} \cap E_{\circledOne} \cap E_{\circledFour} \cap E_{\circledSix}\} = 1 - \PP\{\overline{E}_{T} \cup \overline{E}_{\circledOne} \cup \overline{E}_{\circledFour} \cup \overline{E}_{\circledSix}\} \geq 1 - \frac{T\beta}{K+1}. \notag
    \end{equation}
    This finishes the inductive part of our proof, i.e., for all $k = 0,1,\ldots,K+1$ we have $\PP\{E_k\} \geq 1 - \nicefrac{k\beta}{(K+1)}$.  In particular, for $k = K+1$ we have that with probability at least $1 - \beta$
    \begin{equation*}
        \gap_R(x_{\avg}^K) \leq \frac{5R^2}{\gamma(K+1)}.
    \end{equation*}
    Finally, if 
    \begin{equation*}
        \gamma = \min\left\{\frac{1}{170\ell \ln \tfrac{6(K+1)}{\beta}}, \frac{R}{97200^{\tfrac{1}{\alpha}}(K+1)^{\frac{1}{\alpha}}\sigma \ln^{\frac{\alpha-1}{\alpha}} \tfrac{6(K+1)}{\beta}}\right\}
    \end{equation*}
    then with probability at least $1-\beta$
    \begin{eqnarray*}
        \gap_{R}(\tx^{K}_{\avg}) &\leq& \frac{5R^2}{\gamma(K+1)} = \max\left\{\frac{800 LR^2 \ln \tfrac{6(K+1)}{\beta}}{K+1}, \frac{5 \cdot97200^{\tfrac{1}{\alpha}}\sigma R \ln^{\frac{\alpha-1}{\alpha}} \tfrac{6(K+1)}{\beta}}{(K+1)^{\frac{\alpha-1}{\alpha}}}\right\}\\
        &=& \cO\left(\max\left\{\frac{\ell R^2\ln\frac{K}{\beta}}{K}, \frac{\sigma R \ln^{\frac{\alpha-1}{\alpha}}\frac{K}{\beta}}{K^{\frac{\alpha-1}{\alpha}}}\right\}\right).
    \end{eqnarray*}
    To get $\gap_{R}(\tx^{K}_{\avg}) \leq \varepsilon$ with probability at least $1-\beta$ it is sufficient to choose $K$ such that both terms in the maximum above are $\cO(\varepsilon)$. This leads to
    \begin{equation*}
         K = \cO\left(\frac{\ell R^2}{\varepsilon}\ln\frac{\ell R^2}{\varepsilon\beta}, \left(\frac{\sigma R}{\varepsilon}\right)^{\frac{\alpha}{\alpha-1}}\ln\left(\frac{1}{\beta}\left(\frac{\sigma R}{\varepsilon}\right)^{\frac{\alpha}{\alpha-1}}\right)\right)
    \end{equation*}
    that concludes the proof.
\end{proof}

\subsection{Star-Cocoercive Problems}

\begin{theorem}[Case 2 in Theorem~\ref{thm:clipped_SGDA_main_theorem}]\label{thm:main_result_non_mon_SGDA}
    Let Assumptions~\ref{as:bounded_alpha_moment}, \ref{as:star-cocoercivity} hold for $Q = B_{2R}(x^*)$, where $R \geq \|x^0 - x^*\|$, and    
    \begin{eqnarray}
        0< \gamma &\leq& \min\left\{\frac{1}{170\ell \ln \tfrac{4(K+1)}{\beta}}, \frac{R}{97200^{\tfrac{1}{\alpha}}(K+1)^{\frac{1}{\alpha}}\sigma \ln^{\frac{\alpha-1}{\alpha}} \tfrac{4(K+1)}{\beta}}\right\}, \label{eq:gamma_SGDA_non_mon}\\
        \lambda_{k} \equiv \lambda &=& \frac{R}{60\gamma \ln \tfrac{4(K+1)}{\beta}}, \label{eq:lambda_SGDA_non_mon}
    \end{eqnarray}
    for some $K \geq 0$ and $\beta \in (0,1]$ such that $\ln \tfrac{4(K+1)}{\beta} \geq 1$. Then, after $K$ iterations the iterates produced by \algname{clipped-SGDA} with probability at least $1 - \beta$ satisfy 
    \begin{equation}
        \frac{1}{K+1}\sum^K_{k=0}\|F(x^k)\|^2 \leq \frac{2\ell R^2}{\gamma(K+1)} . \label{eq:main_result_non_mon}
    \end{equation}
     In particular, when $\gamma$ equals the minimum from \eqref{eq:gamma_SGDA_non_mon}, then the iterates produced by \algname{clipped-SGDA} after $K$ iterations with probability at least $1-\beta$ satisfy
    \begin{equation}
         \frac{1}{K+1}\sum^K_{k=0}\|F(x^k)\|^2 = \cO\left(\max\left\{\frac{\ell^2 R^2\ln\frac{K}{\beta}}{K+1}, \frac{\ell \sigma R \ln^{\frac{\alpha-1}{\alpha}}\frac{K}{\beta}}{K^{\frac{\alpha-1}{\alpha}}}\right\}\right), \label{eq:clipped_SGDA_non_monotone_case_2_appendix}
    \end{equation}
    meaning that to achieve $\frac{1}{K+1}\sum\limits^K_{k=0}\|F(x^k)\|^2 \leq \varepsilon$ with probability at least $1 - \beta$ \algname{clipped-SGDA} requires
    \begin{equation}
        K = \cO\left(\frac{\ell^2 R^2}{\varepsilon}\ln\frac{\ell^2 R^2}{\varepsilon\beta}, \left(\frac{\ell \sigma R}{\varepsilon}\right)^{\frac{\alpha}{\alpha-1}}\ln\left( \frac{1}{\beta} \left(\frac{\ell \sigma R}{\varepsilon}\right)^{\frac{\alpha}{\alpha-1}}\right)\right)\quad \text{iterations/oracle calls.} \label{eq:clipped_SGDA_non_monotone_case_complexity_appendix}
    \end{equation}
\end{theorem}
\begin{proof}
    Again, we will closely follow the proof of Theorem~D.2 from \citep{gorbunov2022clipped} and the main difference will be reflected in the application of Bernstein inequality and estimating biases and variances of stochastic terms.

    Let $R_k = \|x^k - x^*\|$ for all $k\geq 0$. As the previous result, the proof is based on on the induction argument and showing that the iterates do not leave some ball around the solution with high probability. More precisely, for each $k = 0,\ldots, K+1$ we define probability event $E_k$ as follows: inequalities
    \begin{gather}
        \|x^t - x^*\|^2 \leq 2R^2, \label{eq:induction_inequality_SGDA}
    \end{gather}
    hold for $t = 0,1,\ldots,k$ simultaneously. We want to prove that $\PP\{E_k\} \geq  1 - \nicefrac{k\beta}{(K+1)}$ for all $k = 0,1,\ldots,K+1$ by induction. One of the important things is that inequalities \eqref{eq:thm_SGDA_technical_gap_0} and \eqref{eq:thm_SGDA_technical_gap_5} are obtained without assuming monotonicity of $F$. Thus, if we do exactly the same steps as in the proof of Theorem~\ref{thm:main_result_gap_SGDA} (up to the replacement of $\ln \frac{6(K+1)}{\beta}$ by $\ln \frac{4(K+1)}{\beta}$), we gain that
    \begin{gather*}
        R_{T+1}^2 \overset{\eqref{eq:thm_SGDA_technical_gap_5}}{\leq} R^2 + \circledOne + \circledTwo + \circledThree + \circledFour + \circledFive ,\\
        \circledTwo \overset{\eqref{eq:bound_2_SGDA_neg_mon}}{\leq} \frac{R^2}{5},\quad \circledThree \overset{\eqref{eq:bound_3_SGDA_neg_mon}}{\leq} \frac{R^2}{5},\quad \circledFive \overset{\eqref{eq:bound_5_SGDA_neg_mon}}{\leq} \frac{R^2}{5},\\
        \sum\limits_{t=0}^{T}\sigma_t^2 \overset{\eqref{eq:bound_1_variances_SGDA_neg_mon}}{\leq}  \frac{R^4}{150\ln\tfrac{4(K+1)}{\beta}},\quad \sum\limits_{t=0}^{T}\widetilde\sigma_t^2 \overset{\eqref{eq:bound_4_variances_SGDA_neg_mon}}{\leq} \frac{R^4}{150\ln\tfrac{4(K+1)}{\beta}}.
    \end{gather*}
    Moreover, we also have (see \eqref{eq:bound_1_SGDA_neg_mon}, \eqref{eq:bound_4_SGDA_neg_mon} and our induction assumption)
    \begin{gather*}
        \PP\{E_{T}\} \geq 1 - \frac{T\beta}{K+1},\\
        \PP\{E_{\circledOne}\} \geq 1 - \frac{\beta}{2(K+1)}, \quad \PP\{E_{\circledFour}\} \geq 1 - \frac{\beta}{2(K+1)},
    \end{gather*}
    where 
    \begin{eqnarray}
        E_{\circledOne}&=& \left\{\text{either} \quad \sum\limits_{t=0}^{T}\sigma_t^2 > \frac{R^4}{150\ln\tfrac{4(K+1)}{\beta}}\quad \text{or}\quad |\circledOne| \leq \frac{R^2}{5}\right\},\notag\\
        E_{\circledFour}&=& \left\{\text{either} \quad \sum\limits_{t=0}^{T}\widetilde\sigma_t^2 > \frac{R^4}{150\ln\tfrac{4(K+1)}{\beta}}\quad \text{or}\quad |\circledFour| \leq \frac{R^2}{5}\right\}.\notag
    \end{eqnarray}
    Thus probability event $E_{T-1} \cap E_{\circledOne} \cap E_{\circledFour}$ implies
    \begin{eqnarray*}
        R_{T+1}^2 \leq R^2 + \circledOne + \circledTwo + \circledThree + \circledFour + \circledFive \leq 2R^2,
    \end{eqnarray*}
    and 
    \begin{equation}
        \PP\{E_{T+1}\} \geq \PP\{E_{T} \cap E_{\circledOne} \cap E_{\circledFour}\} = 1 - \PP\{\overline{E}_{T} \cup \overline{E}_{\circledOne} \cup \overline{E}_{\circledFour}\} \geq 1 - \frac{T\beta}{K+1}.
    \end{equation}
    This finishes the inductive part of our proof, i.e.  for all $k = 0,1,\ldots,K+1$ we have $\PP\{E_k\} \geq 1 - \nicefrac{k\beta}{(K+1)}$. In particular, for $k = K+1$ we have that with probability at least $1 - \beta$
    \begin{eqnarray*}
        \frac{1}{K+1}\sum\limits_{k=0}^{K}\|F(x^k)\|^2 &\overset{\eqref{eq:thm_SGDA_technical_gap_0}}{\leq}& \frac{\ell(R^2 - R_{K+1}^2)}{\gamma(K+1)}  + \frac{\ell(\circledOne + \circledTwo + \circledThree + \circledFour + \circledFive)}{\gamma(K+1)}\\
        &\leq& \frac{2\ell R^2}{\gamma (K+1)}.
    \end{eqnarray*}
     Finally, if 
    \begin{equation*}
        \gamma = \min\left\{\frac{1}{170\ell \ln \tfrac{4(K+1)}{\beta}}, \frac{R}{97200^{\tfrac{1}{\alpha}}(K+1)^{\frac{1}{\alpha}}\sigma \ln^{\frac{\alpha-1}{\alpha}} \tfrac{4(K+1)}{\beta}}\right\}
    \end{equation*}
    then with probability at least $1-\beta$
    \begin{eqnarray*}
        \frac{1}{K+1}\sum\limits_{k=0}^{K}\|F(x^k)\|^2 &\leq& \frac{2\ell R^2}{\gamma(K+1)} = \max\left\{\frac{340 \ell^2 R^2 \ln \tfrac{4(K+1)}{\beta}}{K+1}, \frac{2 \cdot97200^{\tfrac{1}{\alpha}}\ell\sigma R \ln^{\frac{\alpha-1}{\alpha}} \tfrac{4(K+1)}{\beta}}{(K+1)^{\frac{\alpha-1}{\alpha}}}\right\}\\
        &=& \cO\left(\max\left\{\frac{\ell^2 R^2\ln\frac{K}{\beta}}{K}, \frac{\ell \sigma R \ln^{\frac{\alpha-1}{\alpha}}\frac{K}{\beta}}{K^{\frac{\alpha-1}{\alpha}}}\right\}\right).
    \end{eqnarray*}
    To get $\frac{1}{K+1}\sum\limits_{k=0}^{K}\|F(x^k)\|^2 \leq \varepsilon$ with probability at least $1-\beta$ it is sufficient to choose $K$ such that both terms in the maximum above are $\cO(\varepsilon)$. This leads to
    \begin{equation*}
         K = \cO\left(\frac{\ell^2 R^2}{\varepsilon}\ln\frac{\ell^2 R^2}{\varepsilon\beta}, \left(\frac{\ell \sigma R}{\varepsilon}\right)^{\frac{\alpha}{\alpha-1}}\ln\left( \frac{1}{\beta} \left(\frac{\ell \sigma R}{\varepsilon}\right)^{\frac{\alpha}{\alpha-1}}\right)\right)
    \end{equation*}
    that concludes the proof.

\end{proof}

\subsection{Quasi-Strongly Monotone Star-Cocoercive Problems}

As in the monotone case, we use another lemma from \citep{gorbunov2022clipped} that handles the deterministic part of \algname{clipped-SGDA} in the quasi-strongly monotone case.
\begin{lemma}[Lemma~D.3 from \citep{gorbunov2022clipped}]\label{lem:optimization_lemma_str_mon_SGDA}
    Let Assumptions~\ref{as:QSM}, \ref{as:star-cocoercivity}  hold for $Q = B_{2R}(x^*) = \{x\in\R^d\mid \|x - x^*\| \leq 2R\}$, where $R \geq \|x^0 - x^*\|$, and $0 < \gamma \leq \nicefrac{1}{\ell}$. If $x^k$ lie in $B_{2R}(x^*)$ for all $k = 0,1,\ldots, K$ for some $K\geq 0$, then the iterates produced by \algname{clipped-SGDA} satisfy
    \begin{eqnarray}
        \|x^{K+1} - x^*\|^2 &\leq& (1 - \gamma \mu)^{K+1}\|x^0 - x^*\|^2  + 2\gamma \sum\limits_{k=0}^K (1-\gamma\mu)^{K-k} \langle x^k - x^* - \gamma F(x^k), \omega_k \rangle\notag\\
        &&\quad + \gamma^2 \sum\limits_{k=0}^K (1-\gamma\mu)^{K-k}\|\omega_k\|^2, \label{eq:optimization_lemma_SGDA_str_mon}
    \end{eqnarray}
    where $\omega_k$ are defined in \eqref{eq:omega_k_SGDA}.
\end{lemma}

Using this lemma we prove the main convergence result for \algname{clipped-SGDA} in the quasi-strongly monotone case.

\begin{theorem}[Case 2 in Theorem~\ref{thm:clipped_SGDA_main_theorem}]\label{thm:main_result_str_mon_SGDA}
    Let Assumptions~\ref{as:bounded_alpha_moment}, \ref{as:QSM}, \ref{as:star-cocoercivity}, hold for $Q = B_{2R}(x^*) = \{x\in\R^d\mid \|x - x^*\| \leq 2R\}$, where $R \geq \|x^0 - x^*\|$, and
    \begin{eqnarray}
        0< \gamma &\leq& \min\left\{\frac{1}{400 \ell\ln \tfrac{4(K+1)}{\beta}}, \frac{\ln(B_K)}{\mu(K+1)}\right\}, \label{eq:gamma_SGDA_str_mon}\\
        B_K &=& \max\left\{2, \frac{(K+1)^{\frac{2(\alpha-1)}{\alpha}}\mu^2R^2}{5400^{\frac{2}{\alpha}}\sigma^2\ln^{\frac{2(\alpha-1)}{\alpha}}\left(\frac{4(K+1)}{\beta}\right)\ln^2(B_K)} \right\} \label{eq:B_K_SGDA_str_mon_1} \\
        &=& \cO\left(\max\left\{2, \frac{K^{\frac{2(\alpha-1)}{\alpha}}\mu^2R^2}{\sigma^2\ln^{\frac{2(\alpha-1)}{\alpha}}\left(\frac{K}{\beta}\right)\ln^2\left(\max\left\{2, \frac{K^{\frac{2(\alpha-1)}{\alpha}}\mu^2R^2}{\sigma^2\ln^{\frac{2(\alpha-1)}{\alpha}}\left(\frac{K}{\beta}\right)} \right\}\right)} \right\}\right\}, \label{eq:B_K_SGDA_str_mon_2} \\
        \lambda_k &=& \frac{\exp(-\gamma\mu(1 + \nicefrac{k}{2}))R}{120\gamma \ln \tfrac{4(K+1)}{\beta}}, \label{eq:lambda_SGDA_str_mon}
    \end{eqnarray}
    for some $K \geq 0$ and $\beta \in (0,1]$ such that $\ln \tfrac{4(K+1)}{\beta} \geq 1$. Then, after $K$ iterations the iterates produced by \algname{clipped-SGDA} with probability at least $1 - \beta$ satisfy 
    \begin{equation}
        \|x^{K+1} - x^*\|^2 \leq 2\exp(-\gamma\mu(K+1))R^2. \label{eq:main_result_str_mon_SGDA}
    \end{equation}
    In particular, when $\gamma$ equals the minimum from \eqref{eq:gamma_SGDA_str_mon}, then the iterates produced by \algname{clipped-SGDA} after $K$ iterations with probability at least $1-\beta$ satisfy
    \begin{equation}
       \|x^{K} - x^*\|^2 = \cO\left(\max\left\{R^2\exp\left(- \frac{\mu K}{\ell \ln \tfrac{K}{\beta}}\right), \frac{\sigma^2\ln^{\frac{2(\alpha-1)}{\alpha}}\left(\frac{K}{\beta}\right)\ln^2\left(\max\left\{2, \frac{K^{\frac{2(\alpha-1)}{\alpha}}\mu^2R^2}{\sigma^2\ln^{\frac{2(\alpha-1)}{\alpha}}\left(\frac{K}{\beta}\right)} \right\}\right)}{K^{\frac{2(\alpha-1)}{\alpha}}\mu^2}\right\}\right), \label{eq:clipped_SGDA_str_monotone_case_2_appendix}
    \end{equation}
    meaning that to achieve $\|x^{K} - x^*\|^2 \leq \varepsilon$ with probability at least $1 - \beta$ \algname{clipped-SGDA} requires
    \begin{equation}
        K = \cO\left(\frac{\ell}{\mu}\ln\left(\frac{R^2}{\varepsilon}\right)\ln\left(\frac{\ell}{\mu \beta}\ln\frac{R^2}{\varepsilon}\right), \left(\frac{\sigma^2}{\mu^2\varepsilon}\right)^{\frac{\alpha}{2(\alpha-1)}}\ln \left(\frac{1}{\beta} \left(\frac{\sigma^2}{\mu^2\varepsilon}\right)^{\frac{\alpha}{2(\alpha-1)}}\right)\ln^{\frac{\alpha}{\alpha-1}}\left(B_\varepsilon\right)\right) \label{eq:clipped_SGDA_str_monotone_case_complexity_appendix}
    \end{equation}
    iterations/oracle calls, where
    \begin{equation*}
        B_\varepsilon = \max\left\{2, \frac{R^2}{\varepsilon \ln \left(\frac{1}{\beta} \left(\frac{\sigma^2}{\mu^2\varepsilon}\right)^{\frac{\alpha}{2(\alpha-1)}}\right)}\right\}.
    \end{equation*}
\end{theorem}

\begin{proof}
    Again, we will closely follow the proof of Theorem~D.3 from \citep{gorbunov2022clipped} and the main difference will be reflected in the application of Bernstein inequality and estimating biases and variances of stochastic terms.
    
    Let $R_k = \|x^k - x^*\|$ for all $k\geq 0$. As in the previous results, the proof is based on the induction argument and showing that the iterates do not leave some ball around the solution with high probability. More precisely, for each $k = 0,1,\ldots,K+1$ we consider probability event $E_k$ as follows: inequalities
    \begin{equation}
        R_t^2 \leq 2 \exp(-\gamma\mu t) R^2 \label{eq:induction_inequality_str_mon_SGDA}
    \end{equation}
    hold for $t = 0,1,\ldots,k$ simultaneously. We want to prove $\PP\{E_k\} \geq  1 - \nicefrac{k\beta}{(K+1)}$ for all $k = 0,1,\ldots,K+1$ by induction. The base of the induction is trivial: for $k=0$ we have $R_0^2 \leq R^2 < 2R^2$ by definition. Next, assume that for $k = T-1 \leq K$ the statement holds: $\PP\{E_{T-1}\} \geq  1 - \nicefrac{(T-1)\beta}{(K+1)}$. Given this, we need to prove $\PP\{E_{T}\} \geq  1 - \nicefrac{T\beta}{(K+1)}$. Since $R_t^2 \leq 2\exp(-\gamma\mu t) R^2 \leq 2R^2$, we have $x^t \in B_{2R}(x^*)$, where operator $F$ is $\ell$-star-cocoersive. Thus, $E_{T-1}$ implies
    \begin{eqnarray}
        \|F(x^t)\| &\leq& \ell\|x^t - x^*\| \overset{\eqref{eq:induction_inequality_str_mon_SGDA}}{\leq} \sqrt{2}\ell\exp(- \nicefrac{\gamma\mu t}{2})R \overset{\eqref{eq:gamma_SGDA_str_mon},\eqref{eq:lambda_SGDA_str_mon}}{\leq} \frac{\lambda_t}{2} \label{eq:operator_bound_x_t_SGDA_str_mon}
    \end{eqnarray}
    and
    \begin{eqnarray}
        \|\omega_t\|^2 &\leq& 2\|\tF_{\xi}(x^t)\|^2 + 2\|F(x^t)\|^2 \overset{\eqref{eq:operator_bound_x_t_SGDA_str_mon}}{\leq} \frac{5}{2}\lambda_t^2 \overset{\eqref{eq:lambda_SGDA_str_mon}}{\leq} \frac{\exp(-\gamma\mu t)R^2}{4\gamma^2} \label{eq:omega_bound_x_t_SGDA_str_mon}
    \end{eqnarray}
    for all $t = 0, 1, \ldots, T-1$, where we use that $\|a+b\|^2 \leq 2\|a\|^2 + 2\|b\|^2$ holding for all $a,b \in \R^d$. 

    Using Lemma~\ref{lem:optimization_lemma_str_mon_SGDA} and $(1 - \gamma\mu)^T \leq \exp(-\gamma\mu T)$, we obtain that $E_{T-1}$ implies
    \begin{eqnarray}
        R_T^2 &\leq& \exp(-\gamma\mu T)R^2 + 2\gamma \sum\limits_{t=0}^{T-1} (1-\gamma\mu)^{T-1-t} \langle x^t - x^* - \gamma F(x^t), \omega_t \rangle\notag\\
        &&\quad + \gamma^2 \sum\limits_{t=0}^{T-1} (1-\gamma\mu)^{T-1-t}\|\omega_t\|^2. \notag
    \end{eqnarray}
    To handle the sums above, we introduce a new notation:
    \begin{gather}
        \eta_t = \begin{cases} x^t - x^* - \gamma F(x^t),& \text{if } \|x^t - x^* - \gamma F(x^t)\| \leq \sqrt{2}(1 + \gamma \ell) \exp(- \nicefrac{\gamma\mu t}{2})R,\\ 0,& \text{otherwise}, \end{cases} \label{eq:eta_t_SGDA_str_mon}
    \end{gather}
    for $t = 0, 1, \ldots, T-1$. This vector is bounded almost surely:
    \begin{equation}
         \|\eta_t\| \leq \sqrt{2}(1 + \gamma \ell)\exp(-\nicefrac{\gamma\mu t}{2})R \label{eq:eta_t_bound_SGDA_str_mon} 
    \end{equation}
    for all $t = 0, 1, \ldots, T-1$. We also notice that $E_{T-1}$ implies $\|F(x^t)\| \leq \sqrt{2}\ell\exp(-\nicefrac{\gamma\mu t}{2})R$ (due to \eqref{eq:operator_bound_x_t_SGDA_str_mon}) and
    \begin{eqnarray*}
        \|x^t - x^* - \gamma F(x^t)\| &\leq& \|x^t - x^*\| + \gamma \|F(x^t)\|\\
        &\overset{\eqref{eq:operator_bound_x_t_SGDA_str_mon}}{\leq}& \sqrt{2}(1 + \gamma \ell)\exp(-\nicefrac{\gamma\mu t}{2})R
    \end{eqnarray*}
    for $t = 0, 1, \ldots, T-1$. In other words, $E_{T-1}$ implies  $\eta_t = x^t - x^* - \gamma F(x^t)$ for all $t = 0,1,\ldots,T-1$, meaning that from $E_{T-1}$ it follows that
    \begin{eqnarray}
        R_T^2 &\leq& \exp(-\gamma\mu T)R^2 + 2\gamma \sum\limits_{t=0}^{T-1} (1-\gamma\mu)^{T-1-t} \langle \eta_t, \omega_t \rangle + \gamma^2 \sum\limits_{t=0}^{T-1} (1-\gamma\mu)^{T-1-t} \|\omega_t\|^2. \notag
    \end{eqnarray}

    To handle the sums appeared on the right-hand side of the previous inequality we consider unbiased and biased parts of $ \omega_t$:
    \begin{gather}
        \omega_t^u \eqdef \EE_{\xi^t}\left[F_{\xi^t}(x^t)\right] - \tF_{\xi^t}(x^t),\quad \omega_t^b \eqdef F(x^t) - \EE_{\xi_1^t}\left[F_{\xi^t}(x^t)\right], \label{eq:omega_unbias_bias_SGDA_str_mon}
    \end{gather}
    for all $t = 0,\ldots, T-1$. By definition we have  $\omega_t = \omega_t^u + \omega_t^b$ for all $t = 0,\ldots, T-1$. Therefore, $E_{T-1}$ implies
    \begin{eqnarray}
        R_T^2 &\leq& \exp(-\gamma\mu T) R^2 + \underbrace{2\gamma \sum\limits_{t=0}^{T-1} (1-\gamma\mu)^{T-1-t} \langle \eta_t, \omega_t^u \rangle}_{\circledOne} \notag\\ &&\quad + \underbrace{2\gamma \sum\limits_{t=0}^{T-1} (1-\gamma\mu)^{T-1-t} \langle \eta_t, \omega_t^b \rangle}_{\circledTwo} 
        + \underbrace{2\gamma^2 \sum\limits_{t=0}^{T-1} (1-\gamma\mu)^{T-1-t} \EE_{\xi}\left[\|\omega^u_t\|^2\right]}_{\circledThree}\notag\\
        &&\quad + \underbrace{2\gamma^2 \sum\limits_{t=0}^{T-1} (1-\gamma\mu)^{T-1-t}\left( \|\omega^u_t\|^2 -  \EE_{\xi}\left[\|\omega^u_t\|^2\right]\right)}_{\circledFour} + \underbrace{2\gamma^2 \sum\limits_{t=0}^{T-1} (1-\gamma\mu)^{T-1-t} \|\omega^b_t\|^2}_{\circledFive}. \label{eq:SGDA_str_mon_12345_bound}
    \end{eqnarray}
    where we also apply inequality $\|a+b\|^2 \leq 2\|a\|^2 + 2\|b\|^2$ holding for all $a,b \in \R^d$ to upper bound $\|\omega_t\|^2$. It remains to derive good enough high-probability upper-bounds for the terms $\circledOne, \circledTwo, \circledThree, \circledFour, \circledFive$, i.e., to finish our inductive proof we need to show that $\circledOne + \circledTwo + \circledThree + \circledFour + \circledFive  \leq \exp(-\gamma\mu T) R^2$ with high probability. In the subsequent parts of the proof, we will need to use many times the bounds for the norm and second moments of $\omega_{t}^u$ and $\omega_{t}^b$. First, by Lemma \ref{lem:bias_and_variance_clip}, we have with probability $1$ that
    \begin{equation}
        \|\omega_t^u\| \leq 2\lambda_t.\label{eq:omega_magnitude_str_mon}
    \end{equation}
    Moreover, since $E_{T-1}$ implies that $\|F(x^t)\| \leq \nicefrac{\lambda_t}{2}$ and $\|F(x^t)\| \leq \nicefrac{\lambda_t}{2}$ for all $t = 0,1, \ldots, T-1$ (see \eqref{eq:operator_bound_x_t_SGDA_str_mon}), from Lemma~\ref{lem:bias_and_variance_clip} we also have that $E_{T-1}$ implies
    \begin{gather}
        \left\|\omega_t^b\right\| \leq \frac{2^{\alpha}\sigma^\alpha}{\lambda_t^{\alpha-1}}, \label{eq:bias_omega_str_mon}\\
        \EE_{\xi^t}\left[\left\|\omega_t^b\right\|^2\right] \leq 18 \lambda_t^{2-\alpha}\sigma^\alpha, \label{eq:distortion_omega_str_mon}\\
        \EE_{\xi^t}\left[\left\|\omega_t^u\right\|^2\right] \leq 18 \lambda_t^{2-\alpha}\sigma^\alpha, \label{eq:variance_omega_str_mon}
    \end{gather}
    for all $t = 0,1, \ldots, T-1$.

\paragraph{Upper bound for $\circledOne$.} By definition of $\omega_t^u$, we have $\EE_{\xi^t}[\omega_{t}^u] = 0$ and
    \begin{equation*}
        \EE_{\xi^t}\left[2\gamma (1-\gamma\mu)^{T-1-t} \langle \eta_t, \omega_t^u \rangle\right] = 0.
    \end{equation*}
    Next, sum $\circledOne$ has bounded with probability $1$ terms:
    \begin{eqnarray}
        |2\gamma (1-\gamma\mu)^{T-1-t} \langle \eta_t, \omega_t^u \rangle | &\leq& 2\gamma\exp(-\gamma\mu (T - 1 - t)) \|\eta_t\|\cdot \|\omega_t^u\|\notag\\
        &\overset{\eqref{eq:eta_t_bound_SGDA_str_mon},\eqref{eq:omega_magnitude_str_mon}}{\leq}& 4\sqrt{2}\gamma (1 + \gamma \ell) \exp(-\gamma\mu (T - 1 - \nicefrac{t}{2})) R \lambda_t\notag\\
        &\overset{\eqref{eq:gamma_SGDA_str_mon},\eqref{eq:lambda_SGDA_str_mon}}{\leq}& \frac{\exp(-\gamma\mu T)R^2}{5\ln\tfrac{4(K+1)}{\beta}} \eqdef c. \label{eq:SGDA_str_mon_technical_1_1}
    \end{eqnarray}
    The summands also have bounded conditional variances $\sigma_t^2 \eqdef \EE_{\xi^t}\left[4\gamma^2 (1-\gamma\mu)^{2T-2-2t} \langle \eta_t, \omega_t^u \rangle^2\right]$:
    \begin{eqnarray}
        \sigma_t^2 &\leq& \EE_{\xi^t}\left[4\gamma^2\exp(-\gamma\mu (2T - 2 - 2t)) \|\eta_t\|^2\cdot \|\omega_t^u\|^2\right]\notag\\
        &\overset{\eqref{eq:eta_t_bound_SGDA_str_mon}}{\leq}& 8\gamma^2 (1 + \gamma \ell)^2 \exp(-\gamma\mu (2T - 2 - t)) R^2 \EE_{\xi^t}\left[\|\omega_t^u\|^2\right]\notag\\
        &\overset{\eqref{eq:gamma_SGDA_str_mon}}{\leq}& 10\gamma^2\exp(-\gamma\mu (2T - t))R^2 \EE_{\xi^t}\left[\|\omega_t^u\|^2\right]. \label{eq:SGDA_str_mon_technical_1_2}
    \end{eqnarray}

    In other words, we showed that $\{2\gamma (1-\gamma\mu)^{T-1-t} \langle \eta_t, \omega_t^u \rangle\}_{t = 0}^{T-1}$ is a bounded martingale difference sequence with bounded conditional variances $\{\sigma_t^2\}_{t = 0}^{T-1}$. Next, we apply Bernstein's inequality (Lemma~\ref{lem:Bernstein_ineq}) with $X_t = 2\gamma (1-\gamma\mu)^{T-1-t} \langle \eta_t, \omega_t^u \rangle$, parameter $c$ as in \eqref{eq:SGDA_str_mon_technical_1_1}, $b = \tfrac{1}{5}\exp(-\gamma\mu T) R^2$, $F = \tfrac{\exp(-2 \gamma\mu T) R^4}{300\ln\frac{4(K+1)}{\beta}}$:
    \begin{equation*}
        \PP\left\{|\circledOne| > \frac{1}{5}\exp(-\gamma\mu T) R^2 \text{ and } \sum\limits_{t=0}^{T-1}\sigma_t^2 \leq \frac{\exp(- 2\gamma\mu T) R^4}{300\ln\tfrac{4(K+1)}{\beta}}\right\} \leq 2\exp\left(- \frac{b^2}{2F + \nicefrac{2cb}{3}}\right) = \frac{\beta}{2(K+1)}.
    \end{equation*}
    Equivalently, we have
    \begin{equation}
        \PP\{E_{\circledOne}\} \geq 1 - \frac{\beta}{2(K+1)},\quad \text{for}\quad E_{\circledOne} = \left\{\text{either} \quad \sum\limits_{t=0}^{T-1}\sigma_t^2 > \frac{\exp(- 2\gamma\mu T) R^4}{300\ln\tfrac{4(K+1)}{\beta}}\quad \text{or}\quad |\circledOne| \leq \frac{1}{5}\exp(-\gamma\mu T) R^2\right\}. \label{eq:bound_1_SGDA_str_mon}
    \end{equation}
    In addition, $E_{T-1}$ implies that
    \begin{eqnarray}
        \sum\limits_{t=0}^{T-1}\sigma_t^2 &\overset{\eqref{eq:SGDA_str_mon_technical_1_2}}{\leq}& 10\gamma^2\exp(- 2\gamma\mu T)R^2\sum\limits_{t=0}^{T-1} \frac{\EE_{\xi^t}\left[\|\omega_t^u\|^2\right]}{\exp(-\gamma\mu t)}\notag\\ 
        &\overset{\eqref{eq:variance_omega_str_mon}, T \leq K+1}{\leq}& 180\gamma^2\exp(-2\gamma\mu T) R^2 \sigma^\alpha \sum\limits_{t=0}^{K} \frac{\lambda_t^{2-\alpha}}{\exp(-\gamma\mu t)}\notag\\
        &\overset{\eqref{eq:lambda_SGDA_str_mon}}{\leq}& \frac{180\gamma^\alpha\exp(-2\gamma\mu T) R^{4-\alpha} \sigma^\alpha (K+1)\exp(\frac{\gamma\mu\alpha K}{2})}{120^{2-\alpha} \ln^{2-\alpha}\tfrac{4(K+1)}{\beta}}\notag\\
        &\overset{\eqref{eq:gamma_SGDA_str_mon}}{\leq}& \frac{\exp(-2\gamma\mu T)R^4}{300\ln\tfrac{4(K+1)}{\beta}}. \label{eq:bound_1_variances_SGDA_str_mon}
    \end{eqnarray}

    \paragraph{Upper bound for $\circledTwo$.} From $E_{T-1}$ it follows that
    \begin{eqnarray}
        \circledTwo &\leq& 2\gamma \exp(-\gamma\mu (T-1)) \sum\limits_{t=0}^{T-1} \frac{\|\eta_t\|\cdot \|\omega_t^b\|}{\exp(-\gamma\mu t)}\notag\\
        &\overset{\eqref{eq:eta_t_bound_SGDA_str_mon}, \eqref{eq:bias_omega_str_mon}}{\leq}& 2^{1+\alpha}\sqrt{2} \gamma (1+\gamma \ell) \exp(-\gamma\mu (T-1)) R \sigma^\alpha \sum\limits_{t=0}^{T-1} \frac{1}{\lambda_t^{\alpha-1} \exp(-\nicefrac{\gamma\mu t}{2})}\notag\\
        &\overset{\eqref{eq:lambda_SGDA_str_mon}}{\leq}& \frac{2^{3+\alpha}\cdot 120^{\alpha-1}\sqrt{2} \gamma^\alpha(1+\gamma \ell) \exp(-\gamma\mu T) \sigma^{\alpha} (K+1) \exp\left(\frac{\gamma\mu\alpha T}{2}\right) \ln^{\alpha-1}\tfrac{4(K+1)}{\beta}}{R^{\alpha-2}} \notag \\
        &\overset{\eqref{eq:gamma_SGDA_str_mon}}{\leq}& \frac{1}{5}\exp(-\gamma\mu T) R^2. \label{eq:bound_2_SGDA_str_mon}
    \end{eqnarray}

    \paragraph{Upper bound for $\circledThree$.} From $E_{T-1}$ it follows that
    \begin{eqnarray}
        \circledThree &=& 2\gamma^2 \exp(-\gamma\mu (T-1)) \sum\limits_{t=0}^{T-1} \frac{\EE_{\xi^t}\left[\|\omega_t^u\|^2\right]}{\exp(-\gamma\mu t)} \notag\\
        &\overset{\eqref{eq:variance_omega_str_mon}}{\leq}& 144\gamma^2\exp(-\gamma\mu (T-1)) \sigma^\alpha\sum\limits_{t=0}^{T-1} \frac{\lambda_t^{2-\alpha}}{\exp(-\gamma\mu t)} \notag\\
        &\overset{\eqref{eq:lambda_SGDA_str_mon}}{\leq}&  \frac{144\gamma^\alpha R^{2-\alpha}\exp(-\gamma\mu (T-1)) \sigma^\alpha (K+1)\exp(\frac{\gamma\mu\alpha K}{2})}{120^{2-\alpha} \ln^{2-\alpha}\tfrac{4(K+1)}{\beta}} \notag\\
        &\overset{\eqref{eq:gamma_SGDA_str_mon}}{\leq}& \frac{1}{5} \exp(-\gamma\mu T) R^2. \label{eq:bound_3_SGDA_str_mon}
    \end{eqnarray}

    \paragraph{Upper bound for $\circledFour$.} First, we have
    \begin{equation*}
        2\gamma^2 (1-\gamma\mu)^{T-1-t}\EE_{\xi^t}\left[\|\omega_t^u\|^2  -\EE_{\xi^t}\left[\|\omega_t^u\|^2\right] \right] = 0.
    \end{equation*}
    Next, sum $\circledFour$ has bounded with probability $1$ terms:
    \begin{eqnarray}
        2\gamma^2 (1-\gamma\mu)^{T-1-t}\left| \|\omega_t^u\|^2  -\EE_{\xi^t}\left[\|\omega_t^u\|^2\right] \right| 
        &\overset{\eqref{eq:omega_magnitude_str_mon}}{\leq}& \frac{16\gamma^2 \exp(-\gamma\mu T) \lambda_l^2}{\exp(-\gamma\mu (1+t))}\notag\\
        &\overset{\eqref{eq:lambda_SGDA_str_mon}}{\leq}& \frac{\exp(-\gamma\mu T)R^2}{5\ln\tfrac{4(K+1)}{\beta}}\notag\\
        &\eqdef& c. \label{eq:SGDA_str_mon_technical_4_1}
    \end{eqnarray}
    The summands also have conditional variances
    \begin{equation*}
        \widetilde\sigma_t^2 \eqdef \EE_{\xi^t}\left[4\gamma^4 (1-\gamma\mu)^{2T-2-2t} \left|\|\omega_t^u\|^2  -\EE_{\xi^t}\left[\|\omega_t^u\|^2\right] \right|^2 \right]
    \end{equation*}
    that are bounded
    \begin{eqnarray}
        \widetilde\sigma_t^2 &\overset{\eqref{eq:SGDA_str_mon_technical_4_1}}{\leq}& \frac{2\gamma^2\exp(-2\gamma\mu T)R^2}{5\exp(-\gamma\mu (1+t))\ln\tfrac{4(K+1)}{\beta}} \EE_{\xi^t}\left[ \left|\|\omega_t^u\|^2  -\EE_{\xi^t}\left[\|\omega_t^u\|^2\right] \right|\right]\notag\\
        &\leq& \frac{4\gamma^2\exp(-2\gamma\mu T)R^2}{5\exp(-\gamma\mu (1+t))\ln\tfrac{4(K+1)}{\beta}} \EE_{\xi^t}\left[\|\omega_t^u\|^2\right]. \label{eq:SGDA_str_mon_technical_4_2}
    \end{eqnarray}
    In other words, we showed that $\left\{2\gamma^2 (1-\gamma\mu)^{T-1-t}\left( \|\omega_t^u\|^2 -\EE_{\xi^t}\left[\|\omega_t^u\|^2\right]\right)\right\}_{t = 0}^{T-1}$ is a bounded martingale difference sequence with bounded conditional variances $\{\widetilde\sigma_t^2\}_{t = 0}^{T-1}$. Next, we apply Bernstein's inequality (Lemma~\ref{lem:Bernstein_ineq}) with $X_t = 2\gamma^2 (1-\gamma\mu)^{T-1-t}\left( \|\omega_t^u\|^2 -\EE_{\xi^t}\left[\|\omega_t^u\|^2\right]\right)$, parameter $c$ as in \eqref{eq:SGDA_str_mon_technical_4_1}, $b = \tfrac{1}{5}\exp(-\gamma\mu T) R^2$, $G = \tfrac{\exp(-2 \gamma\mu T) R^4}{300\ln\frac{4(K+1)}{\beta}}$:
    \begin{equation*}
        \PP\left\{|\circledFour| > \frac{1}{5}\exp(-\gamma\mu T) R^2 \text{ and } \sum\limits_{l=0}^{T-1}\widetilde\sigma_t^2 \leq \frac{\exp(-2\gamma\mu T) R^4}{294\ln\frac{4(K+1)}{\beta}}\right\} \leq 2\exp\left(- \frac{b^2}{2G + \nicefrac{2cb}{3}}\right) = \frac{\beta}{2(K+1)}.
    \end{equation*}
    Equivalently, we have
    \begin{equation}
        \PP\{E_{\circledFour}\} \geq 1 - \frac{\beta}{2(K+1)},\quad \text{for}\quad E_{\circledFour} = \left\{\text{either} \quad \sum\limits_{t=0}^{T-1}\widetilde\sigma_t^2 > \frac{\exp(-2\gamma\mu T) R^4}{300\ln\tfrac{4(K+1)}{\beta}}\quad \text{or}\quad |\circledFour| \leq \frac{1}{5}\exp(-\gamma\mu T) R^2\right\}. \label{eq:bound_4_SGDA_str_mon}
    \end{equation}
    In addition, $E_{T-1}$ implies that
    \begin{eqnarray}
        \sum\limits_{l=0}^{T-1}\widetilde\sigma_t^2 &\overset{\eqref{eq:SGDA_str_mon_technical_4_2}}{\leq}& \frac{4\gamma^2\exp(-\gamma\mu (2T-1))R^2}{5\ln\tfrac{4(K+1)}{\beta}} \sum\limits_{t=0}^{T-1} \frac{\EE_{\xi^t}\left[\|\omega_l^u\|^2\right]}{\exp(-\gamma\mu t)}\notag\\ &\overset{\eqref{eq:variance_omega_str_mon}, T \leq K+1}{\leq}& \frac{72\gamma^2\exp(-\gamma\mu (2T-1)) R^2 \sigma^\alpha}{5\ln\tfrac{4(K+1)}{\beta}} \sum\limits_{t=0}^{K} \frac{\lambda_t^{2-\alpha}}{\exp(-\gamma\mu t)}\notag\\
        &\overset{\eqref{eq:lambda_SGDA_str_mon}}{\leq}& \frac{72\gamma^\alpha\exp(-\gamma\mu (2T-1)) R^{4-\alpha} \sigma^\alpha (K+1)\exp(\frac{\gamma\mu\alpha K}{2})}{5\cdot 120^{2-\alpha}\ln^{3-\alpha}\tfrac{4(K+1)}{\beta}} \notag\\
        &\overset{\eqref{eq:gamma_SGDA_str_mon}}{\leq}& \frac{\exp(-2\gamma\mu T)R^4}{300\ln\tfrac{4(K+1)}{\beta}}. \label{eq:bound_4_variances_SGDA_str_mon}
    \end{eqnarray}

    \paragraph{Upper bound for $\circledFive$.} From $E_{T-1}$ it follows that
    \begin{eqnarray}
        \circledFive &=&  2\gamma^2 \sum\limits_{l=0}^{T-1} \exp(-\gamma\mu (T-1-t)) \left(\|\omega_t^b\|^2\right)\notag\\
        &\overset{\eqref{eq:bias_omega_str_mon}}{\leq}& 2\cdot 2^{2\alpha}\gamma^2 \exp(-\gamma\mu (T-1)) \sigma^{2\alpha} \sum\limits_{t=0}^{T-1} \frac{1}{\lambda_t^{2\alpha-2} \exp(-\gamma\mu t)} \notag\\
        &\overset{\eqref{eq:lambda_SGDA_str_mon}, T \leq K+1}{\leq}& \frac{2\cdot 2^{2\alpha}\cdot 120^{2\alpha-2}\gamma^{2\alpha} \exp(-\gamma\mu (T-1)) \sigma^{2\alpha} \ln^{2\alpha-2}\tfrac{4(K+1)}{\beta}}{R^{2\alpha-2}} \sum\limits_{t=0}^{K} \exp\left(\gamma\mu(2\alpha-2)\left(1 + \frac{t}{2}\right)\right)\exp(\gamma\mu t)\notag\\
        &\leq& \frac{4\cdot 2^{2\alpha}\cdot 120^{2\alpha-2}\gamma^{2\alpha} \exp(-\gamma\mu (T-3)) \sigma^{2\alpha} \ln^{2\alpha-2}\tfrac{4(K+1)}{\beta}}{R^{2\alpha-2}} \sum\limits_{t=0}^{K} \exp(\gamma\mu \alpha t)\notag\\
        &\leq& \frac{4\cdot 2^{2\alpha}\cdot 120^{2\alpha-2}\gamma^{2\alpha} \exp(-\gamma\mu (T-3)) \sigma^{2\alpha} \ln^{2\alpha-2}\tfrac{4(K+1)}{\beta} (K+1) \exp(\gamma\mu \alpha K)}{R^{2\alpha-2}}\notag\\
        &\overset{\eqref{eq:gamma_SGDA_str_mon}}{\leq}& \frac{1}{5}\exp(-\gamma\mu T) R^2. \label{eq:bound_5_SGDA_str_mon}
    \end{eqnarray}

    Now, we have the upper bounds for  $\circledOne, \circledTwo, \circledThree, \circledFour, \circledFive$. In particular, probability event $E_{T-1}$ implies
    \begin{gather*}
        R_T^2 \overset{\eqref{eq:SGDA_str_mon_12345_bound}}{\leq} \exp(-\gamma\mu T) R^2 + \circledOne + \circledTwo + \circledThree + \circledFour + \circledFive ,\\
        \circledTwo \overset{\eqref{eq:bound_2_SGDA_str_mon}}{\leq} \frac{1}{5}\exp(-\gamma\mu T)R^2,\quad \circledThree \overset{\eqref{eq:bound_3_SGDA_str_mon}}{\leq} \frac{1}{5}\exp(-\gamma\mu T)R^2,\\ \circledFive \overset{\eqref{eq:bound_5_SGDA_str_mon}}{\leq} \frac{1}{5}\exp(-\gamma\mu T)R^2\\
        \sum\limits_{t=0}^{T-1}\sigma_t^2 \overset{\eqref{eq:bound_1_variances_SGDA_str_mon}}{\leq}  \frac{\exp(-2\gamma\mu T)R^4}{300\ln\tfrac{4(K+1)}{\beta}},\quad \sum\limits_{t=0}^{T-1}\widetilde\sigma_t^2 \overset{\eqref{eq:bound_4_variances_SGDA_str_mon}}{\leq} \frac{\exp(-2\gamma\mu T)R^4}{300\ln\tfrac{4(K+1)}{\beta}}.
    \end{gather*}
     Moreover, we also have (see \eqref{eq:bound_1_SGDA_str_mon}, \eqref{eq:bound_4_SGDA_str_mon} and our induction assumption)
     \begin{gather*}
        \PP\{E_{T-1}\} \geq 1 - \frac{(T-1)\beta}{K+1},\\
        \PP\{E_{\circledOne}\} \geq 1 - \frac{\beta}{2(K+1)}, \quad \PP\{E_{\circledFour}\} \geq 1 - \frac{\beta}{2(K+1)}.
    \end{gather*}
    where
    \begin{eqnarray}
        E_{\circledOne}&=&  \left\{\text{either} \quad \sum\limits_{t=0}^{T-1}\sigma_t^2 > \frac{\exp(- 2\gamma\mu T) R^4}{300\ln\tfrac{4(K+1)}{\beta}}\quad \text{or}\quad |\circledOne| \leq \frac{1}{5}\exp(-\gamma\mu T) R^2\right\},\notag\\
        E_{\circledFour}&=& \left\{\text{either} \quad \sum\limits_{t=0}^{T-1}\widetilde\sigma_t^2 > \frac{\exp(-2\gamma\mu T) R^4}{300\ln\tfrac{4(K+1)}{\beta}}\quad \text{or}\quad |\circledFour| \leq \frac{1}{5}\exp(-\gamma\mu T) R^2\right\}.\notag
    \end{eqnarray}
     Thus, probability event $E_{T-1} \cap E_{\circledOne} \cap E_{\circledFour} $ implies
    \begin{eqnarray*}
        R_T^2 &\overset{\eqref{eq:SGDA_str_mon_12345_bound}}{\leq}& \exp(-\gamma\mu T) R^2 + \circledOne + \circledTwo + \circledThree + \circledFour + \circledFive\\
        &\leq& 2\exp(-\gamma\mu T) R^2,
    \end{eqnarray*}
    which is equivalent to \eqref{eq:induction_inequality_str_mon_SGDA} for $t = T$, and
    \begin{equation}
        \PP\{E_T\} \geq \PP\{E_{T-1} \cap E_{\circledOne} \cap E_{\circledFour} \} = 1 - \PP\{\overline{E}_{T-1} \cup \overline{E}_{\circledOne} \cup \overline{E}_{\circledFour} \} \geq 1 - \frac{T\beta}{K+1}. \notag
    \end{equation}
    This finishes the inductive part of our proof, i.e., for all $k = 0,1,\ldots,K+1$ we have $\PP\{E_k\} \geq 1 - \nicefrac{k\beta}{(K+1)}$. In particular, for $k = K+1$ we have that with probability at least $1 - \beta$
    \begin{equation}
        \|x^{K+1} - x^*\|^2 \leq 2\exp(-\gamma\mu (K+1))R^2. \notag
    \end{equation}
    Finally, if 
    \begin{eqnarray*}
        \gamma &=& \min\left\{\frac{1}{400 \ell \ln \tfrac{4(K+1)}{\beta}}, \frac{\ln(B_K)}{\mu(K+1)}\right\}, \notag\\
        B_K &=& \max\left\{2, \frac{(K+1)^{\frac{2(\alpha-1)}{\alpha}}\mu^2R^2}{5400^{\frac{2}{\alpha}}\sigma^2\ln^{\frac{2(\alpha-1)}{\alpha}}\left(\frac{4(K+1)}{\beta}\right)\ln^2(B_K)} \right\}  \\
        &=& \cO\left(\max\left\{2, \frac{K^{\frac{2(\alpha-1)}{\alpha}}\mu^2R^2}{\sigma^2\ln^{\frac{2(\alpha-1)}{\alpha}}\left(\frac{K}{\beta}\right)\ln^2\left(\max\left\{2, \frac{K^{\frac{2(\alpha-1)}{\alpha}}\mu^2R^2}{\sigma^2\ln^{\frac{2(\alpha-1)}{\alpha}}\left(\frac{K}{\beta}\right)} \right\}\right)} \right\}\right\}
    \end{eqnarray*}
    then with probability at least $1-\beta$
    \begin{eqnarray*}
        \|x^{K+1} - x^*\|^2 &\leq& 2\exp(-\gamma\mu (K+1))R^2\\
        &=& 2R^2\max\left\{\exp\left(-\frac{\mu(K+1)}{400 \ell \ln \tfrac{4(K+1)}{\beta}}\right), \frac{1}{B_K} \right\}\\
        &=& \cO\left(\max\left\{R^2\exp\left(- \frac{\mu K}{\ell \ln \tfrac{K}{\beta}}\right), \frac{\sigma^2\ln^{\frac{2(\alpha-1)}{\alpha}}\left(\frac{K}{\beta}\right)\ln^2\left(\max\left\{2, \frac{K^{\frac{2(\alpha-1)}{\alpha}}\mu^2R^2}{\sigma^2\ln^{\frac{2(\alpha-1)}{\alpha}}\left(\frac{K}{\beta}\right)} \right\}\right)}{K^{\frac{2(\alpha-1)}{\alpha}}\mu^2}\right\}\right).
    \end{eqnarray*}
    To get $\|x^{K+1} - x^*\|^2 \leq \varepsilon$ with probability at least $1-\beta$ it is sufficient to choose $K$ such that both terms in the maximum above are $\cO(\varepsilon)$. This leads to
    \begin{equation*}
         K = \cO\left(\frac{\ell}{\mu}\ln\left(\frac{R^2}{\varepsilon}\right)\ln\left(\frac{\ell}{\mu \beta}\ln\frac{R^2}{\varepsilon}\right), \left(\frac{\sigma^2}{\mu^2\varepsilon}\right)^{\frac{\alpha}{2(\alpha-1)}}\ln \left(\frac{1}{\beta} \left(\frac{\sigma^2}{\mu^2\varepsilon}\right)^{\frac{\alpha}{2(\alpha-1)}}\right)\ln^{\frac{\alpha}{\alpha-1}}\left(B_\varepsilon\right)\right),
    \end{equation*}
    where
    \begin{equation*}
        B_\varepsilon = \max\left\{2, \frac{R^2}{\varepsilon \ln \left(\frac{1}{\beta} \left(\frac{\sigma^2}{\mu^2\varepsilon}\right)^{\frac{\alpha}{2(\alpha-1)}}\right)}\right\}.
    \end{equation*}
    This concludes the proof.
\end{proof}

\end{document}